\documentclass[12pt]{article}
\usepackage{amssymb,amsmath,amsthm,tikz,multirow, mathdots}
\usepackage{enumerate} 
\usepackage{hyperref} 
\usepackage [autostyle, english = british]{csquotes}
\usepackage [british]{babel}
\usepackage{csquotes}
\usepackage{scalerel, graphicx}
\usepackage{bm}
\usepackage{authblk}
\usepackage{wasysym, scalerel} 
\usepackage{subcaption}
\usepackage{textcomp}
\usepackage{adjustbox} 
\usepackage{hhline} 
\usepackage{soul} 

\usepackage{comment} 

\colorlet{HBlue}{teal!75!blue!95!white}
\colorlet{HOrange}{orange!80!black}

\usetikzlibrary{calc,arrows, arrows.meta, math} 
\usepackage{tikz,tikz-3dplot}

\newcommand{\quotes}[1]{``#1''} 

\newtheorem*{theorem*}{Theorem}
\newtheorem*{problem*}{Problem}
\newtheorem*{corollary*}{Corollary}

\theoremstyle{definition}
\newtheorem*{definition*}{Definition}

\theoremstyle{plain}
\newtheorem{thm}{Theorem}[section]
\newtheorem{lem}[thm]{Lemma}
\newtheorem{prop}[thm]{Proposition}

\theoremstyle{definition}

\theoremstyle{remark}

\newtheorem*{case*}{Case}
\newtheorem*{subcase*}{Subcase}
\newtheorem*{subsubcase*}{Subsubcase}

\numberwithin{equation}{section}

\newcommand{\al}{\alpha}

\providecommand{\keywords}[1]{\noindent \textit{Keywords:} #1}
\providecommand{\subject}[1]{\noindent \textit{Mathematics Subject Classification:} #1}
\title{Tiling the Sphere with Regular Polygons}
\author[1]{Hoi Ping Luk\thanks{hoi@connect.ust.hk. The research was supported in part by the funding of Academic Career in Pilsen 2024 under Plze\v{n}sk\'y kraj a Z\'apado\v{c}esk\'a univerzita v Plzni.}}
\author[2]{Roman Nedela\footnote{nedela@kma.zcu.cz.}}
\author[3]{Christopher Purcell\footnote{ccppurcell@gmail.com.}}
\affil[1,2,3]{Z\'{a}pado\v{c}esk\'a univerzita v Plzni} 

\begin{document}
\maketitle

\begin{abstract}
		We give a complete classification of edge-to-edge tilings of the
		sphere by regular polygons under a unified framework. Without
		assuming convexity of the tiles or polyhedrality of the underlying
		graph, our proof is independent of the Johnson-Zalgaller
		classification of solids with regular faces (1967), which took
		over 200 pages. We apply a blend of trigonometric, algebraic
		and combinatorial tools of independent interest. \\[2ex]
\keywords{Classification, Spherical tilings, Regular polygons, Division of spaces} \\

\subject{05B45, 52C20, 51M20, 52B10, 51M10}
\end{abstract}

\section{Introduction}

Highly regular geometric objects have been investigated since
antiquity, the Platonic and Archimedean solids being prominent
examples. The more general class of polyhedra having regular
faces was studied by Johnson \cite{J66}, Grunbaum \cite{gj}, Zalgaller
\cite{ZaR} and others in the second half of the 20th century. As well
as the five Platonic solids and thirteen Archimedean solids, this
class includes the prisms, antiprisms and $92$ additional polyhedra
called Johnson solids \cite{J66} and denoted by $J_1,\dots,J_{92}$. The
proof that the collection is complete was given by Zalgaller in 1967
taking 220 pages of the monograph \cite{ZaR,ZaE}.

In the present paper, we classify the spherical tilings\footnote{
We consider exclusively edge-to-edge tilings; a
complementary classification of the remaining tilings was given
recently in~\cite{aehj}.
}
whose tiles
are regular spherical polygons; such tilings can be viewed as
analogous to solids with regular faces. Crucially, our classification
does not depend on the aforementioned classification by Johnson and
Zalgaller. The following theorem is our main result; the tilings are
depicted in Figures~1--4.

\begin{theorem*} 
		The edge-to-edge spherical tilings by regular polygons are 
		\begin{itemize}
				\item the five Platonic tilings,
				\item the thirteen Archimedean tilings,
				\item the twenty-five tilings corresponding to circumscribable
						Johnson solids: 
						\begin{align*}
								J_1, J_2, J_3, J_4, J_5, J_6, J_{11}, J_{19}, J_{27},
								J_{34}, J_{37}, J_{62}, J_{63}, \\ 
								J_{72}, J_{73}, J_{74}, J_{75}, J_{76}, J_{77},
								J_{78}, J_{79}, J_{80}, J_{81}, J_{82}, J_{83},
						\end{align*} 
				\item the infinite families of prisms and antiprisms, 
				\item the infinite families of hosohedra and dihedra.
		\end{itemize} 
\end{theorem*}

\newif\ifdraft
\draftfalse
\ifdraft
FIXME
\else

\begin{figure}[h!]
\centering
\begin{tikzpicture}
\tikzmath{
\s=0.75;
\h=1.75;
\XS=2.15;
}
\node [inner sep=0] (image) at (0,0) 
            {\includegraphics[height=\h cm]{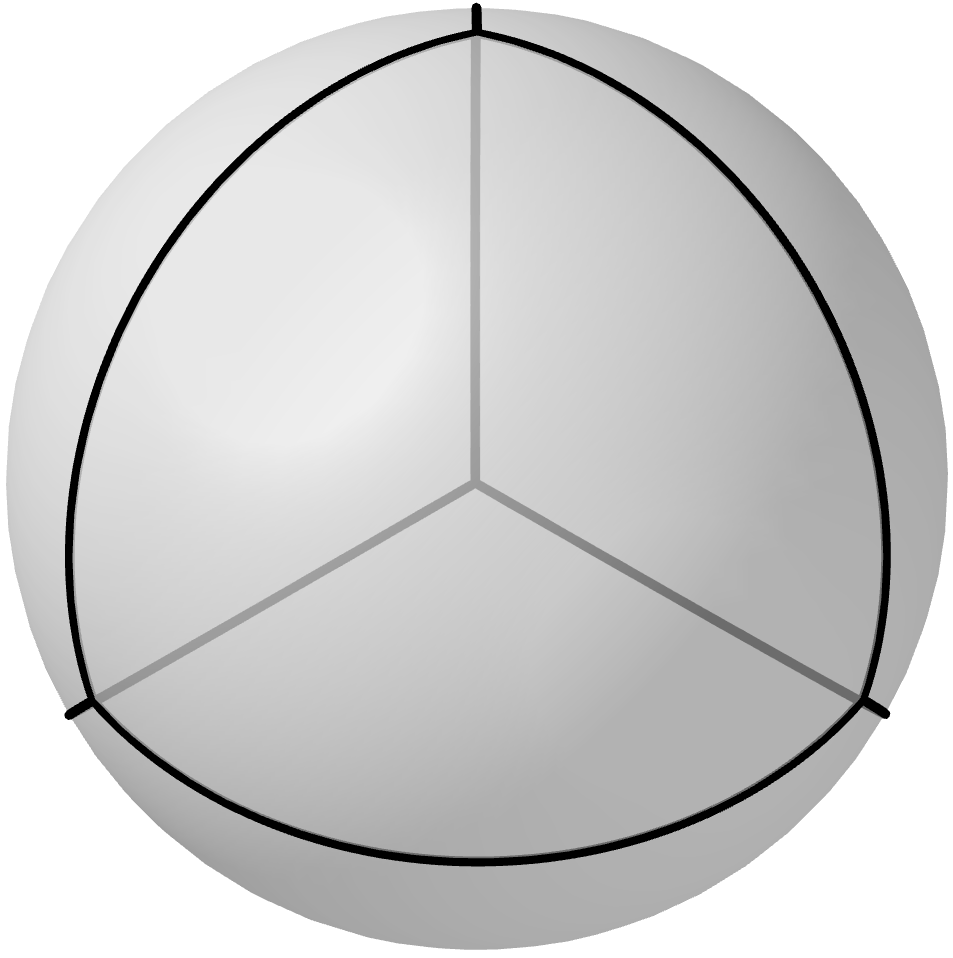}};

\node [inner sep=0] (image) at (\XS,0) 
            {\includegraphics[height=\h cm]{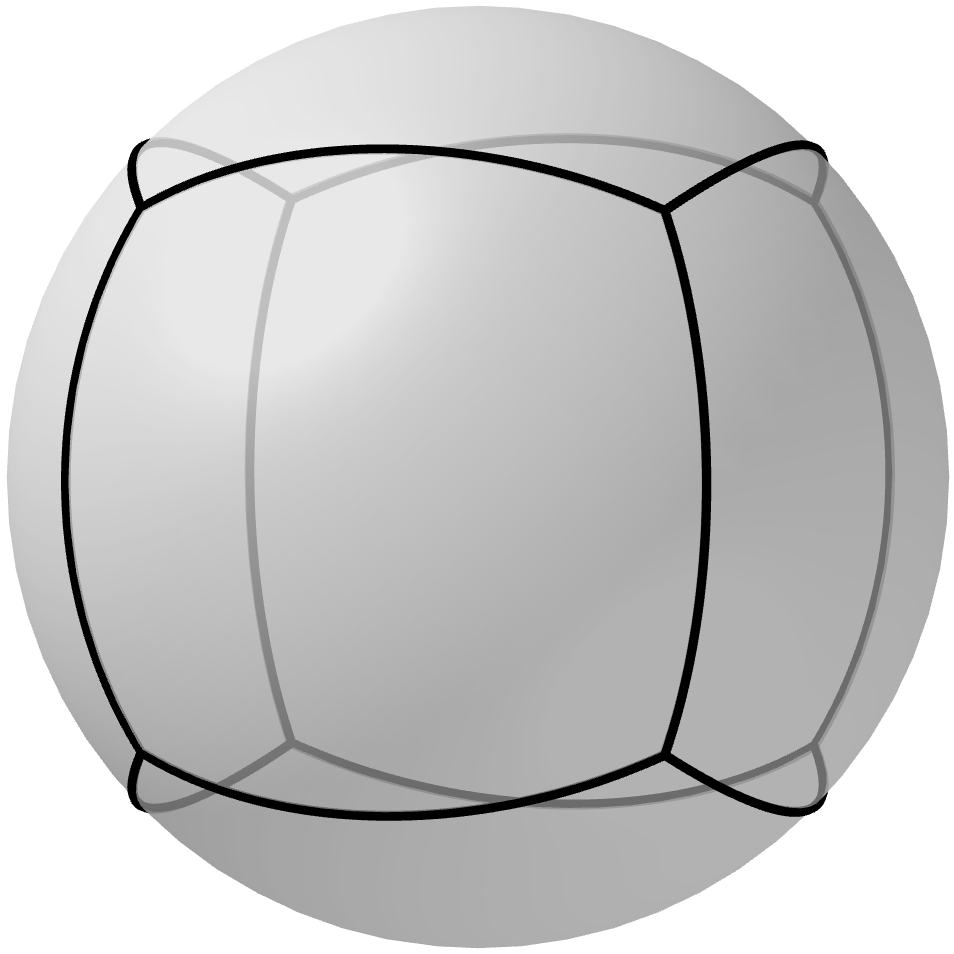}};

\node [inner sep=0] (image) at (2*\XS,0) 
            {\includegraphics[height=\h cm]{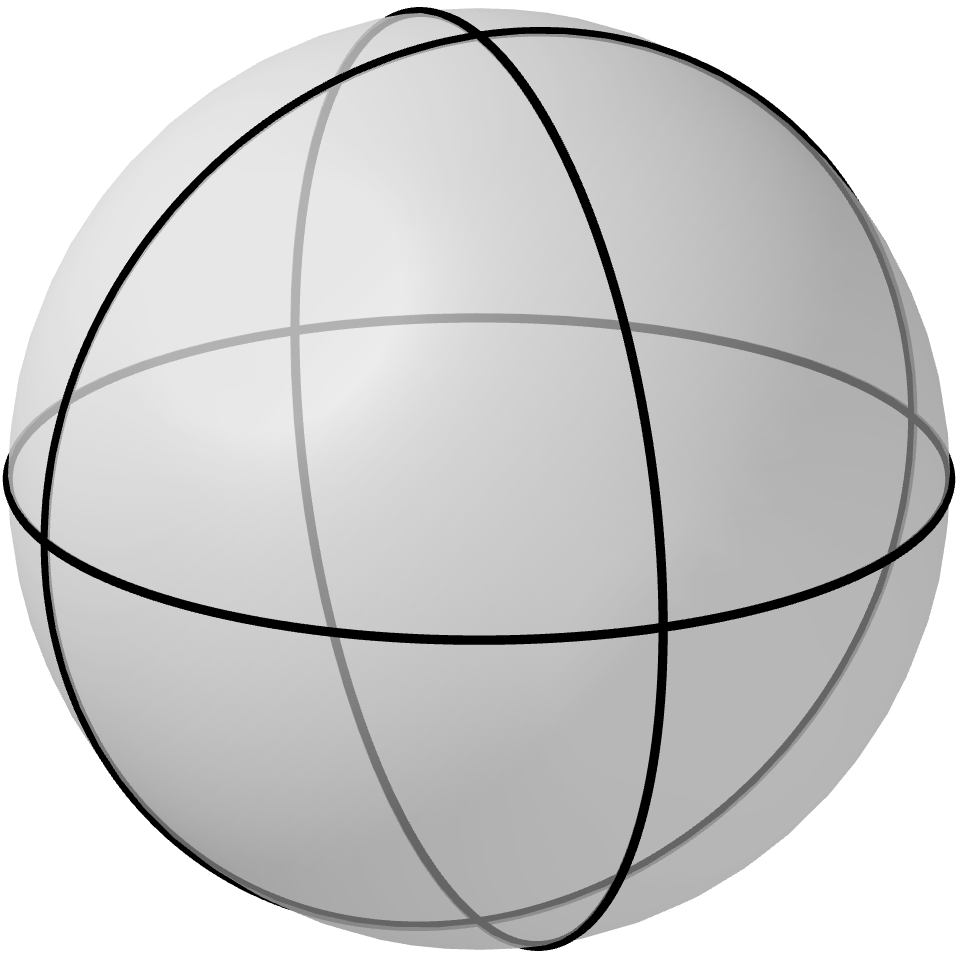}};

\node [inner sep=0] (image) at (3*\XS,0) 
            {\includegraphics[height=\h cm]{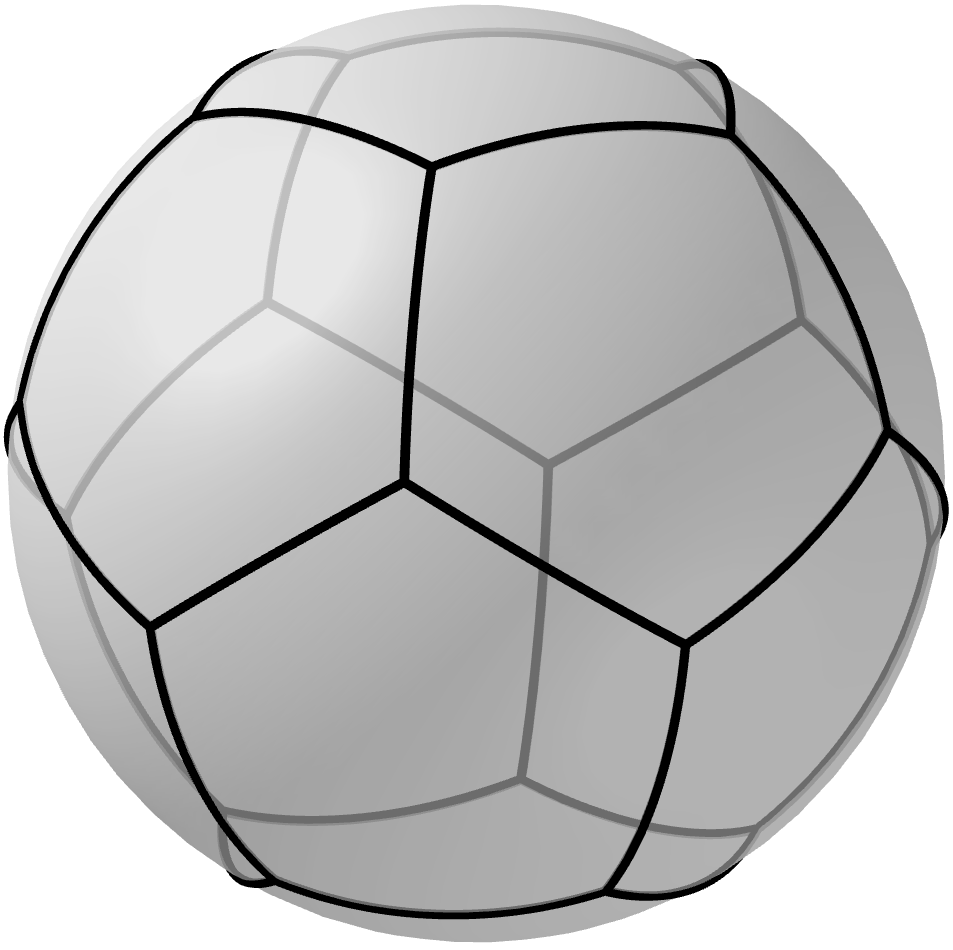}};

\node [inner sep=0] (image) at (4*\XS,0) 
            {\includegraphics[height=\h cm]{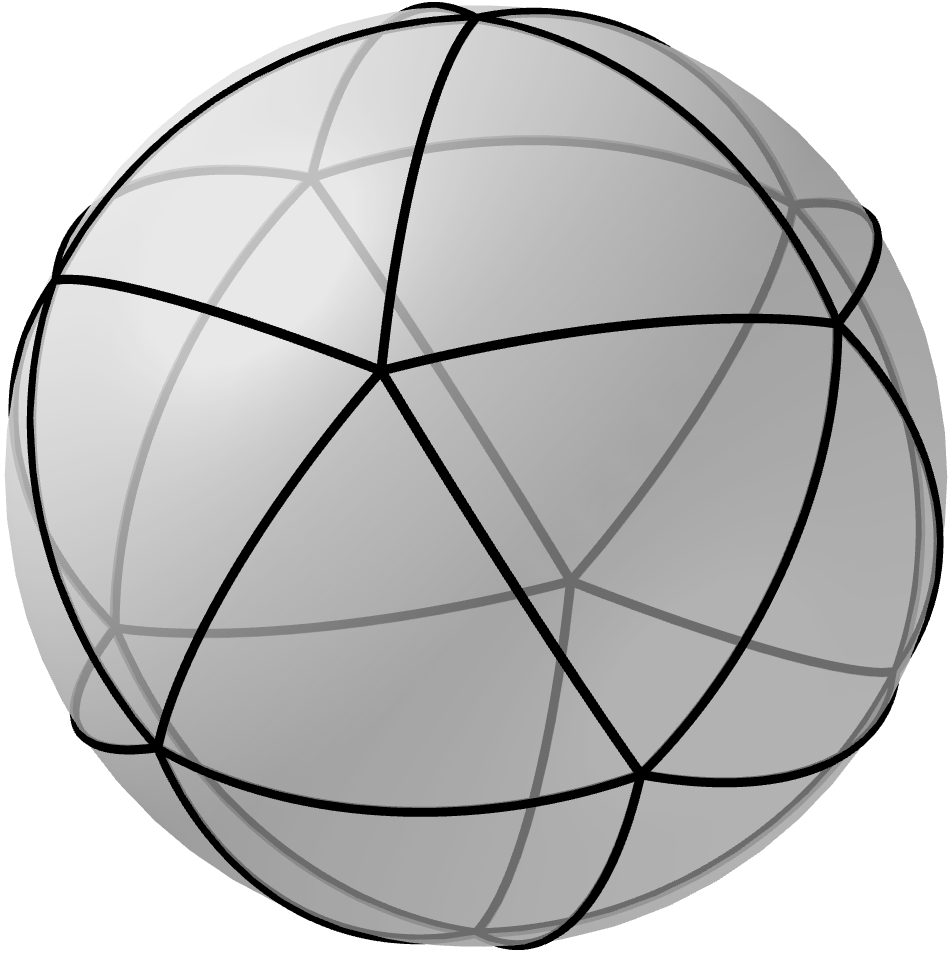}};
\end{tikzpicture}
\caption{Platonic}
\end{figure}


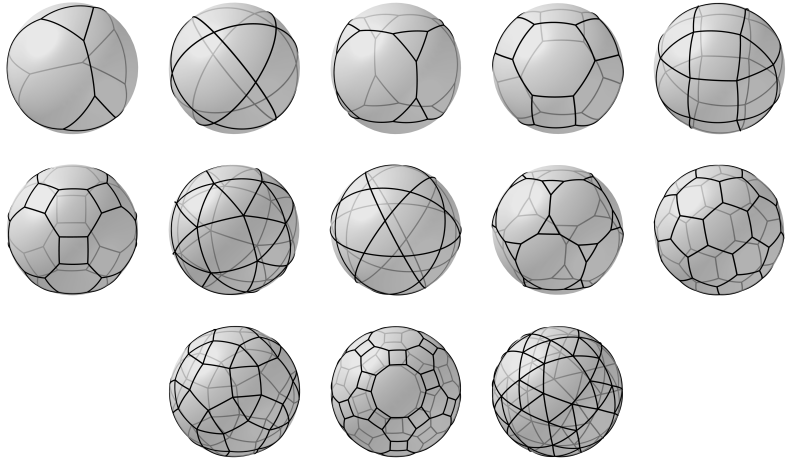
\begin{figure}[h!]
\centering
\begin{tikzpicture}
\tikzmath{
\s=0.75;
\h=1.75;
\XS=2.15;
\YS=2.15;
}
\begin{scope}[] 
\node [inner sep=0] (image) at (0,0) 
            {\includegraphics[height=\h cm]{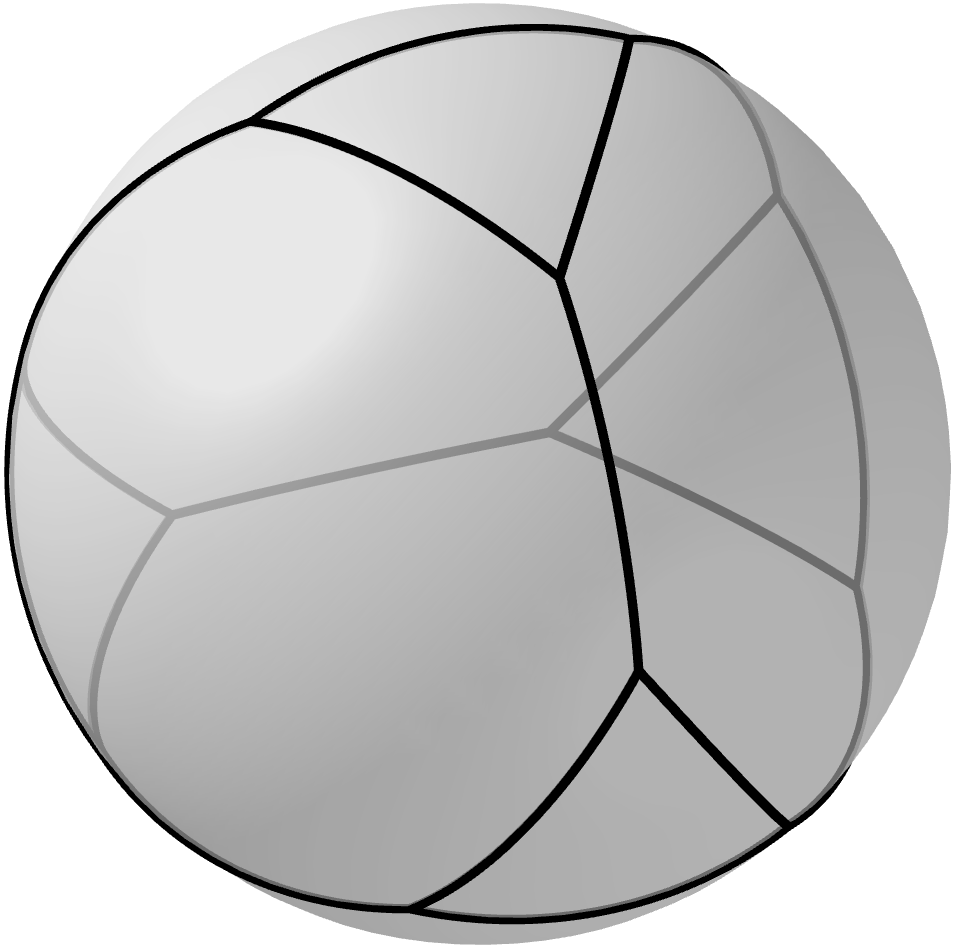}};
\node at (0,-2*\s) {\small };

\node [inner sep=0] (image) at (\XS,0) 
            {\includegraphics[height=\h cm]{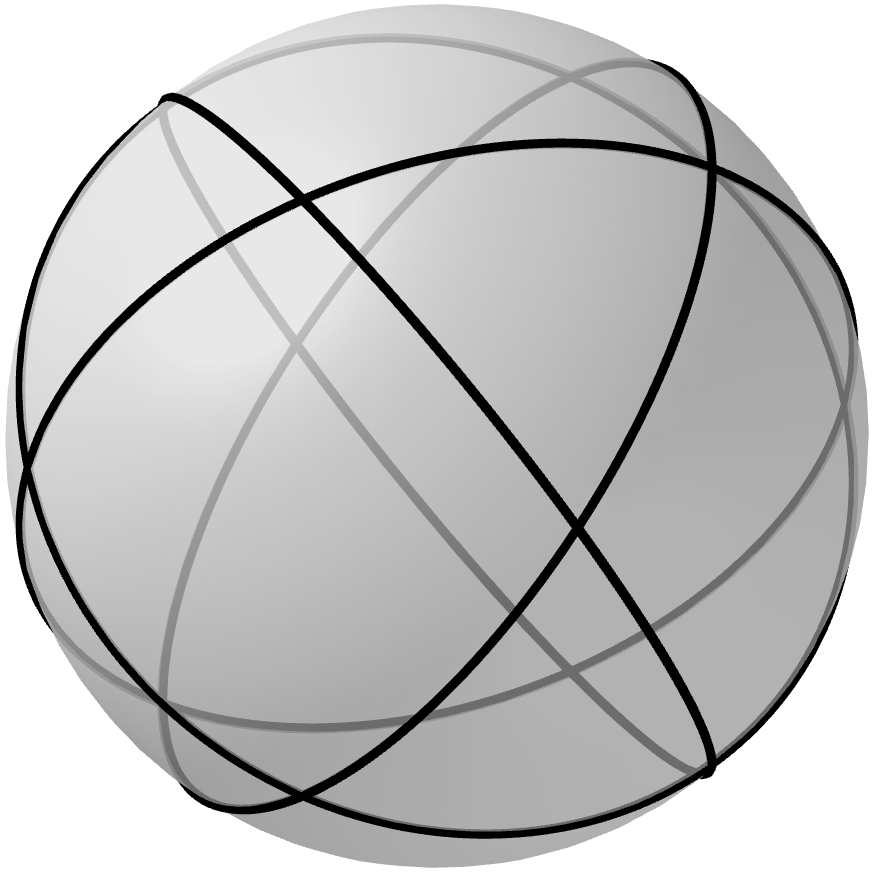}};
\node at (\XS,-2*\s) {\small };

\node [inner sep=0] (image) at (2*\XS,0) 
            {\includegraphics[height=\h cm]{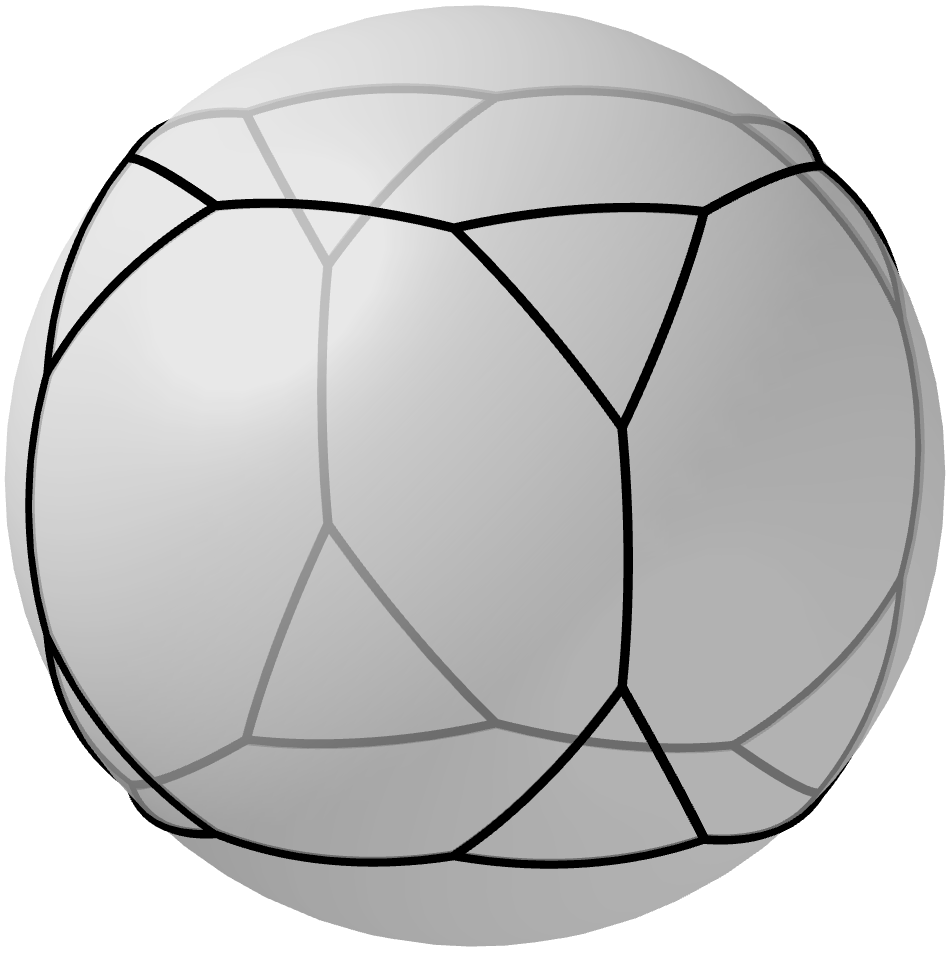}};
\node at (2*\XS,-2*\s) {\small };

\node [inner sep=0] (image) at (3*\XS,0) 
            {\includegraphics[height=\h cm]{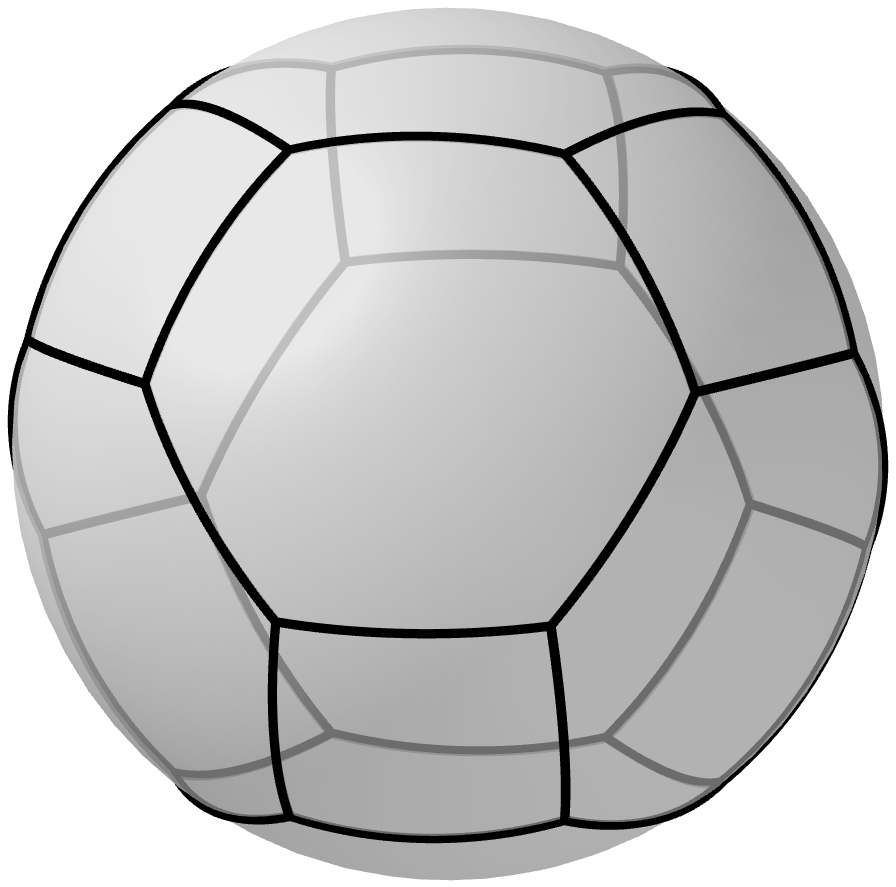}};
\node at (3*\XS,-2*\s) {\small };

\node [inner sep=0] (image) at (4*\XS,0) 
            {\includegraphics[height=\h cm]{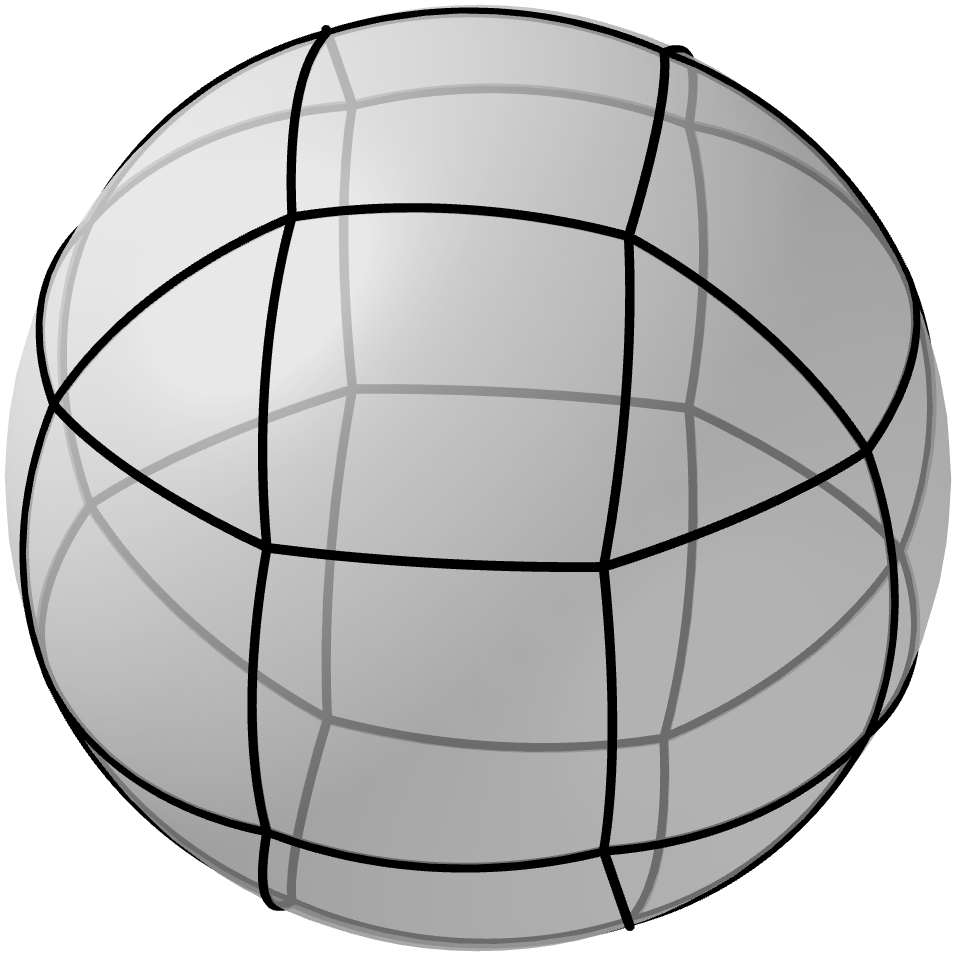}};
\node at (4*\XS,-2*\s) {\small };
\end{scope} 

\begin{scope}[yshift=-\YS cm] 
\node [inner sep=0] (image) at (0,0) 
            {\includegraphics[height=\h cm]{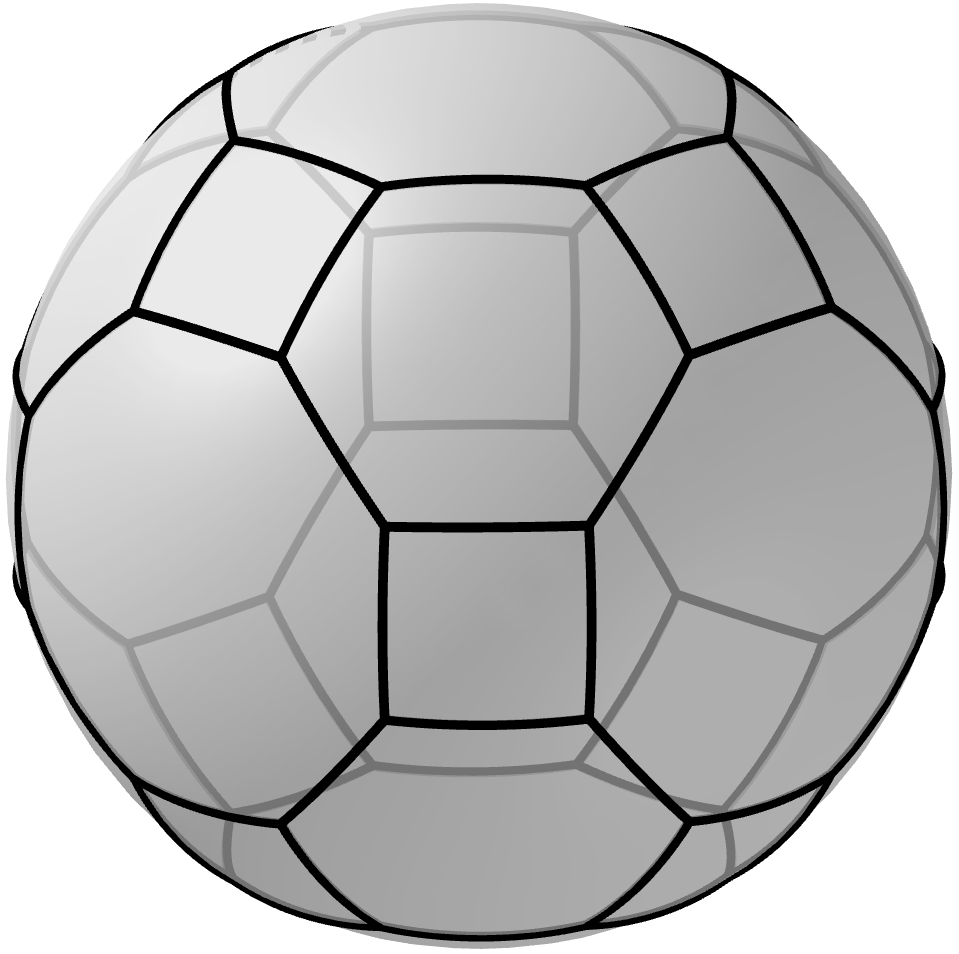}};
\node at (0,-2*\s) {\small };

\node [inner sep=0] (image) at (\XS,0) 
            {\includegraphics[height=\h cm]{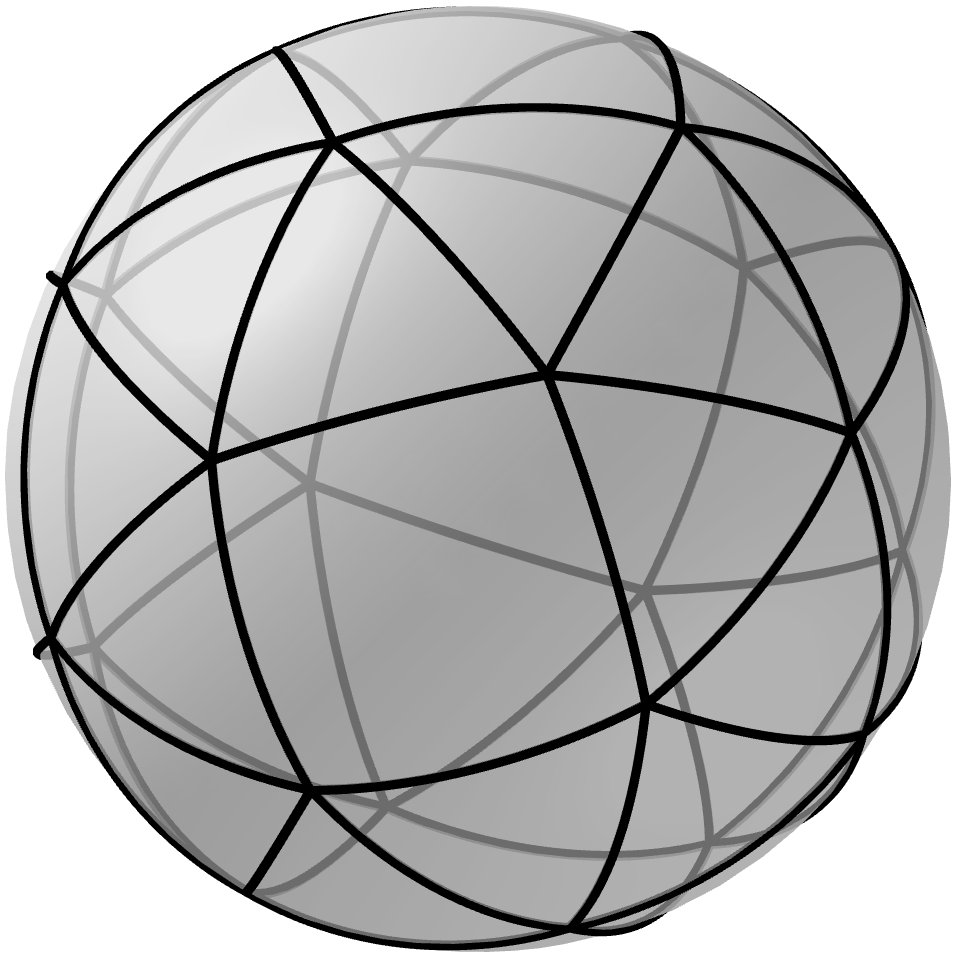}};
\node at (\XS,-2*\s) {\small };

\node [inner sep=0] (image) at (2*\XS,0) 
            {\includegraphics[height=\h cm]{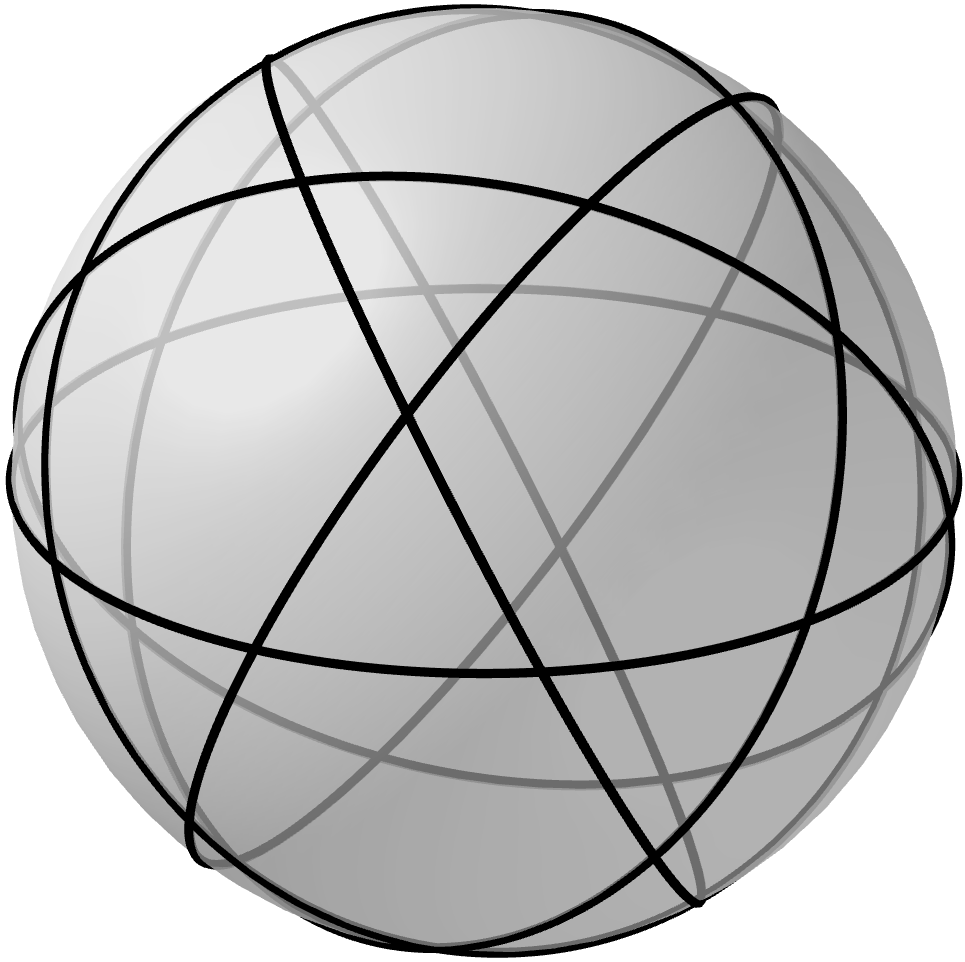}};
\node at (2*\XS,-2*\s) {\small };

\node [inner sep=0] (image) at (3*\XS,0) 
            {\includegraphics[height=\h cm]{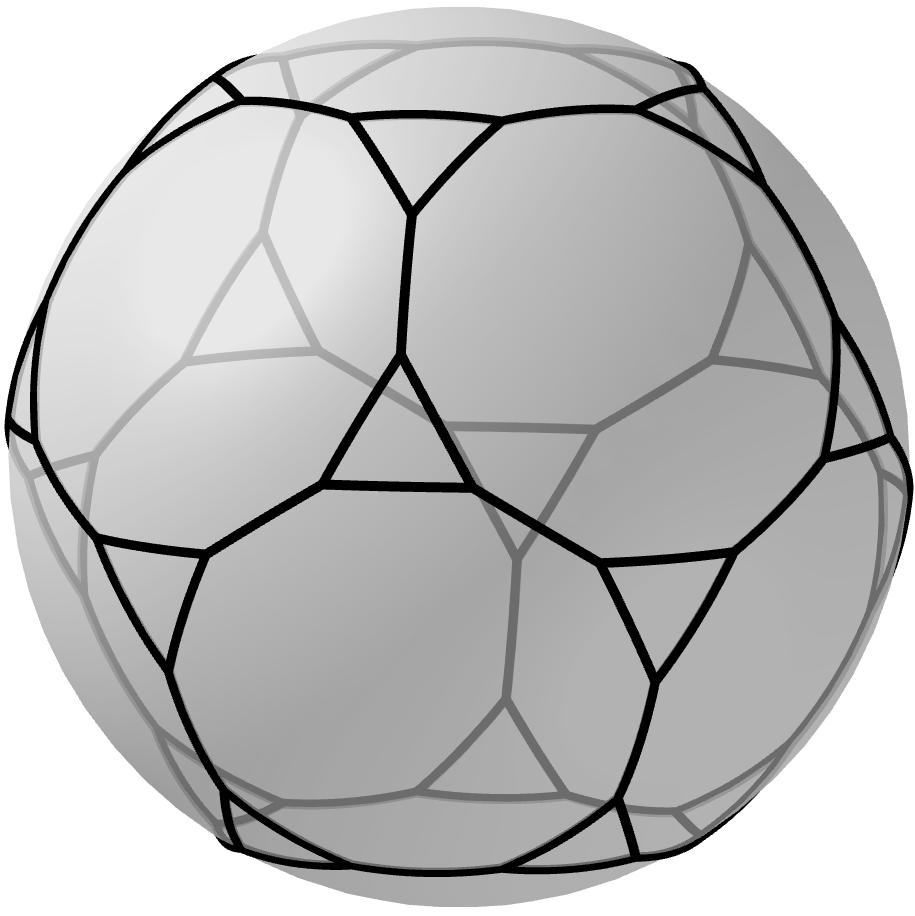}};
\node at (3*\XS,-2*\s) {\small };

\node [inner sep=0] (image) at (4*\XS,0) 
            {\includegraphics[height=\h cm]{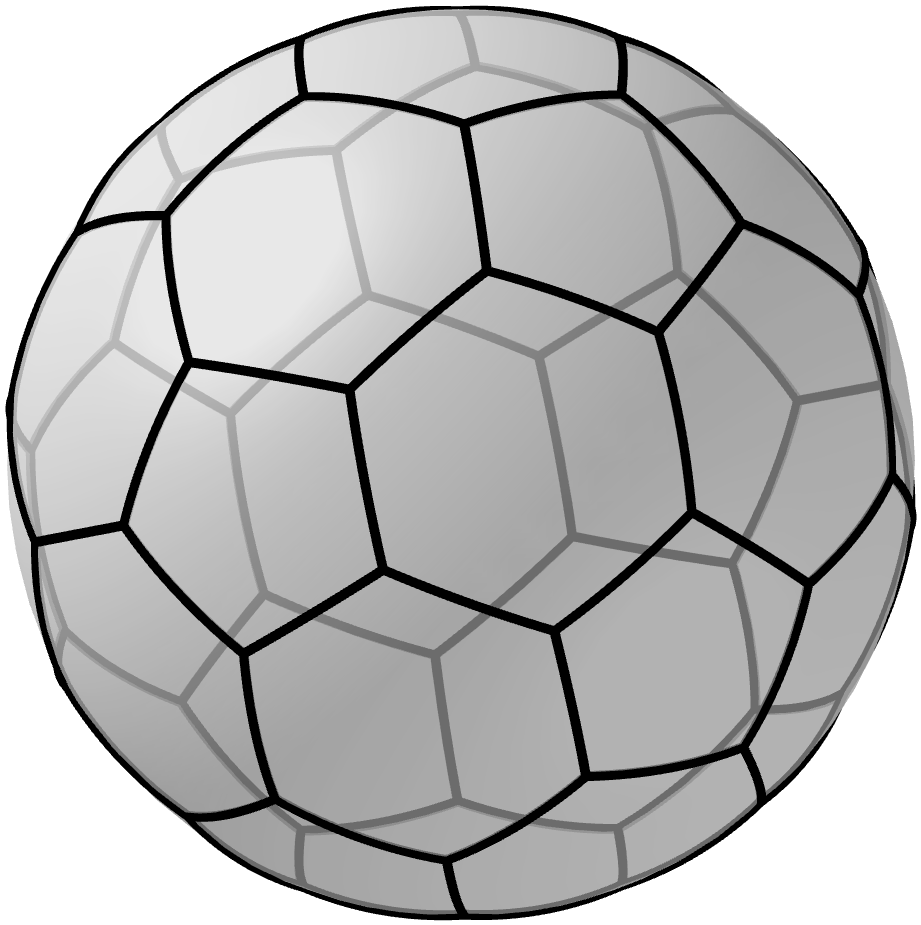}};
\node at (4*\XS,-2*\s) {\small };
\end{scope} 

\begin{scope}[xshift=\XS cm, yshift=-2*\YS cm] 
\node [inner sep=0] (image) at (0,0) 
            {\includegraphics[height=\h cm]{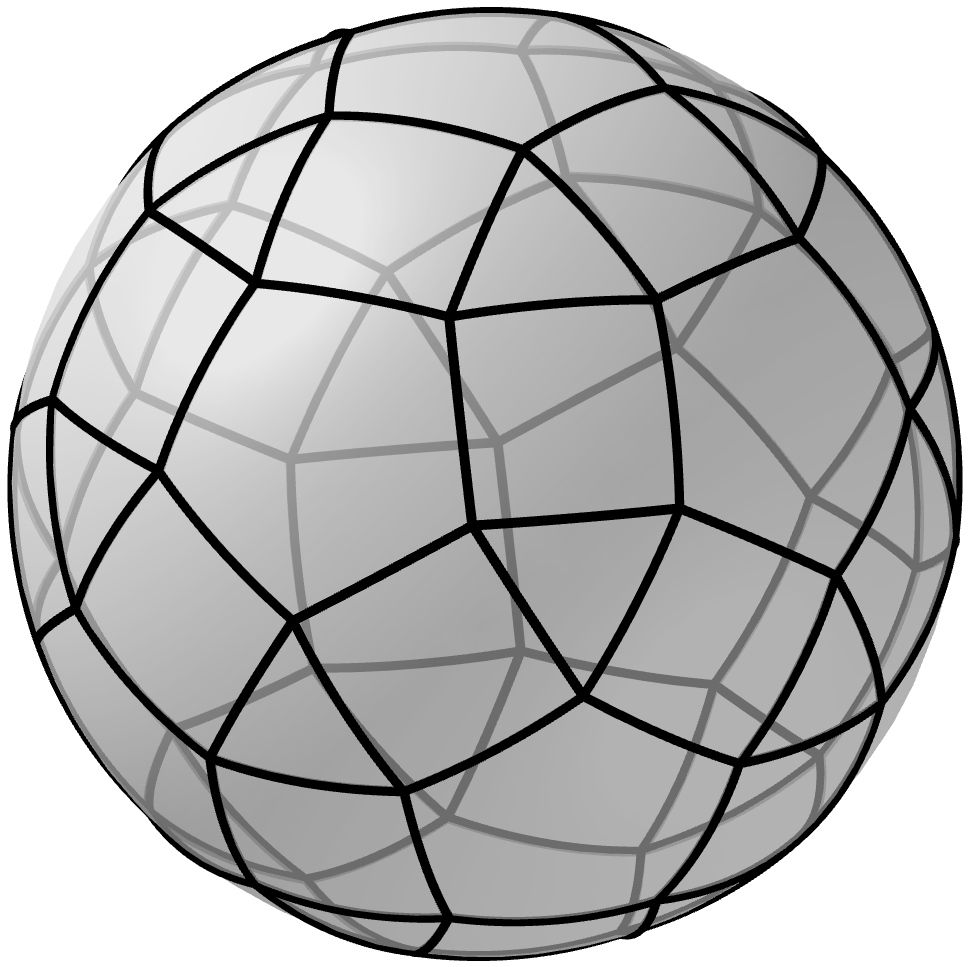}};
\node at (0,-2*\s) {\small };

\node [inner sep=0] (image) at (\XS,0) 
            {\includegraphics[height=\h cm]{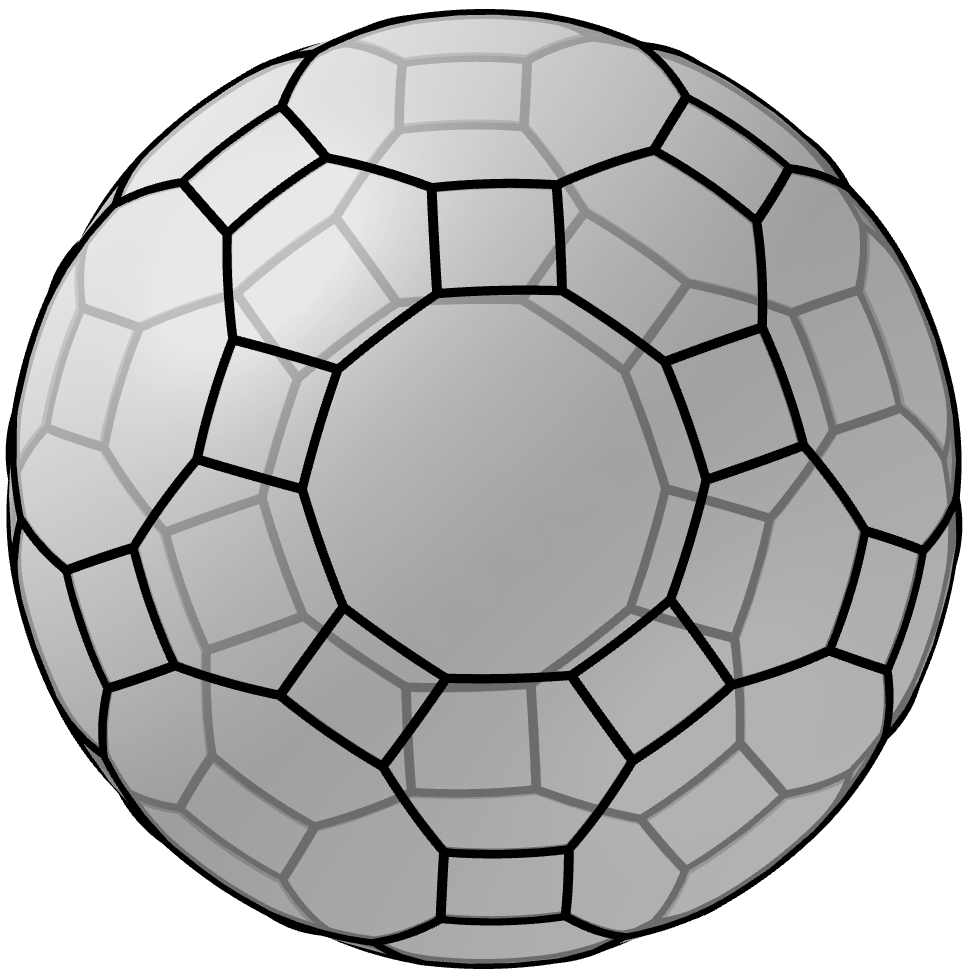}};
\node at (\XS,-2*\s) {\small };

\node [inner sep=0] (image) at (2*\XS,0) 
            {\includegraphics[height=\h cm]{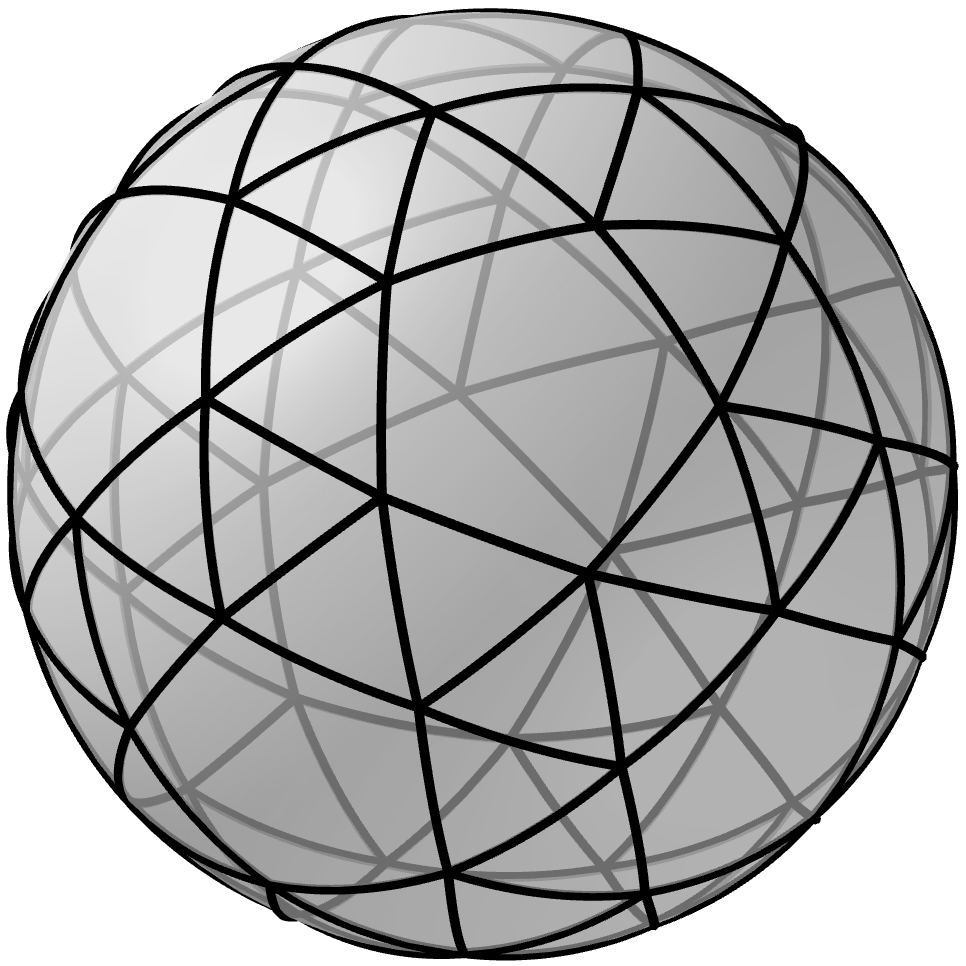}};
\node at (2*\XS,-2*\s) {\small };
\end{scope} 
\end{tikzpicture}
\caption{Archimedean}
\end{figure}


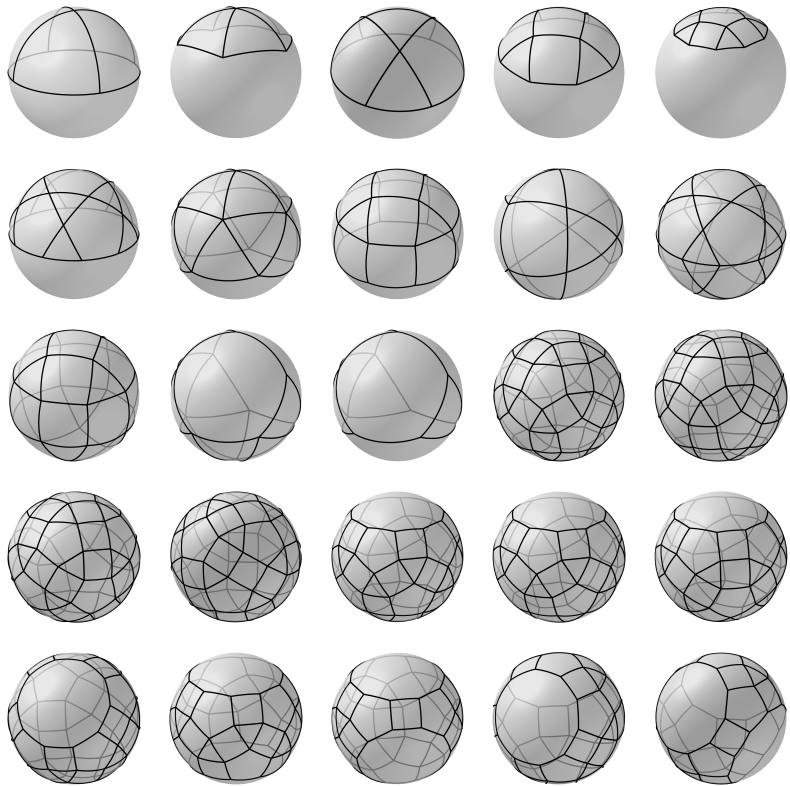
\begin{figure}[h!]
\centering
\begin{tikzpicture}
\tikzmath{
\s=0.75;
\h=1.75;
\XS=2.15;
\YS=2.15;
}
\begin{scope}[] 
\node [inner sep=0] (image) at (0,0) 
            {\includegraphics[height=\h cm]{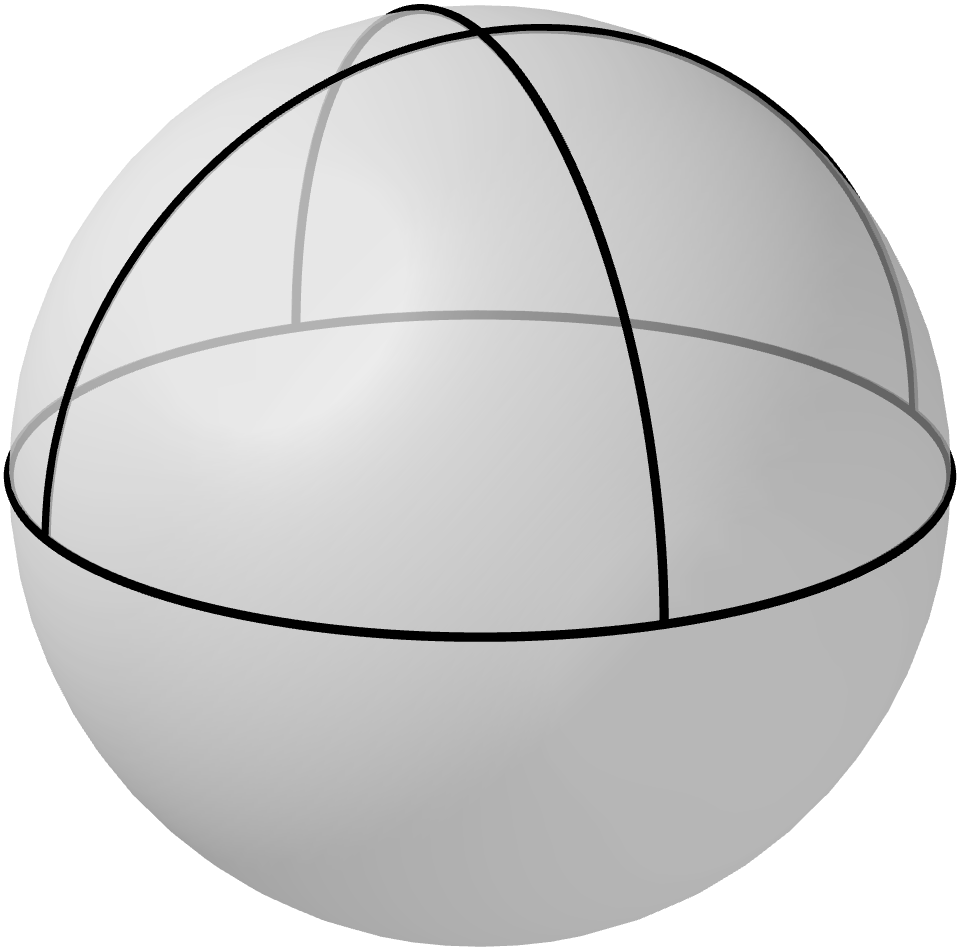}};
\node at (0,-2*\s) {\small };

\node [inner sep=0] (image) at (\XS,0) 
            {\includegraphics[height=\h cm]{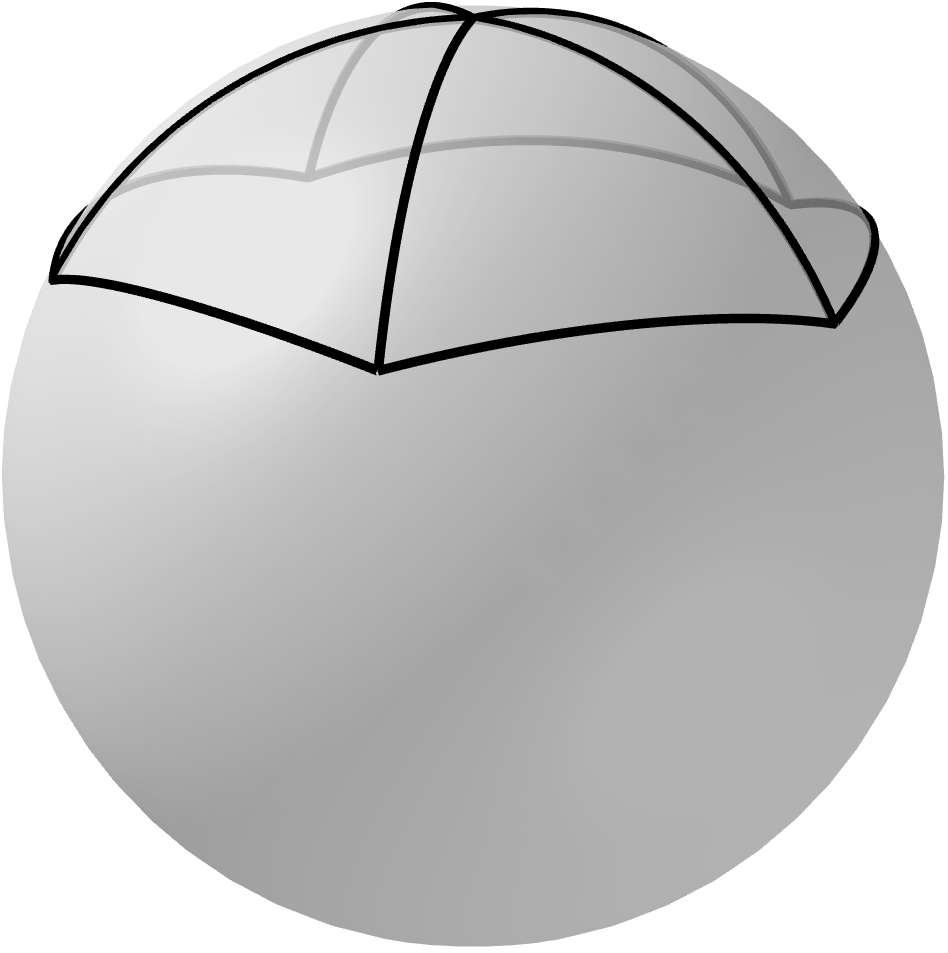}};
\node at (\XS,-2*\s) {\small };

\node [inner sep=0] (image) at (2*\XS,0) 
            {\includegraphics[height=\h cm]{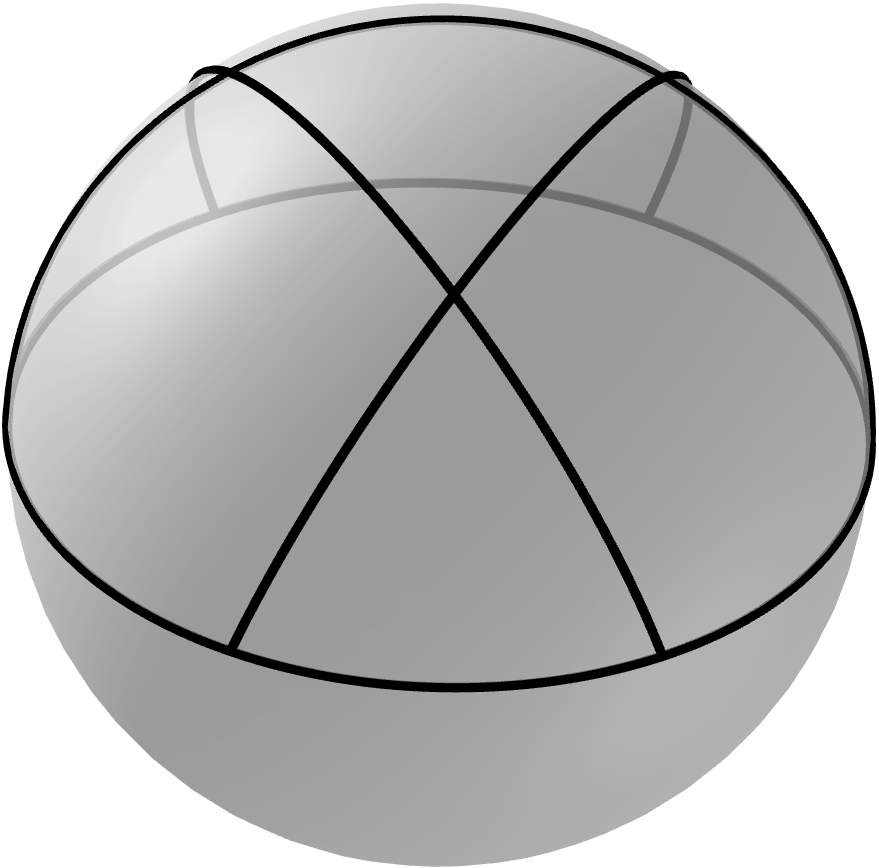}};
\node at (2*\XS,-2*\s) {\small };

\node [inner sep=0] (image) at (3*\XS,0) 
            {\includegraphics[height=\h cm]{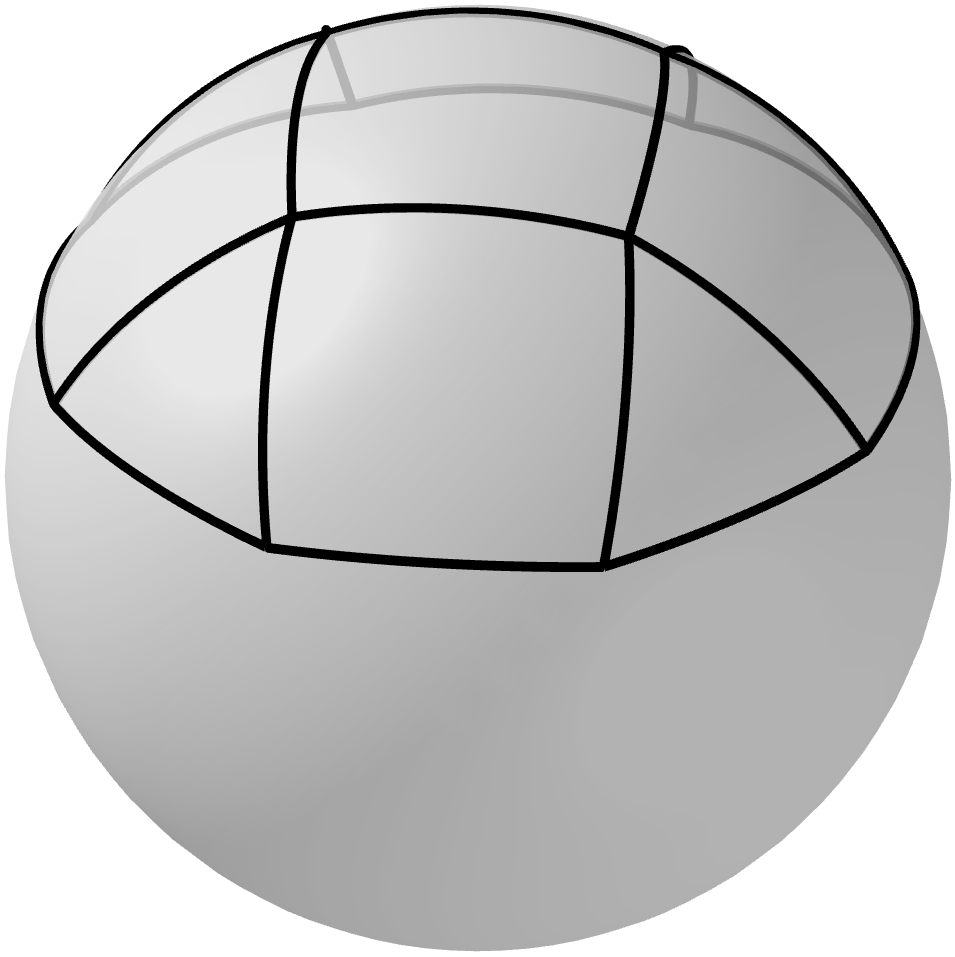}};
\node at (3*\XS,-2*\s) {\small };

\node [inner sep=0] (image) at (4*\XS,0) 
            {\includegraphics[height=\h cm]{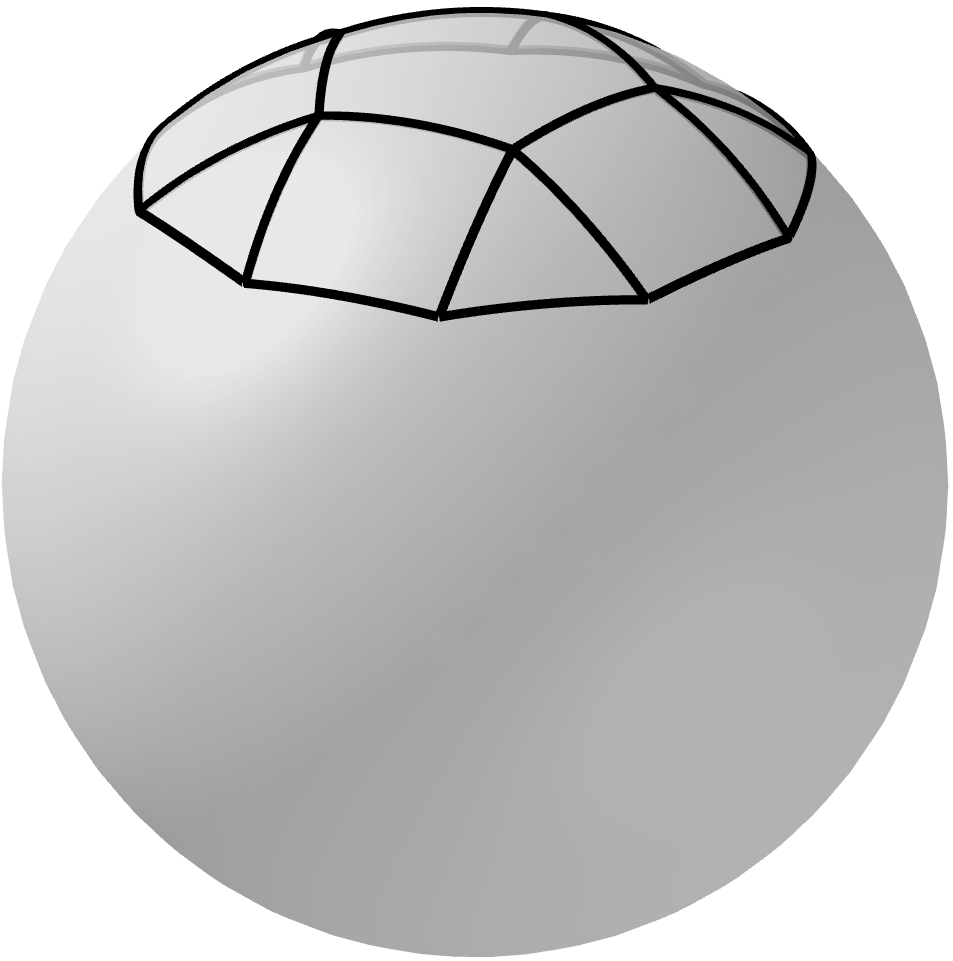}};
\node at (4*\XS,-2*\s) {\small };
\end{scope} 

\begin{scope}[yshift=-\YS cm] 
\node [inner sep=0] (image) at (0,0) 
            {\includegraphics[height=\h cm]{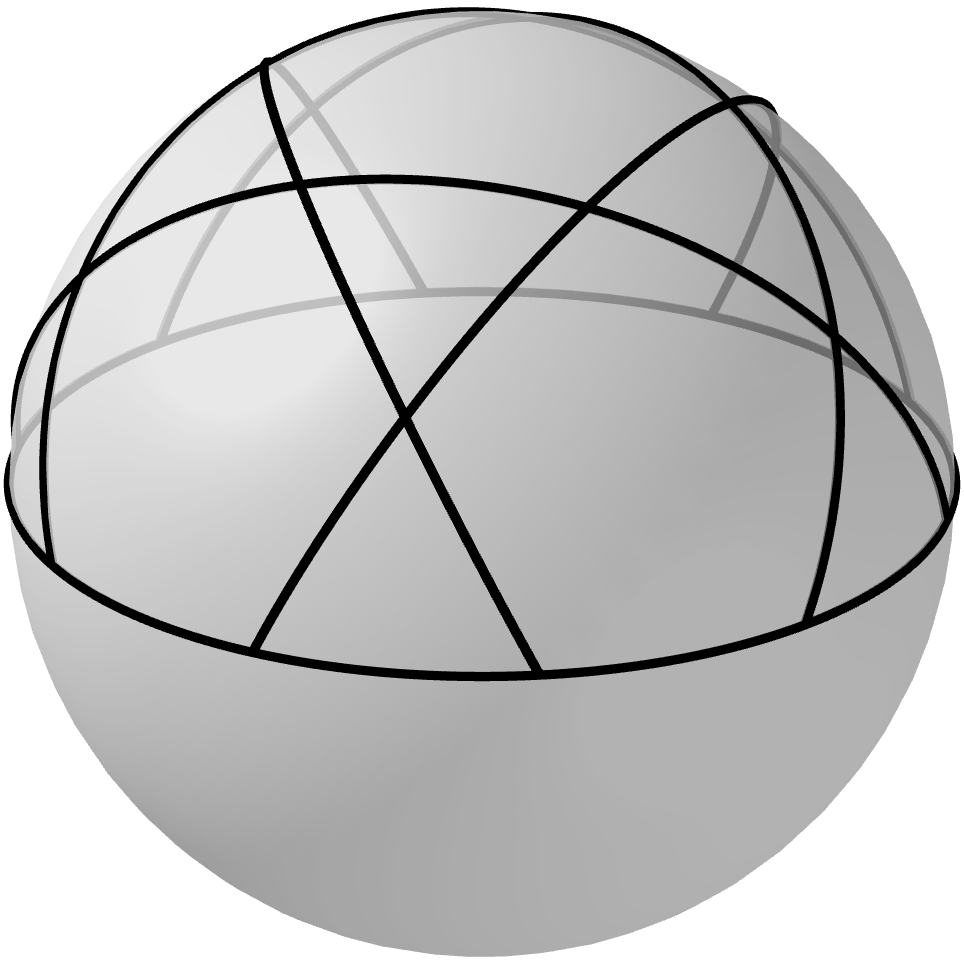}};
\node at (0,-2*\s) {\small };

\node [inner sep=0] (image) at (\XS,0) 
            {\includegraphics[height=\h cm]{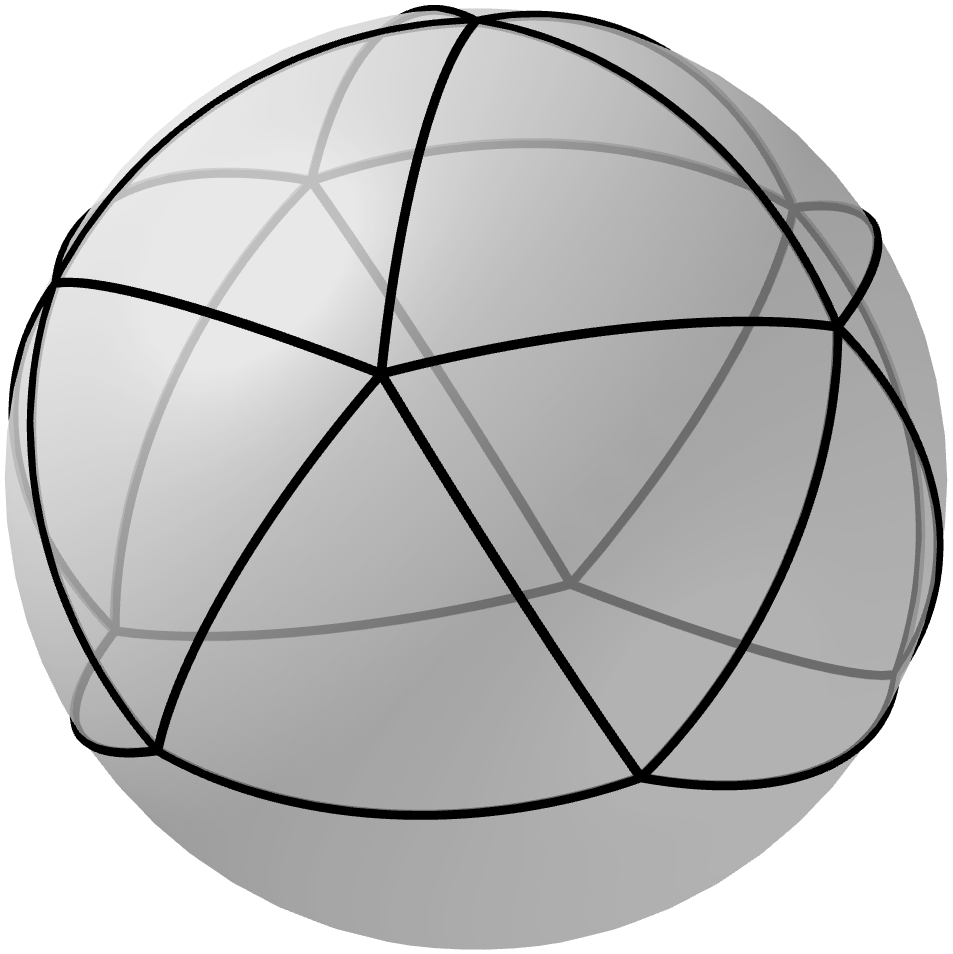}};
\node at (\XS,-2*\s) {\small };

\node [inner sep=0] (image) at (2*\XS,0) 
            {\includegraphics[height=\h cm]{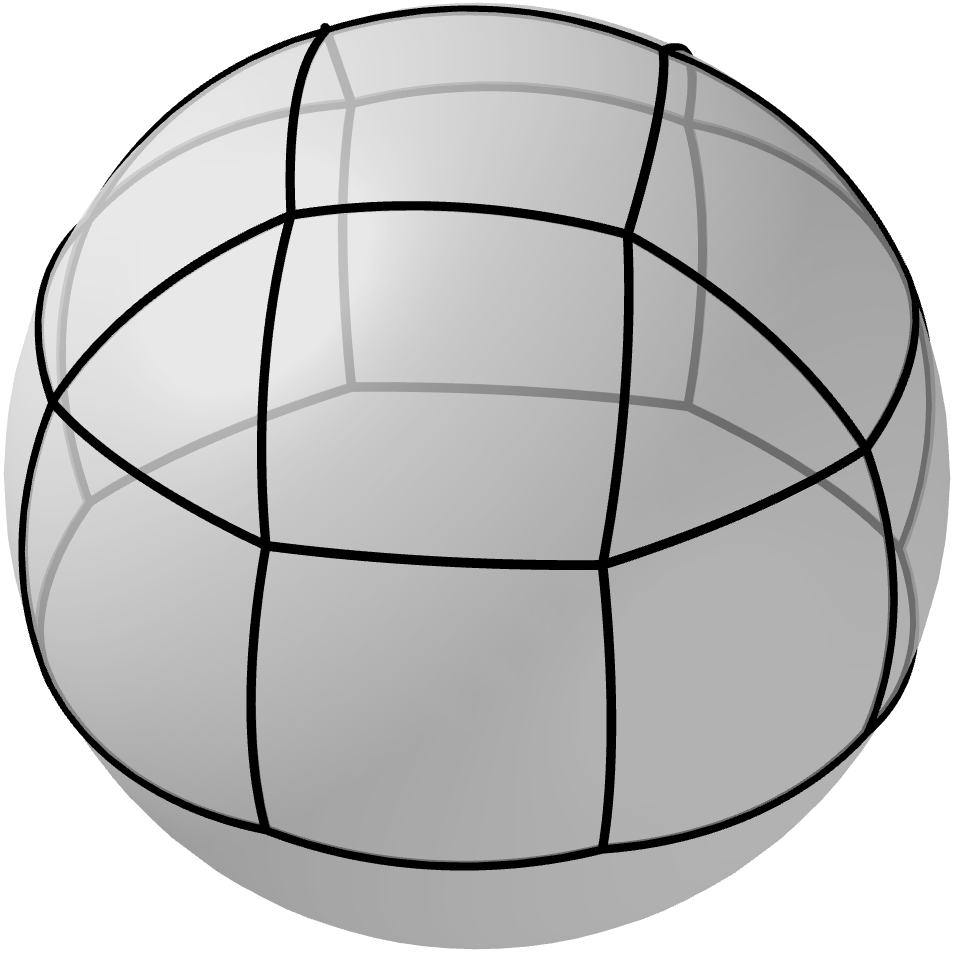}};
\node at (2*\XS,-2*\s) {\small };

\node [inner sep=0] (image) at (3*\XS,0) 
            {\includegraphics[height=\h cm]{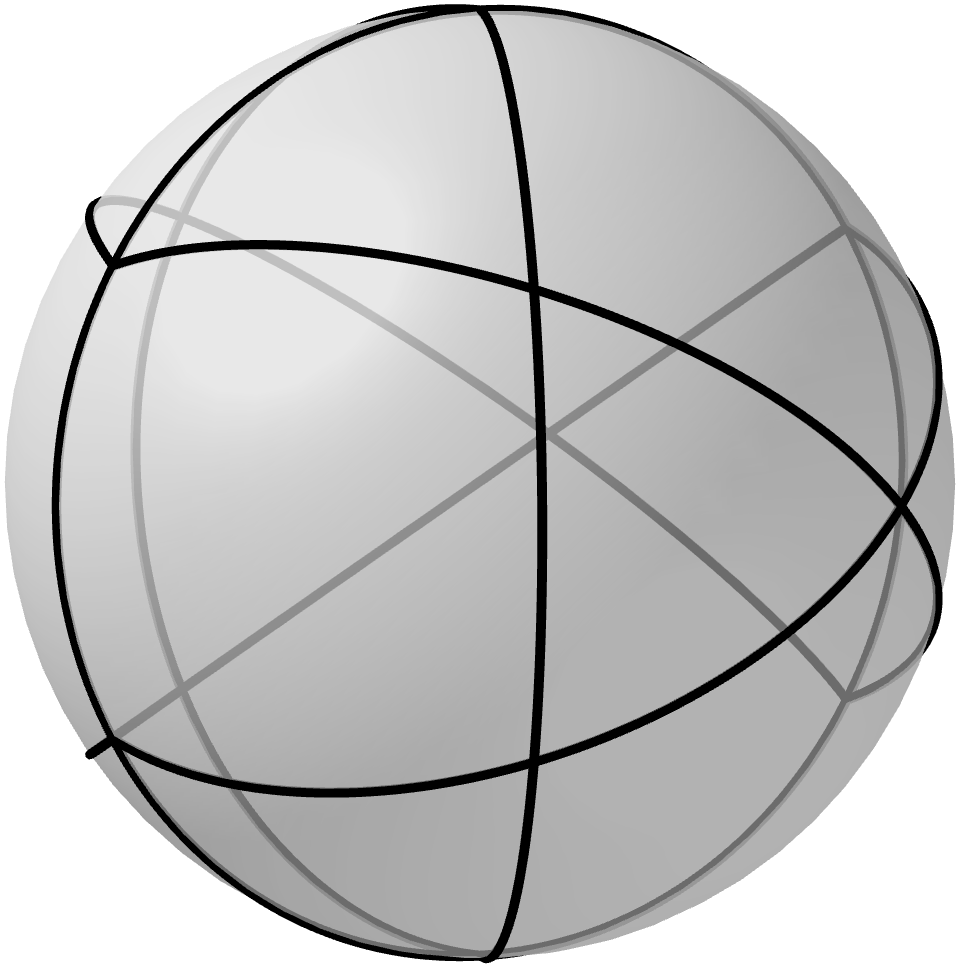}};
\node at (3*\XS,-2*\s) {\small };

\node [inner sep=0] (image) at (4*\XS,0) 
            {\includegraphics[height=\h cm]{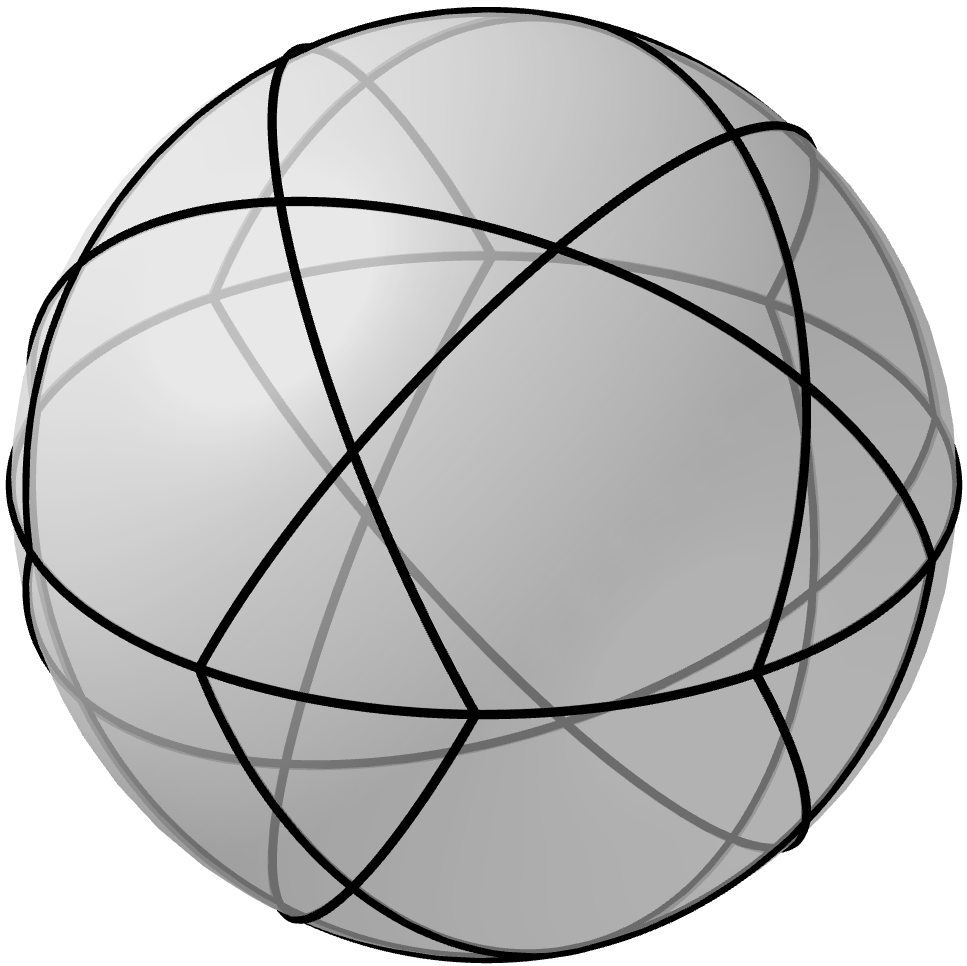}};
\node at (4*\XS,-2*\s) {\small };
\end{scope} 

\begin{scope}[yshift=-2*\YS cm] 
\node [inner sep=0] (image) at (0,0) 
            {\includegraphics[height=\h cm]{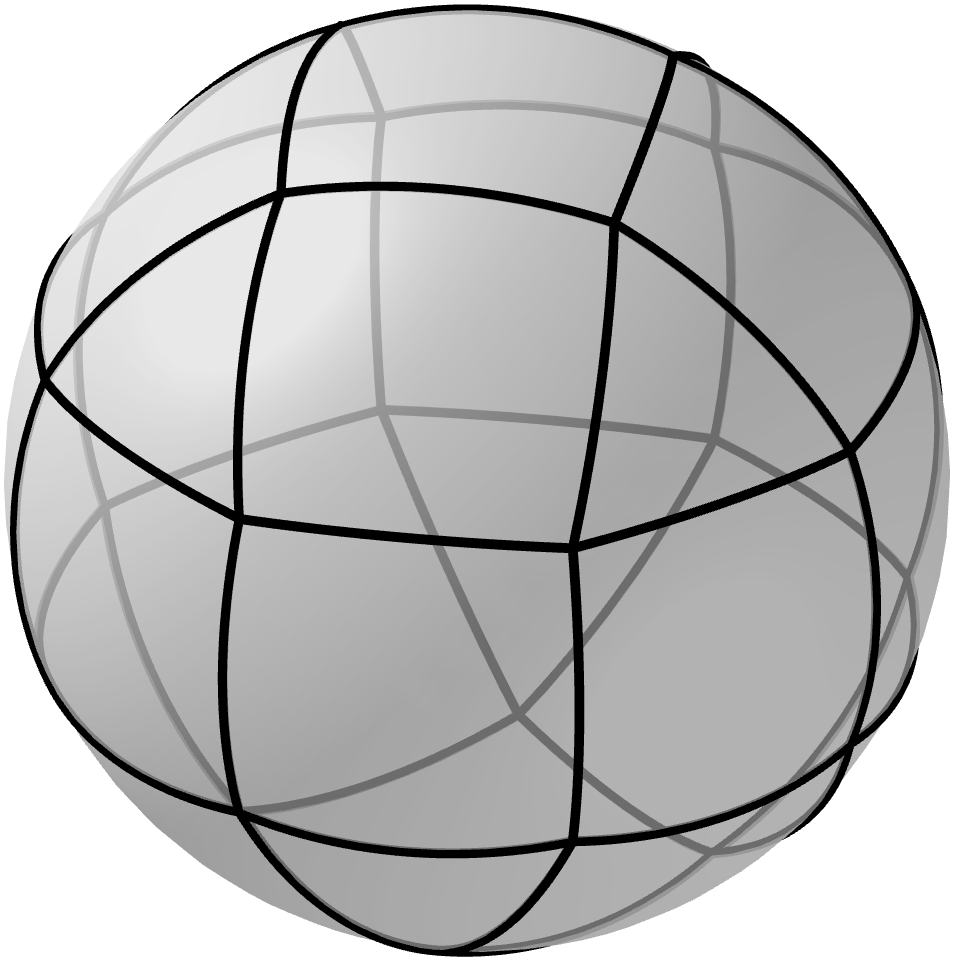}};
\node at (0,-2*\s) {\small };

\node [inner sep=0] (image) at (\XS,0) 
            {\includegraphics[height=\h cm]{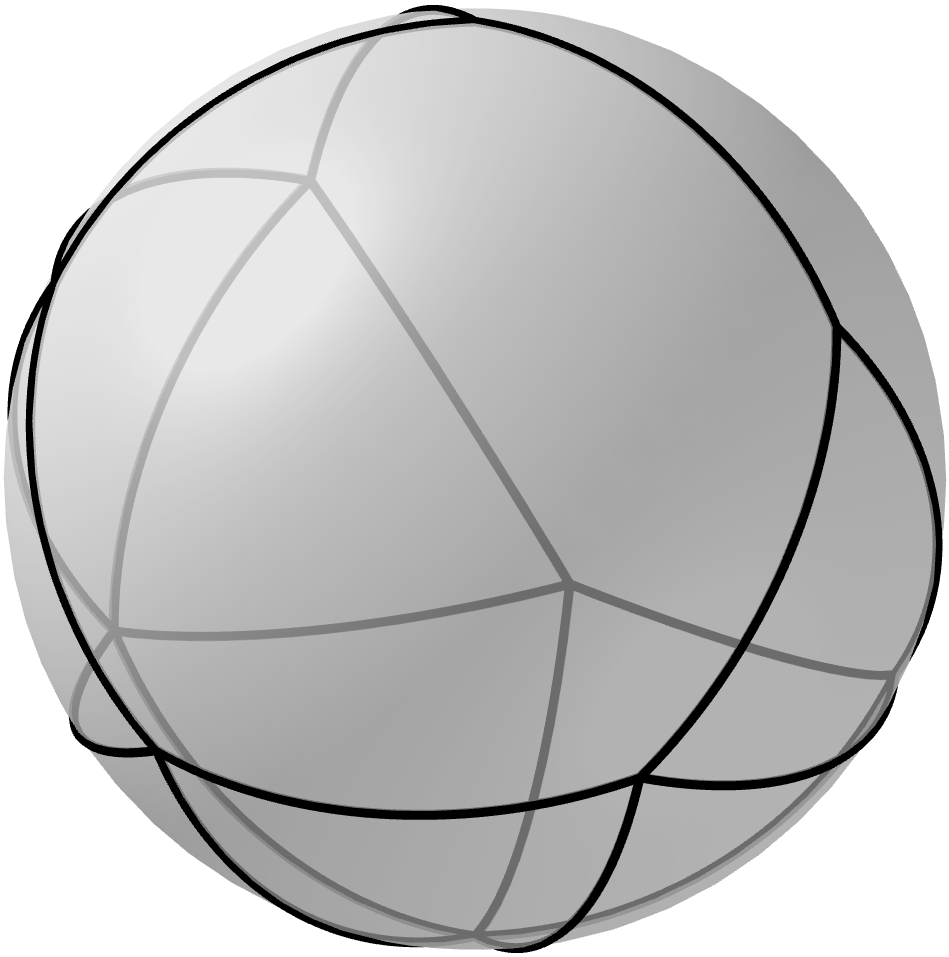}};
\node at (\XS,-2*\s) {\small };

\node [inner sep=0] (image) at (2*\XS,0) 
            {\includegraphics[height=\h cm]{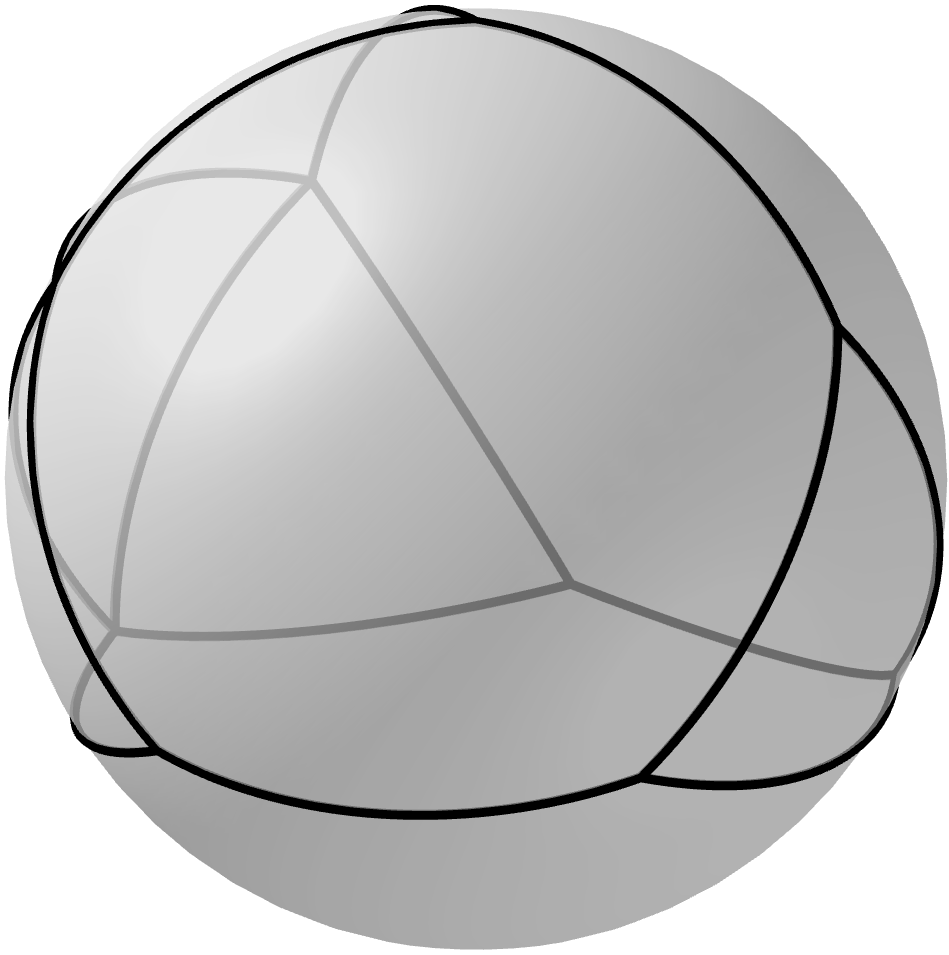}};
\node at (2*\XS,-2*\s) {\small };

\node [inner sep=0] (image) at (3*\XS,0) 
            {\includegraphics[height=\h cm]{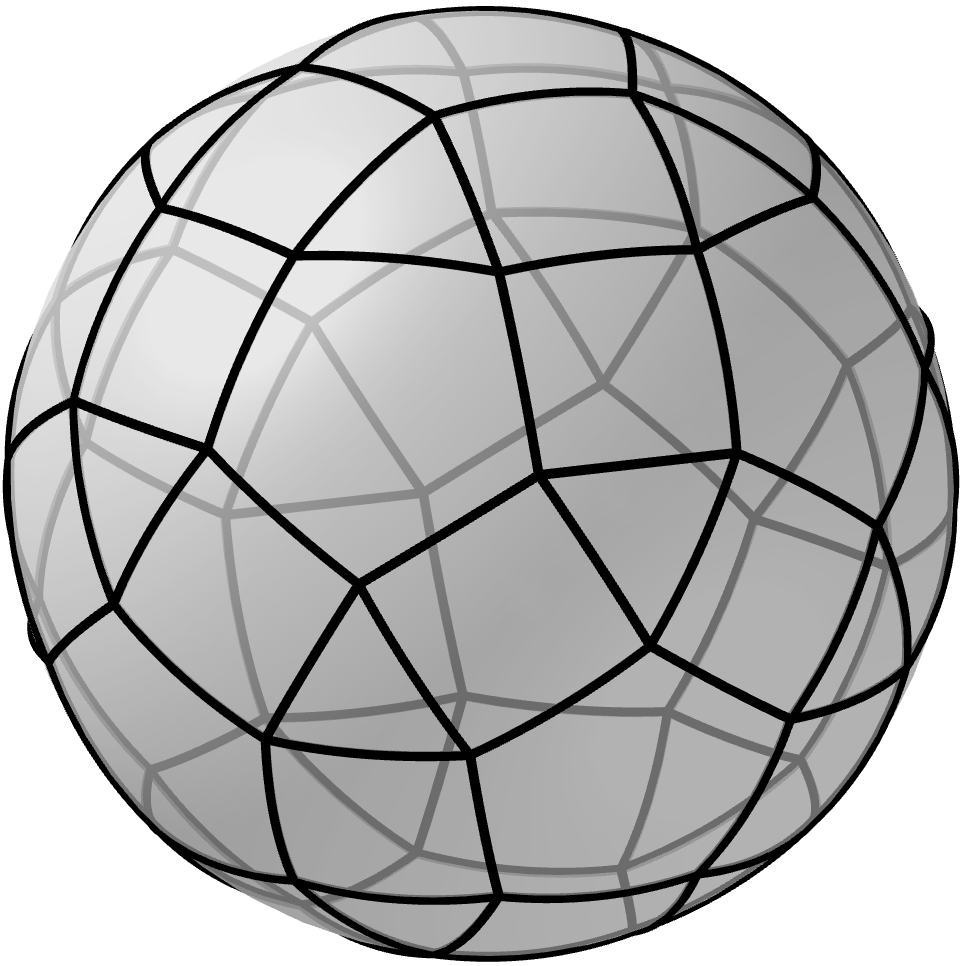}};
\node at (3*\XS,-2*\s) {\small };

\node [inner sep=0] (image) at (4*\XS,0) 
            {\includegraphics[height=\h cm]{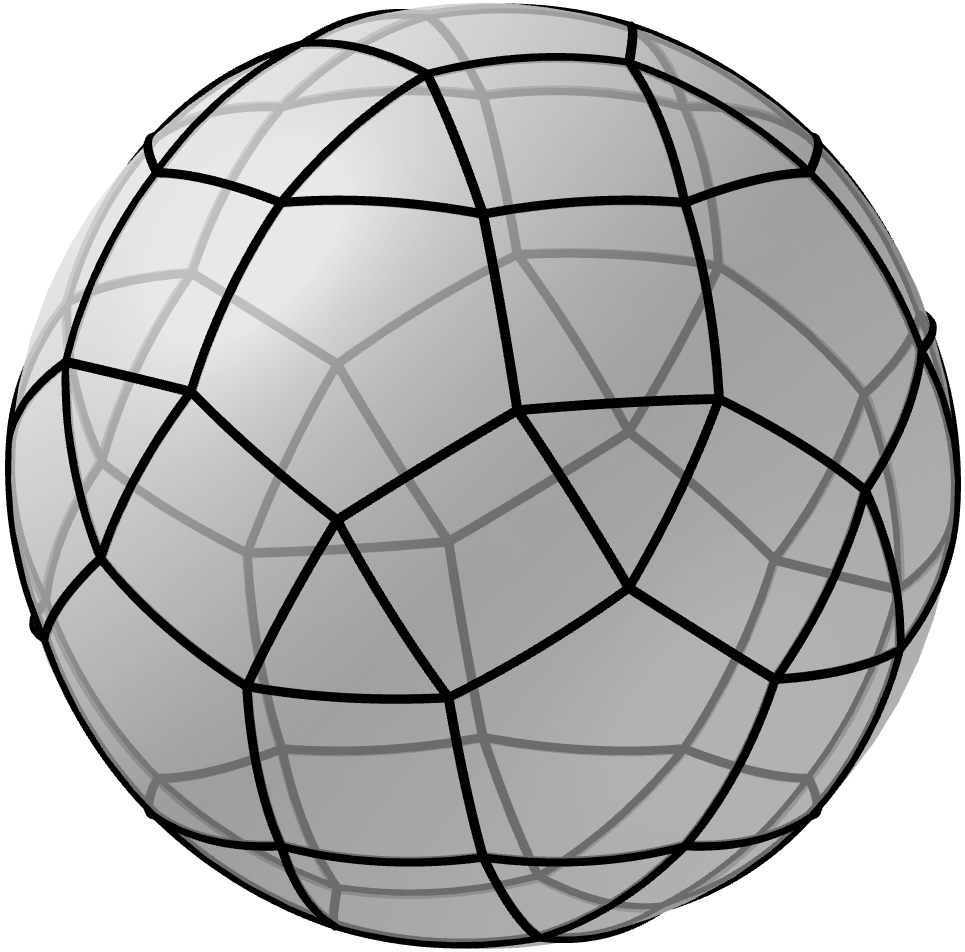}};
\node at (4*\XS,-2*\s) {\small };
\end{scope} 

\begin{scope}[yshift=-3*\YS cm] 
\node [inner sep=0] (image) at (0,0) 
            {\includegraphics[height=\h cm]{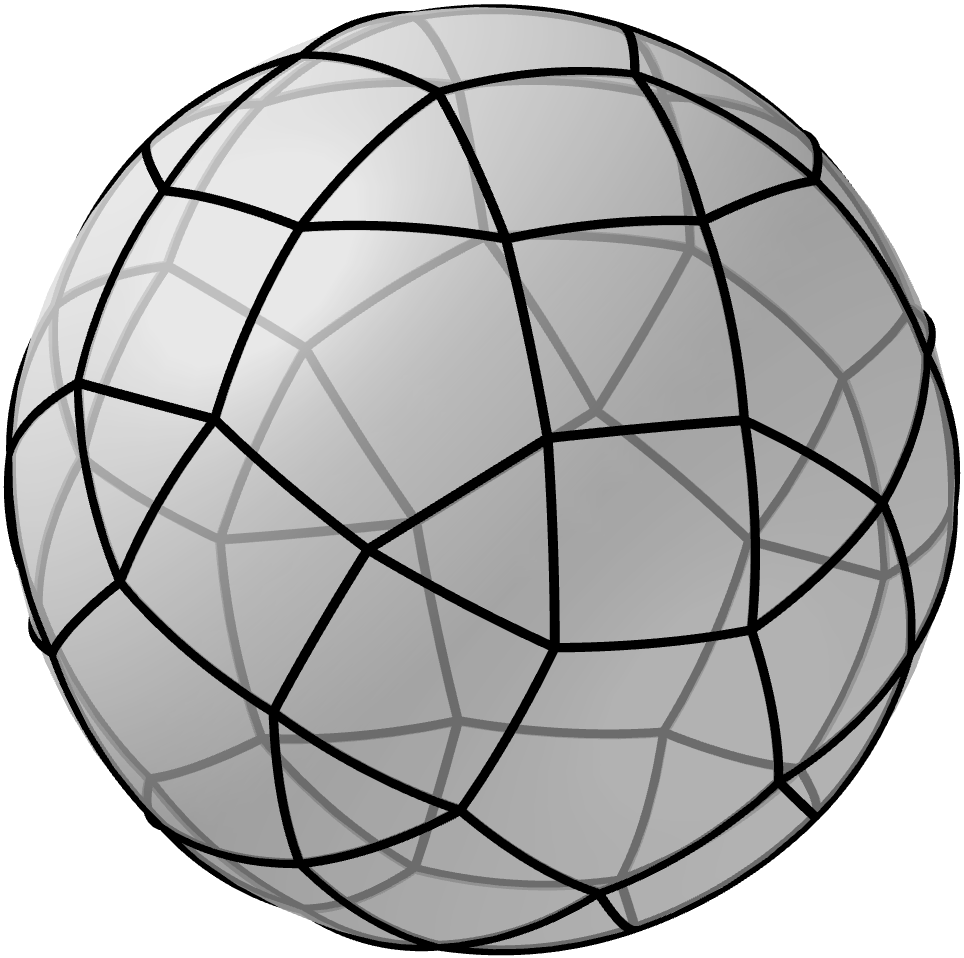}};
\node at (0,-2*\s) {\small };

\node [inner sep=0] (image) at (\XS,0) 
            {\includegraphics[height=\h cm]{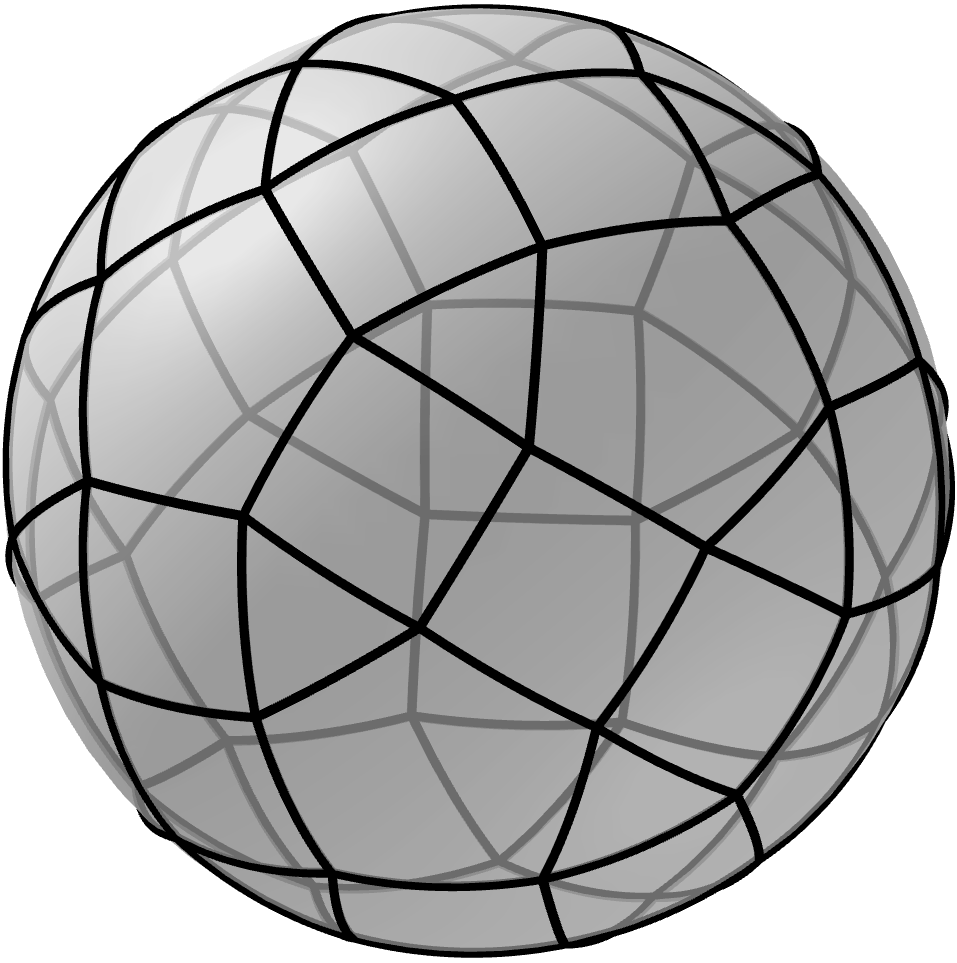}};
\node at (\XS,-2*\s) {\small };

\node [inner sep=0] (image) at (2*\XS,0) 
            {\includegraphics[height=\h cm]{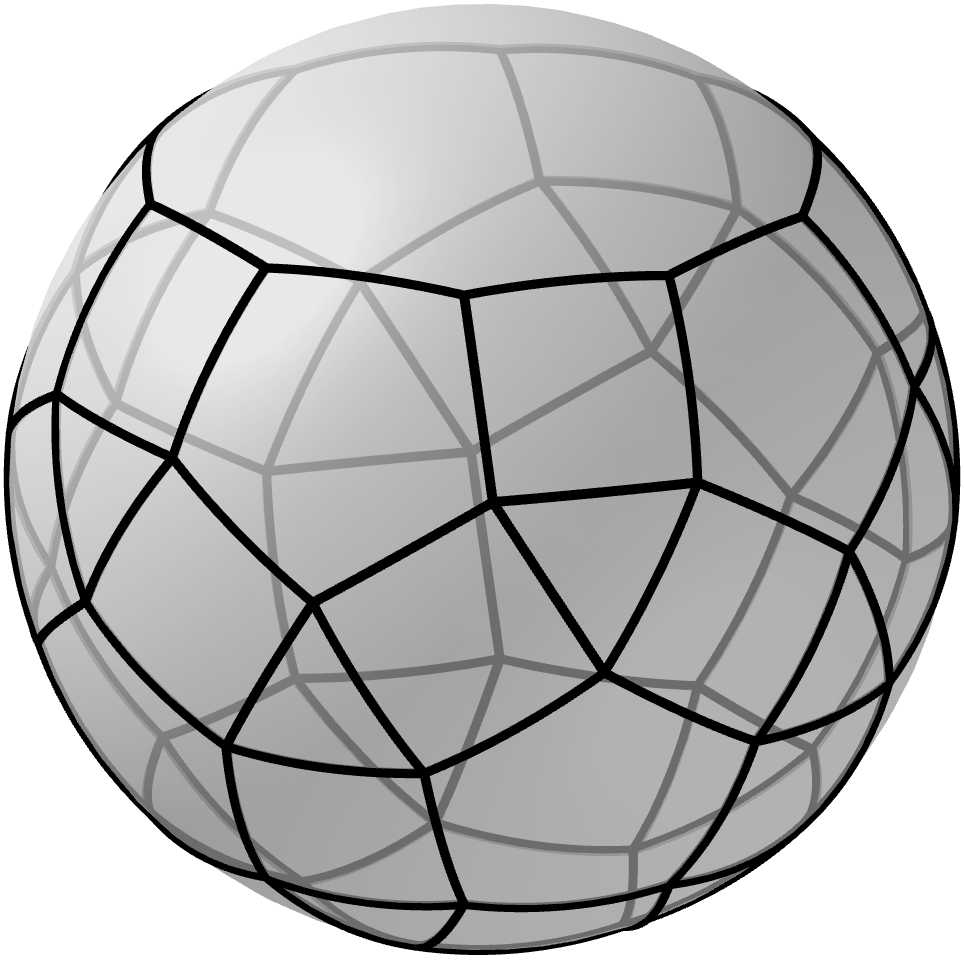}};
\node at (2*\XS,-2*\s) {\small };

\node [inner sep=0] (image) at (3*\XS,0) 
            {\includegraphics[height=\h cm]{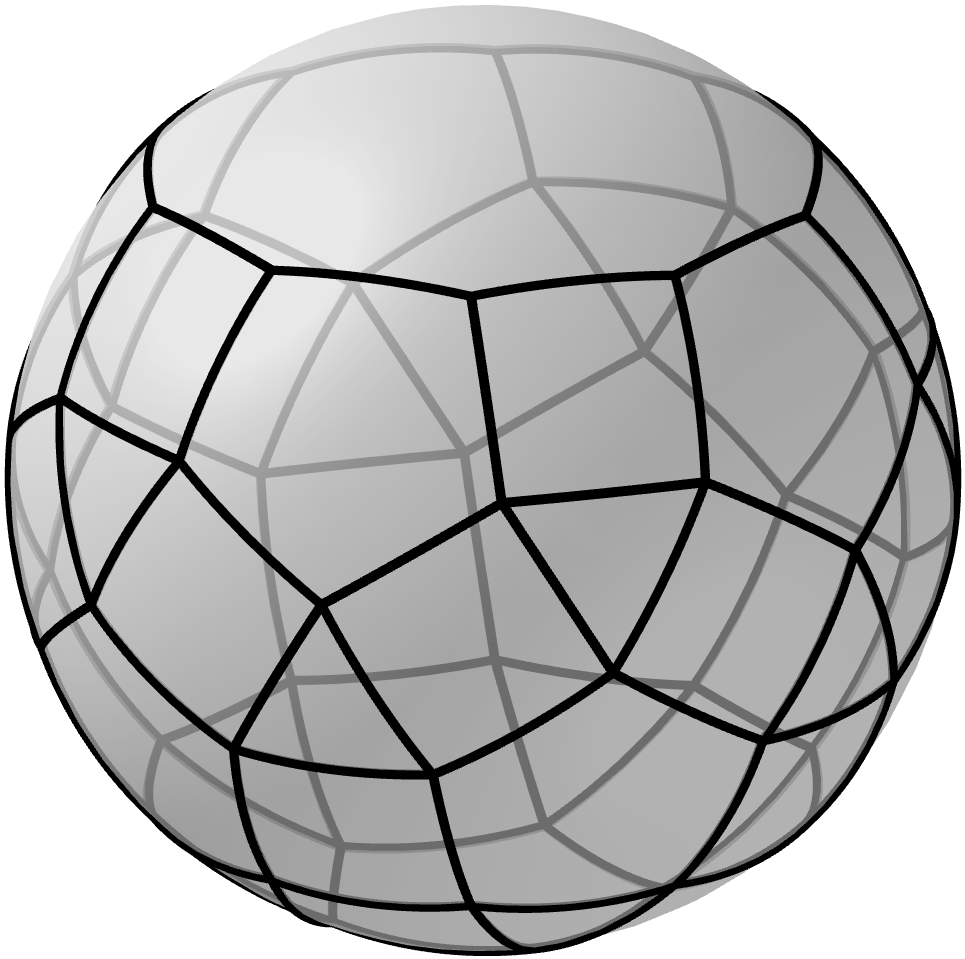}};
\node at (3*\XS,-2*\s) {\small };

\node [inner sep=0] (image) at (4*\XS,0) 
            {\includegraphics[height=\h cm]{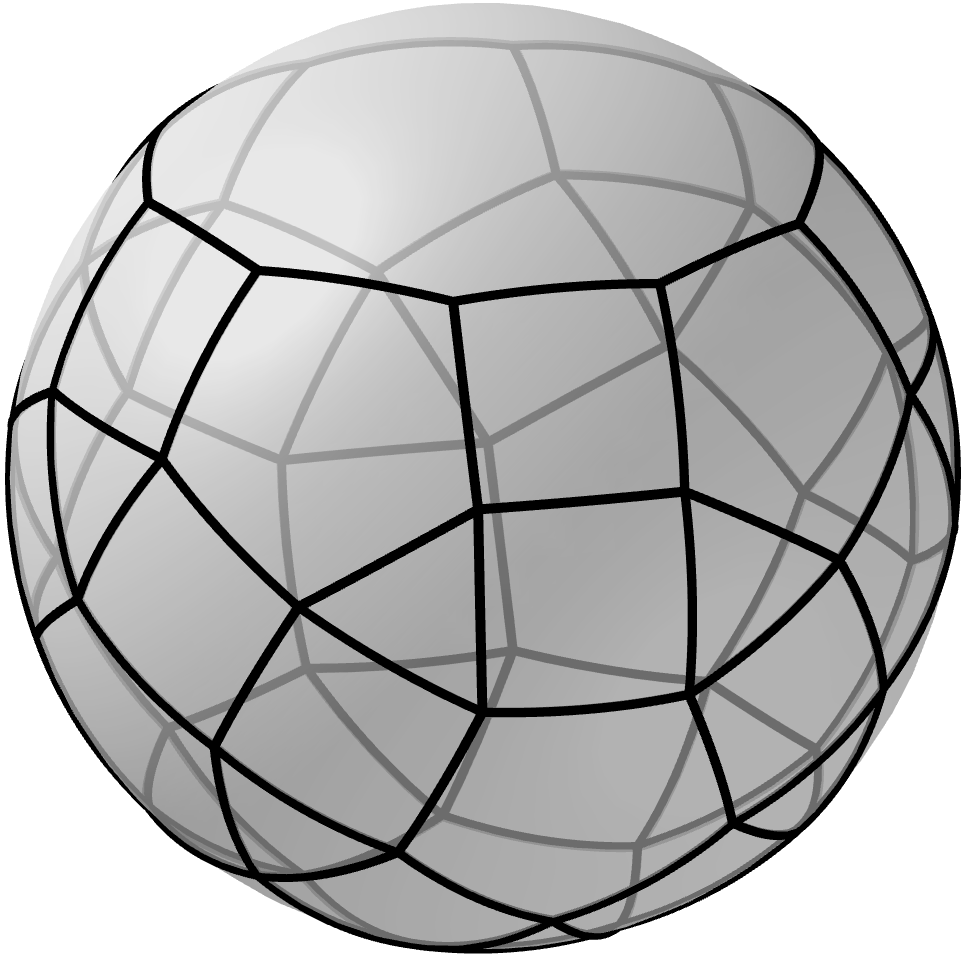}};
\node at (4*\XS,-2*\s) {\small };
\end{scope} 

\begin{scope}[yshift=-4*\YS cm] 
\node [inner sep=0] (image) at (0,0) 
            {\includegraphics[height=\h cm]{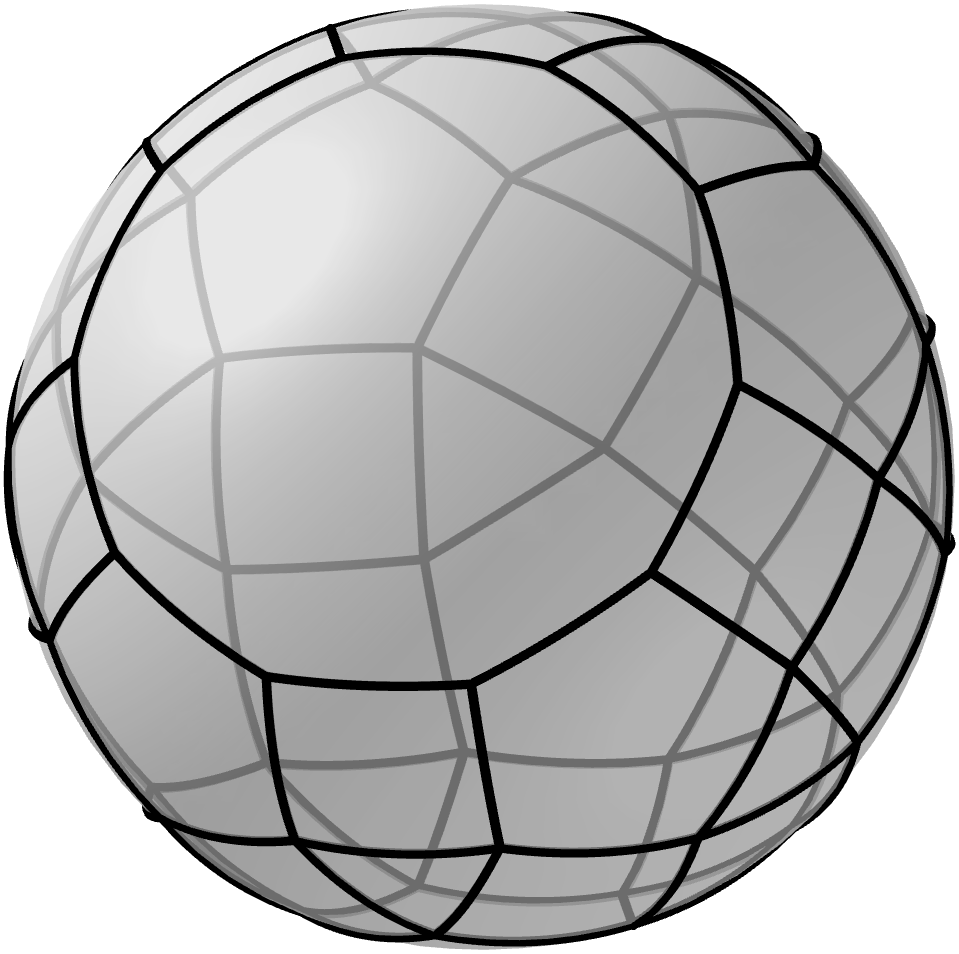}};
\node at (0,-2*\s) {\small };

\node [inner sep=0] (image) at (\XS,0) 
            {\includegraphics[height=\h cm]{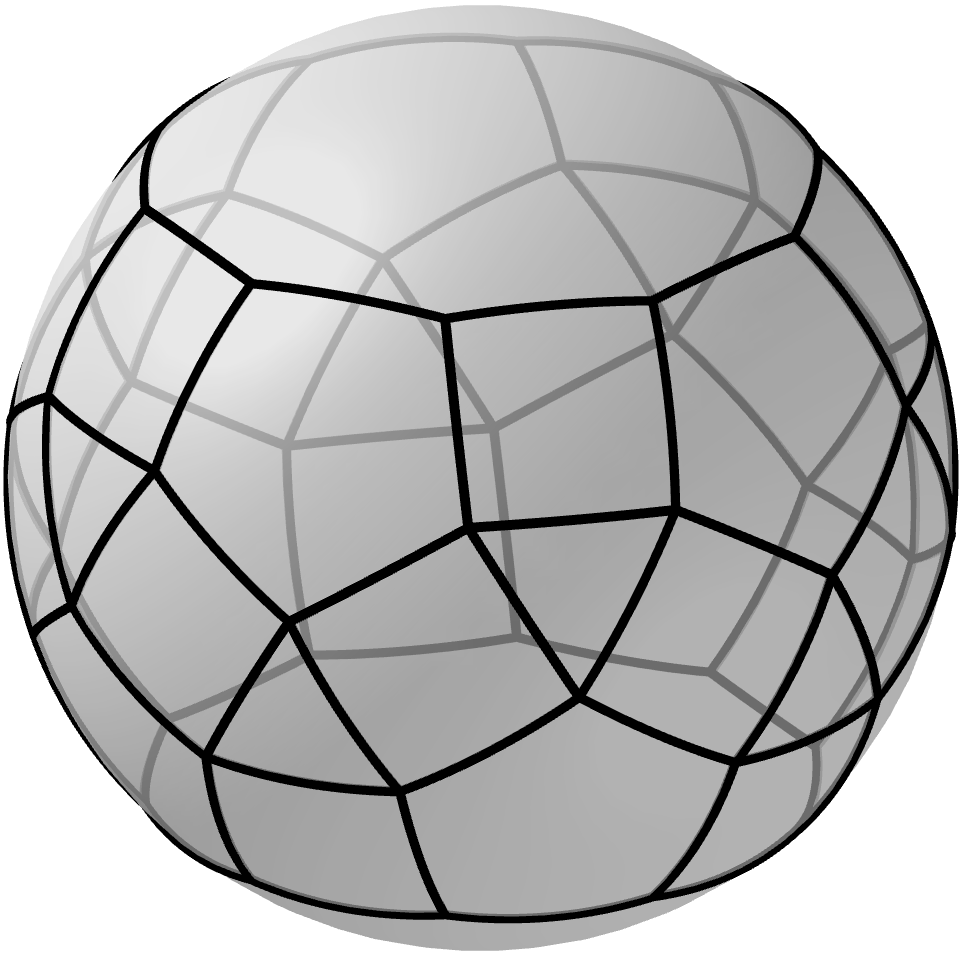}};
\node at (\XS,-2*\s) {\small };

\node [inner sep=0] (image) at (2*\XS,0) 
            {\includegraphics[height=\h cm]{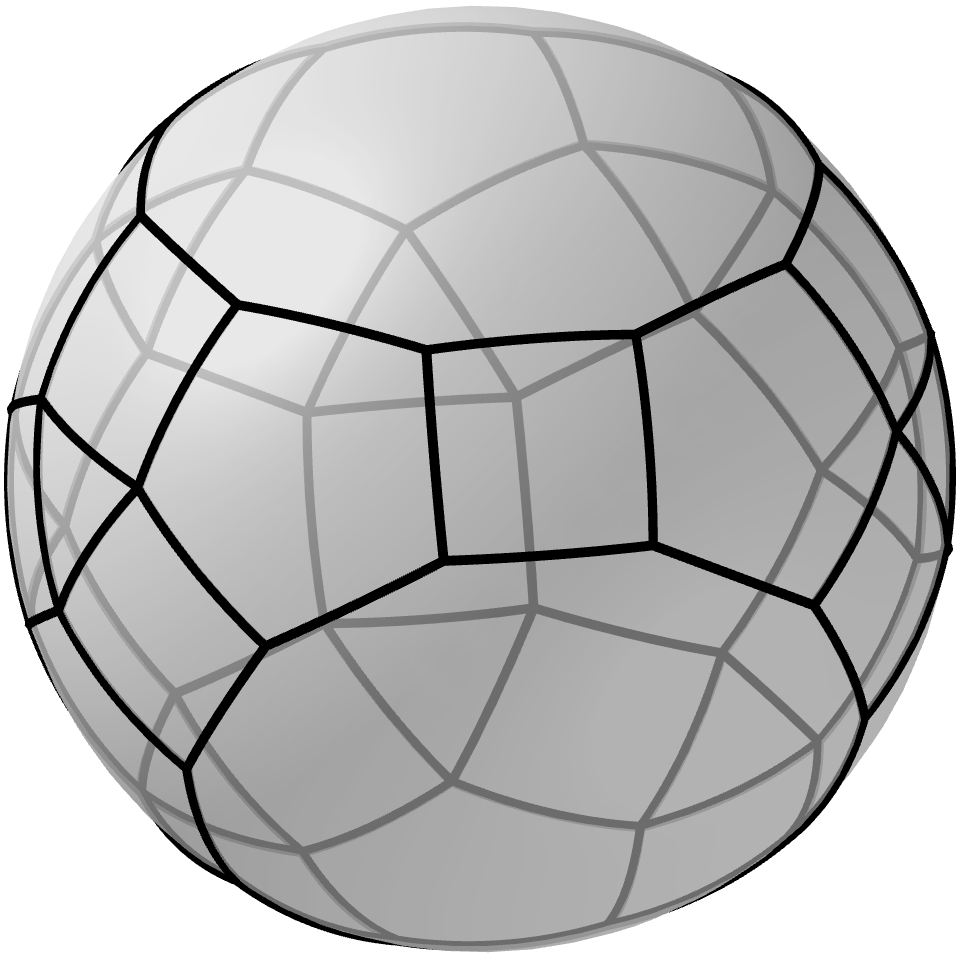}};
\node at (2*\XS,-2*\s) {\small };

\node [inner sep=0] (image) at (3*\XS,0) 
            {\includegraphics[height=\h cm]{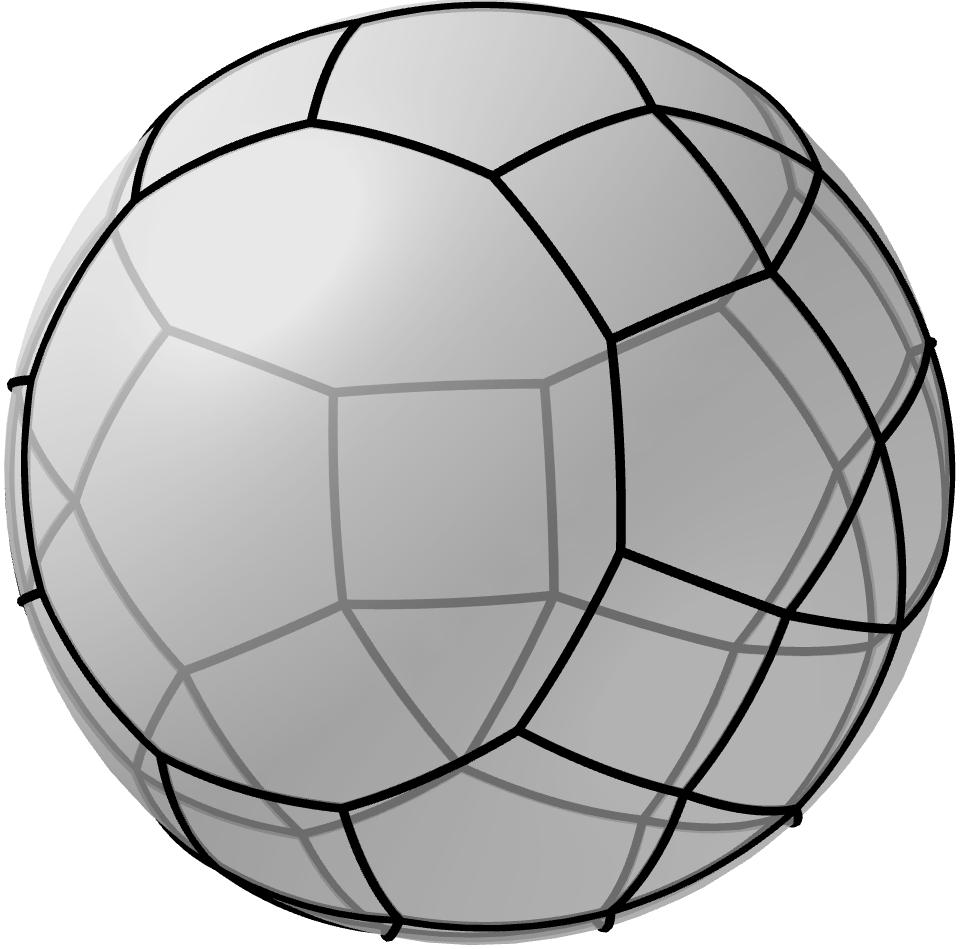}};
\node at (3*\XS,-2*\s) {\small };

\node [inner sep=0] (image) at (4*\XS,0) 
            {\includegraphics[height=\h cm]{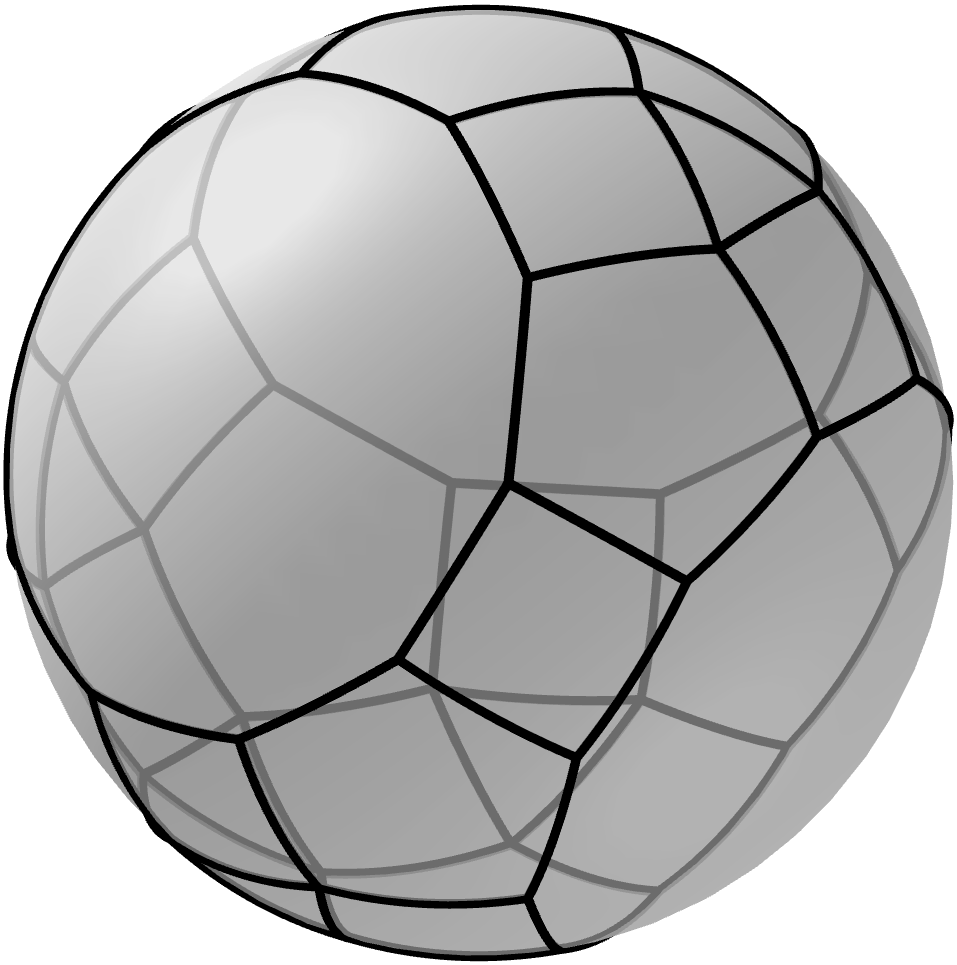}};
\node at (4*\XS,-2*\s) {\small };
\end{scope} 

\end{tikzpicture}
\caption{Circumscribable Johnson solids}
\end{figure}


\begin{figure}[h!]
\centering
\begin{tikzpicture}
\tikzmath{
\s=0.75;
\h=1.75;
\XS=2.15;
\YS=2.15;
}
\begin{scope}[]
\node [inner sep=0] (image) at (0,0) 
            {\includegraphics[height=\h cm]{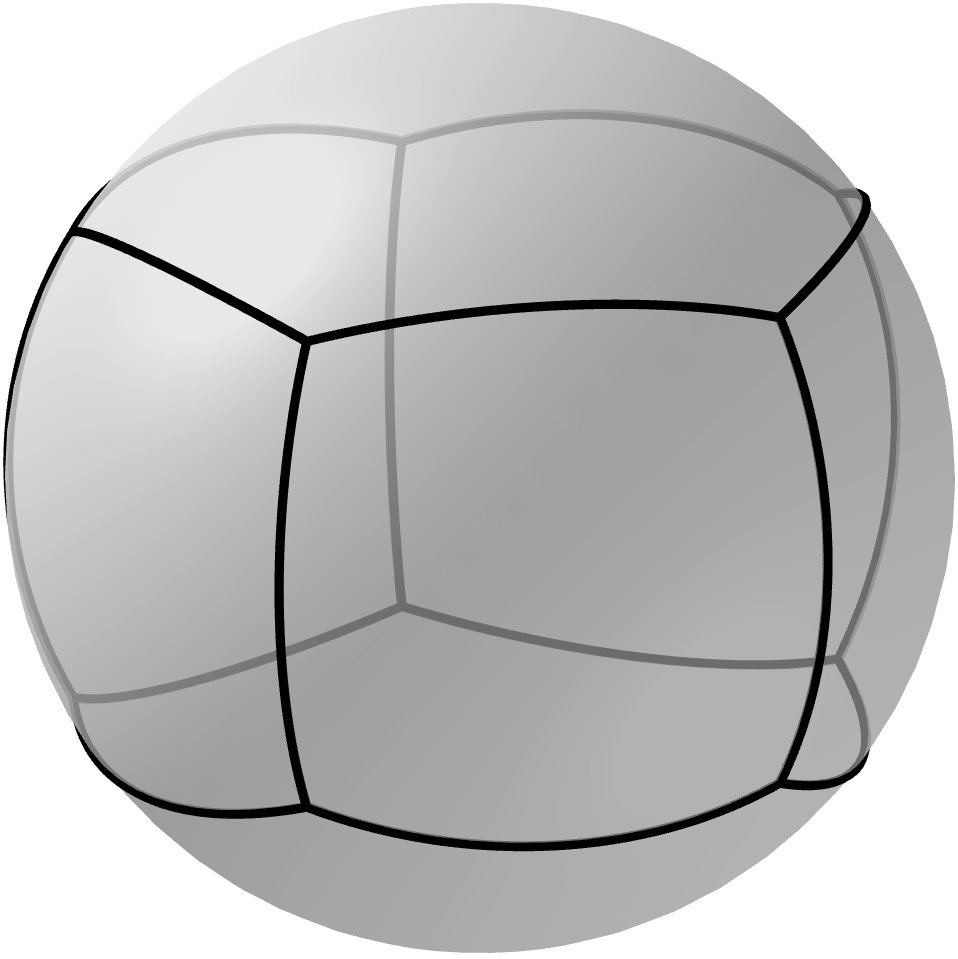}};

\node [inner sep=0] (image) at (\XS,0) 
            {\includegraphics[height=\h cm]{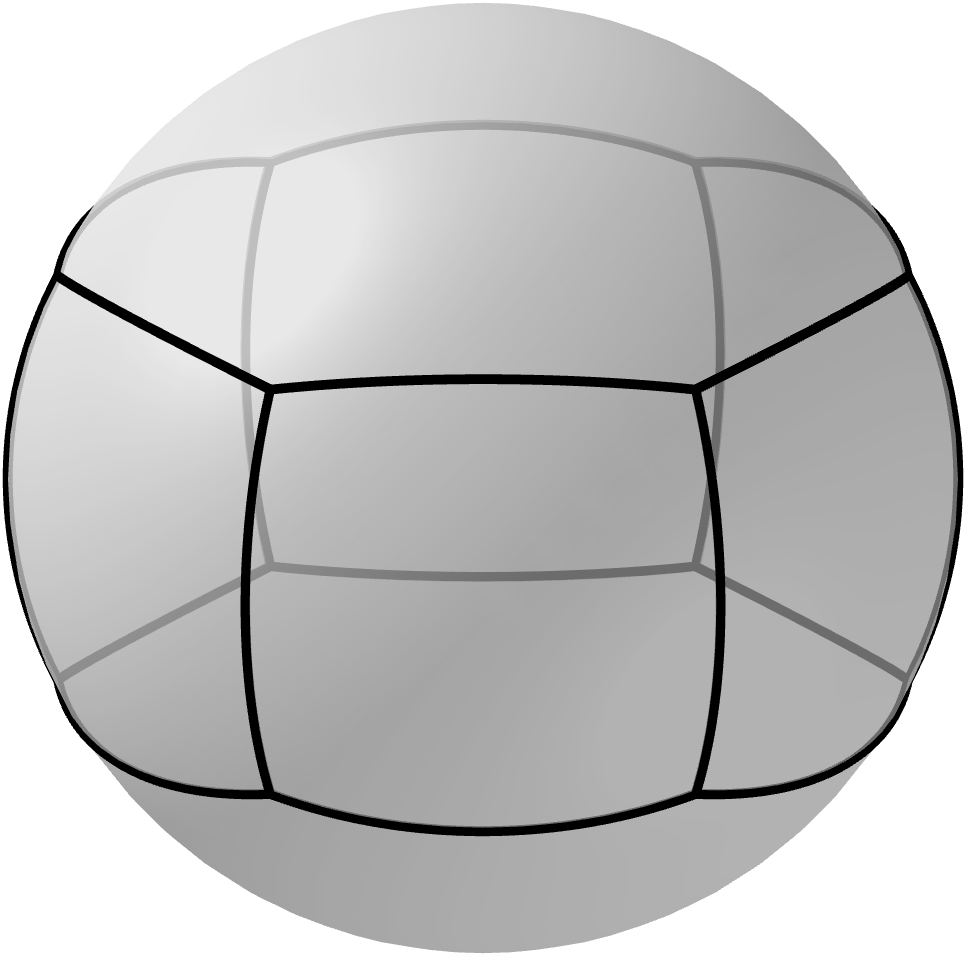}};

\node [inner sep=0] (image) at (2*\XS,0) 
            {\includegraphics[height=\h cm]{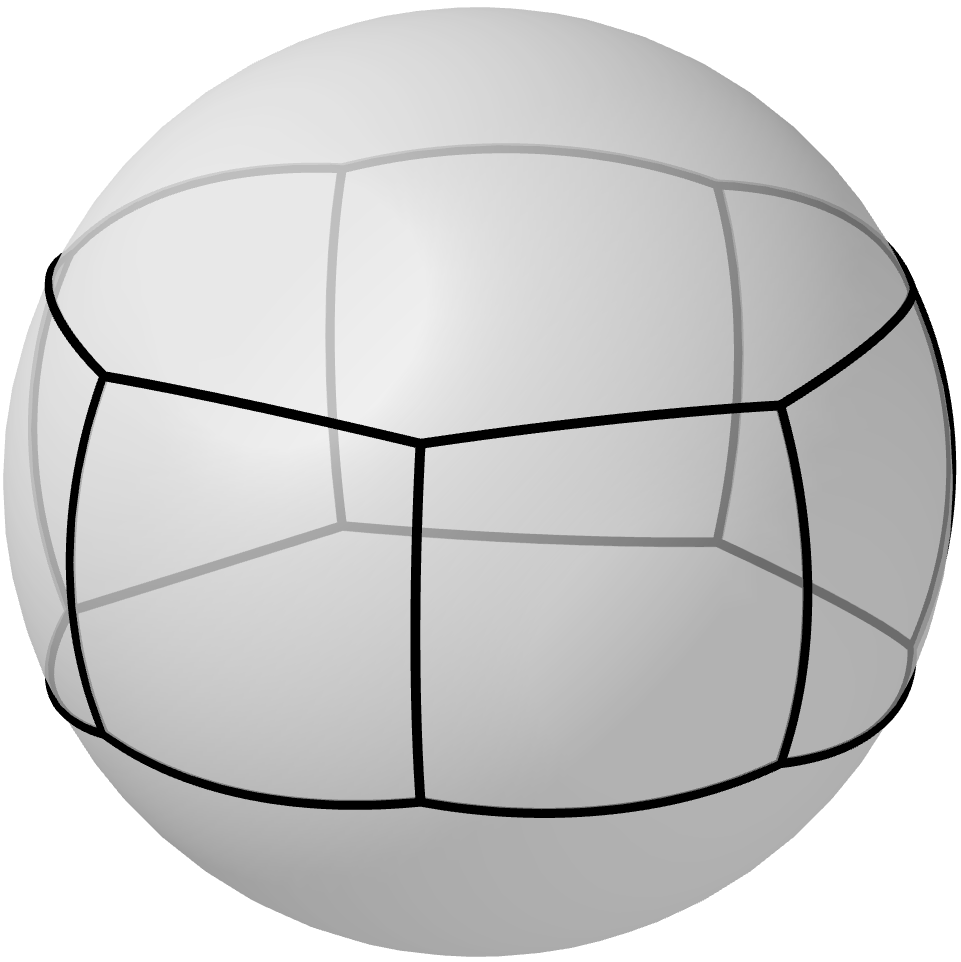}};

\node [inner sep=0] (image) at (3*\XS,0) 
            {\includegraphics[height=\h cm]{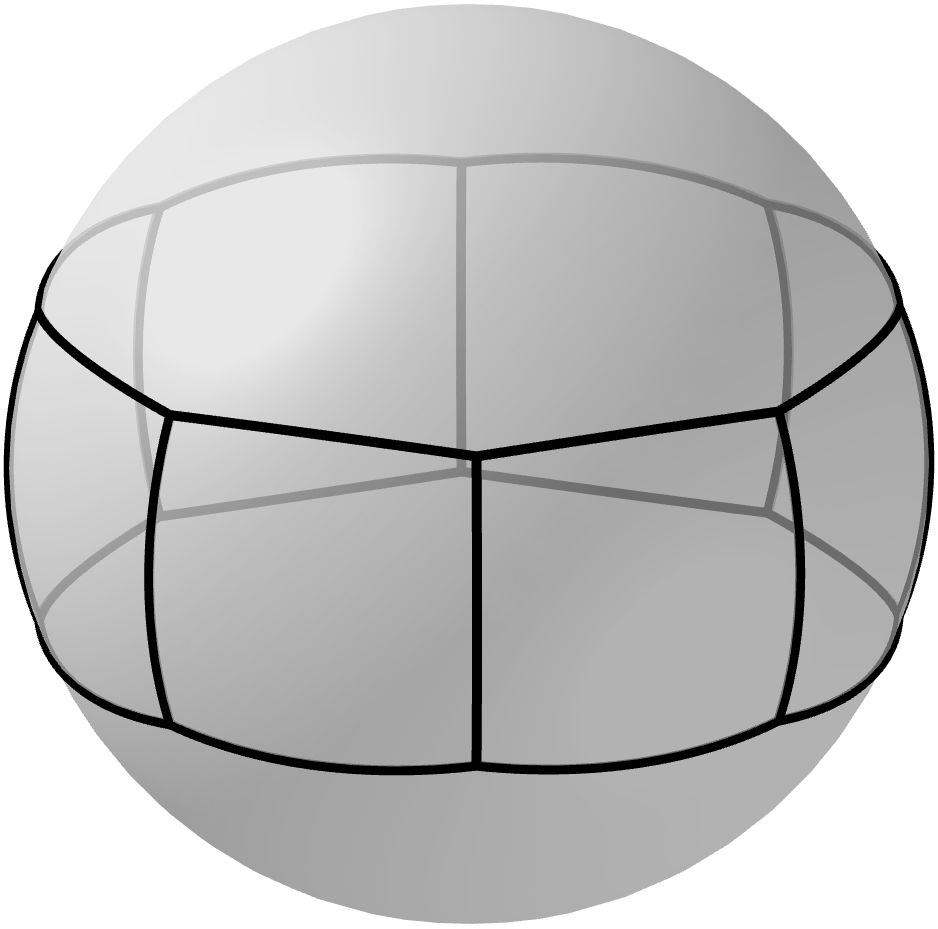}};

\node [inner sep=0] (image) at (4*\XS,0) 
            {\includegraphics[height=\h cm]{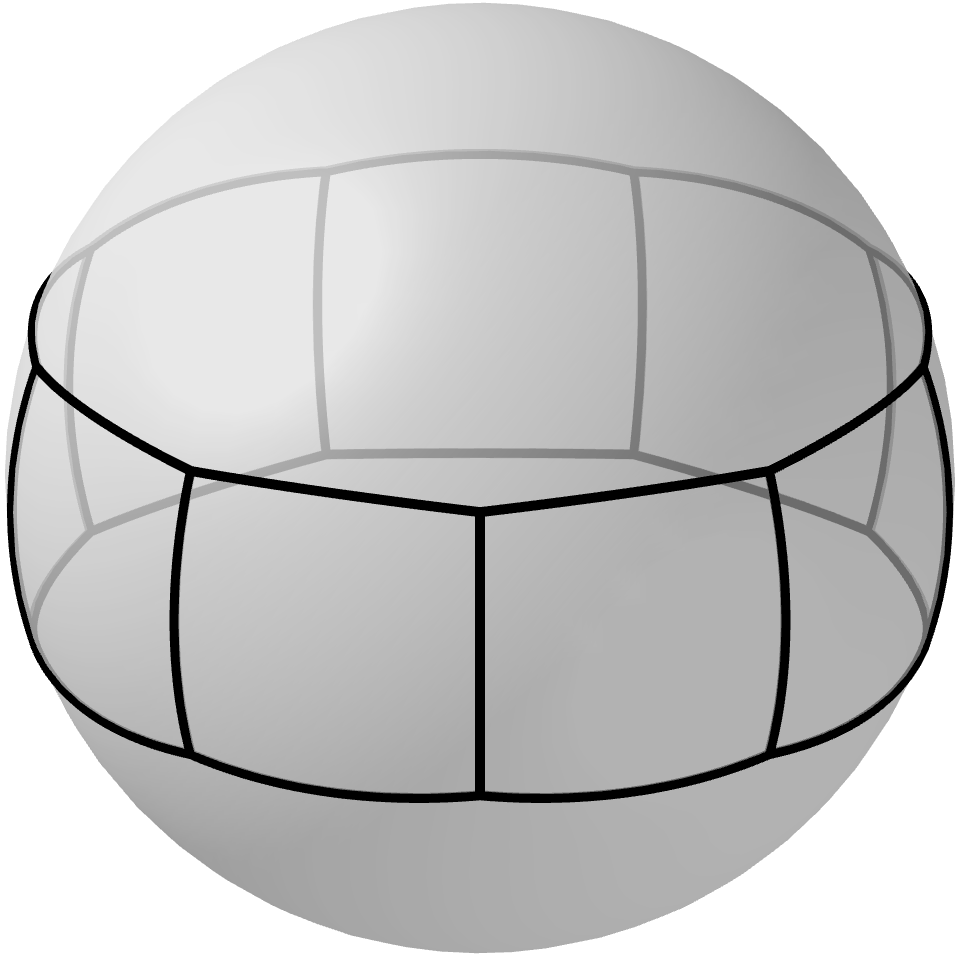}};
\end{scope}

\begin{scope}[yshift=-\YS cm]

\node [inner sep=0] (image) at (0,0) 
            {\includegraphics[height=\h cm]{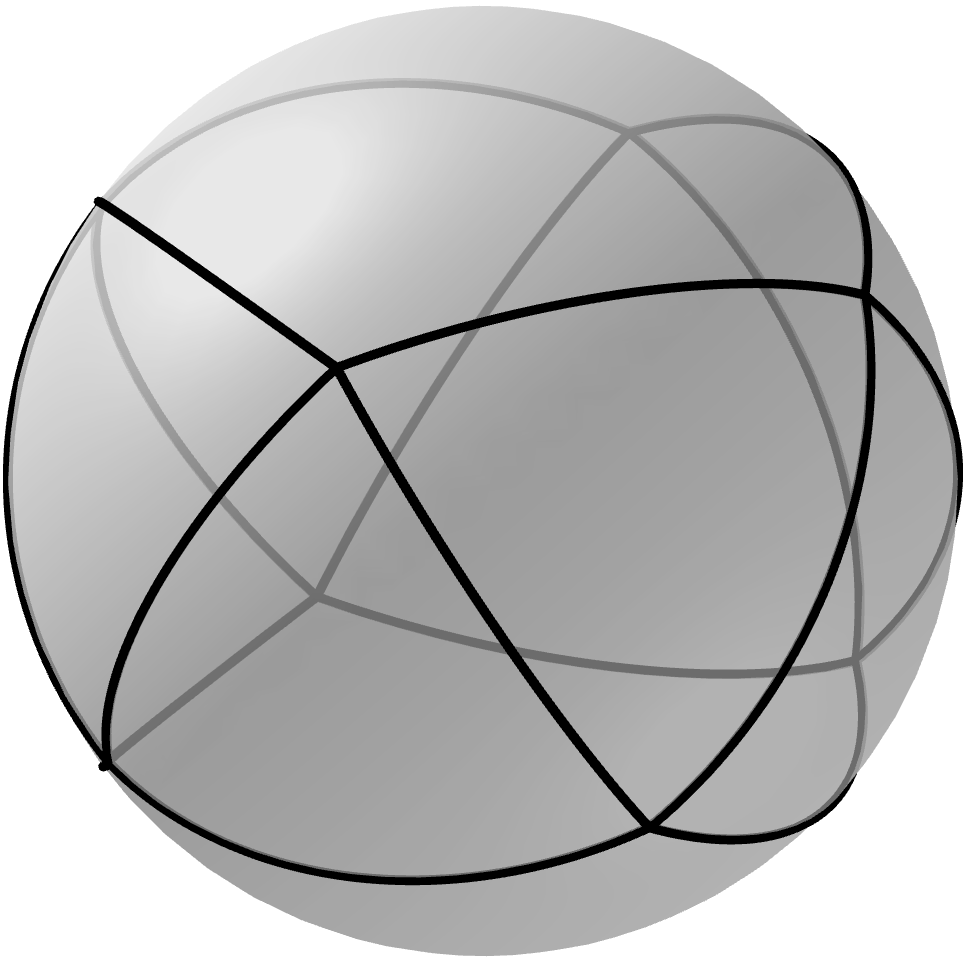}};

\node [inner sep=0] (image) at (\XS,0) 
            {\includegraphics[height=\h cm]{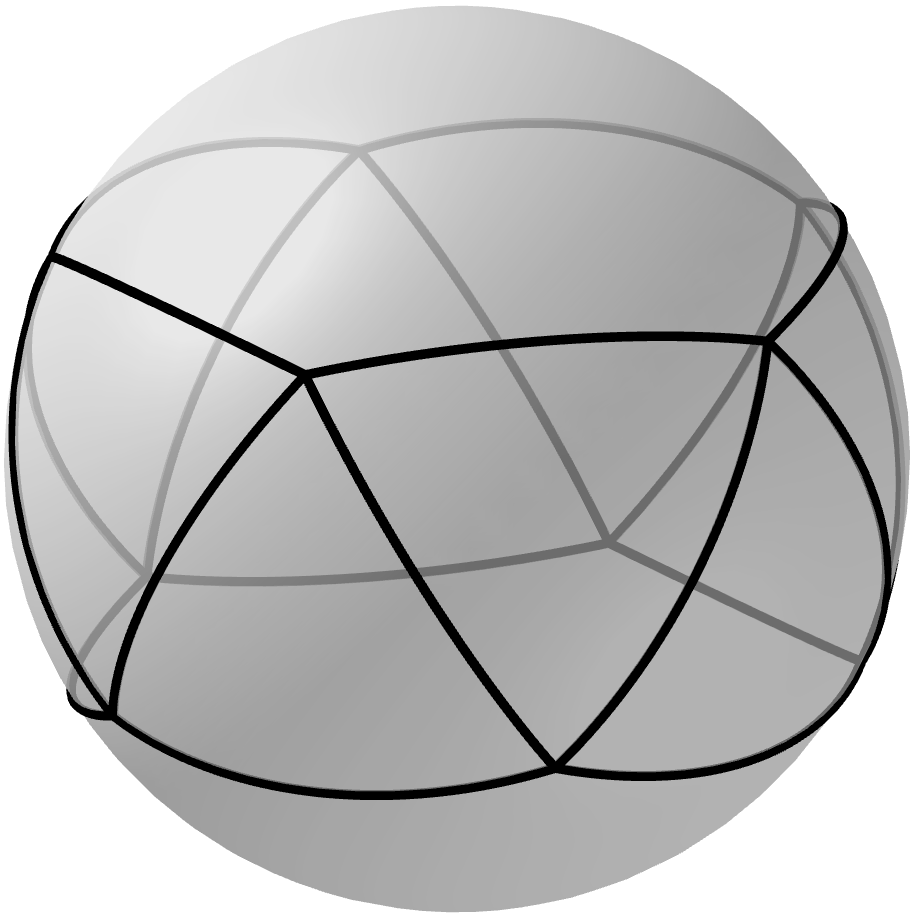}};

\node [inner sep=0] (image) at (2*\XS,0) 
            {\includegraphics[height=\h cm]{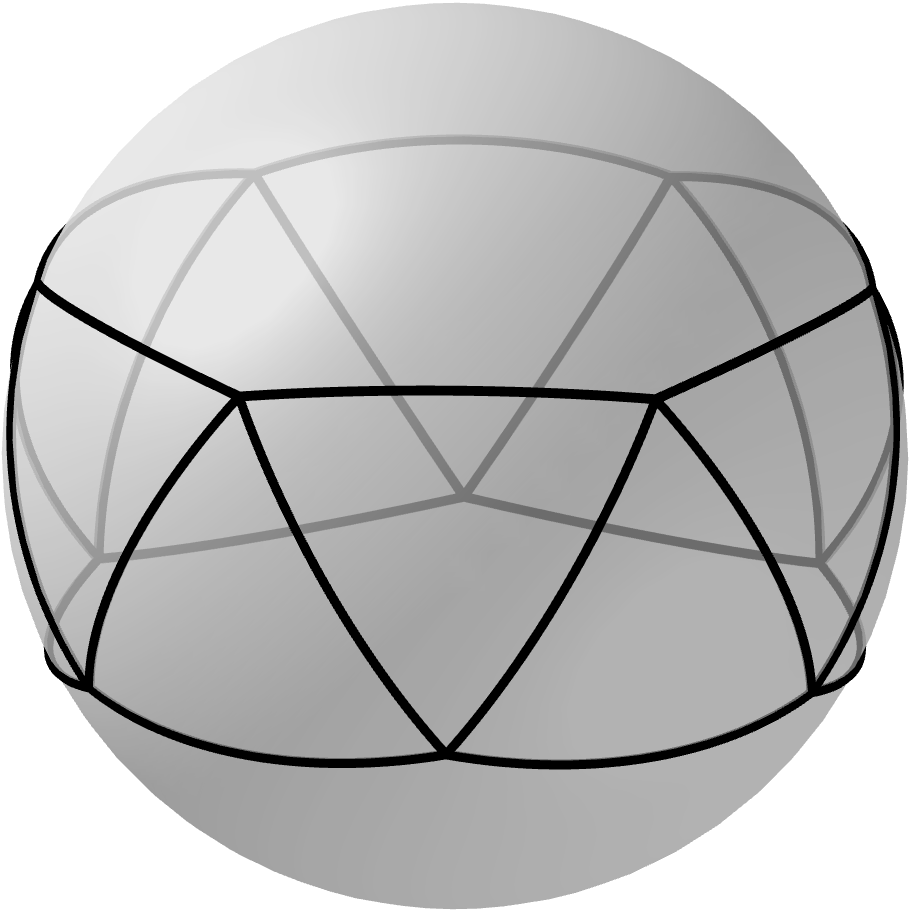}};

\node [inner sep=0] (image) at (3*\XS,0) 
            {\includegraphics[height=\h cm]{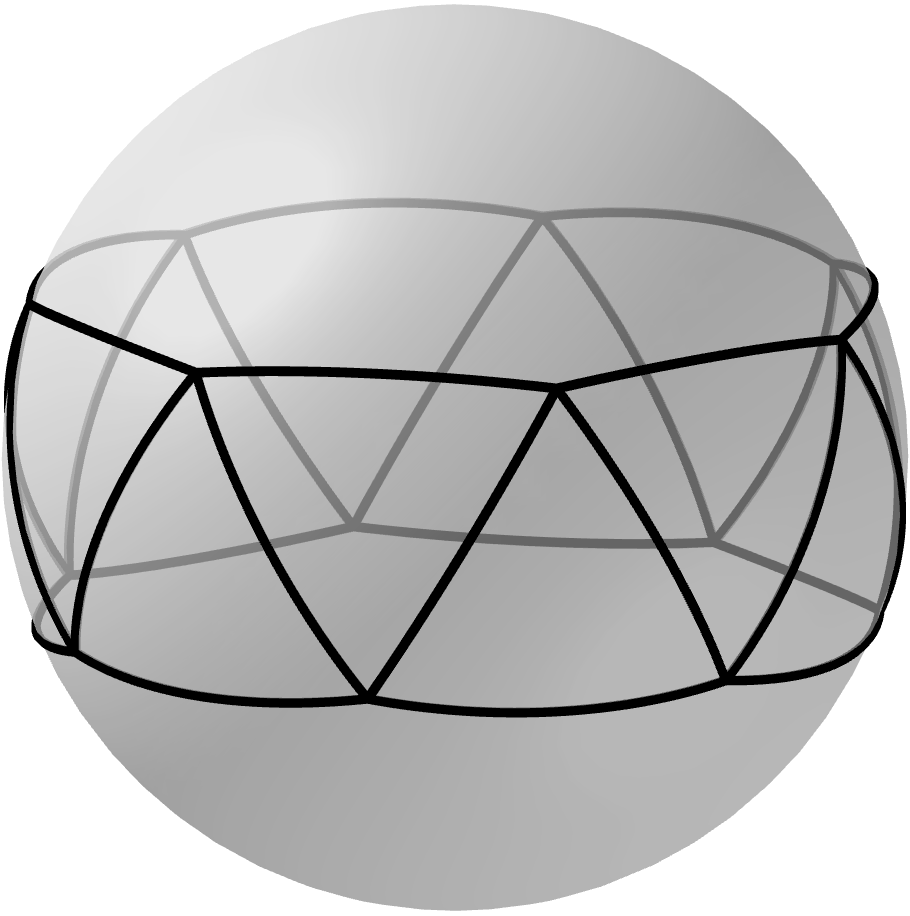}};

\node [inner sep=0] (image) at (4*\XS,0) 
            {\includegraphics[height=\h cm]{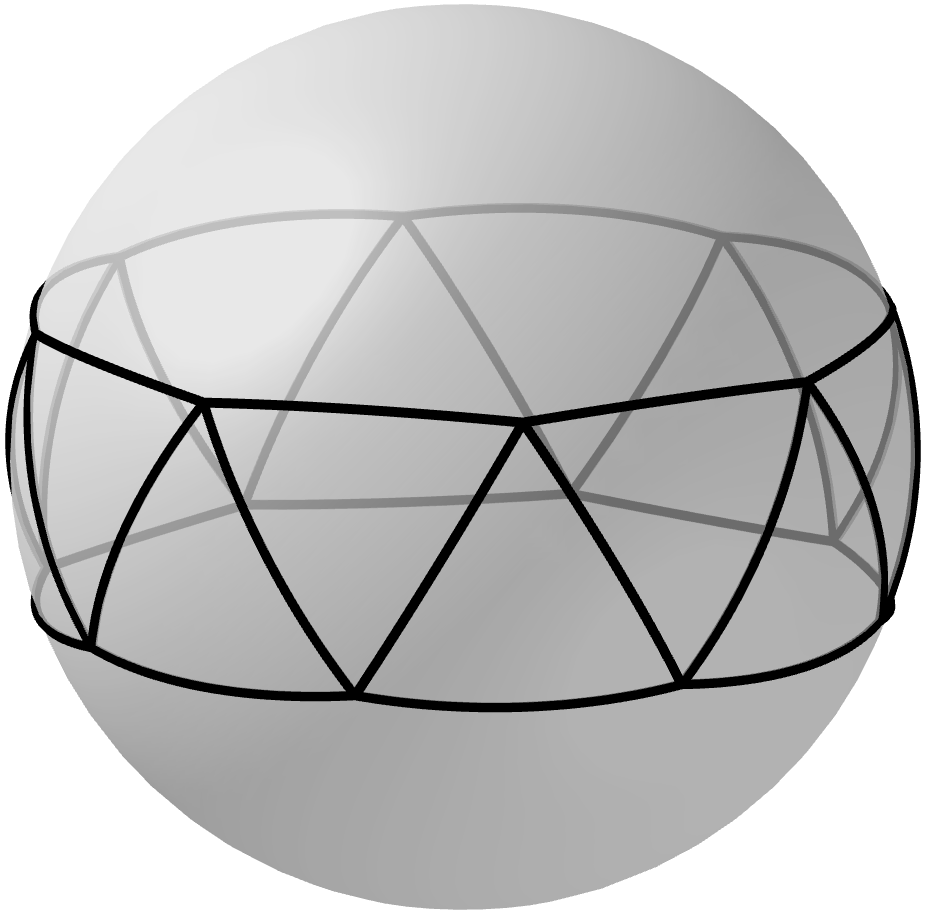}};

\end{scope}

\end{tikzpicture}
\caption{Prisms and antiprisms}
\end{figure}


\begin{figure}[h!]
\centering
\begin{tikzpicture}
\tikzmath{
\s=0.75;
\h=1.75;
\XS=2.15;
\YS=2.15;
}
\begin{scope}[]
\node [inner sep=0] (image) at (0,0) 
            {\includegraphics[height=\h cm]{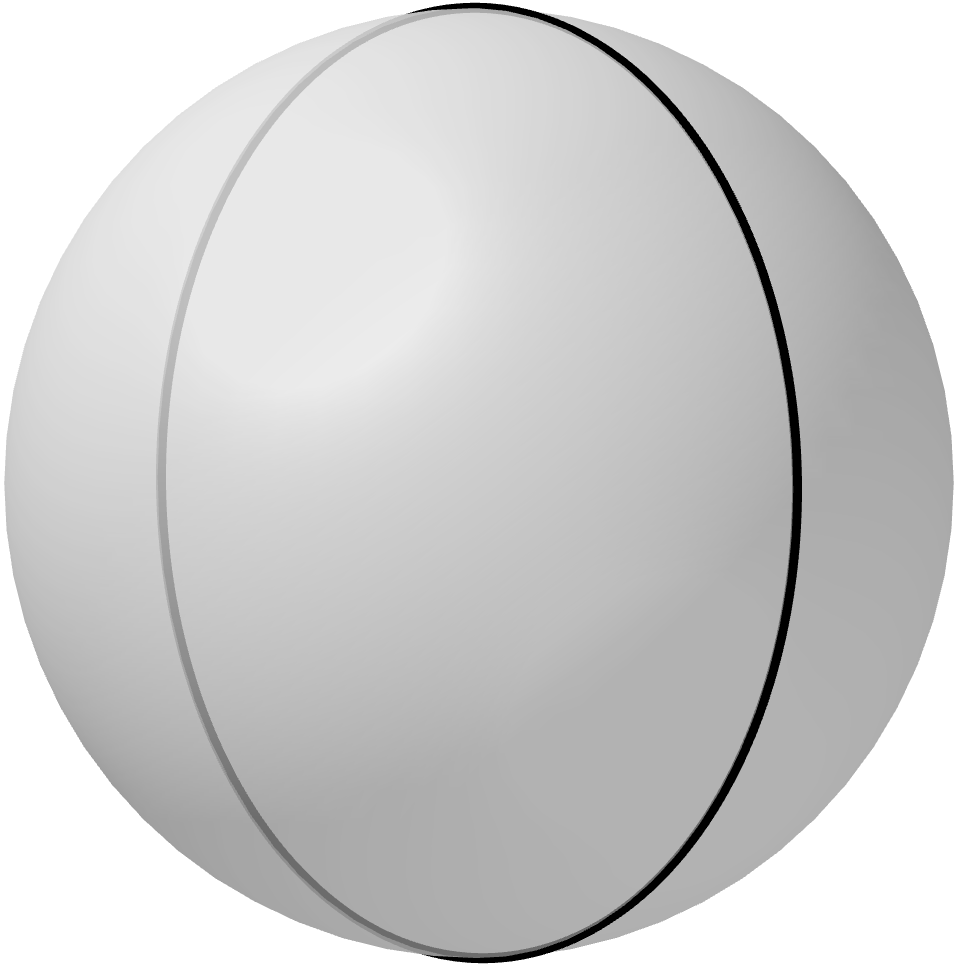}};

\node [inner sep=0] (image) at (\XS,0) 
            {\includegraphics[height=\h cm]{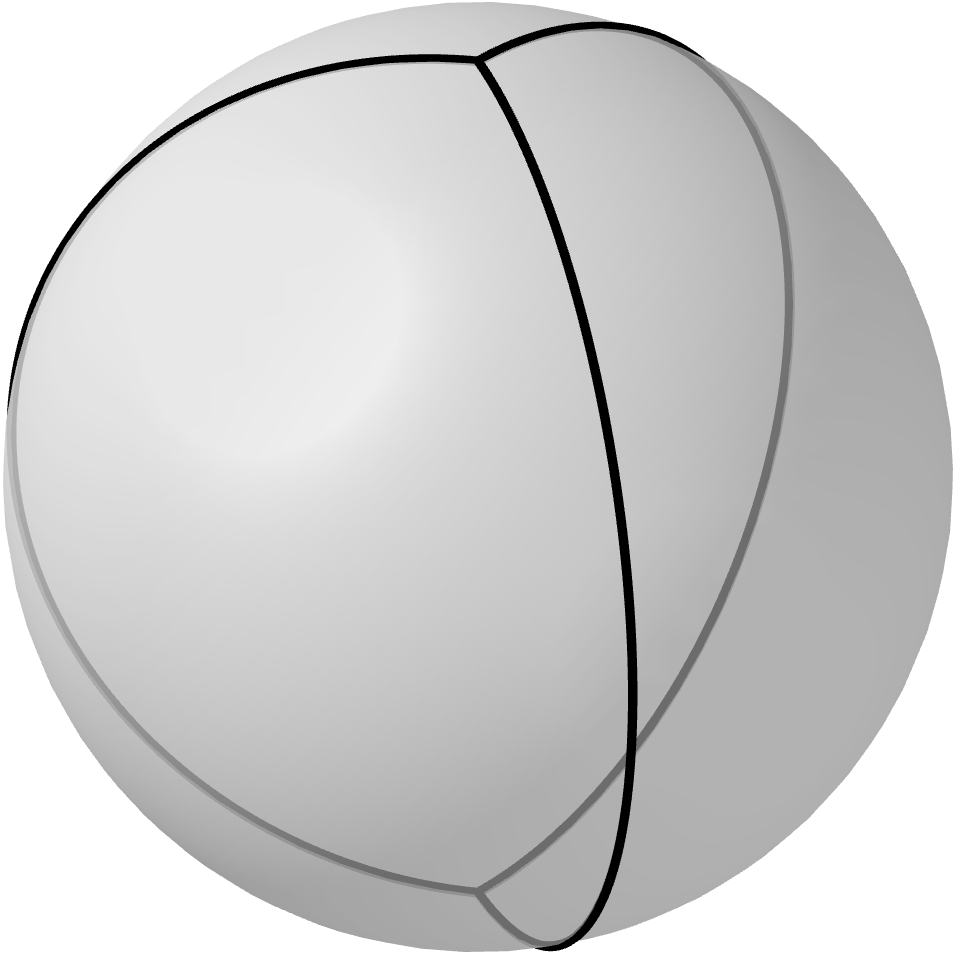}};

\node [inner sep=0] (image) at (2*\XS,0) 
            {\includegraphics[height=\h cm]{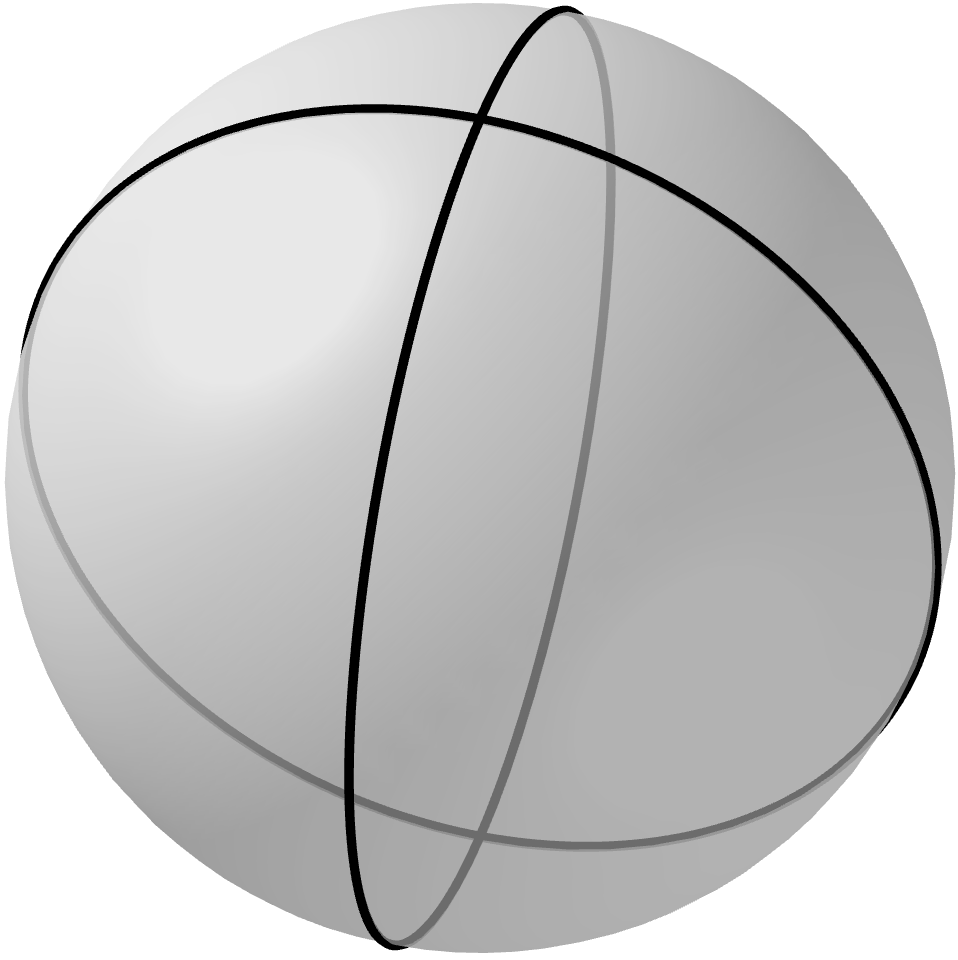}};

\node [inner sep=0] (image) at (3*\XS,0) 
            {\includegraphics[height=\h cm]{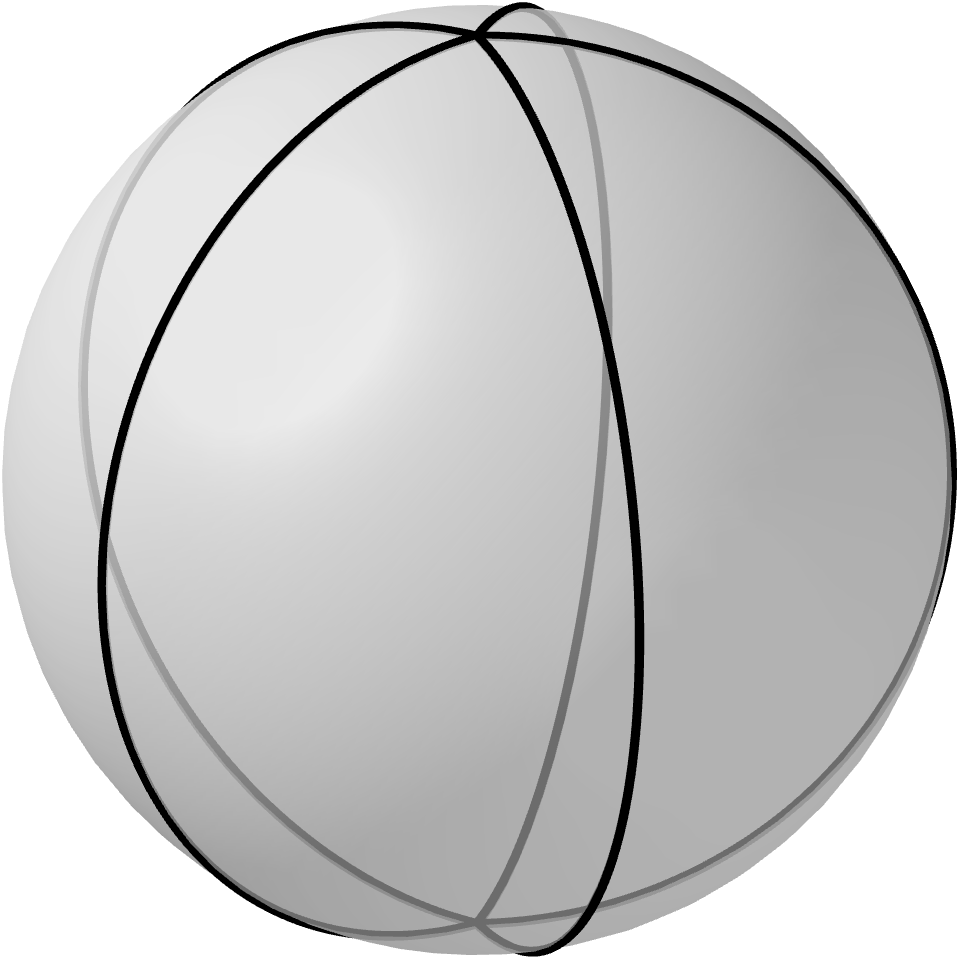}};

\node [inner sep=0] (image) at (4*\XS,0) 
            {\includegraphics[height=\h cm]{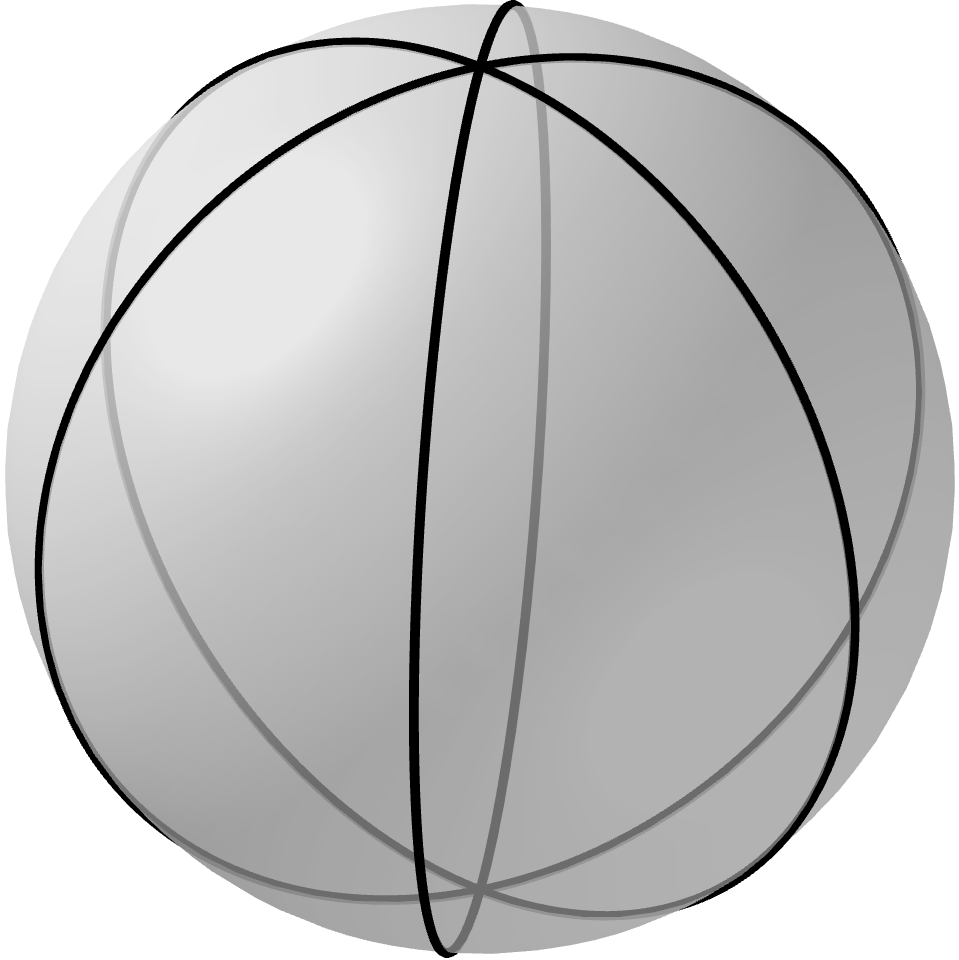}};
\end{scope}

\begin{scope}[yshift=-\YS cm]

\node [inner sep=0] (image) at (0,0) 
            {\includegraphics[height=\h cm]{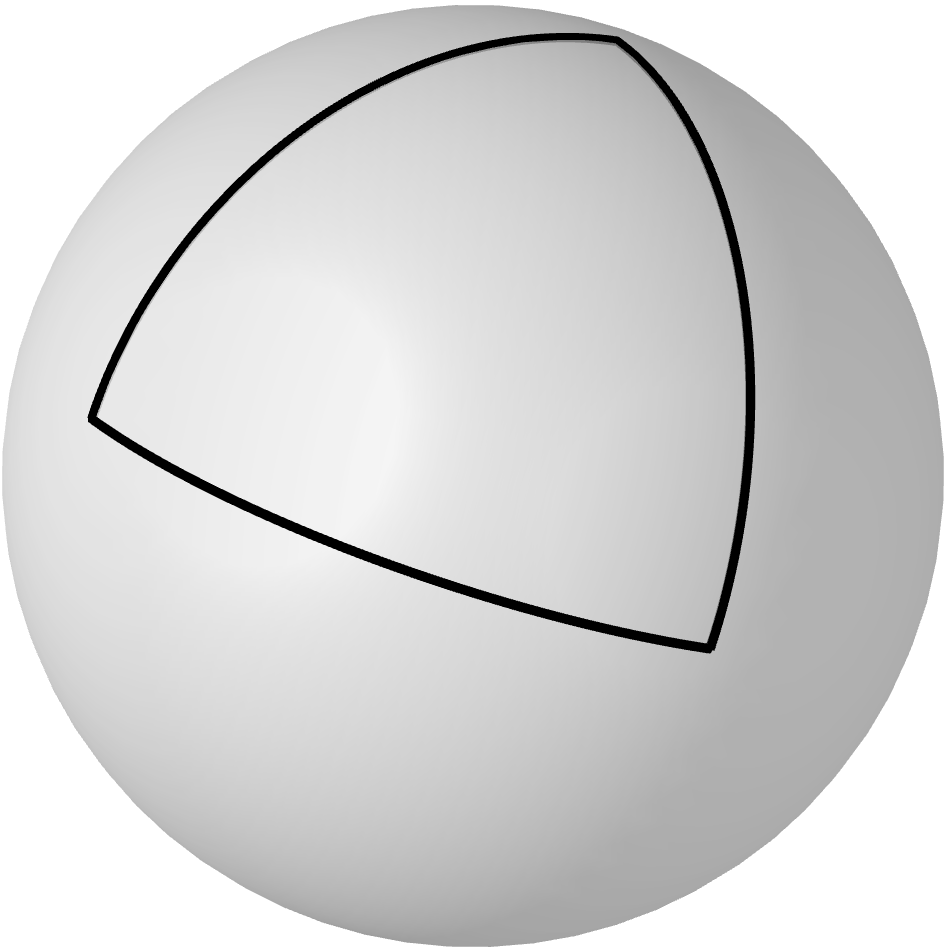}};

\node [inner sep=0] (image) at (\XS,0) 
            {\includegraphics[height=\h cm]{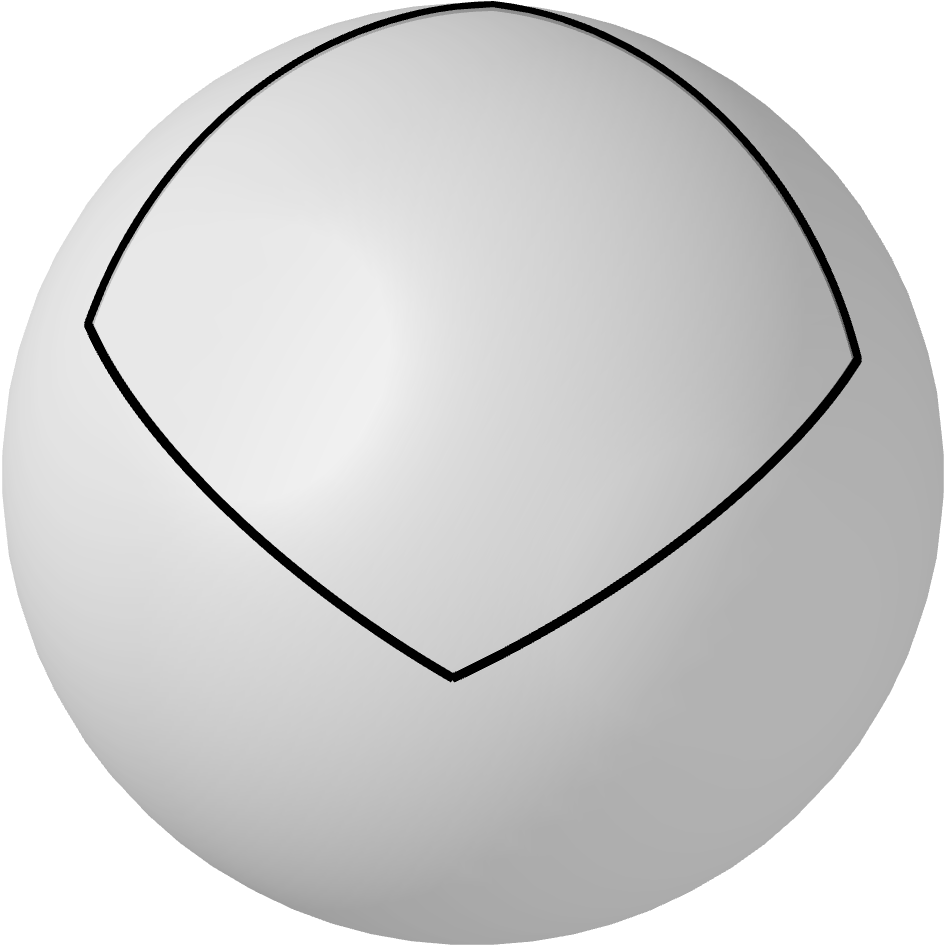}};

\node [inner sep=0] (image) at (2*\XS,0) 
            {\includegraphics[height=\h cm]{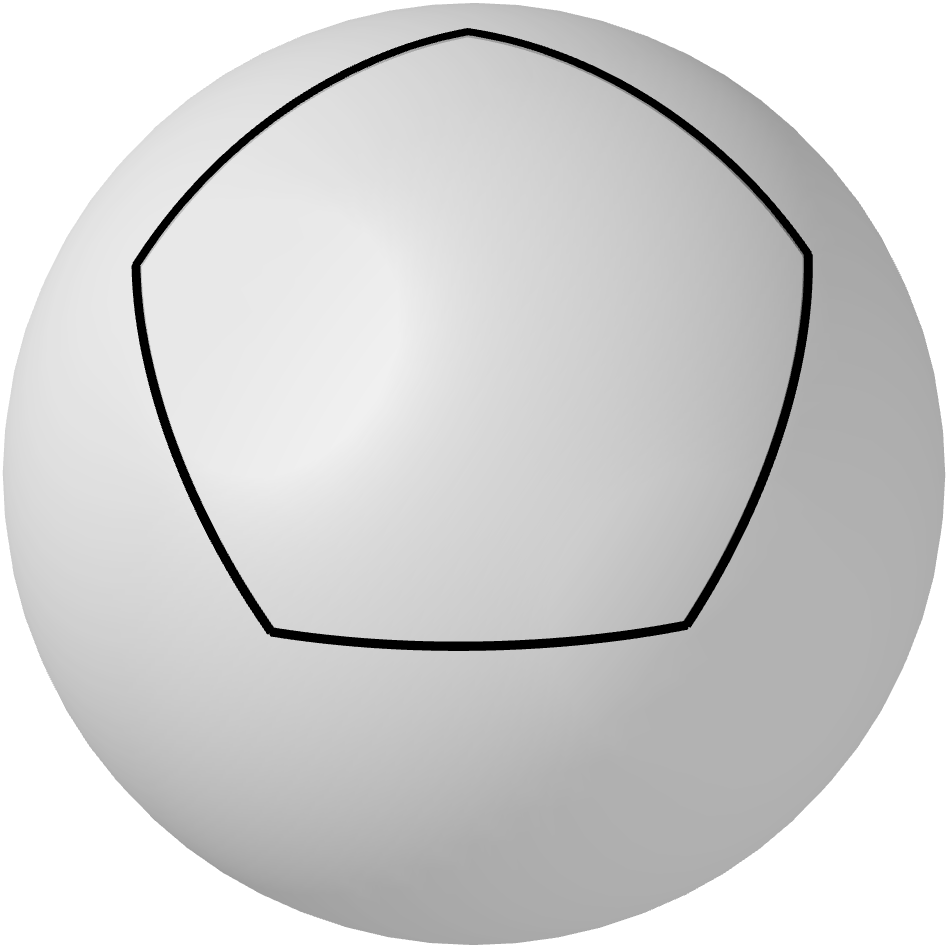}};

\node [inner sep=0] (image) at (3*\XS,0) 
            {\includegraphics[height=\h cm]{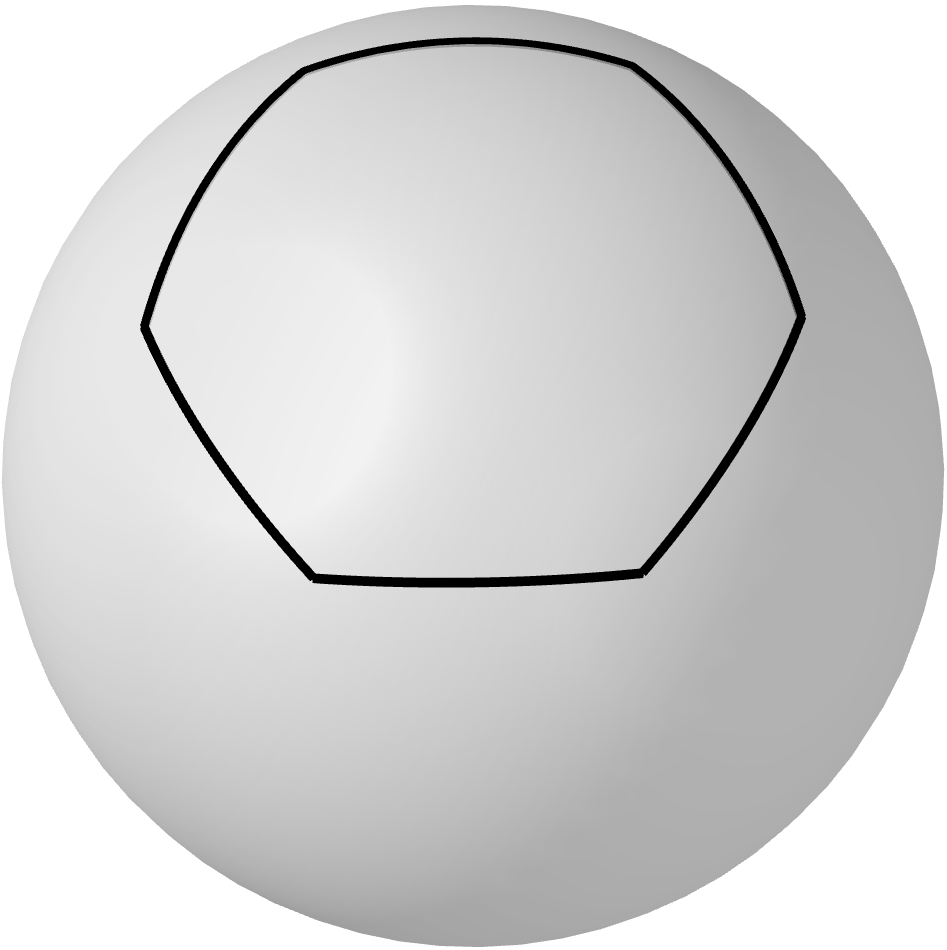}};

\node [inner sep=0] (image) at (4*\XS,0) 
            {\includegraphics[height=\h cm]{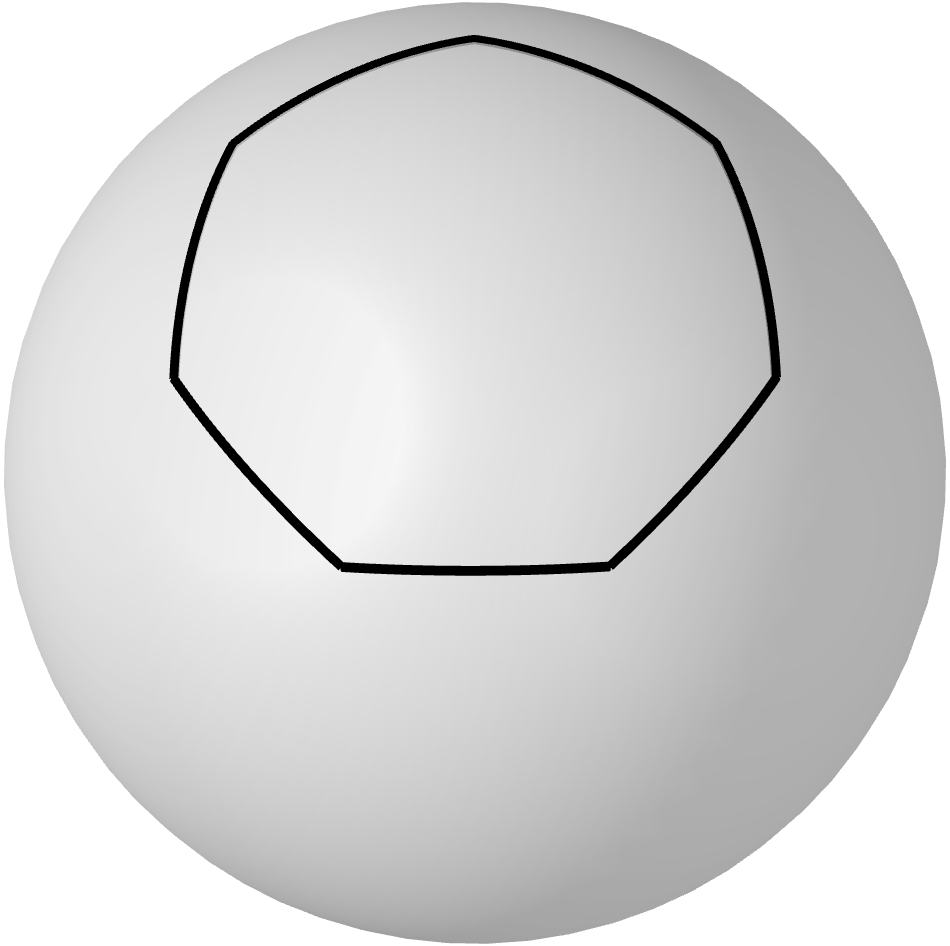}};

\end{scope}

\end{tikzpicture}
\caption{Hosohedra and dihedra}
\end{figure}

\fi


Not only is our theorem proved independently of the Johnson-Zalgaller
classification, but neither convexity of the tiles nor polyhedrality
of the underlying graph is assumed. The global structure is easily
identifiable with that of the solids with regular faces; we abuse
terminology and identify each spherical tiling with its corresponding
solid\footnote{In particular, we refer to the tilings corresponding to
Johnson solids as {\em Johnson tilings.}}. In effect, we
rediscover this structure in the context of spherical geometry.

We note that, armed with the Johnson-Zalgaller classification and the
additional assumptions of convexity and polyhedrality,
one can check
which of the Johnson solids is
circumscribable and obtain a subset of the tilings in our main
theorem. This is part of the approach used to prove Theorem 1.5 in
\cite{ahs} which is unpublished at the time of writing.
Their assumptions exclude hosohedra, dihedra and the following Johnson
solids which have a concave tile: $J_2$, $J_4$ and $J_5$.

We have created interactive 3-dimensional models of the tilings which
can be accessed via the following link: 
\url{https://www.geogebra.org/m/tnzzuugq}. Various data associated
with the tilings, including the length of edges and size of angles, is
presented in Tables~\ref{tab:standalone}, \ref{tab:eD} and
\ref{tab:OIaCeCaD} in Appendix~\ref{app:data}. 

We remark that spherical tilings by
regular polygons often appear as structural models of molecules in
the physics of hard materials, inorganic and organic chemistry and in
microbiology (for example, \cite{gtx,wsw,wjb}) which further motivates
the investigation of the properties of these exceptional structures.   

The rest of the paper is organised as follows. In the next section, we
recall and define terminology and notation, including some basic facts
of spherical geometry. We also prove several preliminary results and
technical lemmas. In Section~\ref{Sec-Tilings}, we prove our main
result.

\section{Preliminaries}

{\bf Spherical tilings.} A \emph{spherical tiling}, or tiling for
short, is a tiling of the sphere in which the surface is
partitioned by great arcs into bounded regions called \emph{spherical
polygons}. In what follows we always assume that the underlying
surface of a tiling is the unit sphere and we exclusively
consider edge-to-edge tilings. Convexity and polyhedrality are not
assumed. The one-skeleton of a tiling is a $2$-connected planar graph
(not necessarily polyhedral) and the two-skeleton projects (by a
stereometric projection) to a $2$-connected
plane graph. A tiling is \emph{convex} if all its polygons are convex,
and it is \emph{strictly convex} if it is convex and the boundary of any
polygon is not formed by a great circle. A stereometric projection
from the centre of a convex polyhedron contained in a unit ball
determines a spherical tiling. Vice-versa, a convex tiling $T$ with a
polyhedral underlying graph determines a polyhedron $P$ which projects
to $T$. 

\medskip 
\noindent
{\bf Angle-valued function.} A given spherical tiling $T$ determines a
function $\Phi$ associating to each angle its value from the interval
$(0,2\pi)$. If $T$ is convex $\Phi$ takes values from $(0,\pi]$. It
transpires that $T$ can be reconstructed from its associated plane
graph $G=G(T)$ and $\Phi$. In what follows we will use this fact
implicitly by drawing the plane graph $G(T)$ and labeling its angles
by angle-values. Two tilings $T_1$ and $T_2$ are isomorphic if and
only if there is a map isomorphism $G(T_1)\to G(T_2)$ preserving the
angle-values. Thus the pair $(G(T),\Phi)$ determines the tiling $T$ up
to isometry.
 
\medskip
\noindent
{\bf Regular spherical polygons.} An equilateral polygon on the
sphere with all angles equal will be called \emph{regular}. 
We denote by $\al_m$ the value of the angle associated to each vertex
of a regular $m$-gon.
A regular $m$-gon is \emph{convex}
if $\alpha_m\leq\pi$ and it is \emph{strictly convex} if
$\alpha_m<\pi$. The boundary of an $m$-gon is a great circle if and
only if it is a hemisphere, which is equivalent to $\alpha_m = \pi$.

A spherical polygon with two sides is called a {\em digon} and
a tiling of a sphere by digons is called a {\em hosohedron}. 

\begin{prop}
		\label{Prop-digon}
		Every tiling of the sphere by regular polygons having at least one digon is a
		hosohedron.
\end{prop}

\begin{proof}
		A regular spherical polygon is a digon if and only if it has edge length
		$\pi$ on the unit sphere. Therefore, if there is at least one
		digon in the tiling, then all the tiles are digons.

		The given digon shares an edge (or two) and both vertices with a
		digon. If there are just two tiles, we are done. Otherwise, the
		new digon also shares an edge and both vertices with a third digon
		and so on. Repeating this argument gives the result.
\end{proof}

A tiling with exactly two tiles is a {\em dihedron}. Such a tiling has
vertices of degree 2. We show that dihedra are the only tilings with
such vertices.

\begin{prop}
		\label{Prop-deg2}
		Every tiling of the sphere by regular polygons having at least one vertex of
		degree 2 is a dihedron.
\end{prop}

\begin{proof}
		Let $\al$ and $\al'$ be the angles of the polygons incident to the
		given vertex, and note that $\al+\al' = 2\pi$.
		Its neighbour is incident to the same two polygons; it cannot be
		incident to a third, otherwise the sum of the angles incident with
		this vertex is larger than $2\pi$.
		Repeating this argument, we observe that every vertex of the
		tiling is of degree 2, hence the result.
\end{proof}

For the rest of the paper, we assume that the degree of every vertex
and the number of edges in every polygon is at least 3. We remark that
hosohedra and dihedra are not realisable as solids.

For $\alpha_m < \pi$, Figure \ref{Fig-reg-mgon-trin} shows the regular
$m$-gon being triangulated into $m$ isosceles triangles with base edge
$x$ and side edges $r$, where $r$ denotes the radius from its centre
to two adjacent vertices with $\alpha = \alpha_m$. 

\begin{figure}[h!]
\centering
\begin{tikzpicture}[>=latex]
\tikzmath{
\r=2;
\th=360/8;
}

\foreach \a in {0,1,2,3,4} {
\tikzset{rotate=\a*\th}
\draw[]
	(90+1.5*\th:\r) -- (90+2.5*\th:\r)
;
}

\foreach \aa in {-1,1} {
\tikzset{xscale=\aa}
\draw[dashed]
	(0,0) -- (270+0.5*\th:\r)
;

\draw[gray!30, dashed]
	(0,0) -- (270+1.5*\th:\r)
;

\node at (270+0.3*\th:0.8*\r) {\footnotesize $\tfrac{1}{2}\alpha$};
\node at (270+0.75*\th:0.8*\r) {\footnotesize $\tfrac{1}{2}\alpha$};

\node at (270+1.5*\th:0.85*\r) {\small $\alpha$};

\node at (270+0.7*\th:0.5*\r) {\small $r$};
\node at (270+1.7*\th:0.5*\r) {\small {\color{gray!50}$r$}};

\node at (270+1*\th:1.1*\r) {\small $x$};
\node at (270+2*\th:1.1*\r) {\small $x$};
}

\node at (0,0.15*\r) {\footnotesize centre};
\node at (270:1.1*\r) {\small $x$};
\node at (270:0.4*\r) {\footnotesize $\tfrac{2}{m}\pi$};

\end{tikzpicture}
\caption{Triangulation of a regular $m$-gon with angles $\alpha:=\alpha_m$ and edges $x$}
\label{Fig-reg-mgon-trin}
\end{figure}
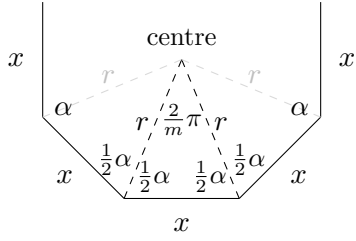

The spherical cosine law for angles on the isosceles triangle in
Figure \ref{Fig-reg-mgon-trin} gives
\begin{align}
		\label{Eq-isos-cosal}
		\cos \tfrac{1}{2}\alpha_m &= - \cos \tfrac{1}{2} \alpha_m \cos
		\tfrac{2}{m}\pi + \sin \tfrac{1}{2}\alpha_m \sin \tfrac{2}{m}\pi
		\cos r, \\ 
		\label{Eq-isos-cos2/m}
		\cos \tfrac{2}{m}\pi  &= -\cos^2
		\tfrac{1}{2}\alpha_m + \sin^2 \tfrac{1}{2}\alpha_m \cos x.
\end{align}

Then \eqref{Eq-isos-cos2/m} and $\cos^2\tfrac{1}{2} \alpha_m =
\frac{1}{2}(1+\cos \alpha_m)$ further imply
\begin{align}
		\label{Eq-x-al}
		\tfrac{1}{2}(1+\cos x) (1+\cos \alpha_m) = \cos x - \cos
		\tfrac{2}{m}\pi. 
\end{align}

Meanwhile, by \eqref{Eq-isos-cosal} and \eqref{Eq-isos-cos2/m}, we get
\begin{align} 
		\label{Eq-cosr}
		&\cot \tfrac{1}{2} \alpha_m =
		\frac{\sin\frac{2}{m}\pi}{1+\cos\frac{2}{m}\pi} \cos r, \\
		\label{Eq-cosx} 
		&\cos x = \cot^2 \tfrac{1}{2}\alpha_m + 
		\frac{\cos \frac{2}{m}\pi}{\sin^2 \frac{1}{2}\alpha_m}.
\end{align}

\begin{lem}
		\label{Lem-ang-asc-seq} 
		If $\al_m$ and $\al_n$ are the respective angles of strictly
		convex polygons having the same edge length, then $m<n$ if and
		only if $\al_m < \al_n$.
\end{lem}

\begin{proof} 
		Let $x$ be the edge length. If $m<n$ then $\cos x - \cos
		\frac{2}{m}\pi > \cos x - \cos \frac{2}{n}\pi$. By
		\eqref{Eq-x-al}, we have $\frac{1}{2}(1
		+ \cos x) ( 1 + \cos \al_m) > \frac{1}{2}(1 + \cos x) (1 + \cos
		\al_m)$ which implies $\cos \al_m > \cos \al_n$; since
		$\al_m,\al_n<\pi$, we deduce that $\al_m <
		\al_n$ as required.
\end{proof}

The following lemma is easy but useful.

\begin{lem} 
		The angles $\alpha_m$ and $\alpha_n$ of a regular $m$-gon and a
		regular $n$-gon with the same edge length satisfy
		\begin{align}
				\label{Eq-alm-aln-linear}
				(1-\cos \alpha_n)(1 + \cos \alpha_m + 2\cos \tfrac{2}{m}\pi)
				= 
				(1-\cos \alpha_m)(1 + \cos \alpha_n + 2\cos \tfrac{2}{n}\pi).
		\end{align}
\end{lem} 

\begin{proof}
		For a regular $m$-gon and a regular $n$-gon with the same edge
		length $x$, (\ref{Eq-cosx}) gives
		\begin{align}
				\label{Eq-alm-aln}
				\cot^2 \tfrac{1}{2}\alpha_m +
				\frac{\cos \frac{2}{m}\pi}{ \sin^2 \frac{1}{2}\alpha_m }
				=
				\cot^2 \tfrac{1}{2}\alpha_n + 
				\frac{\cos \frac{2}{n}\pi}{ \sin^2 \frac{1}{2}\alpha_n }. 
		\end{align}
		By substituting $\cos^2\frac{1}{2}\theta = \tfrac{1}{2}(1+\cos
		\theta)$ and $\sin^2\frac{1}{2}\theta = \tfrac{1}{2}(1-\cos
		\theta)$, we simplify \eqref{Eq-alm-aln} to obtain the desired
		result.
\end{proof}

Since $2\pi - \alpha_n$ and $2\pi - \alpha_m$ satisfy the above
equation, the exterior of every convex $m$-gon (respectively $n$-gon)
with angles $\alpha_m$ (respectively $\alpha_n$) is a complementary
concave $m$-gon (respectively $n$-gon) with angles $2\pi - \alpha_m$
(respectively $\alpha_n$) satisfying the same equation. 

\begin{lem}
		\label{Lem-cong} 
		In an edge-to-edge tiling by regular polygons all $m$-gons for a
		fixed $m\geq 3$ are congruent. In particular, the tiling contains
		at most one regular polygon which is not strictly convex.
\end{lem}

\begin{proof}
		For a fixed $m$ and $\tfrac{1}{2}\alpha_m, r \in (0,\pi]$, the
		equation \eqref{Eq-cosr} implies that $\alpha_m$ is increasing if
		and only if $r$ is increasing. Moreover,
		$\frac{\sin\frac{2}{m}\pi}{1+\cos\frac{2}{m}\pi} \neq 0$ in
		\eqref{Eq-cosr} means $\tfrac{1}{2}\alpha_m = \tfrac{1}{2}\pi$ if
		and only if $r = \tfrac{1}{2}\pi$. When the equalities hold, the
		$m$-gon is a hemisphere. Therefore, a regular $m$-gon with angles
		$\ge\pi$ contains a hemisphere and hence there can be at most one
		such $m$-gon. Since the edge-length $x$ is fixed, equations
		(\ref{Eq-cosr}) and (\ref{Eq-cosx}) imply that given $m$ the
		tiling cannot have both strictly convex and concave regular
		$m$-gon. If all the regular $m$-gons are strictly convex, then the
		statement follows from Lemma~\ref{Lem-ang-asc-seq}.
\end{proof}
 
\medskip
\noindent
{\bf Dehn-Sommerville formulae.} In a tiling, let $f_{m}$ denote the
number of $m$-gons and $v_k$ the number of vertices of degree $k$; let
$v$, $e$ and $f$ denote the total numbers of vertices, edges and faces
respectively. Note that $f = \sum_{m\ge3} f_m$ and $v = \sum_{k\ge3}
v_k$. Recall Euler's polyhedral formula and the Dehn-Sommerville
formulae \cite{DS} below,
\begin{align}
\label{Eq-Euler}
&v-e+f=2, \\
\label{Eq-e-vk}
&2e = \sum_{k\ge3} kv_k, \\
\label{Eq-e-fm}
&2e=\sum_{m\ge3} m f_{m}.
\end{align}
The above gives a proof of the following well-known fact. 

\begin{lem}
		\label{Lem-n345-deg345}
		In an edge-to-edge spherical tiling by polygons with at least
		three edges and with minimum
		vertex degree at least $3$, 
		\begin{enumerate}
				\item there is a triangle, a quadrilateral or a pentagon;
				\item if there is a triangle, then there is a vertex of degree
						$3, 4$ or $5$; otherwise there is a vertex of degree $3$.
		\end{enumerate}
\end{lem}

\begin{proof}
		Let $n \ge 3$ be the minimal integer such that there is at least
		one $n$-gon in the given tiling.
		By \eqref{Eq-e-fm}, we get $f\le\tfrac{2}{n}e$.
		Substituting it and \eqref{Eq-e-vk} into \eqref{Eq-Euler} gives
		\begin{align*}
				2 = v - e + f \le v - (\tfrac{1}{2} - \tfrac{1}{n}) 2 e =
				\tfrac{1}{2n} \sum_{k\ge3} ( 2n - nk + 2 k )v_k. 
		\end{align*}
		Since the left-hand side is positive and $v_k$ is non-negative for
		every $k$,
		the coefficient $(2-k)n+2k$ is positive for some $n$ and $k$.

Since $k\ge3$ and $(2-k)n+2k>0$, we have $2n \ge 3n-6$, i.e., $n<6$.
Hence $n = 3,4$ or $5$, which gives the first statement of the lemma. 

For the second statement, $n\ge3$ and
$(2-k)n+2k>0$ implies $k<6$; if $n\ge4$, then the inequality implies
$k<4$, meaning that there is a degree $3$ vertex.
\end{proof}

\medskip \noindent {\bf Vertices in tilings by regular polygons.}
Since each tile is a {\em regular} polygon and the considered tilings
are edge-to-edge, by Lemma~\ref{Lem-cong} all $m$-gons in the
tiling are congruent. In particular, the angle incident to a vertex is
determined by the number of edges $m\geq 3$ of a polygon it belongs
to. In what follows such an angle will be denoted $\alpha_m$. Note
that for a fixed $m\geq 3$ the angle $\alpha_m$ depends on the
considered tiling. For a vertex $v$ of degree $d\geq 3$ in a tiling by
regular polygons we define its \emph{type} to be a multiset of size
$d$ containing the angle-values of the $d$ angles incident to $v$. We
denote a vertex type by a contracted form of the usual multiset
notation; for instance, a vertex is of type
$\alpha_l^a\alpha_m^b\alpha_n^c$ if it is incident to $a$ angles of
size $\alpha_l$, $b$ angles of size $\alpha_m$ and $c$ angles of size
$\alpha_n$. We define the {\em angle sum} of such a type to be
$a\alpha_l+b\alpha_m+c\alpha_n = 2\pi$. 

With some abuse of notation we
often identify a vertex with its type. Note that the type of a vertex
does not specify the arrangements of the angles (equivalently the
incident tiles). An expanded version is used to serve that purpose.
For example, up to rotation and reflection,
$\alpha_3\alpha_4^2\alpha_5$ has two distinct angle arrangements
$\alpha_3\alpha_4\alpha_5\alpha_4,\alpha_3\alpha_4\alpha_4\alpha_5$ as
shown in Figure \ref{Fig-vertex-arrangement}. We emphasise that two
vertices have the same angle arrangement if they are related by
reflection or rotation; for example $\al_5\al_4\al_4\al_3$ is the same
angle arrangement as $\al_3\al_4\al_4\al_5$.

\begin{figure}[h!]
\centering
\begin{tikzpicture}
\tikzmath{
\XS=5;
\r=1;
\n=4;
\nn=\n-1;
\th=360/\n;
}

\foreach \xs in {0,1} {
\tikzset{xshift=\xs*\XS cm}

\foreach \a in {0,...,\nn} {
\tikzset{rotate=\a*\th}
\draw[]
	(0,0) -- (90:\r) 
;
}

\draw[]
	(\r,0) -- (0,\r)
	(0,\r) -- (-\r,\r)
	(-\r,\r) -- (-\r,0)
;

}

\begin{scope} 

\draw[]
	(0,-\r) -- (\r,-\r)
	(\r,-\r) -- (\r,0)
	(-\r,0) -- (-1.5*\r, -0.5*\r)
	(0,-\r) -- (-0.5*\r, -1.5*\r)
	(-1.5*\r, -0.5*\r) -- (-0.5*\r, -1.5*\r)
;

\node at (0.3*\r, 0.2*\r) {\small $\alpha_3$};
\node at (-0.3*\r, 0.2*\r) {\small $\alpha_4$};
\node at (0.3*\r, -0.25*\r) {\small $\alpha_4$};
\node at (-0.3*\r, -0.25*\r) {\small $\alpha_5$};

\node at (0,-2*\r) {\small $\alpha_3\alpha_4\alpha_5\alpha_4$};

\end{scope} 

\begin{scope}[xshift=\XS cm] 

\draw[]
	(0,-\r) -- (-\r,-\r)
	(-\r,-\r) -- (-\r,0)
	(\r,0) -- (1.5*\r, -0.5*\r)
	(0,-\r) -- (0.5*\r, -1.5*\r)
	(1.5*\r, -0.5*\r) -- (0.5*\r, -1.5*\r)
;

\node at (0.3*\r, 0.2*\r) {\small $\alpha_3$};
\node at (-0.3*\r, 0.2*\r) {\small $\alpha_4$};
\node at (0.3*\r, -0.25*\r) {\small $\alpha_5$};
\node at (-0.3*\r, -0.25*\r) {\small $\alpha_4$};

\node at (0,-2*\r) {\small $\alpha_3\alpha_4\alpha_4\alpha_5$};

\end{scope} 

\end{tikzpicture}
\caption{The two angle arrangements $\alpha_3\alpha_4\alpha_5\alpha_4, \alpha_3\alpha_4\alpha_4\alpha_5$ of $\alpha_3\alpha_4^2\alpha_5$}
\label{Fig-vertex-arrangement}
\end{figure}
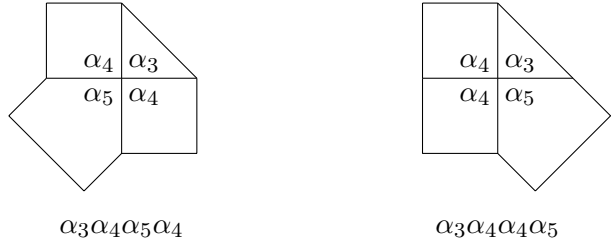

A vertex with partial incident angle information is represented by a
partial type. For example, a vertex having at least $a$ copies of
$\alpha_m$ and $b$ copies of $\alpha_n$ as incident angles is denoted
by $\alpha_m^a\alpha_n^b\cdots$.  
We define the {\em remainder} of a partial type to be the difference
between $2\pi$ and the sum of the incident angles and denote this
function by $R$; for example, the
remainder of $\al_m^a\al_n^b\cdots$ is $R(\al_m^a\al_n^b) = 2\pi - a\al_m
- b\al_n$.

\medskip
\noindent
{\bf Admissible set of vertices.} A set $S$ of types of vertices will
be called \emph{admissible} if there exists a spherical tiling $T$ by
regular polygons such that $S$ is the set of types of vertices of $T$.
The knowledge of the vertex types that can appear in a tiling is the
key to our classification. In this subsection we restrict the types
which can appear in an admissible set. 

The geometry of the sphere and the number of sides in a regular
polygon impose various constraints on the incident angle combinations
at a vertex. Subject to these constraints, we use \quotes{$=$} to
express the list of all possible full or partial vertices for a given
partial vertex. The list can be refined recursively. For example,
$\alpha_m^2\cdots = \alpha_m^3$ denotes that the only possible vertex
for $\alpha_m^2\cdots$ is $\alpha_m^3$. Since the angle sum is
always $2\pi$, the size of an angle limits the number of its
appearance at a vertex. In particular, the sum of the angles of a regular
$m$-gon with angles $\alpha_m$ is $m\alpha_m > (m-2)\pi$. Hence
\begin{align}
		\label{Ineq-alm-lb}
		\alpha_m > (1-\tfrac{2}{m})\pi.
\end{align}

\begin{lem}
		\label{lem:degreebound}
		The degree of a vertex in a tiling by regular polygons belongs to
		the set $\{3,4,5\}$.
\end{lem}

\begin{proof}
		Let $a_m\geq 0$ be the number of $m$-gons incident to a vertex of
		degree $d$. Then 
		$$ 
		2\pi = \sum_{m\geq 3} a_m\alpha_m> \sum_{m\geq
		3}a_m(1-\tfrac{2}{m})\pi\geq \tfrac{d}{3} \pi.
		$$ 
		It follows that $d<6$.
\end{proof}

The inequality $2\pi> \sum_{m\geq 3}a_m(1-\tfrac{2}{m})\pi$ further
restricts the sequence $\{a_m\}_{m\geq 0}$. Thus one can refine
Lemma~\ref{lem:degreebound} to determine the possible vertices in a
tiling as follows. The vertices incident to a
triangle
are
\begin{align}
		\label{List-deg3-al3}
		&\deg 3:& &\alpha_3\cdots = \alpha_3^3, \ \alpha_3^2\alpha_{m}, \
		\alpha_3\alpha_{m}^2, \ \alpha_3\alpha_m\alpha_n; \\
		\label{List-deg4-al3}
		&\deg 4:& &\alpha_3\cdots = \alpha_3^4, \ \alpha_3^3\alpha_{m}, \
		\alpha_3^2\alpha_{m}^2, \  \alpha_3^2\alpha_m\alpha_n, \
		\alpha_3\alpha_4^3, \ \alpha_3\alpha_4^2\alpha_5; \\
		\label{List-deg5-al3}
		&\deg 5:& &\alpha_3\cdots = \alpha_3^5, \ \alpha_3^4\alpha_4, \
		\alpha_3^4\alpha_5.
\end{align}


Similarly, the vertices without an incident triangle are 
\begin{align}
\label{List-al4}
&\alpha_4\cdots=\alpha_4^3, \, \alpha_4^2\alpha_{m},  \, \alpha_4\alpha_m^2 (m\le 7),  \, \alpha_4\alpha_m\alpha_n (m \le 7, m < n \le 19); \\
\label{List-al5}
&\alpha_5\cdots=\alpha_5^3, \, \alpha_4\alpha_5^2,  \, \alpha_5^2\alpha_{m}(m\le9), \, \alpha_5\alpha_6^2, \alpha_5\alpha_6\alpha_7.
\end{align}

We sum up the above in the following lemma.

\begin{lem}
		\label{Lem-all-vertices}
		A vertex in a tiling is one of those in \eqref{List-deg3-al3},
		\eqref{List-deg4-al3}, \eqref{List-deg5-al3}, \eqref{List-al4},
		\eqref{List-al5}. Notably, a vertex has an incident triangle,
		square or pentagon.
\end{lem}

\medskip
\noindent
{\bf Vertex homogeneity.}
Tilings in which every vertex has the same angle arrangement play an
important role in several of our arguments.
We call a
tiling that satisfies this local criterion {\em (strongly)
vertex-homogenous}. We adopt this terminology due to the inconsistent
use of the terms {\em semi-regular} and {\em Archimedean} in the
literature. Both of these are sometimes used to refer to objects
satisfying the global criterion of vertex-transitivity, excluding the
prisms, anti-prisms and Platonic solids and tilings by fiat.
Furthermore, the sets of objects satisfying these criteria do not
coincide: the elongated square gyrobicupola, otherwise known as
$J_{37}$, satisfies the local but not the global criterion. Several
authors have made errors as a result of this issue; the interested
reader is directed to \cite{G08}.

We use the term Archimedean to refer to the tilings corresponding to
the vertex-transitive solids, excluding the prisms, anti-prisms and
Platonic solids. There are thirteen such tilings and they are depicted
in Figure~2. We sometimes make use of Conway's notation to refer to
these objects (see Chapter 21 of \cite{bcg}). The notation is given in
Table~\ref{tab:conway} in Appendix~\ref{app:data}.

The vertex-homogenous tilings were characterised by Sommerville in
\cite{DS2}, albeit using different language; that he does not make the
previously mentioned error is corroborated by Gr\"unbaum in
\cite{G08}. We state the result using our own terminology below.

\begin{prop}
		\label{Prop-vertex-homog}
		A vertex-homogenous tiling is either a Platonic tiling,
		an Archimedean tiling, the Johnson tiling $J_{37}$, a prism or an anti-prism. 
\end{prop}

The possible angle arrangements for a vertex-homogenous
tiling are as follows:
\begin{itemize}
		\item
		$\al_3^3,
		\al_4^3, \al_4^4, \al_3^5 \text{ or } \al_5^3$ if the tiling is
		Platonic; 
		\item
				$\al_3\al_6^2, \al_3\al_4\al_3\al_4, \al_3\al_8^2,
				\al_4\al_6^2, \al_3\al_4^3, \al_4\al_6\al_8, \al_3^4 \al_4,$\\
				$ \al_3\al_5\al_3\al_5, \al_3\al_{10}^2, \al_5\al_6^2,
				\al_3\al_4\al_5\al_4, \al_4\al_6\al_{10} \text{ or }
		\al_3^4\al_5$ \\ if the tiling is Archimedean; 
		\item
				$\al_4^2 \al_m$ or
		$\al_3^3\al_m$ for $m \ge 4$ if the tiling is respectively a prism
		or anti-prism; or 
\item 
		$\al_3\al_4^3$ if the tiling is $J_{37}$.
\end{itemize}

The following useful lemma gives a sufficient condition for a tiling
to be vertex-homogenous.

\begin{lem}
		\label{Lem-vertex-transitive} 
		If a tiling has a vertex of type
		$\al\beta\gamma$ and $\alpha\beta\dots = \alpha\gamma\dots =
		\beta\gamma\dots = \alpha\beta\gamma$, then the tiling is
		vertex-homogenous.
\end{lem}

\begin{proof}
		A neighbour of a vertex of type $\alpha\beta\gamma$ is of type
		$\alpha\beta\cdots$, $\alpha\gamma\cdots$
		or $\beta\gamma\cdots$; that is, such a neighbour is also of type
		$\alpha\beta\gamma$ by assumption.
		Repeating this argument, we deduce that every vertex is of the
		same type, hence the result.
\end{proof}

\medskip\noindent
{\bf Subdivisions of tilings.}
Several tilings can be {\em subdivided} to generate new tilings. We
introduce three types of subdivision and prove that they preserve
regularity of the tiles. If a tiling has an $m$-gon and $\al_m =
2\al_3$, then we add a vertex in the centre of the $m$-gon and add
edges from the vertices of the $m$-gon to the central vertex to obtain
a new tiling; this operation is called a {\em pyramid subdivision}. If
a tiling has an $m$-gon for some even $m$ and $\al_m = \al_3 +\al_4$,
then we add a regular triangle to every other edge of the $m$-gon and
connect the new $\frac{m}{2}$ vertices in a cycle; we will prove that
the created edges do not cross and call this operation a 
{\em cupola subdivision} (see
Figure~\ref{fig:subdiv} for an example). We refer to the inverse of
this operation as {\em diminishing} a cupola.
If a tiling has an $m$-gon and $\al_m =
2\al_4$, 
then we add a square to each edge of the $m$-gon on its interior;
this operation is called a {\em prism subdivision}.
We now prove that these subdivisions can be performed in such a way
that the resulting tiling has regular tiles.

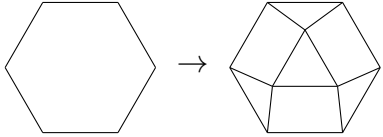
\begin{figure}[h!]
		\centering
		\begin{tikzpicture}
		\tikzmath{
				\r=1;
		}

		\foreach \a in {0,...,5}{
				\draw (\a*60:\r) -- (\a*60+60:\r);
		}		
\end{tikzpicture}
\begin{tikzpicture}
		\node at (0,0.75) {$\rightarrow$};
		\node at (0,0) {};
\end{tikzpicture}
\begin{tikzpicture}
		\tikzmath{
				\r=1;
				\rr=0.5;
		}

		\foreach \a in {0,...,5}{
				\draw (\a*60:\r) -- (\a*60+60:\r);
		}		

		\foreach \a in {0,...,2}{
				\draw (\a*120-30:\rr) -- (\a*120+90:\rr);
				\draw (\a*120:\r) -- (\a*120-30:\rr);
				\draw (\a*120+60:\r) -- (\a*120+90:\rr);
		}
		
\end{tikzpicture}
		\caption{Cupola subdivision for $m=6$}
		\label{fig:subdiv}
\end{figure}

\begin{lem}
		\label{lem:subdiv}
		If a tiling is obtained from a regular tiling by a pyramid
		subdivision, a cupola subdivision or a prism subdivision of an
		$m$-gon, then it is a regular tiling. 
\end{lem}

\begin{proof}

		Consider one of the triangles created by a pyramid subdivision. We
		claim that this triangle is regular. Indeed, it has one edge of
		length $x$, the edge length of the $m$-gon in the original tiling.
		The angles adjacent to this edge are both equal to $\tfrac{1}{2}\al_m =
		\al_3$, since the new vertex is central and the new edges
		subdivide the angle of the $m$-gon. By definition, $\al_3$ is the
		angle of a regular triangle of edge length $x$. By
		angle-side-angle, the triangle is determined; that is, it must be
		regular.
		
		A cupola subdivision creates triangles, quadrilaterals and an
		$\frac{m}{2}$-gon; we must prove that these are simple in the
		sense that their edges do not cross. Consider the triangle formed
		by the centre of the $m$-gon and one of its edges (note that $m
		\ge 6$ since $m$ is even and $\al_m > \al_4$). This isoceles
		triangle has two angles of size $\tfrac{1}{2}\al_m > \al_3$ adjacent to the
		edge of the $m$-gon. Therefore, the regular triangle on that edge
		is strictly contained in the isoceles triangle. Hence, the cycle
		created by the cupola subdivision is the boundary of a simple
		$\frac{m}{2}$-gon. By rotational symmetry, the $\frac{m}{2}$-gon
		is regular. The triangles are regular by definition. 
		A quadrilateral is uniquely determined by three edge lengths and
		the two angles between them; for each created quadrilateral we
		know that three of its edges are of length $x$ and the two angles
		between them are $\al_4$,
		that is, the angle of a regular quadrilateral of edge length $x$.
		We deduce that the created quadrilaterals are regular as required.

		The squares created by a prism subdivision are regular with edge
		length $x$ by definition. Since $\al_m = 2\al_4 > \pi$, the
		$m$-gon is concave; that is, it contains a hemisphere and $x <
		\frac{2}{m}\pi$. The distance from a vertex of the $m$-gon to its
		centre is strictly greater than $\frac{1}{2}\pi$. Therefore, each
		square is contained in an isoceles triangle formed by the centre
		of the $m$-gon and one of its edges as in the above argument.
		Hence, the interiors of any two of these squares do not intersect.
		The cycle of new vertices is the boundary of a second $m$-gon,
		which is regular by rotational symmetry.
\end{proof}

We complete this section by proving the following useful lemma.

\begin{lem}
		\label{Lem-small-remainder}
		If a tiling has a vertex of type $\al_3\al_\ell^i \al_m^j
		\al_n^k$, then $\al_\ell^i \al_m^j \al_n^k\cdots =
		\al_3\al_\ell^i \al_m^j \al_n^k$. If a tiling has no triangle
		and a vertex of type $\al_4\al_\ell^i \al_m^j \al_n^k$, then
		$\al_\ell^i \al_m^j \al_n^k\cdots = \al_4\al_\ell^i \al_m^j
		\al_n^k$.
		
\end{lem}

\begin{proof}
		In the first case, $R(\al_\ell^i \al_m^j \al_n^k) = \al_3
		< \al_p$ for $p>3$ by Lemma~\ref{Lem-ang-asc-seq}. Thus, if there
		is a vertex of
		type $\al_\ell^i \al_m^j \al_n^k\cdots$ other than
		$\al_3\al_\ell^i \al_m^j \al_n^k$, its angle sum cannot be
		equal to $2\pi$, hence the result. The proof for the triangle-free
		case is identical.
\end{proof} 

\section{Weakly Vertex-homogenous Tilings}

In this section we construct every possible tiling which is {\em weakly vertex-homogenous}, meaning that if every vertex is of the same type (but not necessarily of the same angle arrangement). This proves
the following theorem.

\begin{thm}
		\label{thm:weakly}
		A weakly vertex-homogenous tiling is either a Platonic tiling, an
		Archimedean tiling, one of the Johnson tilings $J_{27}$, $J_{34}$,
		$J_{37}$, $J_{72},\dots,J_{75}$, a prism or
		an anti-prism.
\end{thm}

To construct the tilings in Theorem \ref{thm:weakly}, we begin with the following lemma.

\begin{lem}
		\label{Lem-even-m} 
		Along the boundary of an $m$-gon, if its vertices have the same
		angle arrangement $\alpha_l \vert \alpha_m \vert \alpha_n$ for
		distinct labels $l, m, n$ and there is only one $\alpha_m$
		incident to each vertex, then $m$ is even.
\end{lem}

\begin{proof} 
		Without loss of generality, a vertex of the $m$-gon has an angle
		arrangement $\alpha_l \vert \alpha_m \vert \alpha_n$ as shown on
		the left of Figure \ref{Fig-even-mgon-bdry}. Then the hypothesis
		determines the angle arrangement in the next vertex and so on.
		This defines alternating labels $nn$ and $ll$ for the edges of the
		$m$-gon, meaning that the numbers of $l$-gons and $n$-gons along
		the boundary of the $m$-gon must be the same. Hence $m$ is even.

\begin{figure}[h!]
\centering
\begin{tikzpicture}
\tikzmath{
\XS=1.25;
\Y=1;
}

\foreach \xs in {0,1,2} {
\tikzset{xshift=\xs*\XS cm}
\draw[]
	(0,0) -- (\XS,0)
;
}

\foreach \aa in {-1,1} {
\tikzset{xshift=1.5*\XS cm, xscale=\aa}
\draw[]
	(0.5*\XS,0) -- (0.3*\XS,\Y)
	(0.5*\XS,0) -- (0.7*\XS,\Y)
;
\draw[]
	(1.5*\XS,0) -- (1.3*\XS,\Y)
	(1.5*\XS,0) -- (1.7*\XS,\Y)
;

\node at (0.25*\XS,0.25*\Y) {\small $\alpha_l$};
\node at (0.75*\XS,0.25*\Y) {\small $\alpha_n$};
\node at (1.25*\XS,0.25*\Y) {\small $\alpha_n$};
\node at (1.75*\XS,0.25*\Y) {\small $\alpha_l$};

\node at (0.5*\XS,-0.2*\Y) {\small $\alpha_m$};
\node at (1.5*\XS,-0.2*\Y) {\small $\alpha_m$};

\node at (2.5*\XS, 0.5*\Y) {$\cdots$};
}

\end{tikzpicture}
\caption{}
\label{Fig-even-mgon-bdry} \qedhere
\end{figure}
\end{proof}

If the polygons incident to a vertex are all congruent, then the
vertex must be of type $\al_3^3$, $\al_3^4$, $\al_3^5$, $\al_4^3$ or
$\al_5^3$ by Lemma~\ref{Lem-all-vertices}. It is straightforward to
verify that a weakly vertex-homogenous tiling with one of these types
must be Platonic.

The remaining vertices from the lists
\eqref{List-deg3-al3}--\eqref{List-al5} are:
\begin{align}
		\label{List-weakly}
&\alpha_3^2\alpha_{m}, \ \alpha_3\alpha_{m}^2, \
		\alpha_3\alpha_m\alpha_n, \\ \notag &\alpha_3^3\alpha_{m}, \
		\alpha_3^2\alpha_{m}^2, \  \alpha_3^2\alpha_m\alpha_n, \
		\alpha_3\alpha_4^3, \ \alpha_3\alpha_4^2\alpha_5, \\ \notag
																				&\alpha_3^4\alpha_4, \
		\alpha_3^4\alpha_5, \\ \notag &\alpha_4^2\alpha_{m},  \
		\alpha_4\alpha_m^2 (m\le 7),  \ \alpha_4\alpha_m\alpha_n (m \le 7,
		m < n \le 19); \\ \notag &\alpha_4\alpha_5^2, \
		\alpha_5^2\alpha_{m}(m\le9), \ \alpha_5\alpha_6^2,
		\alpha_5\alpha_6\alpha_7.
\end{align}

The vertex types of degree 5 in the above list are $\al_3^4\al_4$ and
$\al_3^4\al_5$. By vertex-to-vertex construction, it is easy to verify
that the weakly vertex-homogenous tilings with these types are,
respectively, the snub cube and the snub dodecahedron. We deal with
vertices of degree 3 and 4 separately in the sequel.

\subsubsection*{Degree 3 vertices}

Observe that weak and strong vertex-homogeneity coincide for tilings
with vertices of degree 3. We omit the modifiers from our terminology
when there is no ambiguity.

An immediate consequence of Lemma~\ref{Lem-even-m} is that if every
vertex in a tiling is of type $\al_\ell\al_m\al_n$, then either $\ell$
is even or $m = n$. 
This restricts the possible vertex types for vertex-homogenous tilings with
degree 3:
\begin{align*}
		&\al_3^3, \ \al_3\al_m^2 \text{ for even } m,\\
		&\al_4^2\al_\ell, \ \al_4\al_m\al_n \text{ for even } m \text{
		and } n, \\
		&\al_5^3, \ \al_5\al_m^2 \text{ for even } m.
\end{align*}

We depict a {\em shrinking} operation on a polygon with vertices of
degree 3 in Figure~\ref{Fig-shrink}. The inverse operation is called
{\em truncation}. 
Applying the shrinking operation to every triangle in a
vertex-homogenous tiling with vertices of type $\al_3\al_{2p}^2$ with
$p \ge 3$
yields a new vertex-homogenous tiling with vertices of type $\al_p^3$;
that is, the new tiling is Platonic. Thus, the original tiling is a
truncation of a Platonic tiling: for $p=3$, $p=4$ and $p=5$ the tilings are the truncated tetrahedron, the truncated cube and
the truncated dodecahedron respectively.
Similarly, the vertex-homogenous tilings with vertex type
$\al_4\al_6^2$ and $\al_5\al_6^2$ are the truncated octahedron and the
truncated icosahedron.

We apply the shrinking operation to every square in a
vertex-homogenous tiling with vertices of type $\al_4\al_{2p}\al_{2q}$
with $q > p \ge 3$
to obtain a new (strongly) vertex-homogenous tiling with angle
arrangement $\al_p\al_q\al_p\al_q$. In the following section, we will
see that
these tilings are Archimedean. In particular, the vertex-homogenous
tilings with vertex type $\al_4\al_6\al_8$ and $\al_4\al_6\al_{10}$
are the truncated cuboctahedron and the truncated icosidodecahedron
respectively.

\begin{figure}[h!]
\centering
\begin{subfigure}[t]{0.32\linewidth}
\centering
\begin{tikzpicture}[>=latex]
\tikzmath{
\r=0.5;
\rr=0.05;
\n=3;
\nn=\n-1;
\th=360/\n;
\XS=2;
}
\raisebox{2ex}{
\foreach \a in {0,...,\nn} {
\tikzset{rotate=\a*\th}
\draw[line width=1.5]
	(90:\r) -- (90+\th:\r)
;
\draw[]
	(90:\r) -- (90:2*\r)
;
};

\draw[->]
	(0.4*\XS,0) -- (0.6*\XS,0)
;

\begin{scope}[xshift=\XS cm]
\foreach \a in {0,...,\nn} {
\tikzset{rotate=\a*\th}
\draw[]
	(0,0) -- (90:\r) 
;
}
\fill[] (0,0) circle (\rr);
\end{scope}
}
\end{tikzpicture}
\caption{$\triangle$ $\to$ $\deg 3$ vertex}
\label{Subfig-shrink-deg3}
\end{subfigure}
\begin{subfigure}[t]{0.32\linewidth}
\centering
\begin{tikzpicture}[>=latex]
\tikzmath{
\r=0.5;
\rr=0.05;
\n=4;
\nn=\n-1;
\th=360/\n;
\XS=3;
}
\foreach \a in {0,...,\nn} {
\tikzset{rotate=\a*\th}
\draw[line width=1.5]
	(90:\r) -- (90+\th:\r)
;
\draw[]
	(90:\r) -- (90:2*\r)
;
};

\draw[->]
	(2.5*\r,0) -- (3.5*\r,0)
;

\begin{scope}[xshift=0.8*\XS cm]
\foreach \a in {0,...,\nn} {
\tikzset{rotate=\a*\th}
\draw[]
	(0,0) -- (90:\r) 
;
}
\fill[] (0,0) circle (\rr);
\end{scope}
\end{tikzpicture}
\caption{$\square$ $\to$ $\deg 4$ vertex}
\label{Subfig-shrink-deg4}
\end{subfigure}
\begin{subfigure}[t]{0.32\linewidth}
\centering
\begin{tikzpicture}[>=latex]
\tikzmath{
\r=0.5;
\rr=0.05;
\n=5;
\nn=\n-1;
\th=360/\n;
\XS=2.5;
}
\foreach \a in {0,...,\nn} {
\tikzset{rotate=\a*\th}
\draw[]
	(90:\r) -- (90+\th:\r)
	(90:\r) -- (90:2*\r)
;
};
\foreach \a in {0,...,\nn} {
\tikzset{rotate=\a*\th}
\draw[line width=1.5]
	(90:\r) -- (90+\th:\r)
;
\draw[]
	(90:\r) -- (90:2*\r)
;
};
\draw[->]
	(2.5*\r,0) -- (3.5*\r,0)
;
\begin{scope}[xshift=\XS cm]
\foreach \a in {0,...,\nn} {
\tikzset{rotate=\a*\th}
\draw[]
	(0,0) -- (90:\r) 
;
}
\fill[] (0,0) circle (\rr);
\end{scope}
\end{tikzpicture}
\caption{{\large $\pentagon$} $\to$ $\deg 5$ vertex}
\label{Subfig-shrink-deg5}
\end{subfigure}
\caption{The shrinking operation}
\label{Fig-shrink}
\end{figure}
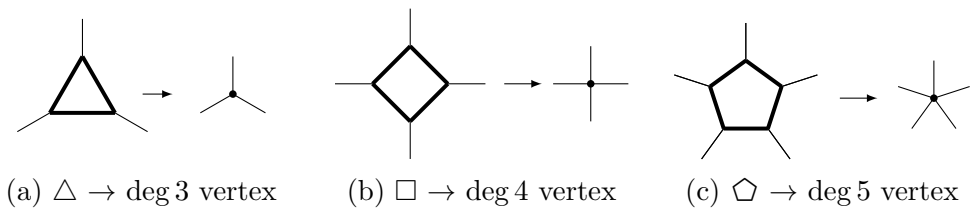

The remaining vertex types of degree 3 are of the form $\al_m\al_4^2$.
We claim that the vertex-homogenous tilings with these types are
prisms. Indeed, each square in such a tiling has exactly two
neighbouring squares on an opposite pair of edges, so there is a cycle
of squares. Each square has an $m$-gon on the other opposite pair of
edges and since $m \ge 3$, two neighbouring squares share the same
pair of $m$-gons. Hence the whole cycle shares the same pair of
$m$-gons by induction, and the claim follows.

\subsubsection*{Degree 4 vertices}

We complete the proof by constructing the tilings with vertices of
degree 4. The degree 4 types listed in \eqref{List-weakly} are $\al_3^3\al_m$, $\al_3^2\al_m^2$,
$\al_3^2\al_m\al_n$, $\al_3\al_4^3$ and $\al_3\al_4^2\al_5$. For each
type we show that the tiling is one of the tilings listed in
Theorem~\ref{thm:weakly} or derive a contradiction.

\medskip\noindent\emph{Type $\al_3^3\al_m$ for $m\ge 4$.} 
The tiling is an anti-prism. The argument is similar to the case of
the prisms. We fix an $m$-gon in the tiling and call a triangle black
if it shares an edge with the $m$-gon and white if it shares exactly
one vertex with the $m$-gon. Each vertex on the triangle is incident
to two black triangles and one white. On the other hand each black
triangle is incident to two white triangles and each white triangle is
incident to two black triangles. Hence there is a cycle of alternating
black and white triangles around the $m$-gon. 
If we pair each black triangle with its white neighbour going
clockwise around the $m$-gon, we see that the pairs form rhombi. Each
rhombus has a pair of rhombi on an opposite pair of edges and a pair of
$m$-gons on the other opposite pair. Since $m \ge 4$, two neighbouring
rhombi share the same pair of $m$-gons. Hence the whole cycle of
rhombi share the same pair of $m$-gons by induction; there are exactly
two $m$-gons and the tiling is an antiprism.

\medskip\noindent\emph{Type $\al_3^2\al_m^2$ for $m \ge 4$.}
The angle sum of $\alpha_3^2\alpha_m^2$ gives $\alpha_m=\pi -
\alpha_3$. Then \eqref{Ineq-alm-lb} further implies
$(1-\frac{2}{m})\pi < \alpha_m < \pi - \frac{1}{3}\pi =
\frac{2}{3}\pi$, which gives $m<6$. Hence $m=4,5$ and
$\alpha_3^2\alpha_m^2=\alpha_3^2\alpha_4^2, \alpha_3^2\alpha_5^2$.

\medskip\noindent\emph{Case $m=4$.} We first determine the tiling when
there is at least one vertex with angle arrangement
$\alpha_3\alpha_3\alpha_4\alpha_4$. Then we determine the tiling when
all vertices have angle arrangement $\al_3\al_4\al_3\al_4$.

We depict a vertex with angle arrangement $\al_3\al_3\al_4\al_4$
in the centre of Figure~\ref{Subfig-Tiling-J27}, with incident
triangles $\mathbf{1}$ and $\mathbf{2}$ and incident squares
$\mathbf{3}$ and $\mathbf{4}$. This vertex has two neighbours of
the same angle arrangement, which are depicted above and below it,
which determines the triangles $\mathbf{5}$ and $\mathbf{6}$ and
the squares $\mathbf{7}$ and $\mathbf{8}$. Continuing, we
determine the tiles $\mathbf{9, 10, 11, 12}$ and obtain the tiling
$J_{27}$.

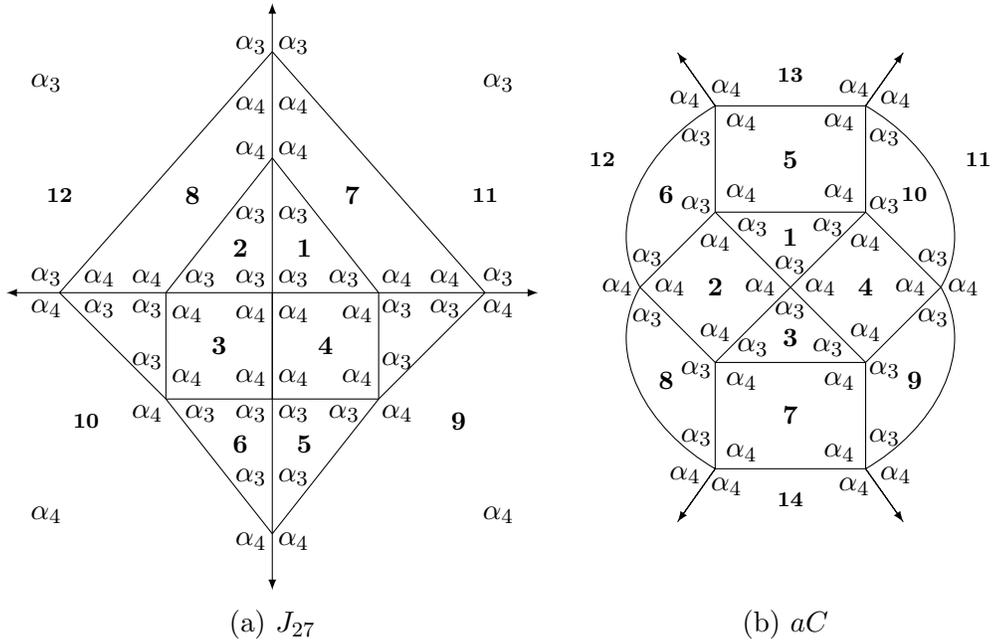
\begin{figure}[h!]
\centering
\begin{subfigure}[t]{0.575\linewidth}
\centering
\begin{tikzpicture}[>=latex]
\tikzmath{
\r=1;
\th=360/4;
\x=\r*cos(\th/2);
}

\foreach \aa in {-1,1} {
\tikzset{xshift=\aa*\x cm, xscale=\aa}
\foreach \a in {0,1,...,3} {
\tikzset{rotate=\a*\th}
\draw[]
	(90-0.5*\th:\r) -- (90+0.5*\th:\r)
;

\node at (0.6*\x,0.6*\x) {\small $\alpha_4$};
}
}

\foreach \bb in {-1,1} {
\tikzset{yscale=\bb}
\draw[]
	(0,\x) -- (0,2.5*\r)
;
\foreach \aa in {-1,1} {
\tikzset{xscale=\aa}
\draw[]
	(0,2.5*\r) -- (2*\x,\x) 
;
\node at (0.4*\x,1.75*\r) {\small $\alpha_3$};
\node at (0.4*\x,0.9*\r) {\small $\alpha_3$};
\node at (1.35*\x,0.9*\r) {\small $\alpha_3$};

\node at (0.4*\x,2.6*\r) {\small $\alpha_4$};
\node at (2.35*\x,0.9*\r) {\small $\alpha_4$};
}
}

\foreach \aa in {-1,1} {
\tikzset{xscale=\aa}
\node at (0.4*\x,3.2*\r) {\small $\alpha_4$};
\node at (3.25*\x,0.9*\r) {\small $\alpha_4$};
\node at (4.25*\x,0.7*\x) {\small $\alpha_4$};
\node at (4.25*\x,-2.25*\r) {\small $\alpha_4$};

\node at (2.35*\x,0.7*\x) {\small $\alpha_3$};
\node at (2.35*\x,-0.2*\r) {\small $\alpha_3$};
\node at (3.25*\x,0.7*\x) {\small $\alpha_3$};

\node at (4.25*\x,1.3*\x) {\small $\alpha_3$};
\node at (0.4*\x,4*\r) {\small $\alpha_3$};
\node at (4.25*\x,3.5*\r) {\small $\alpha_3$};
}

\foreach \aa in {-1,1} {
\tikzset{xscale=\aa}
\draw[]
	(2*\x,\x) -- (4*\x,\x) 
	(0,2.5*\r+2*\x) -- (4*\x,\x)
	(2*\x,-\x) -- (4*\x,\x)
;
\draw[->]
	 (4*\x,\x) -- (5*\x,\x) 
;
}

\draw[]
	(0,2.5*\r) -- (0,2.5*\r+2*\x)
;

\draw[->]
	(0,-2.5*\r) -- (0,-2.5*\r-0.75*\r)
;

\draw[->]
	(0,2.5*\r+2*\x) -- (0,2.5*\r+1.5*\x+\r)
;




\node at (0.6*\x,1.3*\r) {\footnotesize \bm{$1$}};
\node at (-0.6*\x,1.3*\r) {\footnotesize \bm{$2$}};
\node at (-\x,0) {\footnotesize \bm{$3$}};
\node at (\x,0) {\footnotesize \bm{$4$}};

\node at (0.6*\x,-1.3*\r) {\footnotesize \bm{$5$}};
\node at (-0.6*\x,-1.3*\r) {\footnotesize \bm{$6$}};

\node at (1.5*\x,2*\r) {\footnotesize \bm{$7$}};
\node at (-1.5*\x,2*\r) {\footnotesize \bm{$8$}};

\node at (3.5*\x,-\r) {\footnotesize \bm{$9$}};
\node at (-3.5*\x,-\r) {\scriptsize \bm{$10$}};

\node at (4*\x,2*\r) {\scriptsize \bm{$11$}};
\node at (-4*\x,2*\r) {\scriptsize \bm{$12$}};
\end{tikzpicture}
\caption{$J_{27}$}
\label{Subfig-Tiling-J27}
\end{subfigure}
\begin{subfigure}[t]{0.4\linewidth}
\centering
\begin{tikzpicture}[>=latex]
\tikzmath{
\r=1;
\th=360/4;
\x=\r*cos(0.5*\th);
}

\raisebox{5ex}{

\foreach \aa in {-1,1} {
\tikzset{xshift=\aa*\r cm}
\foreach \a in {0,1,2,3} {
\tikzset{rotate=\a*\th}
\draw[]
	(90:\r) -- (90+\th:\r)
;
}
}

\foreach \bb in {-1,1} {
\tikzset{yscale=\bb}
\draw[]
	(\r,\r) -- (-\r,\r)
	(\r,\r+2*\x) -- (-\r,\r+2*\x)
;
\node at (0,0.3*\r) {\small $\alpha_3$};
\node at (0.5*\r,0.8*\r) {\small $\alpha_3$};
\node at (-0.5*\r,0.8*\r) {\small $\alpha_3$};
\foreach \aa in {-1,1} {
\tikzset{xscale=\aa}
\draw[]
	(\r,\r) -- (\r,\r+2*\x)
	(2*\r, 0) to[out=60, in=-30] (\r,\r+2*\x)
;
\draw[->]
	(\r,\r+2*\x) -- (1.5*\r,\r+3*\x)
	(\r,\r+2*\x) -- (1.5*\r,\r+3*\x)
;
}
}

\foreach \aa in {-1,1} {
\tikzset{xscale=\aa}
\node at (0.4*\r,0) {\small $\alpha_4$};
\node at (1.6*\r,0) {\small $\alpha_4$};
\node at (\r,0.6*\r) {\small $\alpha_4$};
\node at (\r,-0.6*\r) {\small $\alpha_4$};

\node at (2.3*\r,0) {\small $\alpha_4$};

\foreach \bb in {-1,1}{
\tikzset{yscale=\bb}
\node at (0.65*\r,1.25*\r) {\small $\alpha_4$};
\node at (0.65*\r,2.2*\r) {\small $\alpha_4$};
\node at (0.85*\r,2.65*\r) {\small $\alpha_4$};
\node at (1.4*\r,2.5*\r) {\small $\alpha_4$};

\node at (1.25*\r,1.1*\r) {\small $\alpha_3$};
\node at (1.25*\r,2*\r) {\small $\alpha_3$};
\node at (1.9*\r,0.4*\r) {\small $\alpha_3$};
}
}

\node at (0,0.95*\x) {\footnotesize \bm{$1$}};
\node at (-\r,0) {\footnotesize \bm{$2$}};
\node at (0,-0.95*\x) {\footnotesize \bm{$3$}};
\node at (\r,0) {\footnotesize \bm{$4$}};

\node at (0,2.4*\x) {\footnotesize \bm{$5$}};
\node at (-1.65*\r,1.75*\x) {\footnotesize \bm{$6$}};
\node at (1.65*\r,1.75*\x) {\scriptsize \bm{$10$}};

\node at (0,-2.4*\x) {\footnotesize \bm{$7$}};
\node at (-1.65*\r,-1.75*\x) {\footnotesize \bm{$8$}};
\node at (1.65*\r,-1.75*\x) {\footnotesize \bm{$9$}};

\node at (2.5*\r,2.4*\x) {\scriptsize \bm{$11$}};
\node at (-2.5*\r,2.4*\x) {\scriptsize \bm{$12$}};

\node at (0,4*\x) {\scriptsize \bm{$13$}};
\node at (0,-4*\x) {\scriptsize \bm{$14$}};
}
\end{tikzpicture}
\caption{$aC$}
\label{Subfig-Cuboctahedron}
\end{subfigure}
\caption{The two tilings with $\alpha_3^2\alpha_4^2$, the triangular orthobicupola $J_{27}$ and the cuboctahedron $aC$}
\end{figure}

We now assume that every vertex in the tiling has angle arrangement
$\al_3\al_4\al_3\al_4$. We depict such a vertex in the centre
of Figure~\ref{Subfig-Cuboctahedron}.
Its
incident tiles $\mathbf{1,2,3,4}$ form a hexagon; since every vertex
has the same angle arrangement, we consider each
vertex on the boundary of this hexagon in turn and deduce the tiles
$\mathbf{5,\dots,10}$. Repeating this process determines the remaining
tiles and the tiling; namely, the cuboctahedron.

\medskip\noindent\emph{Case $m=5$.} 
In this case we begin by determining the tiling when every vertex has angle arrangement
$\al_3\al_5\al_3\al_5$; in such a tiling, every pentagon is adjacent
to five triangles. We then consider tilings with at least one
vertex having angle arrangement $\al_3\al_3\al_5\al_5$. 

We depict a pentagon $\mathbf{1}$ with five adjacent triangles
$\mathbf{2,\dots,6}$ in the
centre of Figure~\ref{Subfig-aD}. Tile $\mathbf{7}$, adjacent to
triangles $\mathbf{2}$ and $\mathbf{3}$, is a pentagon; we determine
the pentagons $\mathbf{8,\dots,11}$ similarly. Continuing, we
determine the triangles $\mathbf{12,\dots,16}$. Since every vertex has the same angle
arrangement, we determine that $\mathbf{17}$ and $\mathbf{18}$ are a
triangle and a pentagon respectively. We determine the rest of the
tiles in the same way, and the tiling is the icosidodecahedron.

Now suppose there is a vertex with angle arrangement
$\al_3\al_3\al_5\al_5$. It has a neighbour with which it shares two
triangles and a neighbour with which it shares two pentagons. Both of
these therefore have the same angle arrangement. Thus, there is a
cycle of such vertices. Observing that the cycle has angle
$\al_3+\al_5=\pi$ on both sides, we see that it is a great circle.
Rotating one of its hemispheres by $\frac{2}{n}\pi$, where $n$ is the
length of the cycle, we obtain a new tiling. The vertices on the cycle
now have angle arrangement $\al_3\al_5\al_3\al_5$.
In fact, every vertex in the new tiling has this arrangement: if a
vertex has angle arrangement $\al_3\al_3\al_5\al_5$, then by the
argument given above there is a cycle of such vertices forming a great
circle. This cycle necessarily intersects the original cycle, which is
a contradiction since these vertices have angle arrangement
$\al_3\al_5\al_3\al_5$ in the new tiling.
Therefore the new tiling is $aD$; the original tiling is obtained from
$aD$ by reversing the operation, that is, by rotating the hemisphere
by $\frac{2}{n}\pi$.
We depict the result of this operation in
Figure~\ref{Subfig-J34}; the tiling is $J_{34}$.

\begin{figure}[h!]
\centering

\begin{subfigure}[t]{0.3\linewidth}
\centering
\begin{tikzpicture}
\tikzmath{
\r=0.45;
\l=1.8*\r;
\R=2.75*\r;
\RR=4.25*\r;
\m=3;
\mm=\m-1;
\n=5;
\nn=\n-1;
\N=2*\n;
\Nn=\N-1;
\th = 360/\m;
\ph = 360/\n;
\Ph = 360/\N;
}

\foreach \a in {0,...,\Nn} {
\tikzset{rotate=\a*\Ph}
\draw[cyan!50, line width=2]
	(90-0.5*\Ph:\R) -- (90+0.5*\Ph:\R)
;
}

\foreach \a in {0,...,\nn} {
\tikzset{rotate=\a*\ph}
\draw[]
	(270:\r) -- (270+\ph:\r)
	(270:\r) -- (270-0.5*\ph:\l)
	(270:\r) -- (270+0.5*\ph:\l)
	(270-0.5*\ph:\l) -- (270-0.5*\Ph:\R) 
	(270+0.5*\ph:\l) -- (270+0.5*\Ph:\R)
	(270-0.5*\Ph:\R) -- (270:3.5*\r)
	(270+0.5*\Ph:\R) -- (270:3.5*\r)
	(270:3.5*\r) -- (270-0.5*\ph:\RR)
	(270:3.5*\r) -- (270+0.5*\ph:\RR)
;
}

\foreach \a in {0,...,\Nn} {
\tikzset{rotate=\a*\Ph}
\draw[]
	(90-0.5*\Ph:\R) -- (90+0.5*\Ph:\R)
;
}

\draw[] (0,0) circle (\RR);

\node at (0,0) {\tiny \bm{$1$}};
\node at (0,1.14*\r) {\tiny \bm{$2$}};
\node at ([rotate=1*\ph]0,1.14*\r) {\tiny \bm{$3$}};
\node at ([rotate=2*\ph]0,1.14*\r) {\tiny \bm{$4$}};
\node at ([rotate=3*\ph]0,1.14*\r) {\tiny \bm{$5$}};
\node at ([rotate=4*\ph]0,1.14*\r) {\tiny \bm{$6$}};

\node at (90+0.5*\ph:0.7*\R) {\tiny \bm{$7$}};
\node at ([rotate=1*\ph]90+0.5*\ph:0.7*\R) {\tiny \bm{$8$}};
\node at ([rotate=2*\ph]90+0.5*\ph:0.7*\R) {\tiny \bm{$9$}};
\node at ([rotate=3*\ph]90+0.5*\ph:0.7*\R) {\tiny \bm{$10$}};
\node at ([rotate=4*\ph]90+0.5*\ph:0.7*\R) {\tiny \bm{$11$}};

\node at (0,2.25*\r) {\tiny \bm{$12$}};
\node at ([rotate=1*\ph]0,2.25*\r) {\tiny \bm{$13$}};
\node at ([rotate=2*\ph]0,2.25*\r) {\tiny \bm{$14$}};
\node at ([rotate=3*\ph]0,2.25*\r) {\tiny \bm{$15$}};
\node at ([rotate=4*\ph]0,2.25*\r) {\tiny \bm{$16$}};

\node at (90+0.5*\ph:3*\r) {\tiny \bm{$18$}};
\node at (90+1.5*\ph:3*\r) {\tiny \bm{$20$}};
\node at (90+2.5*\ph:3*\r) {\tiny \bm{$22$}};
\node at (90+3.5*\ph:3*\r) {\tiny \bm{$24$}};
\node at (90+4.5*\ph:3*\r) {\tiny \bm{$26$}};

\node at (90:1.2*\R) {\tiny \bm{$17$}};
\node at ([rotate=1*\ph]90:1.2*\R) {\tiny \bm{$19$}};
\node at ([rotate=2*\ph]90:1.2*\R) {\tiny \bm{$21$}};
\node at ([rotate=3*\ph]90:1.2*\R) {\tiny \bm{$23$}};
\node at ([rotate=4*\ph]90:1.2*\R) {\tiny \bm{$25$}};

\node at (90+0.5*\ph:0.915*\RR) {\tiny \bm{$27$}};
\node at (90+1.5*\ph:0.915*\RR) {\tiny \bm{$28$}};
\node at (90+2.5*\ph:0.915*\RR) {\tiny \bm{$29$}};
\node at (90+3.5*\ph:0.915*\RR) {\tiny \bm{$30$}};
\node at (90+4.5*\ph:0.915*\RR) {\tiny \bm{$31$}};

\node at (0,1.15*\RR) {\tiny \bm{$32$}};

\end{tikzpicture}
\caption{$aD$}
\label{Subfig-aD}
\end{subfigure}
\begin{subfigure}[t]{0.3\linewidth}
\centering
\begin{tikzpicture}
\tikzmath{
\r=0.45;
\l=1.8*\r;
\R=2.75*\r;
\RR=4.25*\r;
\m=3;
\mm=\m-1;
\n=5;
\nn=\n-1;
\N=2*\n;
\Nn=\N-1;
\th = 360/\m;
\ph = 360/\n;
\Ph = 360/\N;
}

\foreach \a in {0,...,\Nn} {
\tikzset{rotate=\a*\Ph}
\draw[cyan!50, line width=2]
	(90-0.5*\Ph:\R) -- (90+0.5*\Ph:\R)
;
}

\foreach \a in {0,...,\nn} {
\tikzset{rotate=\a*\ph}
\draw[]
	(270:\r) -- (270+\ph:\r)
	(270:\r) -- (270-0.5*\ph:\l)
	(270:\r) -- (270+0.5*\ph:\l)
	(270-0.5*\ph:\l) -- (270-0.5*\Ph:\R) 
	(270+0.5*\ph:\l) -- (270+0.5*\Ph:\R)
	(90-0.5*\Ph:\R) -- (90:3.5*\r)
	(90+0.5*\Ph:\R) -- (90:3.5*\r)
	(90:3.5*\r) -- (90-0.5*\ph:\RR)
	(90:3.5*\r) -- (90+0.5*\ph:\RR)
;
}

\foreach \a in {0,...,\Nn} {
\tikzset{rotate=\a*\Ph}
\draw[]
	(90-0.5*\Ph:\R) -- (90+0.5*\Ph:\R)
;
}

\draw[] (0,0) circle (\RR);

\node at (0,0) {\tiny \bm{$1$}};
\node at (0,1.14*\r) {\tiny \bm{$2$}};
\node at ([rotate=1*\ph]0,1.14*\r) {\tiny \bm{$3$}};
\node at ([rotate=2*\ph]0,1.14*\r) {\tiny \bm{$4$}};
\node at ([rotate=3*\ph]0,1.14*\r) {\tiny \bm{$5$}};
\node at ([rotate=4*\ph]0,1.14*\r) {\tiny \bm{$6$}};

\node at (90+0.5*\ph:0.7*\R) {\tiny \bm{$7$}};
\node at ([rotate=1*\ph]90+0.5*\ph:0.7*\R) {\tiny \bm{$8$}};
\node at ([rotate=2*\ph]90+0.5*\ph:0.7*\R) {\tiny \bm{$9$}};
\node at ([rotate=3*\ph]90+0.5*\ph:0.7*\R) {\tiny \bm{$10$}};
\node at ([rotate=4*\ph]90+0.5*\ph:0.7*\R) {\tiny \bm{$11$}};

\node at (0,2.25*\r) {\tiny \bm{$12$}};
\node at ([rotate=1*\ph]0,2.25*\r) {\tiny \bm{$13$}};
\node at ([rotate=2*\ph]0,2.25*\r) {\tiny \bm{$14$}};
\node at ([rotate=3*\ph]0,2.25*\r) {\tiny \bm{$15$}};
\node at ([rotate=4*\ph]0,2.25*\r) {\tiny \bm{$16$}};

\node at (0,3*\r) {\tiny \bm{$17$}};
\node at ([rotate=1*\ph]0,3*\r) {\tiny \bm{$19$}};
\node at ([rotate=2*\ph]0,3*\r) {\tiny \bm{$21$}};
\node at ([rotate=3*\ph]0,3*\r) {\tiny \bm{$23$}};
\node at ([rotate=4*\ph]0,3*\r) {\tiny \bm{$25$}};

\node at (90+0.5*\ph:1.2*\R) {\tiny \bm{$18$}};
\node at ([rotate=1*\ph]90+0.5*\ph:1.2*\R) {\tiny \bm{$20$}};
\node at ([rotate=2*\ph]90+0.5*\ph:1.2*\R) {\tiny \bm{$22$}};
\node at ([rotate=3*\ph]90+0.5*\ph:1.2*\R) {\tiny \bm{$24$}};
\node at ([rotate=4*\ph]90+0.5*\ph:1.2*\R) {\tiny \bm{$26$}};

\node at (0,0.915*\RR) {\tiny \bm{$27$}};
\node at ([rotate=1*\ph]0,0.915*\RR) {\tiny \bm{$28$}};
\node at ([rotate=2*\ph]0,0.915*\RR) {\tiny \bm{$29$}};
\node at ([rotate=3*\ph]0,0.915*\RR) {\tiny \bm{$30$}};
\node at ([rotate=4*\ph]0,0.915*\RR) {\tiny \bm{$31$}};

\node at (0,1.15*\RR) {\tiny \bm{$32$}};

\end{tikzpicture}
\caption{$J_{34}$}
\label{Subfig-J34}
\end{subfigure}
\caption{The tilings with a pentagon having five adjacent triangles, the icosidodecahedron and Johnson's solid $J_{34}$ -- pentagonal orthobirotunda}
\label{Fig-Tilings-pen-5-adj-tri}
\end{figure}
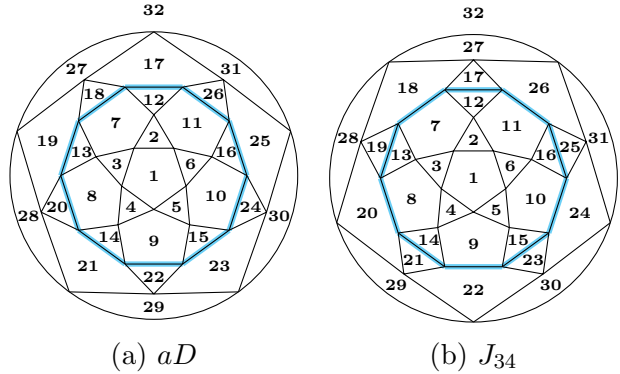

\medskip\noindent\emph{Type $\al_3^2\al_m\al_n$ for $n >m \ge 4$.}
There is no tiling with this type. We first derive a contradiction
when every vertex is of arrangement $\al_3\al_3\al_m\al_n$, then we
derive a contradiction when every vertex is of arrangement
$\al_3\al_m\al_3\al_n$. Finally we show that no tiling can have both
angle arrangements.

If every vertex has angle arrangement $\al_3\al_3\al_m\al_n$, then
every triangle in the tiling is adjacent to another triangle. We
depict a pair $\mathbf{1,2}$ of adjacent triangles
Figure~\ref{Subfig-al3al3almlaln}. The two other tiles incident to
$\mathbf{1}$ are $m$-gons or $n$-gons. Without loss of generality we
can assume that $\mathbf{3}$ is an $m$-gon; we denote the angle of
$\mathbf{4}$ by $\al$. The vertex of $\mathbf{1}$ not shared with
$\mathbf{2}$ has angle arrangement $\al_m\al_3\al\cdots$ which is
a contradiction whether $\al= \al_m$ or $\al=\al_n$. 

\begin{figure}[h!]
\centering
\begin{subfigure}[t]{0.4\linewidth}
\centering
\begin{tikzpicture}
\tikzmath{
\r=0.8;
\n=3;
\nn=\n-1;
\th=360/\n;
\y=\r*cos(0.5*\th);
\x=\r*sin(0.5*\th);
\XS=2*\x+\r;
}

\foreach \aa in {-1,1} {
\foreach \a in {0,...,\nn} {
\tikzset{yscale=\aa, yshift=\y cm, rotate=\a*\th}
\draw[]
	(90:\r) -- (90+\th:\r)
;
}
}

\foreach \aa in {-1,1} {
\tikzset{xscale=\aa}
\draw[]
	(\x,0) -- (\x+\r,0)
;

\node at (0.5*\x,0.35*\y) {\footnotesize $\alpha_3$};
}

\node at (0*\x,\y+0.5*\r) {\footnotesize $\alpha_3$};

\node at (-0.4*\x,\y+\r) {\footnotesize $\alpha$};
\node at (0.5*\x,\y+\r) {\footnotesize $\alpha_m$};
\node at (-1.25*\x,0.35*\y) {\footnotesize $\alpha$};
\node at (1.4*\x,0.35*\y) {\footnotesize $\alpha_m$};

\node at (0,\y+1.4*\r) {?};

\node at (0,\y) {\tiny \bm{$1$}};
\node at (0,-\y) {\tiny \bm{$2$}};
\node at (\x,1.5*\y) {\tiny \bm{$3$}};
\node at (-\x,1.5*\y) {\tiny \bm{$4$}};

\end{tikzpicture}
\caption{}
\label{Subfig-al3al3almlaln}
\end{subfigure}
\begin{subfigure}[t]{0.4\linewidth}
\centering
\begin{tikzpicture}
\tikzmath{
\r=1;
\n=4;
\nn=\n-1;
\th=360/\n;
\x=\r*cos(0.5*\th);
}
\foreach \a in {0,...,\nn} {
\tikzset{rotate=\a*\th}
\draw[]
	(0,0) -- (\r,\r) 
;
}
\foreach \aa in {-1,1} {
{
\tikzset{yscale=\aa}
\draw[]
	(-\r,\r)  -- (\r,\r) 
	(\r,\r) -- (1.75*\r,\r) 
	(-\r,\r) -- (-1.75*\r,\r) 
;
}
\tikzset{xscale=\aa}
\draw[]
	(\r,\r) -- (\r,1.75*\r)
;
}

\node at (0,0.4*\x) {\footnotesize $\alpha_3$};
\node at (0.6*\r,1.2*\x) {\footnotesize $\alpha_3$};
\node at (-0.6*\r,1.2*\x) {\footnotesize $\alpha_3$};

\node at (-0.35*\r,0) {\footnotesize $\alpha_n$};
\node at (-1.7*\x,1.2*\x) {\footnotesize $\alpha_n$};
\node at (-1.7*\x,-1.2*\x) {\footnotesize $\alpha_n$};

\node at (0.35*\r,0) {\footnotesize $\alpha_m$};
\node at (1.75*\x,1.2*\x) {\footnotesize $\alpha_m$};
\node at (1.75*\x,-1.2*\x) {\footnotesize $\alpha_m$};

\node at (0,-0.4*\x) {\footnotesize $\alpha_3$};
\node at (0.6*\r,-1.2*\x) {\footnotesize $\alpha_3$};
\node at (-0.6*\r,-1.2*\x) {\footnotesize $\alpha_3$};

\node at (0.7*\r,1.65*\x) {\footnotesize $\alpha_m$};
\node at (-0.7*\r,1.65*\x) {\footnotesize $\alpha_m$};

\node at (-1.25*\r,1.65*\x) {\footnotesize $\alpha_3$};
\node at (1.75*\x,1.75*\x) {\small $?$};

\node at (0,\x) {\tiny \bm{$1$}};
\node at (-\r,0) {\tiny \bm{$2$}};
\node at (0,-\x) {\tiny \bm{$3$}};
\node at (\r,0) {\tiny \bm{$4$}};

\node at (0,2*\x) {\tiny \bm{$5$}};
\end{tikzpicture}
\caption{}
\label{Subfig-al3almal3aln}
\end{subfigure}
\caption{Deriving a contradiction for type $\al_3^2\al_m\al_n$}
\end{figure}
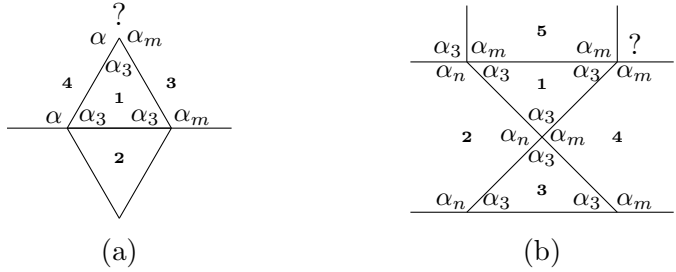

Now, we suppose that every vertex has angle arrangement
$\al_3\al_m\al_3\al_n$. Such a vertex, with incident tiles 
$\mathbf{1,2,3,4}$, is depicted in the centre of
Figure~\ref{Subfig-al3almal3aln}. 
If $\mathbf{5}$ is an $m$-gon, then the vertex incident with
$\mathbf{4}$ and $\mathbf{5}$ is of type $\al_3\al_m^2\cdots$, a
contradiction. Similarly, $\mathbf{5}$ is not an $n$-gon. We conclude
that there is no tiling in this case.

Now we must have vertices with angle arrangement
$\al_3\al_3\al_m\al_n$ and vertices with angle arrangement
$\al_3\al_m\al_3\al_n$. We call these vertices black and white
respectively. We call an edge black (white) if both its endpoints are
black (white); such edges are said to be monochromatic. It is clear
from the angle arrangements that an edge shared by an $n$-gon and an
$m$-gon is black and that an edge shared by two triangle is also
black. We can deduce that an edge shared by a triangle and an $m$-gon
is not monochromatic. If it is black
(Figure~\ref{Subfig-forbidden-black-app}), then the vertex of the triangle
not shared by the $m$-gon is incident with three triangles; if it is
white (Figure~\ref{Subfig-forbidden-white-app}), then that vertex is
incident with two $n$-gons. We claim that two consecutive vertices on
an $m$-gon cannot be white: since the edge between them is
monochromatic, the edge must be shared by an $n$-gon, but an edge
shared by an $m$-gon and an $n$-gon is black
(Figure~\ref{Subfig-forbidden-ww-app}). Furthermore, three consecutive
vertices on an $m$-gon cannot be black: the two edges between them
must be shared by an $n$-gon because they are monochromatic, so the
middle vertex is either of degree 2 or incident to two $n$-gons
(Figure~\ref{Subfig-forbidden-bbb-app}). Finally, three consecutive
vertices on an $m$-gon cannot be coloured white-black-white: the two
edges between them must be shared by triangles, so the middle vertex
is of degree 2 or it has angle arrangement $\al_3\al_n\al_3\al_m$, a
contradiction since such vertices are white
(Figure~\ref{Subfig-forbidden-wbw-app}). We conclude that the sequence of
colours around an $m$-gon is white-black-black-white-black-black and
so on and that $m$ is divisible by 3. This contradicts the fact that
$m\in \{4,5\}$.

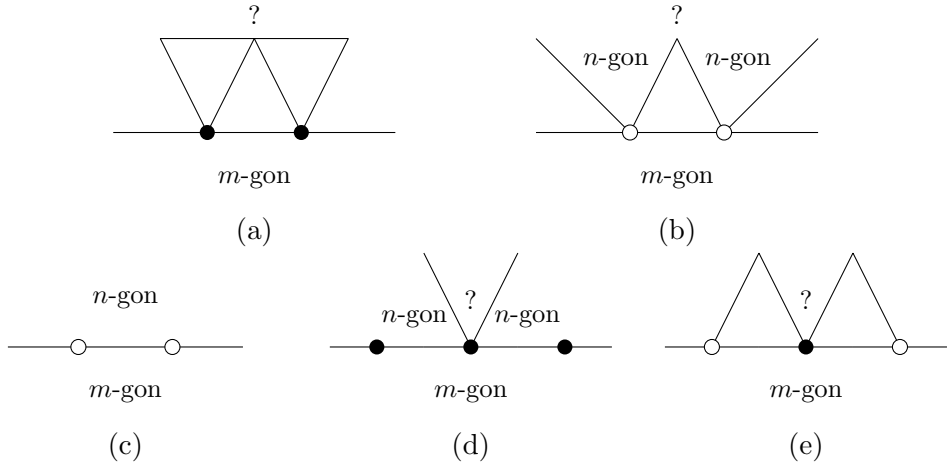
\begin{figure}[h!]
\centering
\begin{subfigure}[t]{0.4\linewidth}
\centering
\begin{tikzpicture}
\tikzmath{
\r=1.25;
\rr=0.08*\r;
}

\foreach \a in {-1,0,1} {
\tikzset{xshift=\a*\r cm}

\draw[]
	(-0.5*\r,0) -- (0.5*\r,0)
;
}

\foreach \a in {0,1} {
\tikzset{xshift=\a*\r cm}
\fill[]
	(-0.5*\r,0) circle (\rr)
;
}

\foreach \aa in {-1,1} {
\tikzset{xscale=\aa}

\draw[]
	(-0.5*\r,0) -- (0,\r)
	(-0.5*\r,0) -- (-1*\r,\r)
	(-1*\r,\r) -- (0,\r)
;
}

\node at (0,1.25*\r) {\footnotesize $?$};

\node at (0,-0.5*\r) {\footnotesize $m$-gon};

\end{tikzpicture}
\caption{}
\label{Subfig-forbidden-black-app}
\end{subfigure}
\begin{subfigure}[t]{0.4\linewidth}
\centering
\begin{tikzpicture}
\tikzmath{
\r=1.25;
\rr=0.08*\r;
}

\foreach \a in {-1,0,1} {
\tikzset{xshift=\a*\r cm}

\draw[]
	(-0.5*\r,0) -- (0.5*\r,0)
;
}

\foreach \aa in {-1,1} {
\tikzset{xscale=\aa}

\draw[]
	(-0.5*\r,0) -- (0,\r)
	(-0.5*\r,0) -- (-1.5*\r,\r)
;

\node at (-0.65*\r,0.75*\r) {\footnotesize $n$-gon};

}

\foreach \a in {0,1} {
\tikzset{xshift=\a*\r cm}
\draw[fill=white]
	(-0.5*\r,0) circle (\rr)
;
}

\node at (0,1.25*\r) {\footnotesize $?$};

\node at (0,-0.5*\r) {\footnotesize $m$-gon};
\end{tikzpicture}
\caption{}
\label{Subfig-forbidden-white-app}
\end{subfigure}

\begin{subfigure}[t]{0.325\linewidth} 
\centering
\begin{tikzpicture}
\tikzmath{
\r=1.25;
\rr=0.08*\r;
}

\draw[]
	(-1.25*\r,0) -- (1.25*\r,0)
;

\foreach \aa in {-1,1} {
\tikzset{xscale=\aa}
\draw[fill=white]
	(-0.5*\r,0) circle (\rr)
;
}

\node at (0,0.5*\r) {\footnotesize $n$-gon};

\node at (0,-0.5*\r) {\footnotesize $m$-gon};

\end{tikzpicture} 
\caption{}
\label{Subfig-forbidden-ww-app}
\end{subfigure} 
\begin{subfigure}[t]{0.325\linewidth} 
\centering
\begin{tikzpicture}
\tikzmath{
\r=1.25;
\rr=0.08*\r;
}

\foreach \a in {-1,0,1} {
\tikzset{xshift=\a*\r cm}

\draw[]
	(-0.5*\r,0) -- (0.5*\r,0)
;
}

\foreach \a in {-1,0,1} {
\tikzset{xshift=\a*\r cm}
\fill[]
	(0,0) circle (\rr)
;
}

\foreach \aa in {-1,1} {
\tikzset{xscale=\aa}
\draw[]
	(0,0) -- (-0.5*\r,\r)
;
\node at (-0.6*\r,0.3*\r) {\footnotesize $n$-gon};
}

\node at (0,0.5*\r) {\footnotesize $?$};

\node at (0,-0.5*\r) {\footnotesize $m$-gon};

\end{tikzpicture}
\caption{}
\label{Subfig-forbidden-bbb-app}
\end{subfigure}
\begin{subfigure}[t]{0.325\linewidth} 
\centering
\begin{tikzpicture}
\tikzmath{
\r=1.25;
\rr=0.08*\r;
}

\foreach \a in {-1,0,1} {
\tikzset{xshift=\a*\r cm}

\draw[]
	(-0.5*\r,0) -- (0.5*\r,0)
;
}

\foreach \aa in {-1,1} {
\tikzset{xscale=\aa}
\draw[]
	(0,0) -- (-0.5*\r,\r)
	(-\r,0) -- (-0.5*\r,\r)
;
}

\fill[]
	(0,0) circle (\rr)
;

\foreach \aa in {-1,1} {
\tikzset{xscale=\aa}
\draw[fill=white]
	(-\r,0) circle (\rr)
;
}

\node at (0,0.5*\r) {\footnotesize $?$};

\node at (0,-0.5*\r) {\footnotesize $m$-gon};

\end{tikzpicture} 
\caption{}
\label{Subfig-forbidden-wbw-app}
\end{subfigure}

\caption{Forbidden arrangements}
\label{Fig-forbidden-2al3almaln-app}
\end{figure}

\medskip\noindent\emph{Type $\al_3\al_4^3$.}
We begin by determining the tilings when there is at least one square
having no adjacent triangles. We then derive a contradiction in the
case in which every square has at least one adjacent triangle.

The square $\mathbf{1}$ in the centre of
Figure~\ref{Subfig-rhombicuboctahedron} is adjacent to four squares
$\mathbf{2,3,4,5}$. It is straightforward to determine that tiles
$\mathbf{6,7,8,9}$ are triangles and that $\mathbf{10,\dots,17}$ are
squares. The situation is identical in Figure~\ref{Subfig-J37}.
Consider the undetermined vertex incident with $\mathbf{10}$ and $\mathbf{17}$; it
must have an incident triangle, which we denote by $\mathbf{18}$ (and
highlight with $\ast$). This triangle is adjacent either to
$\mathbf{17}$, as in Figure~\ref{Subfig-rhombicuboctahedron}, or to
$\mathbf{10}$, as in Figure~\ref{Subfig-J37}. In both cases the tiles
$\mathbf{19,\dots,26}$ are determined similarly. The tilings are, respectively,
the rhombicuboctahedron and $J_{37}$.

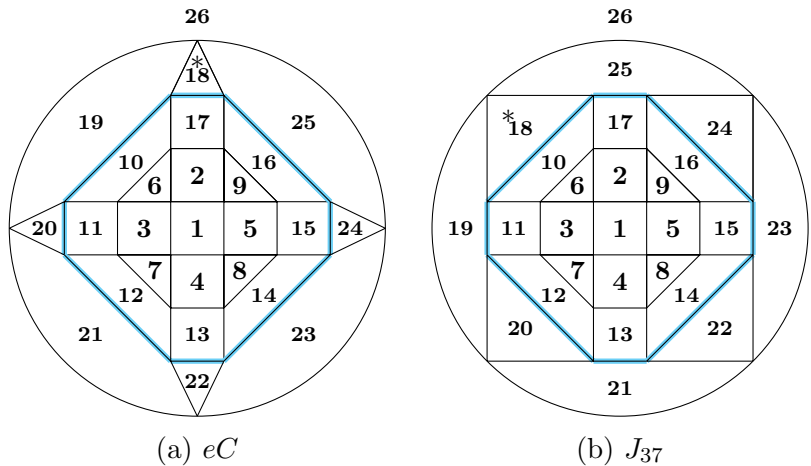
\begin{figure}[h!] 
\centering
\begin{subfigure}[t]{0.4\linewidth}
\centering
\begin{tikzpicture}

\tikzmath{
\r=0.5;
\th = 360/3;
\ph = 360/4;
\x = \r*cos(0.5*\ph);
\R=sqrt(2*(4*\x)^2);
}

\foreach \a in {0,1,2,3,4} {
\tikzset{rotate=\a*\ph}

\draw[cyan!50, line width=2]
	(-\x,5*\x) -- (\x,5*\x) 
	(\x,5*\x) -- (5*\x,\x)
;

\draw[]
	(90-0.5*\ph:\r) -- (90+0.5*\ph:\r)
	(\x,\x) -- (\x,3*\x)
	(-\x,\x) -- (-\x,3*\x)
	(-\x,3*\x) -- (\x,3*\x)
	(\x,3*\x) -- (3*\x,\x)
	(\x,\x) -- (\x,5*\x)
	(-\x,\x) -- (-\x,5*\x)
	(-\x,5*\x) -- (\x,5*\x)
	(\x,5*\x) -- (5*\x,\x)
	%
	%
	(\x,5*\x) -- (0,\R+\r)
	(-\x,5*\x) -- (0,\R+\r)
;
}

\draw[] (0,0) circle (\R+\r);

\node at (0,6.25*\x) {\small $\ast$};

\node at (0,0) {\footnotesize \bm{$1$}};
\node at (0,2*\x) {\footnotesize \bm{$2$}};
\node at (-2*\x,0) {\footnotesize \bm{$3$}};
\node at (0,-2*\x) {\footnotesize \bm{$4$}};
\node at (2*\x,0) {\footnotesize \bm{$5$}};

\node at (-1.6*\x,1.6*\x) {\footnotesize \bm{$6$}};
\node at (-1.6*\x,-1.6*\x) {\footnotesize \bm{$7$}};
\node at (1.6*\x,-1.6*\x) {\footnotesize \bm{$8$}};
\node at (1.6*\x,1.6*\x) {\footnotesize \bm{$9$}};

\node at (-2.5*\x,2.5*\x) {\scriptsize \bm{$10$}};
\node at (-4*\x,0) {\scriptsize \bm{$11$}};

\node at (-2.5*\x,-2.5*\x) {\scriptsize \bm{$12$}};
\node at (0,-4*\x) {\scriptsize \bm{$13$}};

\node at (2.5*\x,-2.5*\x) {\scriptsize \bm{$14$}};
\node at (4*\x,0) {\scriptsize \bm{$15$}};

\node at (2.5*\x,2.5*\x) {\scriptsize \bm{$16$}};
\node at (0,4*\x) {\scriptsize \bm{$17$}};

\node at (0,5.75*\x) {\scriptsize \bm{$18$}};
\node at (-4*\x,4*\x) {\scriptsize \bm{$19$}};

\node at (-5.75*\x,0) {\scriptsize \bm{$20$}};
\node at (-4*\x,-4*\x) {\scriptsize \bm{$21$}};

\node at (0,-5.75*\x) {\scriptsize \bm{$22$}};
\node at (4*\x,-4*\x) {\scriptsize \bm{$23$}};

\node at (5.75*\x,0) {\scriptsize \bm{$24$}};
\node at (4*\x,4*\x) {\scriptsize \bm{$25$}};

\node at (0,8*\x) {\scriptsize \bm{$26$}};

\end{tikzpicture}
\caption{$eC$}
\label{Subfig-rhombicuboctahedron}
\end{subfigure}
\begin{subfigure}[t]{0.4\linewidth}
\centering
\begin{tikzpicture}

\tikzmath{
\r=0.5;
\th = 360/3;
\ph = 360/4;
\x = \r*cos(0.5*\ph);
\R=sqrt(2*(4*\x)^2);
}

\foreach \a in {0,1,2,3,4} {
\tikzset{rotate=\a*\ph}

\draw[cyan!50, line width=2]
	(-\x,5*\x) -- (\x,5*\x) 
	(\x,5*\x) -- (5*\x,\x)
;

\draw[]
	(90-0.5*\ph:\r) -- (90+0.5*\ph:\r)
	(\x,\x) -- (\x,3*\x)
	(-\x,\x) -- (-\x,3*\x)
	(-\x,3*\x) -- (\x,3*\x)
	(\x,3*\x) -- (3*\x,\x)
	(\x,\x) -- (\x,5*\x)
	(-\x,\x) -- (-\x,5*\x)
	(-\x,5*\x) -- (\x,5*\x)
	(\x,5*\x) -- (5*\x,\x)
	(\x,5*\x) -- (5*\x,5*\x)
	(5*\x,\x) -- (5*\x,5*\x)
;
}

\draw[] (0,0) circle (\R+\r);

\node at (-4.2*\x,4.2*\x) {\small $\ast$};

\node at (0,0) {\footnotesize \bm{$1$}};
\node at (0,2*\x) {\footnotesize \bm{$2$}};
\node at (-2*\x,0) {\footnotesize \bm{$3$}};
\node at (0,-2*\x) {\footnotesize \bm{$4$}};
\node at (2*\x,0) {\footnotesize \bm{$5$}};

\node at (-1.6*\x,1.6*\x) {\footnotesize \bm{$6$}};
\node at (-1.6*\x,-1.6*\x) {\footnotesize \bm{$7$}};
\node at (1.6*\x,-1.6*\x) {\footnotesize \bm{$8$}};
\node at (1.6*\x,1.6*\x) {\footnotesize \bm{$9$}};

\node at (-2.5*\x,2.5*\x) {\scriptsize \bm{$10$}};
\node at (-4*\x,0) {\scriptsize \bm{$11$}};

\node at (-2.5*\x,-2.5*\x) {\scriptsize \bm{$12$}};
\node at (0,-4*\x) {\scriptsize \bm{$13$}};

\node at (2.5*\x,-2.5*\x) {\scriptsize \bm{$14$}};
\node at (4*\x,0) {\scriptsize \bm{$15$}};

\node at (2.5*\x,2.5*\x) {\scriptsize \bm{$16$}};
\node at (0,4*\x) {\scriptsize \bm{$17$}};

\node at (-3.75*\x,3.75*\x) {\scriptsize \bm{$18$}};
\node at (-6*\x,0) {\scriptsize \bm{$19$}};

\node at (-3.75*\x,-3.75*\x) {\scriptsize \bm{$20$}};
\node at (0,-6*\x) {\scriptsize \bm{$21$}};

\node at (3.75*\x,-3.75*\x) {\scriptsize \bm{$22$}};
\node at (6*\x,0) {\scriptsize \bm{$23$}};

\node at (3.75*\x,3.75*\x) {\scriptsize \bm{$24$}};
\node at (0,6*\x) {\scriptsize \bm{$25$}};

\node at (0,8*\x) {\scriptsize \bm{$26$}};

\end{tikzpicture}
\caption{$J_{37}$}
\label{Subfig-J37}
\end{subfigure}
\caption{The two tilings with a square adjacent to four squares}
\label{Fig-sq-adj-4-sqs}
\end{figure}

Now we assume that every square has an adjacent triangle.
Since no vertex is incident with two triangles, each square has at
most two adjacent triangles. We first assume that there is a square
with two adjacent triangles and derive a contradiction. We then show
that if there is a square with exactly one adjacent triangles, there
must be a square with two adjacent triangles, completing the argument.

The square $\mathbf{1}$ in the centre of
Figure~\ref{Fig-sq-consec-adj-tri} has two adjacent triangles
$\mathbf{2}$ and $\mathbf{3}$. We determine that tiles
$\mathbf{4,\dots,9}$ are squares. This determines that $\mathbf{10}$
is a square\footnote{It is indeed a square, despite appearing somewhat
triangular in the Figure.}. Every square has an adjacent
triangle by assumption; the triangle adjacent to square $\mathbf{4}$ 
is denoted by $\mathbf{11}$. This determines that $\mathbf{12}$ is a
square, which in turn determines that $\mathbf{13}$ is a triangle. A
symmetrical argument determines tiles $\mathbf{14}$, $\mathbf{15}$ and
$\mathbf{16}$; in particular, $\mathbf{16}$ is a triangle. But there
is a vertex incident with the two triangles $\mathbf{13}$ and
$\mathbf{16}$, which is a contradiction.

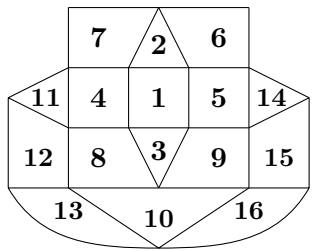
\begin{figure}[h!]
\centering
\begin{tikzpicture}
\tikzmath{
\x=0.4;
\X=3*\x;
\n=4;
\nn=\n-1;
\th=360/\n;
}

\foreach \a in {0,...,\nn} {
\tikzset{rotate=\a*\th}
\draw[]
	(\X,\X) -- (-\X,\X)
;
}

\foreach \aa in {-1,1} {
{
\tikzset{yscale=\aa}
\draw[]
	 (\X, \x) -- (-\X,\x)
;
}
{
\tikzset{xscale=\aa}
\draw[]
	(\x,\x) -- (\x,-\x)
	(\x,\x) -- (0,\X)
	(\x,-\x) -- (0,-\X)
	(\X,\x) -- (\X+2*\x,0)
	(\X,-\x) -- (\X+2*\x,0)
	(\X+2*\x,0) -- (\X+2*\x,-\X)
	(\X,-\X) -- (\X+2*\x,-\X)
	(\X,-\X) -- (0,-\X-2*\x)
	(0,-\X-2*\x) to[out=0,in=240] (\X+2*\x,-\X)
;
}
}

\node at (0,0) {\small \bm{$1$}};
\node at (0,1.75*\x) {\small \bm{$2$}};
\node at (0,-1.75*\x) {\small \bm{$3$}};

\node at (-2*\x,0) {\small \bm{$4$}};
\node at (2*\x,0) {\small \bm{$5$}};

\node at (2*\x,2*\x) {\small \bm{$6$}};
\node at (-2*\x,2*\x) {\small \bm{$7$}};

\node at (2*\x,-2*\x) {\small \bm{$9$}};
\node at (-2*\x,-2*\x) {\small \bm{$8$}};

\node at (0,-4*\x) {\footnotesize \bm{$10$}};
\node at (-3.75*\x,0) {\footnotesize \bm{$11$}};
\node at (-4*\x,-2*\x) {\footnotesize \bm{$12$}};
\node at (-\X,-\X-0.75*\x) {\footnotesize \bm{$13$}};

\node at (3.75*\x,0) {\footnotesize \bm{$14$}};
\node at (4*\x,-2*\x) {\footnotesize \bm{$15$}};
\node at (\X,-\X-0.75*\x) {\footnotesize \bm{$16$}};

\end{tikzpicture}
\caption{Deduction of a square with two adjacent triangles}
\label{Fig-sq-consec-adj-tri} 
\end{figure}

It remains to show that if there is a square with exactly one incident
triangle, there must be a square with two incident triangles. Consider
the square $\mathbf{4}$ in Figure~\ref{Fig-sq-consec-adj-tri}; if it
has exactly one adjacent triangle (namely, $\mathbf{11}$), that
determines that $\mathbf{1}$, $\mathbf{7}$ and $\mathbf{8}$ are squares.
This in turn determines that $\mathbf{2}$ and $\mathbf{3}$ are
triangles, giving the desired result.

\medskip\noindent\emph{Type $\al_3\al_4^2\al_5$.}
A vertex $\alpha_3\alpha_4^2\alpha_{5}$ has two angle arrangements, $\alpha_3\alpha_4\alpha_{5}\alpha_4$ and $\alpha_3\alpha_4\alpha_4\alpha_{5}$. 

If every vertex in the tiling has angle arrangement $\al_3 \al_4 \al_5
\al_4$, then it is straightforward to construct the tiling, namely,
$eD$, using the same process by which we constructed $aC$, $aD$ and
$eC$.

\begin{figure}[h!]
		\centering
		\begin{tikzpicture}
				\tikzmath{
						\r=0.65;
						\rr=0.1*\r;
						\n=5;
						\nn=\n-1;
						\N=2*\n;
						\Nn=\N-1;
						\M=\N+\n;
						\Mm=\M-1;
						\ph=360/\n;
						\x=\r*sin(0.5*\ph);
						\y=\r*cos(0.5*\ph);
						\l=1.055*sqrt( \r^2 + (2*\x)^2 - 2*\r*(2*\x)*cos(90+0.5*\ph) );
						\Ph=360/\N;
						\PH=360/\M;
						\L=3.1*\r;
						\LL=4.1*\r;
				}

				\fill[gray!40]
						(90-0.5*\Ph:\l) -- (90+0.5*\Ph:\l) -- (90+1.5*\Ph:\l) -- (90+2.5*\Ph:\l) 
						-- (90+3.5*\Ph:\l) -- (90+4.5*\Ph:\l) -- (90+5.5*\Ph:\l) -- (90+6.5*\Ph:\l) 
						-- (90+7.5*\Ph:\l) -- (90+8.5*\Ph:\l) -- (90+9.5*\Ph:\l)
						;

				\foreach \a in {0,...,\Nn} {
						\tikzset{rotate=\a*\Ph}
						\draw[]
								(90-0.5*\Ph:\l) -- (90+0.5*\Ph:\l)
								%
								;
				}

				\foreach \a in {0,...,\Mm} {
						\tikzset{rotate=\a*\PH+0.5*\ph}
						\draw[]
								(90:\L) -- (90+\PH:\L)
								%
								;
				}
				\foreach \a in {0,...,\nn} {
						\tikzset{rotate=\a*\ph+0.5*\ph}
						\draw[]
								(90-0.5*\Ph:\l) -- (90-\PH:\L)
								(90+0.5*\Ph:\l) -- (90+\PH:\L)
								;
				}

				\foreach \a in {0,...,\nn} {
						\tikzset{rotate=\a*\ph+0.5*\ph}
						\draw[]
								(90:\r) -- (90+\ph:\r)
								(90:\r) -- (90-0.5*\Ph:\l)
								(90:\r) -- (90+0.5*\Ph:\l)
								;
				}

				\node[inner sep=2,draw,fill=white,shape=circle] at (90+0.5*\Ph:\l) {\scriptsize $1$};
				\node[inner sep=2,draw,fill=white,shape=circle] at (90+1.5*\Ph:\l) {\scriptsize $2$};
				\node[inner sep=1.5,draw,fill=white,shape=circle] at (90+0.5*\ph:\r)
						{\scriptsize $1'$};
				\node[inner sep=2,draw,fill=white,shape=circle] at (90+2.5*\Ph:\l)
						{\scriptsize $3$};
				\node[inner sep=2,draw,fill=white,shape=circle] at (90+3.5*\Ph:\l)
						{\scriptsize $4$};
				\node[inner sep=1.5,draw,fill=white,shape=circle] at (90+1.5*\ph:\r)
						{\scriptsize $2'$};

		\end{tikzpicture}
		\caption{}
		\label{fig:app-P37}
\end{figure}

Suppose every vertex in the tiling is of type $\al_3 \al_4^2
\al_5$ but there is at least one vertex with angle arrangement
$\al_3 \al_4 \al_4 \al_5$. This vertex has a neighbour with which
it shares two squares and and a neighbour with which it shares a
triangle and a pentagon. The angle arrangement of these two
neighbours is clearly also $\al_3 \al_4 \al_4 \al_5$. Repeating
this argument, we deduce that there is a cycle $C$ of vertices
with angle arrangement $\al_3 \al_4 \al_4 \al_5$. We denote the
vertices of $C$ by $1, 2, \dots, n$, see Figure~\ref{fig:app-P37}.
This cycle separates the sphere into two parts; we claim that one
of these is a pentagonal cupola. Consider a triangle having two
vertices $1$ and $2$ on $C$; denote the third vertex of the
triangle by $1'$. The triangle shares the edge between $1$ and $2$
with a pentagon. The other two edges of the triangle are shared
with squares, one of which has vertices {$1'$}, {$2$} and {$3$};
denote its fourth vertex by {$2'$}. Observe that the square shares
the edge between
{$1'$} and {$2'$} with a pentagon. Now, consider
the tile on the same side of $C$ as the triangle having vertices
${3}$ and ${4}$. This tile cannot be a pentagon, since one of its
vertices is {$2'$}, and is therefore a triangle. Repeating this
argument, we see that the tiles on this side of $C$ alternate
between triangles and squares and share a pentagon, hence the
claim.

Rotating the pentagonal cupola by $\frac{1}{5}\pi$, we obtain a
new tiling; the vertices that had angle arrangement $\al_3 \al_4
\al_4 \al_5$ now have arrangement $\al_3 \al_4 \al_5 \al_4$.
Repeating this argument, we see that our tiling can be obtained
from $eD$ by rotating some
non-overlapping set of pentagonal cupolas by $\frac{1}{5}\pi$. Up
to symmetry, there are four sets of non-overlapping pentagonal
cupolas in $eD$. Rotating one cupola gives $J_{72}$. Rotating two
opposite cupolas gives $J_{73}$. Rotating two non-opposite cupolas
gives $J_{74}$. Rotating three cupolas gives $J_{75}$. There is no
set of four or more non-overlapping pentagonal cupolas in $eD$.
See Figure~\ref{fig:app-eDJ72-75} for diagrams of each of these
tilings.

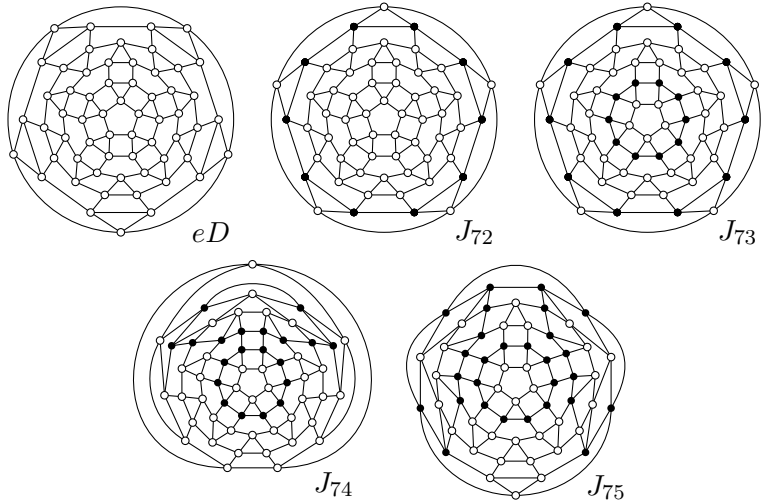
\begin{figure}
		\centering
		\begin{tikzpicture}

				\tikzmath{
						\S = 3.5;
				}

				\begin{scope}[] 

						\tikzmath{
								\XS = \S;
								\r=0.25;
								\rr=0.2*\r;
								\n=5;
								\nn=\n-1;
								\N=2*\n;
								\Nn=\N-1;
								\M=\N+\n;
								\Mm=\M-1;
								\ph=360/\n;
								\x=\r*sin(0.5*\ph);
								\l=1.055*sqrt( \r^2 + (2*\x)^2 - 2*\r*(2*\x)*cos(90+0.5*\ph) );
								\Ph=360/\N;
								\PH=360/\M;
								\L=3.1*\r;
								\LL=4.1*\r;
								\R=5.2*\r;
								\RR=6*\r;
						}

						\foreach \xs in {0,1,2} {
								\tikzset{xshift=\xs*\XS cm}

								\foreach \a in {0,...,\Nn} {
										\tikzset{rotate=\a*\Ph}
										\draw[]
												(90-0.5*\Ph:\l) -- (90+0.5*\Ph:\l)
												(90-0.5*\Ph:\R) -- (90+0.5*\Ph:\R)
												;
								}

								\foreach \a in {0,...,\Mm} {
										\tikzset{rotate=\a*\PH}
										(90-0.5*\PH:\L) -- (90+0.5*\PH:\L)
										;
						}

						\foreach \a in {0,...,\Mm} {
								\tikzset{rotate=\a*\PH}
								\draw[]
										(90-0.5*\PH:\L) -- (90+0.5*\PH:\L)
										(90:\LL) -- (90+\PH:\LL)
										;
						}

						\foreach \a in {0,...,\nn} {
								\tikzset{rotate=\a*\ph}
								\draw[]
										(90-0.5*\Ph:\l) -- (90-0.5*\PH:\L)
										(90+0.5*\Ph:\l) -- (90+0.5*\PH:\L)
										(90-0.5*\PH:\L) -- (90:\LL)
										(90+0.5*\PH:\L) -- (90:\LL)
										(90-1.5*\PH:\L) -- (90-\PH:\LL)
										(90+1.5*\PH:\L) -- (90+\PH:\LL)
										(90-\PH:\LL) -- (90-0.5*\Ph:\R)
										(90+\PH:\LL) -- (90+0.5*\Ph:\R)
										;
						}

						\draw[] (0,0) circle (\RR);

						\foreach \a in {0,...,\Mm} {
								\tikzset{rotate=\a*\PH}
								\draw[fill=white]
										(90+0.5*\PH:\L) circle (\rr) 
										(90:\LL) circle (\rr) 
										;
						}
				}

				\begin{scope}[] 

						\foreach \a in {0,...,\nn} {
								\tikzset{rotate=\a*\ph}
								\draw[]
										(90:\r) -- (90+\ph:\r)
										(90:\r) -- (90-0.5*\Ph:\l)
										(90:\r) -- (90+0.5*\Ph:\l)
										;
						}

						\foreach \a in {0,...,\nn} {
								\tikzset{rotate=\a*\ph+0.5*\ph}
								\draw[]
										(90:\RR) -- (90-0.5*\Ph:\R)
										(90:\RR) -- (90+0.5*\Ph:\R)
										;
						}

						\foreach \a in {0,...,\nn} {
								\tikzset{rotate=\a*\ph}
								\draw[fill=white]
										(90:\r) circle (\rr)
										;
						}

						\foreach \a in {0,...,\Nn} {
								\tikzset{rotate=\a*\Ph}
								\draw[fill=white]
										(90-0.5*\Ph:\l) circle (\rr)
										;
						}

						\foreach \a in {0,...,\Nn} {
								\tikzset{rotate=\a*\Ph}
								\draw[fill=white]
										(90-0.5*\Ph:\R) circle (\rr)
										;
						}

						\foreach \a in {0,...,\nn} {
								\tikzset{rotate=\a*\ph}
								\draw[fill=white]
										(270:\RR) circle (\rr)
										;
						}

						\node at (0.8*\RR,-\RR) {\small $eD$};

				\end{scope}

				\begin{scope}[xshift=\XS cm] 

						\foreach \a in {0,...,\nn} {
								\tikzset{rotate=\a*\ph}
								\draw[]
										(90:\r) -- (90+\ph:\r)
										(90:\r) -- (90-0.5*\Ph:\l)
										(90:\r) -- (90+0.5*\Ph:\l)
										;
						}

						\foreach \a in {0,...,\nn} {
								\tikzset{rotate=\a*\ph}
								\draw[]
										(90:\RR) -- (90-0.5*\Ph:\R)
										(90:\RR) -- (90+0.5*\Ph:\R)
										;

								\foreach \a in {0,...,\Nn}
								\tikzset{rotate=\a*\Ph}
								\fill[]
										(90-0.5*\Ph:\R) circle (\rr)
										;
						}

						\foreach \a in {0,...,\nn} {
								\tikzset{rotate=\a*\ph}
								\draw[fill=white]
										(90:\r) circle (\rr)
										;
						}

						\foreach \a in {0,...,\Nn} {
								\tikzset{rotate=\a*\Ph}
								\draw[fill=white]
										(90-0.5*\Ph:\l) circle (\rr)
										;
						}

						\foreach \a in {0,...,\nn} {
								\tikzset{rotate=\a*\ph}
								\draw[fill=white]
										(90:\RR) circle (\rr)
										;
						}

						\node at (0.8*\RR,-\RR) {\small $J_{72}$};

				\end{scope}

				\begin{scope}[xshift=2*\XS cm] 

						\foreach \a in {0,...,\nn} {
								\tikzset{rotate=\a*\ph+0.5*\ph}
								\draw[]
										(90:\r) -- (90+\ph:\r)
										(90:\r) -- (90-0.5*\Ph:\l)
										(90:\r) -- (90+0.5*\Ph:\l)
										;
						}

						\foreach \a in {0,...,\nn} {
								\tikzset{rotate=\a*\ph}
								\draw[]
										(90:\RR) -- (90-0.5*\Ph:\R)
										(90:\RR) -- (90+0.5*\Ph:\R)
										;
								\fill[]
										(90-0.5*\Ph:\l) circle (\rr)
										(90+0.5*\Ph:\l) circle (\rr)
										;

								\foreach \a in {0,...,\Nn}
								\tikzset{rotate=\a*\Ph}
								\fill[]
										(90-0.5*\Ph:\R) circle (\rr)
										;
						}

						\foreach \a in {0,...,\nn} {
								\tikzset{rotate=\a*\ph}
								\draw[fill=white]
										(90-0.5*\ph:\r) circle (\rr)
										;
						}

						\foreach \a in {0,...,\nn} {
								\tikzset{rotate=\a*\ph}
								\draw[fill=white]
										(90:\RR) circle (\rr)
										;
						}

						\node at (0.8*\RR,-\RR) {\small $J_{73}$};

				\end{scope}

		\end{scope}

		\begin{scope}[yshift=-\S cm] 

				\begin{scope}[xshift=0.5*\S cm]

						\tikzmath{
								\XS=\S;
						\r=0.225;
				\rr=0.05;
		\n=5;
\nn=\n-1;
\N=2*\n;
\Nn=\N-1;
\M=\N+\n;
\Mm=\M-1;
\ph=360/\n;
\x=\r*sin(0.5*\ph);
\l=1.055*sqrt( \r^2 + (2*\x)^2 - 2*\r*(2*\x)*cos(90+0.5*\ph) );
\Ph=360/\N;
\PH=360/\M;
\L=3.1*\r;
\LL=4.1*\r;
\R=5.2*\r;
\RR=6*\r;
\H=16;
\Hh=\H-1;
\Th=360/\H;
\h=4.25*\r;
\O=11;
\Oo=\O-1;
\Ps=360/\O;
\o=5.25*\r;
}

(90-0.5*\Ph:\l) -- (90+0.5*\Ph:\l) -- (90+1.5*\Ph:\l) -- (90+2.5*\Ph:\l) 
-- (90+3.5*\Ph:\l) -- (90+4.5*\Ph:\l) -- (90+5.5*\Ph:\l) -- (90+6.5*\Ph:\l) 
-- (90+7.5*\Ph:\l) -- (90+8.5*\Ph:\l) -- (90+9.5*\Ph:\l)
;

\foreach \a in {0,...,\nn} {
		\tikzset{rotate=\a*\ph 
		}
		(90:\r) -- (90+\ph:\r)
		(90:\r) -- (90-0.5*\Ph:\l)
		(90:\r) -- (90+0.5*\Ph:\l)
		;
}

\foreach \a in {0,...,\Nn} {
		\tikzset{rotate=\a*\Ph}
		\draw[]
				(90-0.5*\Ph:\l) -- (90+0.5*\Ph:\l)
				%
				;
}

\foreach \a in {0,...,\Mm} {
		\tikzset{rotate=\a*\PH}
		(90-0.5*\PH:\L) -- (90+0.5*\PH:\L)
		;
}

\foreach \a in {0,...,\Mm} {
		\tikzset{rotate=\a*\PH}
		\draw[]
				(90-0.5*\PH:\L) -- (90+0.5*\PH:\L)
				%
				;
}

\foreach \a in {0,...,\Hh} {
		\tikzset{rotate=\a*\Th}
		\draw[]
				(90-0.5*\Th:\h) -- (90+0.5*\Th:\h) 
				;
}

\foreach \a in {0,...,\Oo} {
		\tikzset{rotate=\a*\Ps}
		\draw[]
				(90:\o) -- (90+\Ps:\o)
				;
}

\foreach \a in {0,...,\nn} {
		\tikzset{rotate=\a*\ph}
		\draw[]
				(90-0.5*\Ph:\l) -- (90-0.5*\PH:\L)
				(90+0.5*\Ph:\l) -- (90+0.5*\PH:\L)
				%
				%
				;
}

\draw[]
		(90-0.5*\PH:\L) -- (90-0.5*\Th:\h)
		(90+0.5*\PH:\L) -- (90+0.5*\Th:\h)
		(90-1.5*\PH:\L) -- (90-0.5*\Th:\h)
		(90+1.5*\PH:\L) -- (90+0.5*\Th:\h)
		(90-1.5*\PH:\L) -- (90-2.5*\Th:\h)
		(90+1.5*\PH:\L) -- (90+2.5*\Th:\h)
		(90-2.5*\PH:\L) -- (90-3.5*\Th:\h)
		(90+2.5*\PH:\L) -- (90+3.5*\Th:\h)
		(90-3.5*\PH:\L) -- (90-3.5*\Th:\h)
		(90+3.5*\PH:\L) -- (90+3.5*\Th:\h)
		(90-4.5*\PH:\L) -- (90-4.5*\Th:\h)
		(90+4.5*\PH:\L) -- (90+4.5*\Th:\h)
		(90-4.5*\PH:\L) -- (90-5.5*\Th:\h)
		(90+4.5*\PH:\L) -- (90+5.5*\Th:\h)
		(90-5.5*\PH:\L) -- (90-6.5*\Th:\h)
		(90+5.5*\PH:\L) -- (90+6.5*\Th:\h)
		(90-6.5*\PH:\L) -- (90-6.5*\Th:\h)
		(90+6.5*\PH:\L) -- (90+6.5*\Th:\h)	
		(90-7.5*\PH:\L) -- (90-7.5*\Th:\h)
		(90+7.5*\PH:\L) -- (90+7.5*\Th:\h)
		(90-1.5*\Th:\h) -- (90:\o) 
		(90+1.5*\Th:\h) -- (90:\o) 
		(90-1.5*\Th:\h) -- (90-2*\Ps:\o) 
		(90+1.5*\Th:\h) -- (90+2*\Ps:\o) 	
		(90-2.5*\Th:\h) -- (90-2*\Ps:\o) 
		(90+2.5*\Th:\h) -- (90+2*\Ps:\o) 
		(90-4.5*\Th:\h) -- (90-3*\Ps:\o) 
		(90+4.5*\Th:\h) -- (90+3*\Ps:\o) 
		(90-5.5*\Th:\h) -- (90-4*\Ps:\o) 
		(90+5.5*\Th:\h) -- (90+4*\Ps:\o) 
		(90-7.5*\Th:\h) -- (90-5*\Ps:\o) 
		(90+7.5*\Th:\h) -- (90+5*\Ps:\o)
		(90-\Ps:\o) -- (90-2*\Ps:6*\r)
		(90-3*\Ps:\o) -- (90-2*\Ps:6*\r)
		(90+\Ps:\o) -- (90+2*\Ps:6*\r)
		(90+3*\Ps:\o) -- (90+2*\Ps:6*\r)
		(90-\Ps:\o) to[out=120,in=60] (90+\Ps:\o)
		(90-4*\Ps:\o) to[out=30,in=-60] (90-2*\Ps:6*\r)
		(90+4*\Ps:\o) to[out=180-30,in=180+60] (90+2*\Ps:6*\r)
		(90:7*\r) to[out=-10,in=120] (90-2*\Ps:6*\r)
		(90:7*\r) to[out=180+10,in=60] (90+2*\Ps:6*\r)
		(90-5*\Ps:\o) to[out=0,in=0,distance=1.6*\R cm] (90:7*\r)
		(90+5*\Ps:\o) to[out=180,in=180,distance=1.6*\R cm] (90:7*\r)
		; 

\foreach \aa in {-1,1} {
		\tikzset{xscale=\aa}
		\foreach \a in {0,1} {
				\tikzset{rotate=\a*\PH}
				\fill[]
						(90+0.5*\PH:\L) circle (\rr)
						;
		}
		\foreach \a in {2} {
				\tikzset{rotate=\a*\Th}
				\fill[]
						(90+0.5*\Th:\h) circle (\rr)
						;
		}
		\foreach \a in {0,1} {
				\tikzset{rotate=\a*\Ps}
				\fill[]
						(90+\Ps:\o) circle (\rr)
						;
		}
}

\foreach \aa in {-1,1} {
		\tikzset{xscale=\aa}
		\foreach \a in {2,...,6} {
				\tikzset{rotate=\a*\PH}
				\draw[fill=white]
						(90+0.5*\PH:\L) circle (\rr) 
						;
		}
		\foreach \a in {0,1,3,4,5,6,7} {
				\tikzset{rotate=\a*\Th}
				\draw[fill=white]
						(90+0.5*\Th:\h) circle (\rr)
						;
		}
		\foreach \a in {3,4,5} {
				\tikzset{rotate=\a*\Ps}
				\draw[fill=white]
						(90:\o) circle (\rr) 
						;
		}
		\draw[fill=white]
				(90+2*\Ps:6*\r) circle (\rr)
				;
}
\draw[fill=white]
		(90:\o) circle (\rr) 
		(90:7*\r) circle (\rr) 
		(270:\L) circle (\rr) 
		;

\foreach \a in {0,...,\nn} {
		\tikzset{rotate=\a*\ph  + 0.5*\ph
		}
		\draw[]
				(90:\r) -- (90+\ph:\r)
				(90:\r) -- (90-0.5*\Ph:\l)
				(90:\r) -- (90+0.5*\Ph:\l)
				;
		\fill[]
				(90-0.5*\Ph:\l) circle (\rr)
				(90+0.5*\Ph:\l) circle (\rr)
				;
}

\foreach \a in {0,...,\nn} {
		\tikzset{rotate=\a*\ph}
		\draw[fill=white]
				(90-0.5*\ph:\r) circle (\rr)
				;
}

\node at (0.8*\RR,-\RR) {\small $J_{74}$};
\end{scope}


\begin{scope}[xshift=1.5*\S cm]

		\tikzmath{
				\XS=\S;
		\r=0.25;
\rr=0.2*\r;
\n=5;
\nn=\n-1;
\N=2*\n;
\Nn=\N-1;
\M=\N+\n;
\Mm=\M-1;
\ph=360/\n;
\x=\r*sin(0.5*\ph);
\l=1.055*sqrt( \r^2 + (2*\x)^2 - 2*\r*(2*\x)*cos(90+0.5*\ph) );
\Ph=360/\N;
\PH=360/\M;
\L=3.1*\r;
\LL=4.1*\r;
\R=5.2*\r;
\RR=6*\r;
\H=17;
\Hh=\H-1;
\Th=360/\H;
\h=4.25*\r;
\O=12;
\Oo=\O-1;
\Ps=360/\O;
\o=5.25*\r;
}

(90-0.5*\Ph:\l) -- (90+0.5*\Ph:\l) -- (90+1.5*\Ph:\l) -- (90+2.5*\Ph:\l) 
-- (90+3.5*\Ph:\l) -- (90+4.5*\Ph:\l) -- (90+5.5*\Ph:\l) -- (90+6.5*\Ph:\l) 
-- (90+7.5*\Ph:\l) -- (90+8.5*\Ph:\l) -- (90+9.5*\Ph:\l)
;

\foreach \a in {0,...,\nn} {
		\tikzset{rotate=\a*\ph 
		}
		(90:\r) -- (90+\ph:\r)
		(90:\r) -- (90-0.5*\Ph:\l)
		(90:\r) -- (90+0.5*\Ph:\l)
		;
}

\foreach \a in {0,...,\Nn} {
		\tikzset{rotate=\a*\Ph}
		\draw[]
				(90-0.5*\Ph:\l) -- (90+0.5*\Ph:\l)
				%
				;
}

\foreach \a in {0,...,\Mm} {
		\tikzset{rotate=\a*\PH}
		(90-0.5*\PH:\L) -- (90+0.5*\PH:\L)
		;
}

\foreach \a in {0,...,\Mm} {
		\tikzset{rotate=\a*\PH}
		\draw[]
				(90-0.5*\PH:\L) -- (90+0.5*\PH:\L)
				%
				;
}

\foreach \a in {0,...,\Hh} {
		\tikzset{rotate=\a*\Th}
		\draw[]
				(90:\h) -- (90+\Th:\h) 
				;
}

\foreach \a in {0,...,\Oo} {
		\tikzset{rotate=\a*\Ps}
		\draw[]
				(90-0.5*\Ps:\o) -- (90+0.5*\Ps:\o)
				;
}

\foreach \a in {0,...,\nn} {
		\tikzset{rotate=\a*\ph}
		\draw[]
				(90-0.5*\Ph:\l) -- (90-0.5*\PH:\L)
				(90+0.5*\Ph:\l) -- (90+0.5*\PH:\L)
				;
}

\draw[]
		(90-0.5*\PH:\L) -- (90:\h)
		(90+0.5*\PH:\L) -- (90:\h)
		(90-1.5*\PH:\L) -- (90-\Th:\h)
		(90+1.5*\PH:\L) -- (90+\Th:\h)
		(90-1.5*\PH:\L) -- (90-3*\Th:\h)
		(90+1.5*\PH:\L) -- (90+3*\Th:\h)
		(90-2.5*\PH:\L) -- (90-3*\Th:\h)
		(90+2.5*\PH:\L) -- (90+3*\Th:\h)
		(90-3.5*\PH:\L) -- (90-4*\Th:\h)
		(90+3.5*\PH:\L) -- (90+4*\Th:\h)
		(90-4.5*\PH:\L) -- (90-4*\Th:\h)
		(90+4.5*\PH:\L) -- (90+4*\Th:\h)
		(90-4.5*\PH:\L) -- (90-6*\Th:\h)
		(90+4.5*\PH:\L) -- (90+6*\Th:\h)
		(90-5.5*\PH:\L) -- (90-7*\Th:\h)
		(90+5.5*\PH:\L) -- (90+7*\Th:\h)
		(90-6.5*\PH:\L) -- (90-7*\Th:\h)
		(90+6.5*\PH:\L) -- (90+7*\Th:\h)
		(90-7.5*\PH:\L) -- (90-8*\Th:\h)
		(90+7.5*\PH:\L) -- (90+8*\Th:\h)
		(90-\Th:\h) -- (90-0.5*\Ps:\o) 
		(90+\Th:\h) -- (90+0.5*\Ps:\o) 
		(90-2*\Th:\h) -- (90-0.5*\Ps:\o) 
		(90+2*\Th:\h) -- (90+0.5*\Ps:\o) 
		(90-2*\Th:\h) -- (90-2.5*\Ps:\o) 
		(90+2*\Th:\h) -- (90+2.5*\Ps:\o) 
		(90-5*\Th:\h) -- (90-2.5*\Ps:\o) 
		(90+5*\Th:\h) -- (90+2.5*\Ps:\o) 
		(90-5*\Th:\h) -- (90-4.5*\Ps:\o) 
		(90+5*\Th:\h) -- (90+4.5*\Ps:\o) 
		(90-6*\Th:\h) -- (90-4.5*\Ps:\o) 
		(90+6*\Th:\h) -- (90+4.5*\Ps:\o) 	
		(90-8*\Th:\h) -- (90-5.5*\Ps:\o) 
		(90+8*\Th:\h) -- (90+5.5*\Ps:\o) 
		(90-5.5*\Ps:\o) -- (270:6*\r)
		(90+5.5*\Ps:\o) -- (270:6*\r)
		(90-1.5*\Ps:\o) to[out=120,in=60, distance=0.75*\R cm] (90+1.5*\Ps:\o)
		(90-1.5*\Ps:\o) to[out=-30,in=60, distance=0.5*\R cm] (90-3.5*\Ps:\o)
		(90+1.5*\Ps:\o) to[out=180+30,in=180-60, distance=0.5*\R cm] (90+3.5*\Ps:\o)
		(90-3.5*\Ps:\o) to[out=270,in=0] (270:6*\r)
		(90+3.5*\Ps:\o) to[out=270,in=180] (270:6*\r)
		;

\foreach \aa in {-1,1} {
		\tikzset{xscale=\aa}
		\foreach \a in {2,3,4,5} {
				\tikzset{rotate=\a*\PH}
				\fill[]
						(90-0.5*\PH:\L) circle (\rr)
						;
		}
		\foreach \a in {1,6} {
				\tikzset{rotate=\a*\Th}
				\fill[]
						(90:\h) circle (\rr)
						;
		}
		\foreach \a in {0,1,3,4} {
				\tikzset{rotate=\a*\Ps}
				\fill[]
						(90+0.5*\Ps:\o) circle (\rr)
						;
		}
		\foreach \a in {2,...,8} {
				\tikzset{rotate=\a*\Th}
				\draw[fill=white]
						(90:\h) circle (\rr)
						;
		}
		\foreach \a in {2,5} {
				\tikzset{rotate=\a*\Ps}
				\draw[fill=white]
						(90+0.5*\Ps:\o) circle (\rr)
						;
		}
		\draw[fill=white]
				(90+0.5*\PH:\L) circle (\rr)
				(90+5.5*\PH:\L) circle (\rr)
				(90+6.5*\PH:\L) circle (\rr)
				;
}
\draw[fill=white]
		(90+7.5*\PH:\L) circle (\rr)
		(90:\h) circle (\rr)
		(270:6*\r) circle (\rr)
		;

\foreach \a in {0,...,\nn} {
		\tikzset{rotate=\a*\ph  + 0.5*\ph
		}
		\draw[]
				(90:\r) -- (90+\ph:\r)
				(90:\r) -- (90-0.5*\Ph:\l)
				(90:\r) -- (90+0.5*\Ph:\l)
				;

		\fill[]
				(90-0.5*\Ph:\l) circle (\rr)
				(90+0.5*\Ph:\l) circle (\rr)
				;
}

\foreach \a in {0,...,\nn} {
		\tikzset{rotate=\a*\ph}
		\draw[fill=white]
				(90-0.5*\ph:\r) circle (\rr)
				;
}

\node at (0.8*\RR,-0.925*\RR) {\small $J_{75}$};
\end{scope}

\end{scope} 

\end{tikzpicture}

\caption{Tilings $eD, J_{72},\dots,J_{75}$ with
		$\circ=\alpha_3\alpha_4\alpha_5\alpha_4$, $\bullet =
\alpha_3\alpha_4\alpha_4\alpha_5$ } 
\label{fig:app-eDJ72-75}
\end{figure}

\section{Classification}
\label{Sec-Tilings}

In this section, we completely determine the set of tilings for each
of the possible vertex types given in
\eqref{List-deg3-al3}-\eqref{List-al5}. Every tiling appears in the
list given in the main theorem, as required.

\begin{prop}
		\label{Prop-al3=pi/2-2pi/5} 
		If a tiling has a triangle and $\alpha_3 = \frac{1}{2}\pi$, then the
		tiling is either the octahedron or the Johnson tiling $J_1$; if
		$\alpha_3 = \frac{2}{5}\pi$, then the tiling is the icosahedron,
		the pentagonal antiprism or one of the Johnson tilings $J_2$,
		$J_{11}$, $J_{62}$ or $J_{63}$.
\end{prop}

\begin{proof} For $\alpha_3 = \frac{1}{2}\pi$ or $\frac{2}{5}\pi$, if a tiling having a triangle with angle value $\alpha_3$ is weakly vertex-homogenous, then by Theorem \ref{thm:weakly} such a tiling is the octahedron for $\alpha_3 = \frac{1}{2}\pi$ and the icosahedron for $\alpha_3 =\frac{2}{5}\pi$. In the remaining discussion, we assume that there is at least one $m$-gon where $m>3$. 
		\begin{case*}[$\alpha_3 = \frac{1}{2}\pi$] 


				For an $m$-gon with $m>3$, we may
				substitute the values $n = 3$ and $\alpha_n = \frac{1}{2}\pi$ into
				\eqref{Eq-alm-aln-linear} to obtain
				\begin{align*} 
						\cos \alpha_m = -1 - 2\cos \tfrac{2}{m}\pi. 
				\end{align*} 
				The above equation has no solution for $m>4$, since in that
				case the right hand side is less than $-1$; for $m=4$, the
				only solution is $\al_m = \pi$.
				We deduce that the tiling has exactly one
				4-gon which covers a hemisphere by
				Lemma~\ref{Lem-cong}. The other hemisphere is
				triangulated; thus, the tiling is the square pyramid $J_1$.

\end{case*}

\begin{case*}[$\alpha_3=\frac{2}{5}\pi$] 
		%

		For an $m$-gon with $m>3$, we may substitute
		$n = 3$ and $\alpha_3 = \frac{2}{5}\pi$ into
		\eqref{Eq-alm-aln-linear}, recalling that $\cos \frac{2}{5}\pi =
		\frac{1}{4}(\sqrt{5}-1)$, to obtain 
		\begin{align*}
				\cos \alpha_m 
				=
				\tfrac{1}{2}(\sqrt{5}-3+ (\sqrt{5}-5) \cos \tfrac{2}{m}\pi).
		\end{align*}
		For $m \ge 6$, the right hand side of the above equation is less
		than $-1$; therefore, the equation only has solutions for $m = 4$
		and $m = 5$. 
		
		For $m = 4$, the solutions are
		\begin{align*}
				\alpha_4 &= \cos^{-1} \tfrac{1}{2}(\sqrt{5}-3)\text{ and}\\
				\alpha_4 &= 2\pi-\cos^{-1} \tfrac{1}{2}(\sqrt{5}-3).
		\end{align*}

		For $m=5$, the solutions are
		\begin{align*}
				\alpha_5 &=\tfrac{4}{5}\pi\text{ and}\\
				\alpha_5 &= \tfrac{6}{5}\pi.
		\end{align*}

		Since $\alpha_3$ and $\alpha_5$ are multiples of
		$\tfrac{1}{5}\pi$, the same must hold for $\alpha_4$, a
		contradiction. Therefore there are only $3$- and $5$-gons in the
		tiling. 

		If there is at least one pentagon in the tiling and $\alpha_5 =
		\frac{6}{5}\pi$, then the pentagon is concave and hence there is
		exactly one pentagon; the complement of this pentagon is
		triangulated and therefore the tiling is given by the Johnson
		tiling $J_2$.

		If there is at least one pentagon in the tiling and $\alpha_5 =
		\frac{4}{5}\pi = 2\al_3$ then we may perform a pyramid subdivision
		on each pentagon in the tiling, thereby obtaining a new tiling
		with only triangular tiles; that is, we obtain the icosahedron by
		these subdivisions. From this we deduce that the remaining tilings
		in this case are obtained from an icosahedron by deleting an
		independent set of vertices. Deleting one vertex from the
		icosahedron yields $J_{11}$. Up to symmetry, there are two
		possible ways to delete two independent vertices from the
		icosahedron: deleting two vertices at distance 2 yields $J_{62}$;
		deleting two vertices at distance 3 yields the pentagonal
		antiprism. There is exactly one way to delete three independent
		vertices from the icosahedron, which yields $J_{63}$. Since there
		is no set of four or more independent vertices in the icosahedron,
		this completes the proof.\qedhere
\end{case*} 
\end{proof}

\begin{prop}
		\label{Prop-al32alm} 
		If a tiling has a vertex of type $\alpha_3\alpha_m^2$ with
		$m\ge 3$, then the tiling is the tetrahedron ($m = 3$), the
		triangular prism ($m=4$),
		one of the Johnson tilings $J_{62}$ or $J_{63}$ ($m = 5$), the
		truncated tetrahedron $tT$ ($m = 6$), the truncated cube $tC$ ($m
		= 8$), and the truncated dodecahedron $tD$ ($m = 10$). 
\end{prop}

\begin{proof}
		From the angle sum of $\al_3\al_m^2$, we have that $\al_m = \pi -
		\frac{1}{2}\al_3$; that is,
		$R(2\alpha_m) = \alpha_3$. Since $\alpha_3 < \alpha_n$ for any
		$n>3$ by Lemma~\ref{Lem-ang-asc-seq}, we deduce that
		$\alpha_m^2\cdots = \alpha_3\alpha_m^2$.
		We consider $m=3,4,5$ and $m\ge6$ as follows.
		
		If $m = 3$, then by
		Lemma~\ref{Lem-vertex-transitive} the tiling is vertex-homogenous, and by Theorem \ref{thm:weakly} the tiling is the tetrahedron. 

		If $m=4$, then substituting it and $\al_m = \pi - \frac{1}{2}\al_3$ into
		\eqref{Eq-alm-aln-linear}, we obtain
		\[
				\alpha_3 = 4 \tan^{-1} \tfrac{1}{\sqrt{7}},
				\hspace{1em}
				\alpha_4 = \pi - 2 \tan^{-1} \tfrac{1}{\sqrt{7}}.
		\]
		We see that $2\alpha_3 > \alpha_4 = R(\alpha_3\alpha_4)$. Since we
		also have $\alpha_3 < R(\alpha_3\alpha_4) < \alpha_5$, we deduce
		that $\alpha_3\alpha_4\cdots = \alpha_3\alpha_4^2$.
		By Lemma~\ref{Lem-vertex-transitive}, the tiling is vertex
		homogenous and by Theorem \ref{thm:weakly} is therefore the triangular prism. 

		If $m=5$, then it and $\alpha_m = \pi - \frac{1}{2}\alpha_3$ into
		\eqref{Eq-alm-aln-linear}, we obtain
		\[
				\alpha_3 = \tfrac{2}{5}\pi,
				\hspace{1em}
				\alpha_5 = \tfrac{4}{5}\pi.
		\]
		By Proposition~\ref{Prop-al3=pi/2-2pi/5}, the tiling
		is $J_{62}$ or $J_{63}$.

		If $m \ge 6$, then we have $\alpha_3 + \alpha_m > \pi$ by
		\eqref{Ineq-alm-lb}; that is, $R(\alpha_3\alpha_m)<\pi$. On the
		other hand, $3\alpha_3 > \pi$. Thus the angle sums that are at
		most
		$R(\alpha_3\alpha_4)$ are $2\alpha_3$, $\alpha_3+\alpha_4$,
		$\alpha_3+\alpha_5$, $\alpha_4+\alpha_5$ and $\alpha_m$. Since
		$\alpha_3\alpha_4\alpha_5\alpha_m$ is not a valid vertex type by
		Lemma~\ref{Lem-all-vertices}, we conclude that 
		\[
				\alpha_3\alpha_m\cdots \in 
				\{
						\alpha_3^3\alpha_m, \alpha_3^2\alpha_4\alpha_m,
						\alpha_3^2\alpha_5\alpha_m, \alpha_3\alpha_m^2
				\}.
		\]
		If $\alpha_3\alpha_m\cdots \not= \alpha_3\alpha_m^2$, then
		there is a vertex of type $\alpha_3 \alpha_\ell \alpha_m$ in the
		tiling for some $\ell \in \{3,4,5\}$. Rearranging the angle sum
		for this type gives $\alpha_m = 2\pi - 2\alpha_3 - \alpha_\ell$;
		substituting $\alpha_m = \pi - \frac{1}{2}\alpha_3$ into this
		equation gives $3\alpha_3 + \alpha_\ell = 2\pi$. If $\ell = 3$, we
		obtain $\alpha_3 = \frac{2}{5}\pi$ and thus $m = 5$ by
		Proposition~\ref{Prop-al3=pi/2-2pi/5}, a contradiction. But if
		$\ell \ge 4$ then $3\alpha_3 + \alpha_\ell > 2\pi$ by
		\eqref{Ineq-alm-lb}, also a contradiction. We conclude that
		$\alpha_3\alpha_m\cdots = \alpha_3\alpha_m$. Since
		$\alpha_m^2\cdots = \alpha_3\alpha_m$, by Lemma~\ref{Lem-vertex-transitive} the tiling is vertex
		transitive, and by Theorem \ref{thm:weakly}
		it is the truncated tetrahedron
		($m=6$); the truncated cube ($m=8$); and the truncated
		dodecahedron ($m=10$).\qedhere
		%

\end{proof}

\begin{prop}
		\label{Prop-2al3alm} 
		If a tiling has a vertex of type $\alpha_3^2\alpha_m$ with
		$m\ge4$, then the tiling is one of the Johnson tilings $J_1$ or
		$J_2$.
\end{prop}

\begin{proof} 
		A vertex of type $\al_3^2\al_m$ has a neighbour of type
		$\al_3^2\cdots$ and two neighbours of type
		$\al_3\al_m\cdots$. Since $R(\alpha_3\alpha_m) =
		\alpha_3$ and $\alpha_3 < \alpha_n$ for all $n>3$ by
		Lemma~\ref{Lem-ang-asc-seq}, we have that
		$\alpha_3\alpha_m\cdots = \alpha_3^2\alpha_m$.
		Therefore, each vertex of the $m$-gon is of type $\al_3^2\al_m$.
		Two adjacent vertices of this type share a neighbour of type
		$\al_3^2\cdots$. By induction, there is a vertex adjacent to every
		vertex of the $m$-gon, which is of type $\al_3^m$. The result
		follows by Proposition~\ref{Prop-al3=pi/2-2pi/5}.
\end{proof}

\begin{prop}
		\label{Prop-3al3alm} 
		If a tiling has a vertex of type $\alpha_3^3\alpha_m$ with $m \ge
		3$, then the tiling
		is the octahedron, one of the Johnson tilings $J_1$, $J_{11}$,
		$J_{62}$ or $J_{63}$ or an antiprism.
\end{prop}

\begin{proof} 

		If $m=3$, then $\alpha_3 = \tfrac{1}{2}\pi$. By Proposition
		\ref{Prop-al3=pi/2-2pi/5}, the tiling is the octahedron or $J_1$.
		For the rest of the proof we assume that $m \ge 4$.

		Rearranging the angle sum of $\alpha_3^3\alpha_m$ gives $\alpha_m
		= 2\pi - 3\alpha_3$. Since $\alpha_3 < \alpha_m$ by Lemma
		\ref{Lem-ang-asc-seq}, we have $\alpha_3 < \tfrac{1}{2}\pi$.
		Moreover, $R( \alpha_3^2\alpha_m ) = \alpha_3$ and therefore
		$\alpha_3^2 \alpha_m \cdots = \alpha_3^3 \alpha_m$.

		Let $u$ denote the given vertex. There is a neighbour $v$ of $u$
		of type $\alpha_3\alpha_m\cdots$, which is either
		$\alpha_3^3\alpha_m$ or $\alpha_3\alpha_m\alpha_n\cdots$ for some
		$n > 3$. If every vertex is of type $\alpha_3^3\alpha_m$, then by
		Proposition~\ref{Prop-vertex-homog} the tiling is vertex
		homogenous and must be an antiprism. We
		may assume that $u$ is of type $\al_3\al_m\al_n\cdots$.

		Since $R(\alpha_3\alpha_m\alpha_n) \le
		R(\alpha_3^2\alpha_m) = \alpha_3$, we deduce that
		$\alpha_3\alpha_m\alpha_n\cdots = \alpha_3\alpha_m\alpha_n$.
		If $m = n$, then by Proposition \ref{Prop-al32alm} the tiling is
		either $J_{62}$ or $J_{63}$.

		Combining the angle sums for $\al_3^3\al_m$ and $\al_3\al_m\al_n$,
		we find that $\al_n = 2\al_3$. Therefore, we can perform a pyramid
		subdivision on the $n$-gon to obtain another tiling
		(Lemma~\ref{lem:subdiv}). The new vertex in
		this tiling is of type $\al_3^n$. Since $\al_3 < \frac{1}{2}\pi$
		we have $n>4$; by Lemma \ref{Lem-all-vertices}, $n=5$. But then
		$\al_3 = \frac{2}{5}\pi$, and by Proposition
		\ref{Prop-al3=pi/2-2pi/5}, the tiling is $J_{11}$, $J_{62}$ or
		$J_{63}$.
\end{proof}

\begin{prop} 
		\label{Prop-2al32alm} 
		If a tiling has a vertex of type $\alpha_3^2\alpha_m^2$ with
		$m\ge4$, then either
		\begin{enumerate}
				\item $m=4$ and the tiling is the cuboctahedron $aC$ or one of
						the Johnson tilings $J_3$ or $J_{27}$; or
				\item $m=5$ and the tiling is the icosidodecahedron $aD$ or
						one of the Johnson tilings $J_6$ or $J_{34}$.
		\end{enumerate}
\end{prop}

\begin{proof} 
		From the angle sum of $\al_3^2\al_m^2$ we obtain $\al_m = \pi -
		\al_3$. By \eqref{Ineq-alm-lb}, $\al_3 > \frac{1}{3}\pi$ and hence
		$\al_m < \frac{2}{3}\pi$. Again by \eqref{Ineq-alm-lb}, we see that
		$m < 6$; i.e., $m=4$ or $m=5$.

		If every vertex is of the type $\alpha_3^2\alpha_m^2$ for $m\ge4$, then by Theorem \ref{thm:weakly} the tilings are $aC$ and $J_{27}$ for $m=4$; and $aD$ and $J_{34}$ for $m=5$. 
		
		In what follows, we assume that there is a vertex of type different from $\alpha_3^2\alpha_m^2$.
		\begin{case*}
				[$m=4$]
				Substituting $n=3$, $m=4$ and $\al_4 = \pi - \al_3$ into
				\eqref{Eq-alm-aln-linear}, we obtain $\al_3 = \cos^{-1}
				\frac{1}{3}$.



				Let $u$ be a vertex of type $\al_3^2\al_4^2$ and suppose there
				is at least one vertex $v$ of a different type. We may assume
				that $u$ and $v$ are adjacent; if $u$ is $\al_3\al_4\al_3\al_4$ then $v$ is of type $\al_3\al_4\cdots$.
				If $u$ is $\al_3\al_3\al_4\al_4$ and $v$ is
				type $\al_3^2\cdots$ or $\al_4^2\cdots$, then $v$ is incident
				to an $n$-gon with $n>4$ which must also be incident to the
				common neighbour of $u$ and $v$, whose type is
				$\al_3\al_4\cdots$. We may therefore assume that $v$ is of
				type $\al_3\al_4\cdots$. Note that $R(\al_3\al_4) = \pi$ and
				that $2\al_3 < \pi < 2\al_4$. It follows that $v$ is of type
				$\al_3\al_4\al_n$ for some $n$; subtituting $m = 3$, $\al_m =
				\cos^{-1} \frac{1}{3}$ and $\al_n = \pi$ into
				\eqref{Eq-alm-aln-linear}, we obtain $n = 6$. Since $\al_6 =
				\al_3 + \al_4$, we may perform a cupola subdivision on each
				$6$-gon in the tiling to perform a new tiling
				(Lemma~\ref{lem:subdiv}). The vertices on the boundary of this
				tiling, obtained from vertices of type $\al_3\al_4\al_6$, are
				of type $\al_3^2\al_4^2$; the new vertices are of the same
				type. By this operation we obtain $aC$.
				Therefore, the original tiling is obtained from
				$aC$ by diminishing a triangular cupola; that is,
				the tiling is $J_3$.

		\end{case*}

		\begin{case*}
				[$m=5$]
				Substituting $n=3$, $m=5$ and $\al_5 = \pi - \al_3$ into
				\eqref{Eq-alm-aln-linear} we obtain $\al_3 = \cos^{-1}
				\frac{1}{\sqrt{5}}$.



				Let $u$ be of type $\al_3^2\al_5^2$ and let $v$ be a vertex of
				a different type; as in the previous case, we may assume that
				$v$ is of type $\al_3\al_5\al_n$ for some $n$. Substituting
				$\al_3 = \cos^{-1} \frac{1}{\sqrt{5}}$ into
				\eqref{Eq-alm-aln-linear}, we obtain $n = 10$.
				The argument that the tiling must be $J_6$ is identical to the
				argument given in the previous case.\qedhere

		\end{case*}

\end{proof}

\begin{prop}
		\label{Prop-al33al4} 
		If a tiling has a vertex of type $\alpha_3\alpha_4^3$,
		then the tiling is the rhombicuboctahedron $eC$ or one of the
		Johnson tilings $J_4$, $J_{19}$ or $J_{37}$.
\end{prop}

\begin{proof} 

		The angle sum of $\al_3\al_4^3$ implies that $\al_3 = 2\pi
		- 3\al_4$ and in particular that $\al_4 < 2\al_3$. Also,
		$\al_4<\al_m$ for $m>4$ by Lemma~\ref{Lem-ang-asc-seq}. Since
		$R(\al_3\al_4^2) = \al_4$, we have that $\al_3\al_4^2\cdots =
		\al_3\al_4^3$. We also have that $\al_4^3\cdots =
		\al_3\al_4^3$ since $R(\al_4^3) = \al_3$.

		Substituting $m=3$, $n=4$ and $\al_3 = 2\pi-3\al_4$ into
		\eqref{Eq-alm-aln-linear}, we can compute an exact formula for
		$\al_4$, namely:
		\[
				\al_4 = 2\tan^{-1} \sqrt{7 - 4\sqrt{2}}.
		\]

		If every vertex in the tiling is of type $\al_3\al_4^3$, then by
		Proposition~\ref{Prop-vertex-homog} the tiling is either $eC$ or
		$J_{37}$. We may therefore assume that there is at least one
		vertex, denoted by $u$, of a different type. We may further assume
		that $u$ is either of type $\al_3\al_4\cdots$ or of type
		$\al_4^2\cdots$.

		\begin{case*}[$u$ is of type $\al_4^2\cdots$]
				Since $\al_3\al_4^2\cdots = \al_4^3\cdots = \al_3\al_4^3$,
				we have that $u$ is of type $\al_4^2\al_m\cdots$ for some
				$m>4$. In fact, since $R(\al_4^2\al_m) < R(\al_4^3) = \al_3$,
				we have that $u$ is of type $\al_4^2\al_m$. Now $\al_m = \al_3
				+ \al_4 = 2\pi - 2\al_4$; substituting $n=4$,
				$\al_m = 2\pi - 2\al_4$ and our exact value for $\al_4$ given
				above into
				\eqref{Eq-alm-aln-linear}, we obtain $m=8$.

				It is easy to see that every vertex incident with an octagon
				is also of type $\al_4^2\al_8$ by considering that $\al_3 +
				\al_4 < R(\al_8) = 2\al_4 < \al_{n} + \al_{n'}$ for
				any $n,n' \geq 4$. Since
				$\al_8 = \al_3 + \al_4$, we can perform a cupola subdivision
				on each octagon (Lemma~\ref{lem:subdiv}). This operation
				transforms each vertex of type $\al_4^2\al_8$ into one of type
				$\al_3\al_4^3$; the new vertices are also of type
				$\al_3\al_4^3$. Hence the obtained tiling is $eC$ or $J_{37}$.
				Therefore, the original tiling is obtained from $eC$ or
				$J_{37}$ by diminishing a square cupola; that is, the tiling
				is $J_{19}$.
		\end{case*}

		\begin{case*}[$u$ is of type $\al_3\al_4\cdots$]
				We claim that $u$ cannot be of type $\al_3^2\al_4\cdots$. To
				see this, consider that $\al_4 < 2\al_3 < R(\al_3^2\al_4) =
				2\al_4 - \al_3 < \al_4 + \al_3 < 3\al_3$. From this we deduce
				that $\al_3^2\al_4\cdots = \al_3^2\al_4\al_m$ for some $m>4$
				and that $\al_m = 2\al_4 - \al_3 = 5\al_4 - 2\pi$.
				Substituting into \eqref{Eq-alm-aln-linear} we
				find that $m = 5.7325...$, which is a contradiction since $m$
				is an integer; this gives the claim. Since $\al_3\al_4^2\cdots
				= \al_3\al_4^3$ and by assumption $u$ is not of type
				$\al_3\al_4^3$, we deduce that $u$ is of type
				$\al_3\al_4\al_m\cdots$ for some $m>4$. But
				$R(\al_3\al_4\al_n) \leq R(\al_3\al_4^2) = \al_4 <2\al_3$,
				from which we conclude that $u$ is of type $\al_3\al_4\al_m$.

				The angle sum of $\al_3\al_4\al_m$ implies that $\al_m
				= 2\pi - \al_3 - \al_4 = 2\al_4$. Substituting
				into \eqref{Eq-alm-aln-linear}, we obtain $m=8$. Since $\al_8
				= 2\al_4 > \pi$, the octagon is concave; that is, there is
				exactly one. It is easy to see that the neighbour with which
				$u$ shares a triangle and the octagon is also of type
				$\al_3\al_4\al_8$ since $\al_3 < R(\al_3\al_8)<
				\min\{2\al_3,\al_{n>4}\}$. By a similar argument, the
				neighbour with which $u$ shares a square and the octagon (and
				hence every vertex incident with the octagon by induction) is
				of type $\al_3\al_4\al_8$.

				Since $\al_8 = 2\al_4$, we can perform a prism subdivision on
				the octagon. This operation transforms each vertex of type
				$\al_3\al_4\al_8$ into one of type $\al_3\al_4^3$; the new
				vertices are of type $\al_4^2\al_8$. Therefore the obtained
				tiling is the one we found in the previous case, namely,
				$J_{19}$. The original tiling is obtained from $J_{19}$ by
				deleting the vertices of an octagon; that is, the tiling is
				$J_4$. \qedhere
		\end{case*}
\end{proof}

\begin{prop}\label{Prop-al32al4al5} 
		If a tiling has a vertex of type $\alpha_3\alpha_4^2\alpha_5$,
		then the tiling is the rhombicosidodecahedron $eD$ or one of the
		Johnson tilings $J_5$ or $J_{72},\dots,J_{83}$.
\end{prop}

\begin{proof} 
		If every vertex is of the type $\alpha_3\alpha_4^2\alpha_5$, then the tilings are weakly vertex-homogenous. By Theorem \ref{thm:weakly}, they are $eD$ and $J72-75$.
		
		Next we assume that there are vertices of other type and use trigonometry and algebra to prove that the only other possible vertex types in the tiling are
		$\al_3\al_4\al_{10}$ and $\al_4 \al_5 \al_{10}$.

		The angle sum of $\alpha_3\alpha_4^2\alpha_5$ implies
		$\alpha_4 < \pi$ and gives $2\alpha_4=2\pi - (\alpha_3 +
		\alpha_5)$. Taking cosine and sine respectively on both sides
		gives
		\begin{align*}
				&\sin \alpha_3 \sin \alpha_5 - \cos \alpha_3 \cos \alpha_5 +
				2\cos^2 \alpha_4 - 1 = 0, \\
				&\cos \alpha_3\sin \alpha_5 + \sin \alpha_3 \cos \alpha_5 +
				2\cos \alpha_4 \sin \alpha_4 = 0. 
		\end{align*}

		Substituting the pairs $(\alpha_3, \alpha_4), (\alpha_3,
		\alpha_5), (\alpha_4, \alpha_5)$ into \eqref{Eq-alm-aln-linear}
		respectively, with $\cos \frac{2}{5}\pi =
		\frac{1}{4}(\sqrt{5}-1)$, gives
		\begin{align*}
				& 2\cos \alpha_3 - \cos \alpha_4 - 1 = 0, \\
				& \cos^2 \alpha_5 - 3\cos \alpha_3 \cos \alpha_5 + \cos
				\alpha_5 +\cos^2 \alpha_3+\cos \alpha_3 -1 =0, \\
				& 4\cos^2 \alpha_5 - 6\cos \alpha_4 \cos \alpha_5  - 2\cos
				\alpha_5 + \cos^2 \alpha_4 +4\cos \alpha_4  - 1 = 0. 
		\end{align*}

		Combining the above with $\cos^2\alpha_i + \sin^2\alpha_i - 1 = 0$
		for $i = 3,4,5$ gives a system of polynomials, where the variables
		are $x_i = \cos \alpha_i$ and $y_i = \sin \alpha_i$: 
		\begin{align}\label{PolySys-3445}
&2x_4^2+y_3y_5-x_3x_5-1, \quad 2x_4y_4+x_3y_5+x_5y_3, \quad
2x_3-x_4-1, \\ \notag
&x_5^2+x_3^2-3x_3x_5+x_5+x_3-1, \quad
4x_5^2+x_4^2-6x_4x_5-2x_5+4x_4-1, \\ \notag
&x_3^2+y_3^2-1, \quad x_4^2+y_4^2-1, \quad x_5^2+y_5^2-1.  
		\end{align}
		A reduced Gr\"obner basis for the above system is given below: 
		\begin{align}\label{Grobner-3445}
&1600y_4^{11}-2960y_4^9+1484y_4^7-123y_4^5, \\ \notag
&2400y_4^{10}-3840y_4^8+1326y_4^6-160y_4^5y_5+153y_4^4+120y_4^3y_5, \\
\notag
&-24000y_4^{10}+42000y_4^8-18020y_4^6-y_4^4-4y_4^2-8x_4+8, \\ \notag
&-24000y_4^{10}+42000y_4^8-18020y_4^6-y_4^4-4y_4^2-16x_3+16, \\ \notag
&-4800y_4^{10}+8880y_4^8-4452y_4^6-31y_4^4+140y_4^2+160y_4y_5-80x_5+80,
\\ \notag
&-800y_4^9+1440y_4^7-642y_4^5-16y_4^4y_5+5y_4^3+4y_4^2y_5+2y_5+4y_4+2y_3,
\\ \notag
&-4800y_4^{10}+8880y_4^8-3652y_4^6-631y_4^4-480y_4^3y_5+80y_5^2+280y_4^2+320y_4y_5. 
		\end{align}
		The first polynomial in the reduced Gr\"obner basis is univariate.
		Since $y_4 = \sin \alpha_4 > 0$ for $\alpha_4 \in
		(\frac{1}{2}\pi,\pi)$, the suitable  real roots of the univariate
		polynomial are those between $0$ and $1$; these are listed below:
		\begin{align}\label{Roots-y4-3445}
				y_4 = \textstyle\sqrt{ \frac{11-4\sqrt{5}}{2\sqrt{5}} }, \quad
				\sqrt{ \frac{ 4\sqrt{5}+11}{2\sqrt{5} } }, \quad
				\frac{\sqrt{3}}{2}.
		\end{align}
		Each value of $y_4$ can be used to find a zero of the second
		polynomial in the basis, giving a value for $y_5$. Repeating this
		process, we obtain values for $x_3 = \cos \alpha_3$, $x_4 =
		\cos \alpha_4$ and $x_5 = \cos \alpha_5$ which give solutions to
		our original system of trigonometric equations; these are as
		follows:
		\begin{align*}
&x_3 = \tfrac{1}{20} (5 - 2\sqrt{5}),& &x_4 = -\tfrac{1}{10} (5 +
				2\sqrt{5}),& &x_5 = \tfrac{1}{40} (5 + 9\sqrt{5});& \\
									 &x_3 = \tfrac{1}{20} (5 + 2\sqrt{5}),& &x_4=
				\tfrac{1}{10} (2\sqrt{5} - 5),& &x_5= \tfrac{1}{40} (5 -
				9\sqrt{5});&\\
									 &x_3=\tfrac{1}{4},& &x_4=-\tfrac{1}{2},&
									 &x_5=-\tfrac{3\sqrt{5}+1}{8};& \\
									 &x_3=\tfrac{1}{4},& &x_4=-\tfrac{1}{2},&
									 &x_5=\tfrac{3\sqrt{5}-1}{8}.&
		\end{align*}
		Only the second of the above satisfies both $\al_3 < \al_4 <
		\al_5$ and $\al_3 + 2 \al_4 + \al_5 = 2\pi$. We conclude that
		\begin{align}\label{Eq-3445-angles}
&\alpha_3 = \cos^{-1}\tfrac{1}{20} (5 + 2\sqrt{5}) =
				(0.342951...)\pi,& \\ \notag
												 &\alpha_4 = \cos^{-1} \tfrac{1}{10}
				(2\sqrt{5} - 5) = (0.516810...)\pi,&  \\ \notag
																					 &\alpha_5 = \cos^{-1}
				\tfrac{1}{40} (5 - 9\sqrt{5}) = (0.623427...)\pi. &
		\end{align}

		A vertex adjacent to $\alpha_3\alpha_4^2\alpha_{5}$ is one of
		$\alpha_3\alpha_4\cdots, \alpha_3\alpha_5\cdots, \alpha_4^2\cdots,
		\alpha_4\alpha_5\cdots$. Combining the angle values with
		\eqref{List-deg3-al3},  \eqref{List-deg4-al3},
		\eqref{List-deg5-al3}, \eqref{List-al4}, and \eqref{List-al5}
		gives
		\begin{align*}
				\alpha_3\alpha_4\cdots &\in \{\alpha_3\alpha_4\alpha_{n},
				\alpha_3^2\alpha_4\alpha_n, \alpha_3\alpha_4^2\alpha_{5}\}; \\
						\alpha_3\alpha_5\cdots &\in \{\alpha_3\alpha_5\alpha_n,
						\alpha_3^2\alpha_5\alpha_n, \alpha_3\alpha_4^2\alpha_5\};
						\\
								\alpha_4^2\cdots &\in \{\alpha_4^2\alpha_n,
								\alpha_3\alpha_4^2\alpha_5\}; \\  
										\alpha_4\alpha_5\cdots &\in
										\{\alpha_4\alpha_5\alpha_n,
										\alpha_3\alpha_4^2\alpha_5\}.
								\end{align*}
								The angle sums of $\alpha_3\alpha_4^2\alpha_{5}$ and
								one of the other vertices give an expression for
								$\alpha_n$ which we can substitute into
								\eqref{Eq-alm-aln-linear} to determine $n$; we give
								the list of obtained values below:
								\begin{align*}
&\alpha_3\alpha_4\alpha_{n}:& 
&\alpha_n = \alpha_4 + \alpha_5,&
&n=10;& \\
&\alpha_3^2\alpha_4\alpha_{n}:&
&\alpha_n = \alpha_4 + \alpha_5 - \alpha_3,&
&n=8.093977...;& \\
&\alpha_3\alpha_5\alpha_n:&
&\alpha_n = 2\alpha_4,&
&n=13.551639...;& \\
&\alpha_3^2\alpha_5\alpha_n:&
&\alpha_n = 2\alpha_4 - \alpha_3,&
&\text{no solution for }n;& \\
&\alpha_4^2\alpha_n:&
&\alpha_n=\alpha_3+\alpha_5,&
&n=13.551639...;& \\
&\alpha_4\alpha_5\alpha_n:&
&\alpha_n=\alpha_3+\alpha_4,&
&n=10.&
								\end{align*}
								Hence 
								\begin{align*} 
		\alpha_3\alpha_4\cdots &\in \{\alpha_3\alpha_4\alpha_{10},
		\alpha_3\alpha_4^2\alpha_{5}\}; \\
		\alpha_3\alpha_5\cdots &= \alpha_3\alpha_4^2\alpha_5; \\
		\alpha_4^2\cdots &= \alpha_3\alpha_4^2\alpha_5; \\
		\alpha_4\alpha_5\cdots &\in \{\alpha_4\alpha_5\alpha_{10},
		\alpha_3\alpha_4^2\alpha_5\}. 
								\end{align*} 
								Therefore the set of possible vertex types is 
								\begin{align*} 
										\{ \alpha_3\alpha_4^2\alpha_{5}, \,
										\alpha_3\alpha_4\alpha_{10}, \,
								\alpha_4\alpha_5\alpha_{10} \}.
						\end{align*}

If there is a vertex of type $\al_3 \al_4 \al_{10}$ then $\al_{10} =
\al_4 + \al_5$; similarly, if there is a vertex of type $\al_3 \al_5
\al_{10}$ then $\al_{10} = \al_3 + \al_4$. By
Lemma~\ref{Lem-cong}, at most one of these vertex types can appear
in the tiling.

Suppose there is a vertex of type $\al_3 \al_4 \al_{10}$. Then,
denoting the vertices of the decagon by $1,\dots,10$, we follow an
analogous argument to the case in which all vertices are of type
$\al_3 \al_4^2 \al_5$, see Figure~\ref{fig:app-P37}. We deduce that the
complement of the decagon in the tiling is a pentagonal cupola; that
is, the tiling is $J_5$.

Suppose there is a vertex of type $\al_4 \al_5 \al_{10}$. Then, since
$\al_{10} = \al_3 + \al_4$, we can perform a cupola subdivision on the
decagon to obtain a new tiling (Lemma~\ref{lem:subdiv}). The vertices
of the decagon which had type $\al_4 \al_5 \al_{10}$ in the original
tiling have type $\al_3 \al_4^2 \al_5$ in the new tiling; note that
they may have angle arrangement $\al_3 \al_4 \al_4 \al_5$, in which
case we rotate the cupola by $\frac{1}{5}\pi$. Performing these
operations at each decagon, we obtain $eD$. Therefore, our original
tiling can be obtained from $eD$ by rotating and diminishing
non-overlapping pentagonal cupolas. Diminishing one cupola gives
$J_{76}$. Diminishing one cupola and rotating an opposite cupola gives
$J_{77}$. Diminishing one cupola and rotation a non-opposite cupola
gives $J_{78}$. Diminishing one cupola and rotating two cupolas gives
$J_{79}$. Diminishing two opposite cupolas gives $J_{80}$. Diminishing
two non-opposite cupolas gives $J_{81}$. Diminishing two non-opposite
cupolas and rotating a third gives $J_{82}$. Finally, diminishing
three cupolas gives $J_{83}$. See Figure~\ref{Fig-J76-J83} for
diagrams of each of these tilings.\qedhere

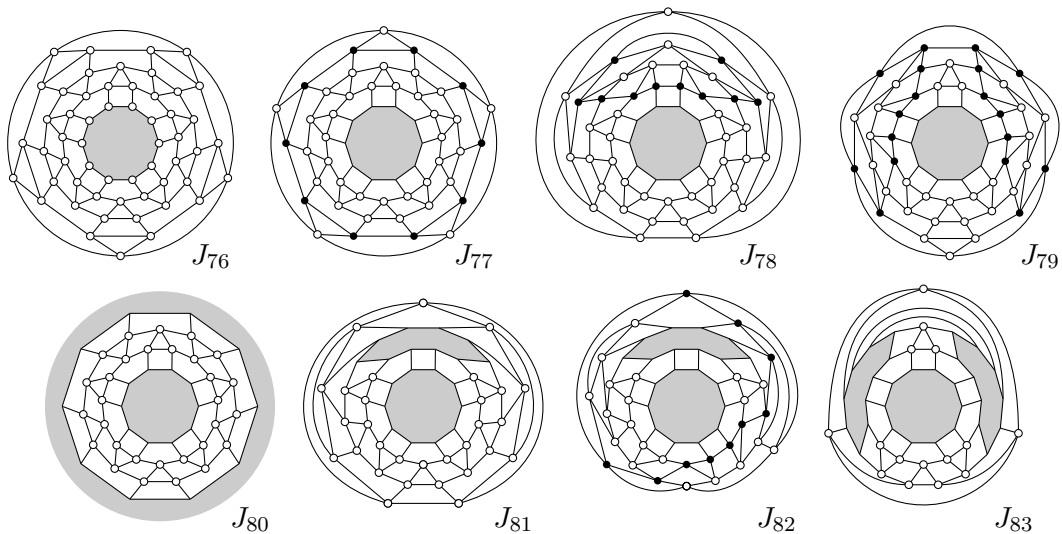
\begin{figure}[h!] 
\centering
\begin{tikzpicture}

\tikzmath{
\S = 3.5;
}

\begin{scope} 
				
\begin{scope}

				\tikzmath{
								\XS = \S;
								\r=0.25;
								\rr=0.2*\r;
								\n=5;
								\nn=\n-1;
								\N=2*\n;
								\Nn=\N-1;
								\M=\N+\n;
								\Mm=\M-1;
								\ph=360/\n;
								\x=\r*sin(0.5*\ph);
								\l=1.055*sqrt( \r^2 + (2*\x)^2 - 2*\r*(2*\x)*cos(90+0.5*\ph) );
								\Ph=360/\N;
								\PH=360/\M;
								\L=3.1*\r;
								\LL=4.1*\r;
								\R=5.2*\r;
								\RR=6*\r;
						}

								\fill[gray!40]
										(90-0.5*\Ph:\l) -- (90+0.5*\Ph:\l) -- (90+1.5*\Ph:\l) -- (90+2.5*\Ph:\l) 
										-- (90+3.5*\Ph:\l) -- (90+4.5*\Ph:\l) -- (90+5.5*\Ph:\l) -- (90+6.5*\Ph:\l) 
										-- (90+7.5*\Ph:\l) -- (90+8.5*\Ph:\l) -- (90+9.5*\Ph:\l)
										;

								\foreach \a in {0,...,\Nn} {
										\tikzset{rotate=\a*\Ph}
										\draw[]
												(90-0.5*\Ph:\l) -- (90+0.5*\Ph:\l)
												(90-0.5*\Ph:\R) -- (90+0.5*\Ph:\R)
												;
								}

								\foreach \a in {0,...,\Mm} {
										\tikzset{rotate=\a*\PH}
										(90-0.5*\PH:\L) -- (90+0.5*\PH:\L)
										;
						}

						\foreach \a in {0,...,\Mm} {
								\tikzset{rotate=\a*\PH}
								\draw[]
										(90-0.5*\PH:\L) -- (90+0.5*\PH:\L)
										(90:\LL) -- (90+\PH:\LL)
										;
						}

						\foreach \a in {0,...,\nn} {
								\tikzset{rotate=\a*\ph}
								\draw[]
										(90-0.5*\Ph:\l) -- (90-0.5*\PH:\L)
										(90+0.5*\Ph:\l) -- (90+0.5*\PH:\L)
										(90-0.5*\PH:\L) -- (90:\LL)
										(90+0.5*\PH:\L) -- (90:\LL)
										(90-1.5*\PH:\L) -- (90-\PH:\LL)
										(90+1.5*\PH:\L) -- (90+\PH:\LL)
										(90-\PH:\LL) -- (90-0.5*\Ph:\R)
										(90+\PH:\LL) -- (90+0.5*\Ph:\R)
										;
						}

						\foreach \a in {0,...,\Mm} {
								\tikzset{rotate=\a*\PH}
								\draw[fill=white]
										(90+0.5*\PH:\L) circle (\rr) 
										(90:\LL) circle (\rr) 
										;
						}
				
				\draw[] (0,0) circle (\RR);

						\foreach \a in {0,...,\nn} {
								\tikzset{rotate=\a*\ph}
								\draw[]
										%
										;
						}

						\foreach \a in {0,...,\nn} {
								\tikzset{rotate=\a*\ph+0.5*\ph}
								\draw[]
										(90:\RR) -- (90-0.5*\Ph:\R)
										(90:\RR) -- (90+0.5*\Ph:\R)
										;
						}

						\foreach \a in {0,...,\nn} {
								\tikzset{rotate=\a*\ph}
								\draw[fill=white]
										(270:\RR) circle (\rr)
										;
						}

						\foreach \a in {0,...,\Nn} {
								\tikzset{rotate=\a*\Ph}
								\draw[fill=white]
										(90-0.5*\Ph:\l) circle (\rr)
										;
						}

						\foreach \a in {0,...,\Nn} {
								\tikzset{rotate=\a*\Ph}
								\draw[fill=white]
										(90-0.5*\Ph:\R) circle (\rr)
										;
						}

						\node at (0.8*\RR,-\RR) {\small $J_{76}$};

\end{scope}

\begin{scope}[xshift=\S cm]

						\tikzmath{
								\XS = \S;
								\r=0.25;
								\rr=0.2*\r;
								\n=5;
								\nn=\n-1;
								\N=2*\n;
								\Nn=\N-1;
								\M=\N+\n;
								\Mm=\M-1;
								\ph=360/\n;
								\x=\r*sin(0.5*\ph);
								\l=1.055*sqrt( \r^2 + (2*\x)^2 - 2*\r*(2*\x)*cos(90+0.5*\ph) );
								\Ph=360/\N;
								\PH=360/\M;
								\L=3.1*\r;
								\LL=4.1*\r;
								\R=5.2*\r;
								\RR=6*\r;
						}


						(90-0.5*\Ph:\R) -- (90+0.5*\Ph:\R) -- (90+1.5*\Ph:\R) -- (90+2.5*\Ph:\R) 
						-- (90+3.5*\Ph:\R) -- (90+4.5*\Ph:\R) -- (90+5.5*\Ph:\R) -- (90+6.5*\Ph:\R) 
						-- (90+7.5*\Ph:\R) -- (90+8.5*\Ph:\R) -- (90+9.5*\Ph:\R)
						;

				\fill[gray!40]
						(90-0.5*\Ph:\l) -- (90+0.5*\Ph:\l) -- (90+1.5*\Ph:\l) -- (90+2.5*\Ph:\l) 
						-- (90+3.5*\Ph:\l) -- (90+4.5*\Ph:\l) -- (90+5.5*\Ph:\l) -- (90+6.5*\Ph:\l) 
						-- (90+7.5*\Ph:\l) -- (90+8.5*\Ph:\l) -- (90+9.5*\Ph:\l)
						;

				\foreach \a in {0,...,\Nn} {
						\tikzset{rotate=\a*\Ph}
						\draw[]
								(90-0.5*\Ph:\l) -- (90+0.5*\Ph:\l)
								(90-0.5*\Ph:\R) -- (90+0.5*\Ph:\R)
								;
				}

				\foreach \a in {0,...,\Mm} {
						\tikzset{rotate=\a*\PH}
						(90-0.5*\PH:\L) -- (90+0.5*\PH:\L)
						;
		}

		\foreach \a in {0,...,\Mm} {
				\tikzset{rotate=\a*\PH}
				\draw[]
						(90-0.5*\PH:\L) -- (90+0.5*\PH:\L)
						(90:\LL) -- (90+\PH:\LL)
						;
		}

		\foreach \a in {0,...,\nn} {
				\tikzset{rotate=\a*\ph}
				\draw[]
						(90-0.5*\Ph:\l) -- (90-0.5*\PH:\L)
						(90+0.5*\Ph:\l) -- (90+0.5*\PH:\L)
						(90-0.5*\PH:\L) -- (90:\LL)
						(90+0.5*\PH:\L) -- (90:\LL)
						(90-1.5*\PH:\L) -- (90-\PH:\LL)
						(90+1.5*\PH:\L) -- (90+\PH:\LL)
						(90-\PH:\LL) -- (90-0.5*\Ph:\R)
						(90+\PH:\LL) -- (90+0.5*\Ph:\R)
						;
		}

		\foreach \a in {0,...,\Mm} {
				\tikzset{rotate=\a*\PH}
				\draw[fill=white]
						(90+0.5*\PH:\L) circle (\rr) 
						(90:\LL) circle (\rr) 
						;
		}

				\draw[] (0,0) circle (\RR);

				\foreach \a in {0,...,\nn} {
						\tikzset{rotate=\a*\ph}
						\draw[]
								(90:\RR) -- (90-0.5*\Ph:\R)
								(90:\RR) -- (90+0.5*\Ph:\R)
								;
				}

				\foreach \a in {0,...,\nn} {
						\tikzset{rotate=\a*\ph+0.5*\ph}
						\draw[]
								%
								;
						\fill[]
								(90-0.5*\Ph:\R) circle (\rr)
								(90+0.5*\Ph:\R) circle (\rr)
								;
				}

				\foreach \a in {0,...,\nn} {
						\tikzset{rotate=\a*\ph}
						\draw[fill=white]
								(90:\RR) circle (\rr)
								;
				}

				\node at (0.8*\RR,-\RR) {\small $J_{77}$};

\end{scope}

\begin{scope}[xshift=2.08*\S cm]

				\tikzmath{
						\XS=\S;
				\r=0.25;
		\rr=0.05;
\n=5;
\nn=\n-1;
\N=2*\n;
\Nn=\N-1;
\M=\N+\n;
\Mm=\M-1;
\ph=360/\n;
\x=\r*sin(0.5*\ph);
\l=1.055*sqrt( \r^2 + (2*\x)^2 - 2*\r*(2*\x)*cos(90+0.5*\ph) );
\Ph=360/\N;
\PH=360/\M;
\L=3.1*\r;
\LL=4.1*\r;
\R=5.2*\r;
\RR=6*\r;
\H=16;
\Hh=\H-1;
\Th=360/\H;
\h=4.25*\r;
\O=11;
\Oo=\O-1;
\Ps=360/\O;
\o=5.25*\r;
}

\fill[gray!40]
		(90-0.5*\Ph:\l) -- (90+0.5*\Ph:\l) -- (90+1.5*\Ph:\l) -- (90+2.5*\Ph:\l) 
		-- (90+3.5*\Ph:\l) -- (90+4.5*\Ph:\l) -- (90+5.5*\Ph:\l) -- (90+6.5*\Ph:\l) 
		-- (90+7.5*\Ph:\l) -- (90+8.5*\Ph:\l) -- (90+9.5*\Ph:\l)
		;

\foreach \a in {0,...,\nn} {
		\tikzset{rotate=\a*\ph 
		}
		(90:\r) -- (90+\ph:\r)
		(90:\r) -- (90-0.5*\Ph:\l)
		(90:\r) -- (90+0.5*\Ph:\l)
		;
}

\foreach \a in {0,...,\Nn} {
		\tikzset{rotate=\a*\Ph}
		\draw[]
				(90-0.5*\Ph:\l) -- (90+0.5*\Ph:\l)
				%
				;
}

\foreach \a in {0,...,\Mm} {
		\tikzset{rotate=\a*\PH}
		(90-0.5*\PH:\L) -- (90+0.5*\PH:\L)
		;
}

\foreach \a in {0,...,\Mm} {
		\tikzset{rotate=\a*\PH}
		\draw[]
				(90-0.5*\PH:\L) -- (90+0.5*\PH:\L)
				%
				;
}

\foreach \a in {0,...,\Hh} {
		\tikzset{rotate=\a*\Th}
		\draw[]
				(90-0.5*\Th:\h) -- (90+0.5*\Th:\h) 
				;
}

\foreach \a in {0,...,\Oo} {
		\tikzset{rotate=\a*\Ps}
		\draw[]
				(90:\o) -- (90+\Ps:\o)
				;
}

\foreach \a in {0,...,\nn} {
		\tikzset{rotate=\a*\ph}
		\draw[]
				(90-0.5*\Ph:\l) -- (90-0.5*\PH:\L)
				(90+0.5*\Ph:\l) -- (90+0.5*\PH:\L)
				%
				%
				;
}

\draw[]
		(90-0.5*\PH:\L) -- (90-0.5*\Th:\h)
		(90+0.5*\PH:\L) -- (90+0.5*\Th:\h)
		(90-1.5*\PH:\L) -- (90-0.5*\Th:\h)
		(90+1.5*\PH:\L) -- (90+0.5*\Th:\h)
		(90-1.5*\PH:\L) -- (90-2.5*\Th:\h)
		(90+1.5*\PH:\L) -- (90+2.5*\Th:\h)
		(90-2.5*\PH:\L) -- (90-3.5*\Th:\h)
		(90+2.5*\PH:\L) -- (90+3.5*\Th:\h)
		(90-3.5*\PH:\L) -- (90-3.5*\Th:\h)
		(90+3.5*\PH:\L) -- (90+3.5*\Th:\h)
		(90-4.5*\PH:\L) -- (90-4.5*\Th:\h)
		(90+4.5*\PH:\L) -- (90+4.5*\Th:\h)
		(90-4.5*\PH:\L) -- (90-5.5*\Th:\h)
		(90+4.5*\PH:\L) -- (90+5.5*\Th:\h)
		(90-5.5*\PH:\L) -- (90-6.5*\Th:\h)
		(90+5.5*\PH:\L) -- (90+6.5*\Th:\h)
		(90-6.5*\PH:\L) -- (90-6.5*\Th:\h)
		(90+6.5*\PH:\L) -- (90+6.5*\Th:\h)	
		(90-7.5*\PH:\L) -- (90-7.5*\Th:\h)
		(90+7.5*\PH:\L) -- (90+7.5*\Th:\h)
		(90-1.5*\Th:\h) -- (90:\o) 
		(90+1.5*\Th:\h) -- (90:\o) 
		(90-1.5*\Th:\h) -- (90-2*\Ps:\o) 
		(90+1.5*\Th:\h) -- (90+2*\Ps:\o) 	
		(90-2.5*\Th:\h) -- (90-2*\Ps:\o) 
		(90+2.5*\Th:\h) -- (90+2*\Ps:\o) 
		(90-4.5*\Th:\h) -- (90-3*\Ps:\o) 
		(90+4.5*\Th:\h) -- (90+3*\Ps:\o) 
		(90-5.5*\Th:\h) -- (90-4*\Ps:\o) 
		(90+5.5*\Th:\h) -- (90+4*\Ps:\o) 
		(90-7.5*\Th:\h) -- (90-5*\Ps:\o) 
		(90+7.5*\Th:\h) -- (90+5*\Ps:\o)
		(90-\Ps:\o) -- (90-2*\Ps:6*\r)
		(90-3*\Ps:\o) -- (90-2*\Ps:6*\r)
		(90+\Ps:\o) -- (90+2*\Ps:6*\r)
		(90+3*\Ps:\o) -- (90+2*\Ps:6*\r)
		(90-\Ps:\o) to[out=120,in=60] (90+\Ps:\o)
		(90-4*\Ps:\o) to[out=30,in=-60] (90-2*\Ps:6*\r)
		(90+4*\Ps:\o) to[out=180-30,in=180+60] (90+2*\Ps:6*\r)
		(90:7*\r) to[out=-10,in=120] (90-2*\Ps:6*\r)
		(90:7*\r) to[out=180+10,in=60] (90+2*\Ps:6*\r)
		(90-5*\Ps:\o) to[out=0,in=0,distance=1.6*\R cm] (90:7*\r)
		(90+5*\Ps:\o) to[out=180,in=180,distance=1.6*\R cm] (90:7*\r)
		; 

\foreach \aa in {-1,1} {
		\tikzset{xscale=\aa}
		\foreach \a in {0,1} {
				\tikzset{rotate=\a*\PH}
				\fill[]
						(90+0.5*\PH:\L) circle (\rr)
						;
		}
		\foreach \a in {2} {
				\tikzset{rotate=\a*\Th}
				\fill[]
						(90+0.5*\Th:\h) circle (\rr)
						;
		}
		\foreach \a in {0,1} {
				\tikzset{rotate=\a*\Ps}
				\fill[]
						(90+\Ps:\o) circle (\rr)
						;
		}
}

\foreach \aa in {-1,1} {
		\tikzset{xscale=\aa}
		\foreach \a in {2,...,6} {
				\tikzset{rotate=\a*\PH}
				\draw[fill=white]
						(90+0.5*\PH:\L) circle (\rr) 
						;
		}
		\foreach \a in {0,1,3,4,5,6,7} {
				\tikzset{rotate=\a*\Th}
				\draw[fill=white]
						(90+0.5*\Th:\h) circle (\rr)
						;
		}
		\foreach \a in {3,4,5} {
				\tikzset{rotate=\a*\Ps}
				\draw[fill=white]
						(90:\o) circle (\rr) 
						;
		}
		\draw[fill=white]
				(90+2*\Ps:6*\r) circle (\rr)
				;
}
\draw[fill=white]
		(90:\o) circle (\rr) 
		(90:7*\r) circle (\rr) 
		(270:\L) circle (\rr) 
		;
\node at (0.8*\RR,-\RR) {\small $J_{78}$};

\end{scope}

\begin{scope}[xshift=3.15*\S cm]

		\tikzmath{
				\XS=\S;
		\r=0.25;
\rr=0.2*\r;
\n=5;
\nn=\n-1;
\N=2*\n;
\Nn=\N-1;
\M=\N+\n;
\Mm=\M-1;
\ph=360/\n;
\x=\r*sin(0.5*\ph);
\l=1.055*sqrt( \r^2 + (2*\x)^2 - 2*\r*(2*\x)*cos(90+0.5*\ph) );
\Ph=360/\N;
\PH=360/\M;
\L=3.1*\r;
\LL=4.1*\r;
\R=5.2*\r;
\RR=6*\r;
\H=17;
\Hh=\H-1;
\Th=360/\H;
\h=4.25*\r;
\O=12;
\Oo=\O-1;
\Ps=360/\O;
\o=5.25*\r;
}

\fill[gray!40]
		(90-0.5*\Ph:\l) -- (90+0.5*\Ph:\l) -- (90+1.5*\Ph:\l) -- (90+2.5*\Ph:\l) 
		-- (90+3.5*\Ph:\l) -- (90+4.5*\Ph:\l) -- (90+5.5*\Ph:\l) -- (90+6.5*\Ph:\l) 
		-- (90+7.5*\Ph:\l) -- (90+8.5*\Ph:\l) -- (90+9.5*\Ph:\l)
		;

\foreach \a in {0,...,\nn} {
		\tikzset{rotate=\a*\ph 
		}
		(90:\r) -- (90+\ph:\r)
		(90:\r) -- (90-0.5*\Ph:\l)
		(90:\r) -- (90+0.5*\Ph:\l)
		;
}

\foreach \a in {0,...,\Nn} {
		\tikzset{rotate=\a*\Ph}
		\draw[]
				(90-0.5*\Ph:\l) -- (90+0.5*\Ph:\l)
				%
				;
}

\foreach \a in {0,...,\Mm} {
		\tikzset{rotate=\a*\PH}
		(90-0.5*\PH:\L) -- (90+0.5*\PH:\L)
		;
}

\foreach \a in {0,...,\Mm} {
		\tikzset{rotate=\a*\PH}
		\draw[]
				(90-0.5*\PH:\L) -- (90+0.5*\PH:\L)
				%
				;
}

\foreach \a in {0,...,\Hh} {
		\tikzset{rotate=\a*\Th}
		\draw[]
				(90:\h) -- (90+\Th:\h) 
				;
}

\foreach \a in {0,...,\Oo} {
		\tikzset{rotate=\a*\Ps}
		\draw[]
				(90-0.5*\Ps:\o) -- (90+0.5*\Ps:\o)
				;
}

\foreach \a in {0,...,\nn} {
		\tikzset{rotate=\a*\ph}
		\draw[]
				(90-0.5*\Ph:\l) -- (90-0.5*\PH:\L)
				(90+0.5*\Ph:\l) -- (90+0.5*\PH:\L)
				;
}

\draw[]
		(90-0.5*\PH:\L) -- (90:\h)
		(90+0.5*\PH:\L) -- (90:\h)
		(90-1.5*\PH:\L) -- (90-\Th:\h)
		(90+1.5*\PH:\L) -- (90+\Th:\h)
		(90-1.5*\PH:\L) -- (90-3*\Th:\h)
		(90+1.5*\PH:\L) -- (90+3*\Th:\h)
		(90-2.5*\PH:\L) -- (90-3*\Th:\h)
		(90+2.5*\PH:\L) -- (90+3*\Th:\h)
		(90-3.5*\PH:\L) -- (90-4*\Th:\h)
		(90+3.5*\PH:\L) -- (90+4*\Th:\h)
		(90-4.5*\PH:\L) -- (90-4*\Th:\h)
		(90+4.5*\PH:\L) -- (90+4*\Th:\h)
		(90-4.5*\PH:\L) -- (90-6*\Th:\h)
		(90+4.5*\PH:\L) -- (90+6*\Th:\h)
		(90-5.5*\PH:\L) -- (90-7*\Th:\h)
		(90+5.5*\PH:\L) -- (90+7*\Th:\h)
		(90-6.5*\PH:\L) -- (90-7*\Th:\h)
		(90+6.5*\PH:\L) -- (90+7*\Th:\h)
		(90-7.5*\PH:\L) -- (90-8*\Th:\h)
		(90+7.5*\PH:\L) -- (90+8*\Th:\h)
		(90-\Th:\h) -- (90-0.5*\Ps:\o) 
		(90+\Th:\h) -- (90+0.5*\Ps:\o) 
		(90-2*\Th:\h) -- (90-0.5*\Ps:\o) 
		(90+2*\Th:\h) -- (90+0.5*\Ps:\o) 
		(90-2*\Th:\h) -- (90-2.5*\Ps:\o) 
		(90+2*\Th:\h) -- (90+2.5*\Ps:\o) 
		(90-5*\Th:\h) -- (90-2.5*\Ps:\o) 
		(90+5*\Th:\h) -- (90+2.5*\Ps:\o) 
		(90-5*\Th:\h) -- (90-4.5*\Ps:\o) 
		(90+5*\Th:\h) -- (90+4.5*\Ps:\o) 
		(90-6*\Th:\h) -- (90-4.5*\Ps:\o) 
		(90+6*\Th:\h) -- (90+4.5*\Ps:\o) 	
		(90-8*\Th:\h) -- (90-5.5*\Ps:\o) 
		(90+8*\Th:\h) -- (90+5.5*\Ps:\o) 
		(90-5.5*\Ps:\o) -- (270:6*\r)
		(90+5.5*\Ps:\o) -- (270:6*\r)
		(90-1.5*\Ps:\o) to[out=120,in=60, distance=0.75*\R cm] (90+1.5*\Ps:\o)
		(90-1.5*\Ps:\o) to[out=-30,in=60, distance=0.5*\R cm] (90-3.5*\Ps:\o)
		(90+1.5*\Ps:\o) to[out=180+30,in=180-60, distance=0.5*\R cm] (90+3.5*\Ps:\o)
		(90-3.5*\Ps:\o) to[out=270,in=0] (270:6*\r)
		(90+3.5*\Ps:\o) to[out=270,in=180] (270:6*\r)
		;

\foreach \aa in {-1,1} {
		\tikzset{xscale=\aa}
		\foreach \a in {2,3,4,5} {
				\tikzset{rotate=\a*\PH}
				\fill[]
						(90-0.5*\PH:\L) circle (\rr)
						;
		}
		\foreach \a in {1,6} {
				\tikzset{rotate=\a*\Th}
				\fill[]
						(90:\h) circle (\rr)
						;
		}
		\foreach \a in {0,1,3,4} {
				\tikzset{rotate=\a*\Ps}
				\fill[]
						(90+0.5*\Ps:\o) circle (\rr)
						;
		}
		\foreach \a in {2,...,8} {
				\tikzset{rotate=\a*\Th}
				\draw[fill=white]
						(90:\h) circle (\rr)
						;
		}
		\foreach \a in {2,5} {
				\tikzset{rotate=\a*\Ps}
				\draw[fill=white]
						(90+0.5*\Ps:\o) circle (\rr)
						;
		}
		\draw[fill=white]
				(90+0.5*\PH:\L) circle (\rr)
				(90+5.5*\PH:\L) circle (\rr)
				(90+6.5*\PH:\L) circle (\rr)
				;
}
\draw[fill=white]
		(90+7.5*\PH:\L) circle (\rr)
		(90:\h) circle (\rr)
		(270:6*\r) circle (\rr)
		;

\node at (0.8*\RR,-\RR) {\small $J_{79}$};
\end{scope} 

\end{scope} 


\begin{scope}[xshift=0.15*\S cm, yshift=-\S cm] 

\begin{scope}[]

\tikzmath{
								\XS = \S;
								\r=0.25;
								\rr=0.2*\r;
								\n=5;
								\nn=\n-1;
								\N=2*\n;
								\Nn=\N-1;
								\M=\N+\n;
								\Mm=\M-1;
								\ph=360/\n;
								\x=\r*sin(0.5*\ph);
								\l=1.055*sqrt( \r^2 + (2*\x)^2 - 2*\r*(2*\x)*cos(90+0.5*\ph) );
								\Ph=360/\N;
								\PH=360/\M;
								\L=3.1*\r;
								\LL=4.1*\r;
								\R=5.2*\r;
								\RR=6*\r;
						}

						{

								\fill[gray!40] (0,0) circle (1.02*\RR);

								\fill[white]
										(90-0.5*\Ph:\R) -- (90+0.5*\Ph:\R) -- (90+1.5*\Ph:\R) -- (90+2.5*\Ph:\R) 
										-- (90+3.5*\Ph:\R) -- (90+4.5*\Ph:\R) -- (90+5.5*\Ph:\R) -- (90+6.5*\Ph:\R) 
										-- (90+7.5*\Ph:\R) -- (90+8.5*\Ph:\R) -- (90+9.5*\Ph:\R)
										;
						}

								\fill[gray!40]
										(90-0.5*\Ph:\l) -- (90+0.5*\Ph:\l) -- (90+1.5*\Ph:\l) -- (90+2.5*\Ph:\l) 
										-- (90+3.5*\Ph:\l) -- (90+4.5*\Ph:\l) -- (90+5.5*\Ph:\l) -- (90+6.5*\Ph:\l) 
										-- (90+7.5*\Ph:\l) -- (90+8.5*\Ph:\l) -- (90+9.5*\Ph:\l)
										;

								\foreach \a in {0,...,\Nn} {
										\tikzset{rotate=\a*\Ph}
										\draw[]
												(90-0.5*\Ph:\l) -- (90+0.5*\Ph:\l)
												(90-0.5*\Ph:\R) -- (90+0.5*\Ph:\R)
												;
								}

								\foreach \a in {0,...,\Mm} {
										\tikzset{rotate=\a*\PH}
										(90-0.5*\PH:\L) -- (90+0.5*\PH:\L)
										;
						}

						\foreach \a in {0,...,\Mm} {
								\tikzset{rotate=\a*\PH}
								\draw[]
										(90-0.5*\PH:\L) -- (90+0.5*\PH:\L)
										(90:\LL) -- (90+\PH:\LL)
										;
						}

						\foreach \a in {0,...,\nn} {
								\tikzset{rotate=\a*\ph}
								\draw[]
										(90-0.5*\Ph:\l) -- (90-0.5*\PH:\L)
										(90+0.5*\Ph:\l) -- (90+0.5*\PH:\L)
										(90-0.5*\PH:\L) -- (90:\LL)
										(90+0.5*\PH:\L) -- (90:\LL)
										(90-1.5*\PH:\L) -- (90-\PH:\LL)
										(90+1.5*\PH:\L) -- (90+\PH:\LL)
										(90-\PH:\LL) -- (90-0.5*\Ph:\R)
										(90+\PH:\LL) -- (90+0.5*\Ph:\R)
										;
						}

						\foreach \a in {0,...,\Mm} {
								\tikzset{rotate=\a*\PH}
								\draw[fill=white]
										(90+0.5*\PH:\L) circle (\rr) 
										(90:\LL) circle (\rr) 
										;
						}

\node at (0.8*\RR,-\RR) {\small $J_{80}$};

\end{scope}

\begin{scope}[xshift=1*\S cm] 

						\tikzmath{
								\r=0.25;
						\rr=0.2*\r;
				\n=5;
		\nn=\n-1;
\N=2*\n;
\Nn=\N-1;
\M=\N+\n;
\Mm=\M-1;
\ph=360/\n;
\x=\r*sin(0.5*\ph);
\l=1.055*sqrt( \r^2 + (2*\x)^2 - 2*\r*(2*\x)*cos(90+0.5*\ph) );
\Ph=360/\N;
\PH=360/\M;
\L=3.1*\r;
\LL=4.1*\r;
\H=16;
\Hh=\H-1;
\Th=360/\H;
\h=4.25*\r;
\O=9;
\Oo=\O-1;
\Ps=360/\O;
\o=5.5*\r;
\R=5.2*\r;
\RR=6*\r;
}

\fill[gray!40]
		(90-0.5*\Ph:\l) -- (90+0.5*\Ph:\l) -- (90+1.5*\Ph:\l) -- (90+2.5*\Ph:\l) 
		-- (90+3.5*\Ph:\l) -- (90+4.5*\Ph:\l) -- (90+5.5*\Ph:\l) -- (90+6.5*\Ph:\l) 
		-- (90+7.5*\Ph:\l) -- (90+8.5*\Ph:\l) -- (90+9.5*\Ph:\l)
		;

\fill[gray!40]
		(90+1.5*\PH:\L) -- (90+0.5*\PH:\L) -- (90-0.5*\PH:\L) -- (90-1.5*\PH:\L) -- 
		(90-2.5*\Th:\h) -- (90-1.5*\Th:\h) -- (90-0.5*\Th:\h)  -- (90+0.5*\Th:\h) -- (90+1.5*\Th:\h) -- (90+2.5*\Th:\h) 
		-- cycle
		;

\foreach \a in {0,...,\Nn} {
		\tikzset{rotate=\a*\Ph}
		\draw[]
				(90-0.5*\Ph:\l) -- (90+0.5*\Ph:\l)
				%
				;
}

\foreach \a in {0,...,\Mm} {
		\tikzset{rotate=\a*\PH}
		\draw[]
				(90-0.5*\PH:\L) -- (90+0.5*\PH:\L)
				;
}
\foreach \a in {0,...,\nn} {
		\tikzset{rotate=\a*\ph+0.5*\ph}
		\draw[]
				(90-0.5*\Ph:\l) -- (90-\PH:\L)
				(90+0.5*\Ph:\l) -- (90+\PH:\L)
				;
}

\foreach \a in {0,...,\Hh} {
		\tikzset{rotate=\a*\Th}
		\draw[]
				(90-0.5*\Th:\h) -- (90+0.5*\Th:\h) 
				;
}

\foreach \a in {0,...,\Oo} {
		\tikzset{rotate=\a*\Ps}
		\draw[]
				(90:\o) -- (90+\Ps:\o)
				;
}

\foreach \aa in {-1,1} {
		\tikzset{xscale=\aa}
		\draw[]
				(90-1.5*\PH:\L) -- (90-2.5*\Th:\h)
				(90-2.5*\PH:\L) -- (90-3.5*\Th:\h)
				(90-3.5*\PH:\L) -- (90-3.5*\Th:\h)
				(90-4.5*\PH:\L) -- (90-4.5*\Th:\h)
				(90-4.5*\PH:\L) -- (90-5.5*\Th:\h)
				(90-5.5*\PH:\L) -- (90-6.5*\Th:\h)
				(90-6.5*\PH:\L) -- (90-6.5*\Th:\h)
				(90-7.5*\PH:\L) -- (90-7.5*\Th:\h)
				(90-0.5*\Th:\h) -- (90-\Ps:\o)
				(90-\Ps:\o) to[out=-30,in=45] (90-3*\Ps:\o)
				(90-1.5*\Th:\h) to[out=-30,in=135] (90-2*\Ps:\o)
				(90-4.5*\Th:\h) -- (90-2*\Ps:\o)
				(90-5.5*\Th:\h) -- (90-3*\Ps:\o) 
				(90-7.5*\Th:\h) -- (90-4*\Ps:\o) 
				(90-4*\Ps:\o) to[out=15,in=0,distance=7.25*\r cm] (90:\o)
				;
}


\foreach \aa in {-1,1} {
		\tikzset{xscale=\aa}
		;
\foreach \a in {0,1,2,3,4} {
		{
				\tikzset{rotate=\a*\PH}
				\draw[fill=white]
						(90+2.5*\PH:\L) circle (\rr)
						;
		}
		{
				\tikzset{rotate=\a*\Th}
				\draw[fill=white]
						(90+3.5*\Th:\h) circle (\rr)
						;
		}
		{
				\tikzset{rotate=\a*\Ps}
				\draw[fill=white]
						(90+\Ps:\o) circle (\rr)
						;
		}
}
\draw[fill=white]
		(90+7.5*\PH:\L) circle (\rr)
		(90:\o) circle (\rr)
		;
}

\node at (0.8*\RR,-\RR) {\small $J_{81}$};
\end{scope} 

\begin{scope}[xshift=2*\S cm] 

		\tikzmath{
				\r=0.25;
		\rr=0.2*\r;
\n=5;
\nn=\n-1;
\N=2*\n;
\Nn=\N-1;
\M=\N+\n;
\Mm=\M-1;
\ph=360/\n;
\x=\r*sin(0.5*\ph);
\l=1.055*sqrt( \r^2 + (2*\x)^2 - 2*\r*(2*\x)*cos(90+0.5*\ph) );
\Ph=360/\N;
\PH=360/\M;
\L=3.1*\r;
\LL=4.1*\r;
\H=17;
\Hh=\H-1;
\Th=360/\H;
\h=4.25*\r;
\R=5.2*\r;
\RR=6*\r;
}

\fill[gray!40]
		(90-0.5*\Ph:\l) -- (90+0.5*\Ph:\l) -- (90+1.5*\Ph:\l) -- (90+2.5*\Ph:\l) 
		-- (90+3.5*\Ph:\l) -- (90+4.5*\Ph:\l) -- (90+5.5*\Ph:\l) -- (90+6.5*\Ph:\l) 
		-- (90+7.5*\Ph:\l) -- (90+8.5*\Ph:\l) -- (90+9.5*\Ph:\l)
		;

\fill[gray!40]
		(90+1.5*\PH:\L) -- (90+0.5*\PH:\L) -- (90-0.5*\PH:\L) -- (90-1.5*\PH:\L) -- 
		(90-2.5*\Th:\h) -- (90-1.5*\Th:\h) -- (90-0.5*\Th:\h)  -- (90+0.5*\Th:\h) -- (90+1.5*\Th:\h) -- (90+2.5*\Th:\h) 
		-- cycle
		;

\foreach \a in {0,...,\Nn} {
		\tikzset{rotate=\a*\Ph}
		\draw[]
				(90-0.5*\Ph:\l) -- (90+0.5*\Ph:\l)
				%
				;
}

\foreach \a in {0,...,\Mm} {
		\tikzset{rotate=\a*\PH}
		\draw[]
				(90-0.5*\PH:\L) -- (90+0.5*\PH:\L)
				;
}
\foreach \a in {0,...,\nn} {
		\tikzset{rotate=\a*\ph+0.5*\ph}
		\draw[]
				(90-0.5*\Ph:\l) -- (90-\PH:\L)
				(90+0.5*\Ph:\l) -- (90+\PH:\L)
				;
}

\foreach \a in {0,...,\Hh} {
		\tikzset{rotate=\a*\Th}
		\draw[]
				(90-0.5*\Th:\h) -- (90+0.5*\Th:\h) 
				;
}

\foreach \aa in {-1,1} {
		\tikzset{xscale=\aa}
		\draw[]
				(90-1.5*\PH:\L) -- (90-2.5*\Th:\h)
				(90-2.5*\PH:\L) -- (90-3.5*\Th:\h)
				(90-3.5*\PH:\L) -- (90-3.5*\Th:\h)
				(90-4.5*\PH:\L) -- (90-4.5*\Th:\h)
				(90-7.5*\PH:\L) -- (90-7.5*\Th:\h)
				(90-0.5*\Th:\h) -- (90-1.5*\Th:\R)
				(90-1.5*\Th:\R) -- (90-2.5*\PH:\R) 
				(90-1.5*\Th:\h) -- (90-2.5*\PH:\R) 
				(90-4.5*\Th:\h) -- (90-2.5*\PH:\R) 
				(90-1.5*\Th:\R) -- (90:\RR)
				;
}

\draw[]
		(90-4.5*\PH:\L) -- (90-6.5*\Th:\h)
		(90-5.5*\PH:\L) -- (90-6.5*\Th:\h)
		(90-6.5*\PH:\L) -- (90-7.5*\Th:\h)
		(90+4.5*\PH:\L) -- (90+5.5*\Th:\h)
		(90+5.5*\PH:\L) -- (90+6.5*\Th:\h)
		(90+6.5*\PH:\L) -- (90+6.5*\Th:\h)
		(90-2.5*\PH:\R)  to[out=-70, in=45] (90-5.5*\Th:\h) 
		(90+2.5*\PH:\R) -- (90+4*\Th:\R)
		(90+5.5*\Th:\h) -- (90+4*\Th:\R)
		(90-5.5*\Th:\h) -- (90-5.5*\Th:\R)
		(270:\h) to[out=-25,in=230] (90-5.5*\Th:\R)
		(90-1.5*\Th:\R) to[out=-20,in=60,distance=3*\r cm] (90-5.5*\Th:\R)
		(90-5.5*\Th:\R) to[out=45,in=0,distance=4.5*\r cm] (90:\RR)
		(270:\h) to[out=200,in=310] (-\h,-\L)
		(90+7.5*\Th:\h) -- (-\h,-\L)
		(-\h,-\L) -- (90+4*\Th:\R)
		(90+1.5*\Th:\R) to[out=180+20,in=100,distance=2*\r cm] (90+4*\Th:\R) 
		(-\h,-\L) to[out=135,in=180,distance=5*\r cm] (90:\RR)
		;

\foreach \aa in {-1,1} {
		\tikzset{xscale=\aa}
		\draw[fill=white]
				(90+2.5*\PH:\L) circle (\rr)
				(90+3.5*\PH:\L) circle (\rr)
				(90+3.5*\Th:\h) circle (\rr)
				(90+5.5*\Th:\h) circle (\rr)
				(90+6.5*\Th:\h) circle (\rr)
				;
}

\draw[fill=white]
		(90+4.5*\PH:\L) circle (\rr)
		(90+5.5*\PH:\L) circle (\rr)
		(90+6.5*\PH:\L) circle (\rr)
		(90-5.5*\Th:\R) circle (\rr)
		(90+1.5*\Th:\R) circle (\rr)
		(90+2.5*\PH:\R) circle (\rr)
		(90+3.5*\PH:\R) circle (\rr)
		(90-7.5*\Th:\h) circle (\rr)
		(90+4.5*\Th:\h) circle (\rr)
		(90+8.5*\Th:\h) circle (\rr)
		(270:\h) circle (\rr)
		;

\fill[]
		(90-4.5*\PH:\L) circle (\rr)
		(90-5.5*\PH:\L) circle (\rr)
		(90-6.5*\PH:\L) circle (\rr)
		(90-7.5*\PH:\L) circle (\rr)
		(90-1.5*\Th:\R) circle (\rr)
		(90-2.5*\PH:\R) circle (\rr)
		(90-4.5*\Th:\h) circle (\rr)
		(90+7.5*\Th:\h) circle (\rr)
		(-\h,-\L) circle (\rr)
		(90:\RR) circle (\rr)
		;

\node at (0.8*\RR,-\RR) {\small $J_{82}$};

\end{scope} 

\begin{scope}[xshift=2.9*\S cm] 

		\tikzmath{
				\r=0.25;
		\rr=0.2*\r;
\n=5;
\nn=\n-1;
\N=2*\n;
\Nn=\N-1;
\M=\N+\n;
\Mm=\M-1;
\ph=360/\n;
\x=\r*sin(0.5*\ph);
\l=1.055*sqrt( \r^2 + (2*\x)^2 - 2*\r*(2*\x)*cos(90+0.5*\ph) );
\Ph=360/\N;
\PH=360/\M;
\L=3.1*\r;
\LL=4.1*\r;
\H=17;
\Hh=\H-1;
\Th=360/\H;
\h=4.25*\r;
\O=11;
\Oo=\O-1;
\Ps=360/\O;
\o=5.25*\r;
\R=5.2*\r;
\RR=6*\r;
}

\fill[gray!40]
		(90-0.5*\Ph:\l) -- (90+0.5*\Ph:\l) -- (90+1.5*\Ph:\l) -- (90+2.5*\Ph:\l) 
		-- (90+3.5*\Ph:\l) -- (90+4.5*\Ph:\l) -- (90+5.5*\Ph:\l) -- (90+6.5*\Ph:\l) 
		-- (90+7.5*\Ph:\l) -- (90+8.5*\Ph:\l) -- (90+9.5*\Ph:\l)
		;

\foreach \aa in {-1,1} {
		\tikzset{xscale=\aa}
		\fill[gray!40]
				(90+4.5*\PH:\L) -- (90+3.5*\PH:\L) -- (90+2.5*\PH:\L) -- (90+1.5*\PH:\L) -- 
				(90+\Th:\h) -- (90+2*\Th:\h) -- (90+3*\Th:\h) -- (90+4*\Th:\h) -- (90+5*\Th:\h) -- (90+6*\Th:\h)
				-- cycle
				;
}

\foreach \a in {0,...,\Nn} {
		\tikzset{rotate=\a*\Ph}
		\draw[]
				(90-0.5*\Ph:\l) -- (90+0.5*\Ph:\l)
				%
				;
}
\foreach \a in {0,...,\Mm} {
		\tikzset{rotate=\a*\PH}
		\draw[]
				(90-0.5*\PH:\L) -- (90+0.5*\PH:\L)
				;
}
\foreach \a in {0,...,\nn} {
		\tikzset{rotate=\a*\ph+0.5*\ph}
		\draw[]
				(90-0.5*\Ph:\l) -- (90-\PH:\L)
				(90+0.5*\Ph:\l) -- (90+\PH:\L)
				;
}

\foreach \a in {0,...,\Hh} {
		\tikzset{rotate=\a*\Th}
		\draw[]
				(90:\h) -- (90+\Th:\h) 
				;
}

\foreach \aa in {-1,1} {
		\tikzset{xscale=\aa}
		\draw[]
				(90-0.5*\PH:\L) -- (90:\h)
				(90-1.5*\PH:\L) -- (90-\Th:\h)
				(90-4.5*\PH:\L) -- (90-6*\Th:\h)
				(90-5.5*\PH:\L) -- (90-7*\Th:\h)
				(90-6.5*\PH:\L) -- (90-7*\Th:\h)
				(90-7.5*\PH:\L) -- (90-8*\Th:\h)
				(90-4*\Th:\h) to[out=90,in=-15, distance=3.25*\r cm] (90:6.25*\r)
				(90-5*\Th:\h) -- (90-5*\Th:\h+\r)
				(90-5*\Th:\h+\r) to[out=90,in=0, distance=4*\r cm] (90:6.25*\r)
				(90-8*\Th:\h) to[out=0,in=240] (90-5*\Th:\h+\r)
				;
}

\draw[]
		(90+2*\Th:\h) to[out=60,in=120, distance=2.5*\r cm] (90-2*\Th:\h)
		(90+3*\Th:\h) to[out=80,in=100, distance=4.5*\r cm] (90-3*\Th:\h)
		(90+5*\Th:\h+\r) arc (90+5*\Th:90+5*\Th+7*\Th:\h+\r)
		;

\foreach \aa in {-1,1} {
		\tikzset{xscale=\aa}
		;

\draw[fill=white]
		(90+0.5*\PH:\L) circle (\rr)
		(90+5.5*\PH:\L) circle (\rr)
		(90+6.5*\PH:\L) circle (\rr)
		(90+7*\Th:\h) circle (\rr) 
		(90+8*\Th:\h) circle (\rr) 
		(90+5*\Th:\h+\r) circle (\rr) 
		;

}

\draw[fill=white]
		(90+7.5*\PH:\L) circle (\rr)
		(90:\h) circle (\rr) 
		(90:6.25*\r) circle (\rr) 
		;

\node at (0.8*\RR,-\RR) {\small $J_{83}$};

\end{scope} 

\end{scope} 

\end{tikzpicture}
\caption{Tilings $J_{76},\dots,J_{83}$ with
				$\circ=\alpha_3\alpha_4\alpha_5\alpha_4$, $\bullet =
				\alpha_3\alpha_4\alpha_4\alpha_5$}
\label{Fig-J76-J83}
\end{figure}
\end{proof}

In Table~\ref{Table-eD-DR}, we summarise how to modify $eD$ to obtain
$J_{72},\dots,J_{83}$. The entries show how many cupolas must be
rotated or diminished.
When exactly two operations are
necessary, a superscript $o$ denotes that they should be performed on
opposite cupolas and a superscript $n$ denotes that they should be
performed on non-opposite cupolas.

\begin{table}[h!]
\begin{center}
\bgroup
\def\arraystretch{1.25}
    \begin{tabular}[]{| c | c | c | c | c | c | c |} 
    \hline
    Tiling & $J_{72}$ & $J_{73}$ & $J_{74}$ & $J_{75}$ & $J_{76}$ & $J_{77}$ \\
    \hline
    \hline
    dim & $0$ & $1^o$ & $1^n$ & $0$ & $1$ & $1^o$ \\
    \hline
    rot & $1$ & $1^o$ & $1^n$ & $3$ & $0$ & $1^o$ \\
    \hline
    \end{tabular}
\egroup
\end{center}
\begin{center}
\bgroup
\def\arraystretch{1.25}
    \begin{tabular}[]{| c | c | c | c | c | c | c |} 
    \hline
    Tiling & $J_{78}$ & $J_{79}$ & $J_{80}$ & $J_{81}$ & $J_{82}$ & $J_{83}$ \\
    \hline
    \hline
    dim & $1^n$ & $1$ & $2^{o}$ & $2^{n}$ & $2$ & $3$ \\
    \hline
    rot & $1^n$ & $2$ & $0$ & $0$ & $1$ & $0$ \\
    \hline
    \end{tabular}
\egroup
\end{center}
\caption{Modifying $eD$ into $J_{72},...,J_{83}$ by diminishing and
rotating cupolas}
\label{Table-eD-DR}
\end{table}

\begin{prop}
		\label{Prop-4al3al4-4al3al5} 
		If a tiling has a vertex of type $\alpha_3^4 \alpha_4$, then the
		tiling is the snub cube $sC$. If a tiling has a vertex of type
		$\alpha_3^4 \alpha_5$, then the tiling is the snub dodecahedron
		$sD$.
\end{prop}

\begin{proof} 
		If every vertex is of type $\al_3^4 \al_4$ then the tiling is the
		$sC$ by Proposition~\ref{Prop-vertex-homog}. By the same
		result, the tiling is a $sD$ if every vertex is of
		type $\al_3^4 \al_5$. We argue that, in either case, every vertex
		is of the same type; that is, we show that the tiling is
		vertex-homogenous. The argument is the same for both cases; let $m
		\in \{4,5\}$. We begin by showing that $\al_3\cdots = \al_3 \al_m
		\cdots = \al_3^4 \al_m$.

		Consider a vertex of type $\al_3^3\cdots$. It cannot have be of
		type $\al_3^3$ because $3\al_3 = 2\pi-\al_4$ by the vertex angle
		sum of $\al_3^3\al_4$. If it has degree 4 then it is
		$\al_3^3\al_n$ for some $n>m$. If it is degree 5 then it must be
		$\al_3^4 \al_m$. Now consider a vertex of type $\al_3 \al_m
		\cdots$. If it has degree 3 then it is of type $\al_3 \al_m \al_n$
		for some $n \ge m$. If it is of degree 4 then it is of type
		$\al_3^2 \al_m \al_n$ for some $n \ge m$ or of type $\al_3
		\al_4^3$ or of type $\al_3 \al_4^2 \al_5$. If it is of degree 5
		then it must be of type $\al_3^4 \al_m$. To summarise the above,
		we have:
		\begin{align*}
				\al_3^3\cdots &\in \{\al_3^3\al_{n>m}, \al_3^4\al_m\}, \\
				\al_3 \al_m\cdots &\in 
				\{\al_3\al_m\al_{n\ge m}, \al_3^2\al_m\al_{n \ge m}, \al_3
				\al_4^3, \al_3\al_4^2\al_5, \al_3^4 \al_m\}.
		\end{align*}
		If there is a vertex of type $\al_3^3\al_n$ then by
		Proposition~\ref{Prop-3al3alm} the tiling is $J_1$, $J_{11}$,
		$J_{62}$ or $J_{63}$; it is straightforward to check that there is
		no vertex of type $\al_3^4 \al_m$, for any $m$, in these tilings,
		which gives a contradiction. Therefore $\al_3^3\cdots =
		\al_3^4\al_m$.

		For each vertex type in the set of possibilities for
		$\al_3\al_m\cdots$, other than $\al_3^4\al_m$, we will observe
		that the tiling is one of $eD,J_2,J_4,J_5,
		J_{11}, J_{19}, J_{37},$ $J_{62},J_{63}, J_{72},\dots,J_{83}$. It is
		straightforward to check that there is no vertex of type $\al_3^4
		\al_m$, for any $m$, in these tilings, from which we derive a
		contradiction. We will conclude that $\al_3\al_m \cdots = \al_3^3
		\al_m$.

		If there is a vertex of type $\al_3^2 \al_m \al_n$ for $n \ge m$,
		then $\al_n = 2\pi - 2\al_3 - \al_m = 2\al_3$. Therefore, we can
		perform a pyramid subdivision to an $n$-gon to obtain a new tiling
		by Lemma~\ref{lem:subdiv}; the new vertex in this tiling is of
		type $\al_3^n$ and so $n<6$. If $n=4$ then $\al_3 =
		\frac{1}{2}\pi$ and if $n=5$ then $\al_3 = \frac{2}{5}\pi$. By
		Proposition~\ref{Prop-al3=pi/2-2pi/5}, the tiling is $J_2$,
		$J_{11}$, $J_{62}$ or $J_{63}$.
		
		If there is a vertex of type $\al_3 \al_4^3$, then by
		Proposition~\ref{Prop-al33al4} the tiling is $J_4$, $J_{19}$ or
		$J_{37}$.
		
		If there is a vertex of type $\al_3 \al_4^2 \al_5$, then by
		Proposition~\ref{Prop-al32al4al5} the tiling is $eD$, $J_5$ or one
		of $J_{72},\dots,J_{83}$. 

		If there is a vertex of type $\al_3 \al_m \al_n$ for $n \ge m$,
		then $\al_n = 3\al_3 > \pi$. Therefore,
		by Lemma~\ref{Lem-cong}, there is exactly one $n$-gon
		and all its vertices are of type $\al_3\al_m\al_n$.
		Now consider a triangle sharing an edge with the $n$-gon; its two
		other edges are shared with $m$-gons. Therefore the vertex of the
		triangle not shared with the $m$-gon is of type
		$\al_3\al_m^2\cdots$. But then it must be of type $\al_3\al_4^3$
		or $\al_3\al_4^2\al_5$ and we derive a contradiction as above.

		To complete the proof, we show that every vertex is of type
		$\al_3^4\al_m$. If there is a vertex of another type, we can
		assume that it has a neighbour of type $\al_3^4\al_4$. Denote a
		vertex of type $\al_3^4 \al_m$ by $0$ and its five neighbours by
		$1,\dots,5$ so that the vertices 1 and 5 are incident with the
		$m$-gon; see Figure~\ref{fig:P38}. Observe that 1 and 5 are of
		type $\al_3 \al_m \cdots = \al_3^4 \al_m$ and 2, 3 and 4 are of
		type $\al_3^2 \cdots$. We demonstrate that 2, 3 and 4 are of type
		$\al_3^3\cdots = \al_3^4 \al_m$. Since 1 and 2 do not share an
		$m$-gon, and 1 is of type $\al_3^4 \al_m$, we see that 1 and 2
		share two triangles.
		We deduce that 2 is of type $\al_3^3\cdots = \al_3^4 \al_m$
		because 2 shares a triangle with 3 that is not shared with 1. But
		2 shares exactly one triangle with 3 and hence it shares a
		triangle and an $m$-gon. This shows that the vertex 3 is of type
		$\al_3 \al_m\cdots = \al_3^4 \al_m$; the proof that the vertex 4
		is of the same type is symmetrical, hence the result.\qedhere

		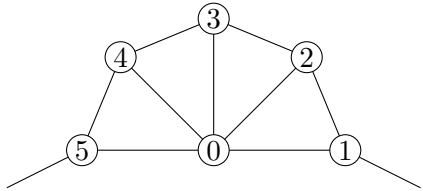
\begin{figure}[h!]
				\centering
				\begin{tikzpicture}
		\tikzmath{
				\r=1.75;
		}
		
		\foreach \a in {0,...,4}{
				\draw (0,0) -- (\a*45:\r);
		}

		\foreach \a in {0,...,3}{
				\draw (\a*45:\r) -- (\a*45+45:\r);
		}

		\draw (\r,0) -- (\r+1,-0.5);
		\draw (-\r,0) -- (-\r-1,-0.5);
		\node[inner sep=1,draw,fill=white,shape=circle] at (0,0) {\small
		$0$};
		\foreach \a in {1,...,5}{
				\node[inner sep=1,draw,fill=white,shape=circle] at (\a*45-45:\r)
				{\small $\a$};
		}
		\end{tikzpicture}
				\caption{Deduction of $\al_3^4\al_m$}
				\label{fig:P38}
		\end{figure}

\end{proof}

\begin{prop}
		\label{Prop-al3almaln}
		If a tiling has a vertex of type $\al_3 \al_m \al_n$ with $n > m \ge
		4$, then $m=4$, $n=10$ and the tiling is the Johnson tiling $J_5$.
\end{prop}

\begin{proof}
		Observe that none of the
		tilings in Proposition~\ref{Prop-vertex-homog} has vertex type
		$\al_3 \al_m \al_n$ for $n > m \ge 4$. So the tiling cannot be
		vertex-homogenous.

		Suppose that every vertex incident with a triangle is of degree 3
		with three distinct angles; that is, every such vertex is of type
		$\al_3 \al_k \al_\ell$ for some $k,\ell$ such that $\ell > k \ge
		4$. It follows that $\al_3 \al_m \cdots = \al_3 \al_n \cdots =
		\al_3 \al_m \al_n$. Also, $\al_m \al_n \cdots = \al_3 \al_m \al_n$
		by Lemma~\ref{Lem-small-remainder}. But by
		Lemma~\ref{Lem-vertex-transitive}, the tiling is
		vertex-homogenous, which is a contradiction.

		We have deduced that there is at least one vertex incident with a
		triangle whose type is not $\al_3 \al_k \al_\ell$ with $\ell > k
		\ge 4$. This vertex is either of type $\al_3^3$, $\al_3^2 \al_p$
		or $\al_3 \al_p^2$ for some $p \ge 4$, or it is of degree greater
		than 3. For each of these types, the existence of such a vertex
		contradicts the existence of a vertex of type $\al_3 \al_m \al_n$
		(Propositions~\ref{Prop-al3=pi/2-2pi/5}--\ref{Prop-2al3almaln}),
		with the exception of the type $\al_3 \al_4^2 \al_5$, in which
		case $m=4$, $n=10$ and the tiling is $J_5$ as required.
\end{proof}

\begin{prop} 
		\label{Prop-2al3almaln} 
		There is no tiling with a vertex of type
		$\alpha_3^2\alpha_m\alpha_n$ for
		$n>m\ge4$.
\end{prop}

\begin{proof} 
		Suppose there is such a tiling for some fixed $m$ and $n$.
		Rearranging the angle sum for this type, we obtain $\al_m +
		\al_n = 2\pi - 2\al_3 < \frac{4}{3}\pi$ which implies that $\al_m
		< \frac{2}{3}\pi < \al_6$ by \eqref{Ineq-alm-lb}. Hence, $m \in
		\{4,5\}$.

		Observe that none of the
		tilings in Proposition~\ref{Prop-vertex-homog} has vertex type
		$\al_3^2 \al_m \al_n$ for $n > m \ge 4$. So
		if there is such a vertex, the tiling is not
		vertex-homogenous; we can assume that this vertex has a neighbour
		of type $\al_3^2\cdots$, $\al_3\al_m\cdots$, $\al_3\al_n\cdots$ or
		$\al_m\al_n\cdots$. 

		If there is a vertex of type $\al_3^2\al_{p}$ then $p>5$ and by
		Proposition~\ref{Prop-2al3alm} there is no such tiling. There is
		no vertex of type $\al_3^3\al_{p}$ since none of the tilings in
		Proposition~\ref{Prop-3al3alm} has a vertex of type
		$\al_3^2\al_m\al_n$. Thus $\al_3^2\cdots = \al_3^2\al_p\al_q$ for
		some $q > p \ge 4$.
		
		If there is a vertex of type $\al_3\al_k\al_\ell$ where $k \in
		\{m,n\}$ then $\ell>4$; we get a contradiction since none of the
		tilings in Propositions~\ref{Prop-al32alm} and
		\ref{Prop-al3almaln} has a vertex of type
		$\al_3^2\al_m\al_n$. Clearly a vertex of type $\al_3\al_k\cdots$
		of degree 4 is of type $\al_3^2\al_m\al_n$. If such a vertex is of
		degree 5, then it is either $\al_3^4\al_4$ or $\al_3^4\al_5$ and
		we get a contradiction since neither of the tilings in
		Proposition~\ref{Prop-4al3al4-4al3al5} has a vertex of type
		$\al_3^2\al_m\al_n$. Thus $\al_3\al_m\cdots = \al_3^2\al_m\al_n$.

		If there is a vertex of $\al_\ell\al_m\al_n$ for some $\ell$,
		then $\al_\ell = 2\al_3$, and we can perform a pyramid subdivision
		on the $\ell$-gon to obtain a new tiling by
		Lemma~\ref{lem:subdiv}. The new vertex in this tiling is of type
		$\al_3^\ell$ and hence $\al_3 \in \{\frac{2}{3}\pi,
				\frac{1}{4}\pi,
		\frac{2}{5}\pi\}$; we get a contradiction because none of the
		tilings in Proposition~\ref{Prop-al3=pi/2-2pi/5} have a vertex of
		type $\al_3^2\al_m\al_n$. Thus $\al_m\al_n\cdots =
		\al_3^2\al_m\al_n$.

		To summarise the above, every vertex in the tiling is of type
		$\al_3^2\al_p\al_q$ for some $p$ and $q$ such that $q > p \ge 4$.
		We argue that for every vertex, $p=m$ and $q=n$. Suppose the
		contrary; without loss of generality, $p<m<n<q$. We demonstrated
		that $p$ and $m$ are in $\{4,5\}$, so $p=4$ and $m=5$. The angle
		sum of $\al_3^2\al_5\al_n$ and \eqref{Ineq-alm-lb} give
		$(1-\frac{2}{n})\pi < \alpha_n = 2\pi - 2\alpha_3 - \alpha_5 <
		2\pi - 2\cdot \frac{1}{3}\pi - \frac{3}{5}\pi  =
		\frac{11}{15}\pi$, which implies $5 < n < \frac{15}{2}$. Hence
		$n\in\{ 6,7\}$. Applying the Gr\"obner basis technique from
		Proposition~\ref{Prop-al32al4al5}, we obtain exact values for
		$\al_3$, $\al_5$ and $\al_n$: when $n=6$, we obtain two solutions
		to the set of equations for which $\al_5>\al_6$ giving a
		contradiction; when $n=7$, the angle values are given below.
		\begin{align*}
				\al_3 &=(0.3357023573924277...)\pi,\\ \al_5
							&=(0.6056764771694325...)\pi,\\ \al_7
							&=(0.7229188642174295...)\pi.
		\end{align*}
		Substituting this value for $\al_3$ into
		\eqref{Eq-alm-aln-linear}, we obtain 
		\[
				\al_4 = (0.5041121622358487...)\pi.
		\] Using the vertex angle formula for $\al_3^2\al_4\al_n$, we
		futher obtain
		\[
				\al_n = (0.8244831229792959...)\pi.
		\] Finally, we substitute this value and the value for $\al_3$
		into \eqref{Eq-alm-aln-linear} to obtain $n=10.56076889342715...$,
		a contradiction.

		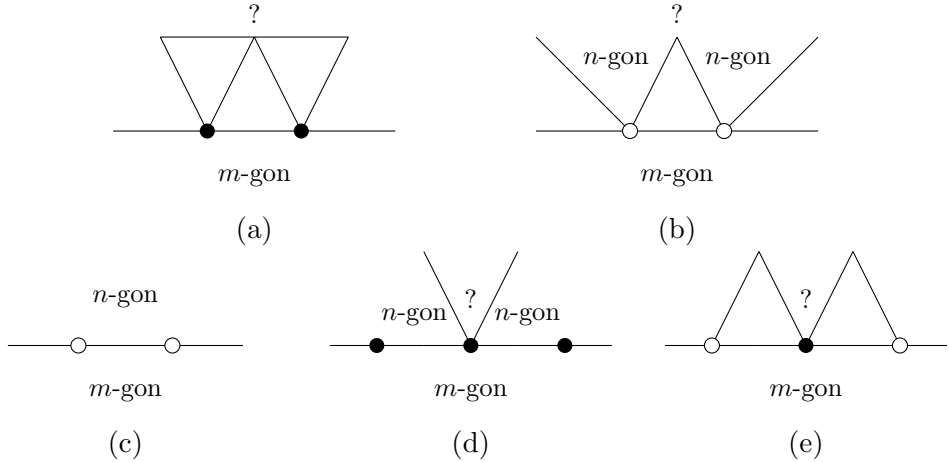
\begin{figure}[h!]
\centering
\begin{subfigure}[t]{0.4\linewidth}
\centering
\begin{tikzpicture}
\tikzmath{
\r=1.25;
\rr=0.08*\r;
}

\foreach \a in {-1,0,1} {
\tikzset{xshift=\a*\r cm}

\draw[]
	(-0.5*\r,0) -- (0.5*\r,0)
;
}

\foreach \a in {0,1} {
\tikzset{xshift=\a*\r cm}
\fill[]
	(-0.5*\r,0) circle (\rr)
;
}

\foreach \aa in {-1,1} {
\tikzset{xscale=\aa}

\draw[]
	(-0.5*\r,0) -- (0,\r)
	(-0.5*\r,0) -- (-1*\r,\r)
	(-1*\r,\r) -- (0,\r)
;
}

\node at (0,1.25*\r) {\footnotesize $?$};

\node at (0,-0.5*\r) {\footnotesize $m$-gon};

\end{tikzpicture}
\caption{}
\label{Subfig-forbidden-black}
\end{subfigure}
\begin{subfigure}[t]{0.4\linewidth}
\centering
\begin{tikzpicture}
\tikzmath{
\r=1.25;
\rr=0.08*\r;
}

\foreach \a in {-1,0,1} {
\tikzset{xshift=\a*\r cm}

\draw[]
	(-0.5*\r,0) -- (0.5*\r,0)
;
}

\foreach \aa in {-1,1} {
\tikzset{xscale=\aa}

\draw[]
	(-0.5*\r,0) -- (0,\r)
	(-0.5*\r,0) -- (-1.5*\r,\r)
;

\node at (-0.65*\r,0.75*\r) {\footnotesize $n$-gon};

}

\foreach \a in {0,1} {
\tikzset{xshift=\a*\r cm}
\draw[fill=white]
	(-0.5*\r,0) circle (\rr)
;
}

\node at (0,1.25*\r) {\footnotesize $?$};

\node at (0,-0.5*\r) {\footnotesize $m$-gon};
\end{tikzpicture}
\caption{}
\label{Subfig-forbidden-white}
\end{subfigure}

\begin{subfigure}[t]{0.325\linewidth} 
\centering
\begin{tikzpicture}
\tikzmath{
\r=1.25;
\rr=0.08*\r;
}

\draw[]
	(-1.25*\r,0) -- (1.25*\r,0)
;

\foreach \aa in {-1,1} {
\tikzset{xscale=\aa}
\draw[fill=white]
	(-0.5*\r,0) circle (\rr)
;
}

\node at (0,0.5*\r) {\footnotesize $n$-gon};

\node at (0,-0.5*\r) {\footnotesize $m$-gon};

\end{tikzpicture} 
\caption{}
\label{Subfig-forbidden-ww}
\end{subfigure} 
\begin{subfigure}[t]{0.325\linewidth} 
\centering
\begin{tikzpicture}
\tikzmath{
\r=1.25;
\rr=0.08*\r;
}

\foreach \a in {-1,0,1} {
\tikzset{xshift=\a*\r cm}

\draw[]
	(-0.5*\r,0) -- (0.5*\r,0)
;
}

\foreach \a in {-1,0,1} {
\tikzset{xshift=\a*\r cm}
\fill[]
	(0,0) circle (\rr)
;
}

\foreach \aa in {-1,1} {
\tikzset{xscale=\aa}
\draw[]
	(0,0) -- (-0.5*\r,\r)
;
\node at (-0.6*\r,0.3*\r) {\footnotesize $n$-gon};
}

\node at (0,0.5*\r) {\footnotesize $?$};

\node at (0,-0.5*\r) {\footnotesize $m$-gon};

\end{tikzpicture}
\caption{}
\label{Subfig-forbidden-bbb}
\end{subfigure}
\begin{subfigure}[t]{0.325\linewidth} 
\centering
\begin{tikzpicture}
\tikzmath{
\r=1.25;
\rr=0.08*\r;
}

\foreach \a in {-1,0,1} {
\tikzset{xshift=\a*\r cm}

\draw[]
	(-0.5*\r,0) -- (0.5*\r,0)
;
}

\foreach \aa in {-1,1} {
\tikzset{xscale=\aa}
\draw[]
	(0,0) -- (-0.5*\r,\r)
	(-\r,0) -- (-0.5*\r,\r)
;
}

\fill[]
	(0,0) circle (\rr)
;

\foreach \aa in {-1,1} {
\tikzset{xscale=\aa}
\draw[fill=white]
	(-\r,0) circle (\rr)
;
}

\node at (0,0.5*\r) {\footnotesize $?$};

\node at (0,-0.5*\r) {\footnotesize $m$-gon};

\end{tikzpicture} 
\caption{}
\label{Subfig-forbidden-wbw}
\end{subfigure}

\caption{Forbidden arrangements}
\label{Fig-forbidden-2al3almaln}
\end{figure}

		Since the tiling is not vertex-homogenous and every vertex is of
		type $\al_3^2\al_m\al_n$, we must have vertices with angle
		arrangement $\al_3\al_3\al_m\al_n$ and with angle arrangement
		$\al_3\al_m\al_3\al_n$. We call these vertices black and white
		respectively. We call an edge black (white) if both its endpoints
		are black (white); such edges are said to be monochromatic. It is
		clear from the angle arrangements that an edge shared by an
		$n$-gon and an $m$-gon is black and that an edge shared by two
		triangle is also black. We can deduce that an edge shared by a
		triangle and an $m$-gon is not monochromatic. If it is black
		(Figure~\ref{Subfig-forbidden-black}), then the vertex of the
		triangle not shared by the $m$-gon is incident with three
		triangles; if it is white (Figure~\ref{Subfig-forbidden-white}),
		then that vertex is incident with two $n$-gons. We claim that two
		consecutive vertices on an $m$-gon cannot be white: since the edge
		between them is monochromatic, the edge must be shared by an
		$n$-gon, but an edge shared by an $m$-gon and an $n$-gon is black
		(Figure~\ref{Subfig-forbidden-ww}). Furthermore, three consecutive
		vertices on an $m$-gon cannot be black: the two edges between them
		must be shared by an $n$-gon because they are monochromatic, so
		the middle vertex is either of degree 2 or incident to two
		$n$-gons (Figure~\ref{Subfig-forbidden-bbb}). Finally, three
		consecutive vertices on an $m$-gon cannot be coloured
		white-black-white: the two edges between them must be shared by
		triangles, so the middle vertex is of degree 2 or it has angle
		arrangement $\al_3\al_n\al_3\al_m$, a contradiction since such
		vertices are white (Figure~\ref{Subfig-forbidden-wbw}). We
		conclude that the sequence of colours around an $m$-gon is
		white-black-black-white-black-black and so on and that $m$ is
		divisible by 3. This contradicts the fact that $m\in \{4,5\}$.
		\qedhere

\end{proof}


\subsection*{Tilings without triangles}
We have completely characterised the tilings with at least one
triangle. For each vertex type in the lists \eqref{List-al4} and
\eqref{List-al5} of vertex types without $\al_3$, we check whether it
has appeared in one of the tilings given in
Propositions~\ref{Prop-al3=pi/2-2pi/5}--\ref{Prop-al3almaln} (see
Tables~\ref{tab:conway}--\ref{tab:OIaCeCaD}). We discover that $\al_4^2\al_8$ appears in $J_{19}$ and
that $\al_4\al_5\al_{10}$ appears in $J_{76},\dots,J_{83}$; every
other type from those lists does not appear in any tiling with at
least one triangle.  
Therefore, the following lemma and the subsequent proposition complete
the proof of the main theorem.

\begin{lem}
		\label{Lem-tri-free}
		If a tiling has no triangle, then it is vertex-homogenous.
\end{lem}

\begin{proof}
		By consulting the lists of possible vertex types in
		\eqref{List-al4} and \eqref{List-al5}, we see that every vertex in
		a triangle-free tiling is of degree $3$. Consider a triangle-free
		tiling with a vertex of type $\al_\ell \al_m \al_n$.
		Since $R(\al_\ell\al_m) = \al_n$ and since all vertices are of
		degree 3, we have $\al_\ell\al_m\cdots = \al_m\al_n\al_n$. By a symmetrical
		argument, we deduce that $\al_\ell\al_m\cdots =
		\al_\ell\al_n\cdots = \al_m\al_n\cdots = \al_\ell\al_m\al_n$, and
		the result follows from Lemma~\ref{Lem-vertex-transitive}.
\end{proof}

The following proposition follows from Lemma~\ref{Lem-tri-free} and
Proposition~\ref{Prop-vertex-homog}.

\begin{prop}
		\label{Prop-tri-free}
		If a tiling has a vertex without an incident triangle, then
		the tiling is given below.
		\begin{enumerate}

				\item $\al_4^3$: the cube

				\item $\al_4^2\al_{m>4}$: a prism or the Johnson tiling
						$J_{19}$ ($m=8$)

				\item $\al_4\al_{m>4}^2$: the truncated octahedron $tO$

				\item $\al_4\al_{m>4}\al_{n>m}$: the truncated cuboctahedron
						$bC$ or one of the Johnson tilings
						$J_{76},\dots,J_{83}$ ($m=5$, $n=10$)

				\item $\al_5^2\al_{m>4}$: the dodecahedron

				\item $\al_5\al_{m>5}^2$: the truncated icosahedron $tI$

				\item $\al_5\al_{m>5}\al_{n>6}$: the truncated
						icosidodecahedron $bD$

		\end{enumerate}
\end{prop}

\appendix

\section*{Appendix}
\setcounter{section}{1}
\renewcommand\thesection{A}

\subsection{Proof of Theorem \ref{thm:weakly}}
\label{app:weakly}

\subsection{Tiling Data}
\label{app:data}

In this section we present the various data associated with the
tilings listed in the main result. For each tiling we give the vertex
types and their multiplicities, as well as the values for the angles
of the polygons. We make use of Conway's notation for the Archimedean
tilings given in Table~\ref{tab:conway}.

\begin{table}
\begin{center}
\def\arraystretch{1.25}
    \begin{tabular}[]{| c | c |} 
    		\hline
    		Conway's notation & Name \\
    		\hline
		\hline 
		$tT$ & truncated tetrahedron \\
		\hline
		$aC = aO = eT$ & cuboctahedron \\
		\hline
		$tC$ & truncated cube \\
		\hline
		$tO = bT$ & truncated octahedron \\
		\hline 
		$eC = eO$ & rhombicuboctahedron \\
		\hline 
		$bC = bO$ & truncated cuboctahedron \\
		\hline 
		$sC = sO$ & snub cube \\
		\hline 
		$aD = aI$ & icosidodecahedron \\
		\hline 
		$tD$ & truncated dodecahedron \\
		\hline
		$tI$ & truncated icosahedron \\
		\hline
		$eD = eI$ & rhombicosidodeca­hedron \\
		\hline
		$bD = bI$ & truncated icosidodecahedron \\
		\hline
		$sD = sI$ & snub dodecahedron \\
		\hline
    \end{tabular}
\end{center}
\caption{Conway's notation for the 13 Archimedean tilings}
\label{tab:conway}
\end{table}

As we have seen, several tilings are obtained from other tilings by
operations that preserve the edge length and therefore the value of
$\al_m$ for each (convex) $m$-gon. The tilings that are not related to
any other tiling in this way are presented in
Table~\ref{tab:standalone}. The
values are given in an exact form with the exception of the tiling
$sD$; the value $\cos \al_3$ for that tiling is defined to be one a
root of the polynomial $64\xi^6+128\xi^5+64\xi^4-24\xi^3-24\xi^2+1$
(approximately, $\xi = 0.471575629621941...$) which was obtained via
the Gr\"obner basis.

\begin{table}
\begin{center}
\bgroup
\def\arraystretch{1.25}
\begin{adjustbox}{width=\columnwidth,center}
    \begin{tabular}[]{| c | c | c | c | c |} 
		\hline
		Tiling & Vertex Types & Angles \& Edge  &  Polygons  \\
		\hline
		\hline 
		$T$ & $\{ 4\alpha_3^3 \}$ & $\alpha_3=\tfrac{2}{3}\pi$, \quad $x=\cos^{-1} \tfrac{-1}{3}$ & $f_3=4$ \\
		\hline
		$C$ & $\{ 8\alpha_4^3 \}$ & $\alpha_4=\tfrac{2}{3}\pi$, \quad  $x=\cos^{-1} \tfrac{1}{3}$ & $f_4=6$ \\
		\hline
		$D$ & $\{ 20\alpha_5^3 \}$ & $\alpha_5=\tfrac{2}{3}\pi$, \quad $x=\cos^{-1}\tfrac{\sqrt{5}}{3}$ & $f_5=12$ \\
		\hline
		$tT$ & $\{ 12\alpha_3\alpha_6^2 \}$ & $\alpha_3=4\cot^{-1}\sqrt{11}$, \quad $\alpha_6=\pi - \tfrac{1}{2}\alpha_3$, \quad $x=\cos^{-1}\tfrac{7}{11}$ & $f_3=4$, $f_6=4$ \\
		\hline
		$tC$ & $\{ 24\alpha_3\alpha_8^2 \}$ & $\alpha_3 = 4\cot^{-1} \textstyle \sqrt{7+4\sqrt{2}}$, \quad $\alpha_8=\pi - \tfrac{1}{2}\alpha_3$, \quad $x=\cos^{-1} \tfrac{1}{17}(3+8\sqrt{2})$ & $f_3=8$, $f_8=6$ \\
		\hline
		$tO$ & $\{ 24\alpha_4\alpha_6^2 \}$ & $\alpha_4=4\cot^{-1}\sqrt{5}$, \quad $\alpha_6=\pi - \tfrac{1}{2}\alpha_4$, \quad $x=\cos^{-1} \tfrac{4}{5}$ & $f_4=6$, $f_6=8$ \\
		\hline
		$tD$ & $\{ 60\alpha_3\alpha_{10}^2 \}$ & $\alpha_3 =4\cot^{-1} \textstyle \sqrt{9+2\sqrt{5}}$, \quad $\alpha_{10}=\pi - \tfrac{1}{2}\alpha_3$, \quad $x=\cos^{-1}\tfrac{1}{61} (24 + 15\sqrt{5})$ & $f_3=20$, $f_{10}=12$ \\
		\hline
		$tI$ & $\{60\alpha_5\alpha_6^2\}$ & $\alpha_5 = 4\tan^{-1} \textstyle \sqrt{\tfrac{1}{109} (17 + 6\sqrt{5})}$, \quad $\alpha_{6}=\pi - \tfrac{1}{2}\alpha_5$, \quad $x=\cos^{-1} \tfrac{1}{109} (80 + 9\sqrt{5})$ & $f_5=12$, $f_6=20$ \\
		\hline
		\hline
		$sC$ & $\{ 24\alpha_3^4\alpha_4 \}$ & \parbox[c]{12.5cm}{$\alpha_3 = 2\cot^{-1} \textstyle\sqrt{\tfrac{19}{21} + \tfrac{1}{21} \textstyle\sqrt[3]{4528 - 336\textstyle\sqrt{33}} + \tfrac{2}{21} \sqrt[3]{ 556 + 42 \sqrt{33} } }$, \\ $\alpha_4= 2\pi - 4\alpha_3$, \quad $x=\cos^{-1} \tfrac{1}{21} (-1 + \sqrt[3]{566 - 42\sqrt{33}} + \sqrt[3]{566 + 42 \sqrt{33}})$} & $f_3=32$, $f_4=6$ \\
		\hline
		$sD$ & $\{ 60\alpha_3^4\alpha_5 \}$ & $\alpha_3 = \cos^{-1} \xi$, \quad $\alpha_5=2\pi - 4\alpha_3$, \quad $x=\cos^{-1} \tfrac{\xi}{1-\xi}$ & $f_3=80$, $f_5=12$ \\
		\hline		
		\hline
		$aC$ & $\{ 12(\alpha_3\alpha_4)^2 \}$ & $\alpha_3 = \cos^{-1} \tfrac{1}{3}$, \quad $\alpha_4=\pi-\alpha_3$, \quad $x=\tfrac{1}{3}\pi$ & $f_3=8$, $f_4=6$ \\
		\hline
		$aD$ & $\{ 30(\alpha_3\alpha_5)^2 \}$ & $\alpha_3 = \cos^{-1} \tfrac{1}{\sqrt{5}}$, \quad $\alpha_5=\pi - \alpha_3$, \quad $x=\tfrac{1}{5}\pi$ & $f_3 = 20$, $f_5=12$ \\
		\hline
		\hline
		$bC$ & $\{ 48\alpha_4\alpha_6\alpha_8 \}$ & \parbox[c]{9cm}{$\alpha_4=\cos^{-1}\tfrac{1}{12} (\sqrt{2}-2)$, \quad $\alpha_6=\cos^{-1}\tfrac{1}{8} (\sqrt{2}-6)$, \\ $\alpha_8=\cos^{-1} \tfrac{-1}{12} (6\sqrt{2} + 1)$, \quad $x=\cos^{-1} \tfrac{1}{97} (71+12\sqrt{2})$} & $f_4=12$, $f_6=8$, $f_8=6$ \\
		\hline
		$bD$ & $\{ 120\alpha_4\alpha_6\alpha_{10} \}$ & \parbox[c]{9.5cm}{$\alpha_4=\cos^{-1} \tfrac{1}{30}(2\sqrt{5}-5)$, \quad $\alpha_6=\cos^{-1} \tfrac{1}{20}(2\sqrt{5}-15)$, \\ $\alpha_{10}=\cos^{-1} \tfrac{-1}{24} (9 + 5\sqrt{5})$, \quad $x=\cos^{-1}\tfrac{1}{241} (179 + 24\sqrt{5})$} & $f_4=30$, $f_6=20$, $f_{10}=12$ \\
		\hline
	\end{tabular}
\end{adjustbox}
\egroup
\end{center}
\caption{Standalone tilings}
\label{tab:standalone}
\end{table}

In Tables~\ref{tab:eD} and~\ref{tab:OIaCeCaD} we give the data
associated with the tilings that can be grouped together according to
their edge length. Table~\ref{tab:eD} includes the tilings related to
$eD$; the remaining tilings appear in Table~\ref{tab:OIaCeCaD}.

\begin{table}[h!]
\begin{center}
\bgroup
\def\arraystretch{2}
\begin{adjustbox}{width=\columnwidth,center}
    \begin{tabular}[]{| c | c | c | c | c |} 
		\hline
		Tilings & Vertex Types & Angles \& Edge  &  Polygons  \\
		\hline
		\hline
		$J_5$ & $\{ 5\alpha_3\alpha_4\alpha_5\alpha_4, 10\alpha_3\alpha_4\alpha_{10} \}$ & \parbox[c]{5cm}{\begin{align*} &\alpha_3 =\cos^{-1}\tfrac{1}{20} (5 + 2\sqrt{5}), \\ &\alpha_4 = \cos^{-1} \tfrac{1}{10} (2\sqrt{5} - 5), \\ &\alpha_5 = \cos^{-1}  \tfrac{1}{40} (5 - 9\sqrt{5}),\\  &\alpha_{10}=2\pi - (\alpha_3 + \alpha_4), \\ &x=\cos^{-1}\tfrac{1}{41}(19 + 8\sqrt{5}) \end{align*}} & $f_3=5$, $f_4=5$, $f_5=1$, $f_{10}=1$  \\
		\hline
		$eD$ & $\{ 60\alpha_3\alpha_4\alpha_5\alpha_4 \}$ & \multirow{13}{*}{\parbox[c]{5cm}{\begin{align*} &\alpha_3 =\cos^{-1}\tfrac{1}{20} (5 + 2\sqrt{5}), \\ &\alpha_4 = \cos^{-1} \tfrac{1}{10} (2\sqrt{5} - 5), \\ &\alpha_5 = \cos^{-1}  \tfrac{1}{40} (5 - 9\sqrt{5}),\\  &\alpha_{10}=2\pi - (\alpha_4 + \alpha_5), \\ &x=\cos^{-1}\tfrac{1}{41}(19 + 8\sqrt{5}) \end{align*}}} & $f_3=20$, $f_4=30$, $f_5=12$ \\
		\cline{1-2}\cline{4-4}
		$J_{72}$ & $\{ 50\alpha_3\alpha_4\alpha_5\alpha_4, 10\alpha_3\alpha_4^2\alpha_5 \}$ & & $f_3=20$, $f_4=30$, $f_5=12$ \\
		\cline{1-2}\cline{4-4}
		$J_{73}$ & $\{ 40\alpha_3\alpha_4\alpha_5\alpha_4, 20\alpha_3\alpha_4^2\alpha_5 \}$ & & $f_3=20$, $f_4=30$, $f_5=12$ \\
		\cline{1-2}\cline{4-4}
		$J_{74}$ & $\{ 40\alpha_3\alpha_4\alpha_5\alpha_4, 20\alpha_3\alpha_4^2\alpha_5 \}$ & & $f_3=20$, $f_4=30$, $f_5=12$ \\
		\cline{1-2}\cline{4-4}
		$J_{75}$ & $\{ 30\alpha_3\alpha_4\alpha_5\alpha_4, 30\alpha_3\alpha_4^2\alpha_5 \}$ & & $f_3=20$, $f_4=30$, $f_5=12$ \\
		\cline{1-2}\cline{4-4}
		$J_{76}$ & $\{ 45\alpha_3\alpha_4\alpha_5\alpha_4, 10\alpha_4\alpha_5\alpha_{10} \}$ & & $f_3=15$, $f_4=25$, $f_5=11$, $f_{10}=1$  \\
		\cline{1-2}\cline{4-4}
		$J_{77}$ & $\{ 35 \alpha_3\alpha_4\alpha_5\alpha_4, 10\alpha_3\alpha_4^2\alpha_5, 10\alpha_4\alpha_5\alpha_{10} \}$ & & $f_3=15$, $f_4=25$, $f_5=11$, $f_{10}=1$ \\
		\cline{1-2}\cline{4-4}
		$J_{78}$ & $\{ 35\alpha_3\alpha_4\alpha_5\alpha_4, 10\alpha_3\alpha_4^2\alpha_5, 10\alpha_4\alpha_5\alpha_{10} \}$ & & $f_3=15$, $f_4=25$, $f_5=11$, $f_{10}=1$ \\
		\cline{1-2}\cline{4-4}
		$J_{79}$ & $\{ 25\alpha_3\alpha_4\alpha_5\alpha_4, 20 \alpha_3\alpha_4^2\alpha_5, 10\alpha_4\alpha_5\alpha_{10} \}$ & & $f_3=15$, $f_4=25$, $f_5=11$, $f_{10}=1$ \\
		\cline{1-2}\cline{4-4}
		$J_{80}$ & $\{ 30 \alpha_3\alpha_4\alpha_5\alpha_4, 20\alpha_4\alpha_5\alpha_{10} \}$ & & $f_3=10$, $f_4=20$, $f_5=10$, $f_{10}=2$ \\
		\cline{1-2}\cline{4-4}
		$J_{81}$ & $\{ 30\alpha_3\alpha_4\alpha_5\alpha_4, 20 \alpha_4\alpha_5\alpha_{10} \}$ & & $f_3=10$, $f_4=20$, $f_5=10$, $f_{10}=2$ \\
		\cline{1-2}\cline{4-4}
		$J_{82}$ & $\{ 20 \alpha_3\alpha_4\alpha_5\alpha_4, 10\alpha_3\alpha_4^2\alpha_5, 20\alpha_4\alpha_5\alpha_{10} \}$ & & $f_3=10$, $f_4=20$, $f_5=10$, $f_{10}=2$ \\
		\cline{1-2}\cline{4-4}
		$J_{83}$ & $\{ 15\alpha_3\alpha_4\alpha_5\alpha_4, 30\alpha_4\alpha_5\alpha_{10}\}$ & & $f_3=5$, $f_4=15$, $f_5=9$, $f_{10}=3$ \\
		\hline
	\end{tabular}
\end{adjustbox}
\egroup
\end{center}
\caption{Tilings related to $eD$}
\label{tab:eD}
\end{table}

\begin{table}
\begin{center}
\bgroup
\def\arraystretch{1.25}
\begin{adjustbox}{width=\columnwidth,center}
    \begin{tabular}[]{| c | c | c | c | c |} 
		\hline
		Tiling & Vertex Types & Angles \& Edge  &  Polygons  \\
		\hline
		\hline 
		$O$ & $\{ 6\alpha_3^4 \}$ & \multirow{2}{*}{$\alpha_3 = \tfrac{1}{2}\pi$, \quad $\alpha_4=\pi$, \quad $x=\tfrac{1}{2}\pi$} & $f_3=8$ \\
		\cline{1-2}\cline{4-4}
		$J_1$ & $\{ 4\alpha_3^2\alpha_m, 1\alpha_3^4 \}$ & & $f_3=4$, $f_4=1$ \\
		\hline
		\hline
		$J_2$ & $\{ 5\alpha_3^2\alpha_m, 1\alpha_3^5 \}$ & $\alpha_3 = \tfrac{2}{5}\pi$, \quad $\alpha_5=\tfrac{6}{5}\pi$, \quad $x=\cos^{-1}\tfrac{1}{\sqrt{5}}$ & $f_3=5$, $f_5=1$ \\
		\hline
		$I$ & $\{ 12\alpha_3^5 \}$ & \multirow{4}{*}{\parbox[c]{5cm}{\begin{align*}  \alpha_3 = \tfrac{2}{5}\pi, \quad \alpha_5=\tfrac{4}{5}\pi,  \quad x=\cos^{-1}\tfrac{1}{\sqrt{5}} \end{align*}}}  & $f_3=20$ \\
		\cline{1-2}\cline{4-4}
		$J_{11}$ & $\{ 5\alpha_3^3\alpha_5, 6\alpha_3^5 \}$ & & $f_3=15$, $f_5=1$ \\
		\cline{1-2}\cline{4-4}
		$J_{62}$ & $\{ 2\alpha_3\alpha_5^2, 6\alpha_3^3\alpha_5, 2\alpha_3^5 \}$ & & $f_3=10$, $f_5=2$  \\
		\cline{1-2}\cline{4-4}
		$J_{63}$ & $\{ 6 \alpha_3\alpha_5^2, 3 \alpha_3^3\alpha_5 \}$ & & $f_3=5$, $f_5=3$ \\
		\hline
		\hline
		$aC$ & $\{ 12(\alpha_3\alpha_4)^2 \}$ & \multirow{3}{*}{\parbox[c]{6cm}{$\alpha_3=\cos^{-1} \tfrac{1}{3}$, \quad $\alpha_4=\pi-\alpha_3$, \\ $\alpha_6=2\pi-(\alpha_3+\alpha_4)$, \quad $x=\tfrac{1}{3}\pi$}} & $f_3=8$, $f_4=6$ \\
		\cline{1-2}\cline{4-4}
		$J_3$ & $\{ 6\alpha_3\alpha_4\alpha_6, 3(\alpha_3\alpha_4)^2 \}$ & & $f_3=4$, $f_4=3$, $f_6=1$ \\
		\cline{1-2}\cline{4-4}
		$J_{27}$ & $\{ 6\alpha_3^2\alpha_4^2, 6(\alpha_3\alpha_4)^2 \}$ & & $f_3=8$, $f_4=6$ \\
		\hline
		\hline
		$eC$ & $\{ 24\alpha_3\alpha_4^3 \}$ & \multirow{4}{*}{\parbox[c]{7.5cm}{$\alpha_3 = 2\pi - 3\alpha_4$, \quad $\alpha_4=2\tan^{-1} \textstyle\sqrt{7-4\sqrt{2}}$, \\ $\alpha_{10} = 2\pi - 2\alpha_4$, \quad $x=\cos^{-1}\tfrac{1}{17} (7 + 4 \sqrt{2})$}}  & $f_3=8$, $f_4=18$ \\
		\cline{1-2}\cline{4-4}
		$J_4$ & $\{ 8\alpha_3\alpha_4\alpha_8, 4\alpha_3\alpha_4^3 \}$ & & $f_3=4$, $f_4=5$, $f_{8}=1$ \\
		\cline{1-2}\cline{4-4}
		$J_{19}$ & $\{ 12\alpha_3\alpha_4^3, 8\alpha_4^2\alpha_8 \}$ & & $f_3=4$, $f_4=13$, $f_{10}=1$   \\
		\cline{1-2}\cline{4-4}
		$J_{37}$ & $\{ 24 \alpha_3\alpha_4^3 \}$ & & $f_3=8$, $f_4=18$  \\
		\hline
		\hline
		$aD$ & $\{ 30(\alpha_3\alpha_5)^2 \}$ & \multirow{3}{*}{\parbox[c]{5.5cm}{$\alpha_3 =  \cos^{-1} \tfrac{1}{\sqrt{5}}$, \quad $\alpha_5 = \pi - \alpha_3$, \\ $\alpha_{10} =  \pi$, \quad $x=\tfrac{1}{5}\pi$}} & $f_3 = 20$, $f_5=12$ \\
		\cline{1-2}\cline{4-4}
		$J_6$ & $\{ 10(\alpha_3\alpha_5)^2, 10\alpha_3\alpha_5\alpha_{10} \}$ & & $f_3=10$, $f_5=6$, $f_{10}=1$ \\
		\cline{1-2}\cline{4-4}
		$J_{34}$ & $\{ 20(\alpha_3\alpha_5)^2, 10\alpha_3^2\alpha_5^2 \}$ & & $f_3=20$, $f_5=12$  \\
		\hline
	\end{tabular}
\end{adjustbox}
\egroup
\end{center}
\caption{Grouped tilings}
\label{tab:OIaCeCaD}
\end{table}


\begin{thebibliography}{99}
		\addcontentsline{toc}{section}{References}

		\bibitem{aehj}
		C.~Adams, C.~Edgar, P.~Hollander and L.~Jacoby,
		\newblock The non-edge-to-edge tilings of the sphere by regular polygons.
		\newblock {\em Discrete \& Computational Geometry}, 72:1029--1085, 2024.

		\bibitem{ahs}
		Y.~Akama, B.~Hua and Y.~Su,
		\newblock Areas of spherical polyhedral surfaces with regular faces. 
		\newblock {\em preprint}, {\tt arXiv:1804.11033v1}, 2018.

		\bibitem{bcg}
		H.~Burgiel, J.~H.~Conway and C.~Goodman-Strauss,
		\newblock The Symmetries of Things,
		\newblock {\em CRC Press} 2008.

		\bibitem{cly}
		H.~M.~Cheung, H.~P.~Luk and M.~Yan,
		\newblock Tilings of the sphere by congruent quadrilaterals or triangles,
		\newblock {\em preprint}, arXiv:2204.02736, 2022.

		\bibitem{cc}
		G. Chiarotti and P. Chiaradia
		\newblock Condensed Matter 
		\newblock in: {\em Physics of Solid Surfaces}
Volume 45B Springer (2018).  



		\bibitem{gtx}
		Y.~Gong, Y.~Tao, N.~Xu, C.~Sun, X.~Wang and Z.~Su,
		\newblock Two polyoxovanadate-based metal–organic polyhedra with undiscovered “near-miss Johnson solid” geometry, 
		\newblock {\em Chemical communications}, 72,2019.

		\bibitem{G08}
		B.~Gr\"unbaum,
		\newblock An enduring error,
		\newblock {\em Elemente der Mathematik}, 64 (2009), 3:89--101.

		\bibitem{gj}
		B.~Gr\"unbaum, N.~W.~Johnson,
		\newblock The faces of a regular-faced polyhedron,
		\newblock {\em Journal of the London Mathematical Society} 1
		(1965) 1:577--586.

		\bibitem{J66}
		N.~Johnson, Convex solids with regular faces, Canadian Journal of Mathematics 18 (1966), 169–200.

		\bibitem{ns}
		R.~Nedela, M.~\v{S}koviera,
		\newblock Maps,
		\newblock in: {\it Handbook of Graph Theory}, 2nd. Ed., CRC Press, 2014, 826--827.

		\bibitem{DS} D.~M.~Y.~Sommerville, The relations connecting the angle sums and volume of a polytope in space of n dimensions. Proceedings of the Royal Society Series A (1927), 115:103–19

		\bibitem{DS2}
		D.~M.~Y.~Somerville,
		\newblock Semi-regular networks of the plane in absolute geometry,
		\newblock {\em Earth and Environmental Science Transactions of the
		Royal Society of Edinburgh} 41 (3) 725--745 (1906).

		\bibitem{wsw}
		Z.~Wang, H.~Su, X.~Wang, Q.~Zhao, C.~Tung, D.~Sun, and L.~Zheng,
		\newblock Johnson Solids: Anion-Templated Silver Thiolate Clusters
		Cappedby Sulfonate, 
		\newblock {\em Chem. Eur. J.} 24, 1640--1650 (2018).

		\bibitem{wjb}
		T.~Wu, Z.~Jiang, Q.~Bai, \dots, M.~Wang, X.~Li, P.~Wang,
		\newblock Supramolecular triangular orthobicupola: Selfassembly of
		a giant Johnson solid $J_27$. 
		\newblock {\em Chem} 7, 2429–2441, (2021).

		\bibitem{ZaR} Zalgaller, V.A., Convex polyhedra with regular faces. Zap. Naucn. Sem. Leningrad. Otdel. Mat. Inst. Steklov. (LOMI) 2, 220 (1967).

		\bibitem{ZaE} Zalgaller, V.A., Convex polyhedra with regular faces. Translated from Russian. Seminars in Mathematics, V. A. Steklov Mathematical Institute, Leningrad, Vol. 2. Consultants
		Bureau, New York (1969)

\end{thebibliography}
\end{document}